\documentclass[11pt,reqno]{amsart}

\pdfoutput=1

\usepackage{LL2_macros}

\begin{document}

\title[Asymptotic stability of solitary waves for the 1D focusing cubic NLS]{Asymptotic stability of solitary waves for the 1D focusing cubic Schr\"odinger equation under even perturbations}

\begin{abstract}
We establish the full asymptotic stability of solitary waves for the focusing cubic Schr\"odinger equation on the line under small even perturbations in weighted Sobolev norms.
The strategy of our proof combines a space-time resonances approach based on the distorted Fourier transform to capture modified scattering effects with modulation techniques to take into account the symmetries of the problem, namely the invariance under scaling and phase shifts.
A major challenge is the slow local decay of the radiation term caused by the threshold resonances of the non-selfadjoint linearized matrix Schr\"odinger operator around the solitary waves.
Our analysis hinges on two remarkable null structures that we uncover in the quadratic nonlinearities of the evolution equation for the radiation term as well as of the modulation equations. 
\end{abstract}

\author[Y. Li]{Yongming Li}
\address{Department of Mathematics \\ Texas A\&M University \\ College Station, TX 77843, USA}
\email{liyo0008@tamu.edu}

\author[J. L\"uhrmann]{Jonas L\"uhrmann}
\address{Department of Mathematics \\ Texas A\&M University \\ College Station, TX 77843, USA}
\email{luhrmann@tamu.edu}

\thanks{
The authors were partially supported by NSF grant DMS-1954707 and by NSF CAREER grant DMS-2235233.
}

\maketitle 

\tableofcontents

\section{Introduction} \label{sec:intro}

\subsection{The focusing cubic Schr\"odinger equation on the line}

We consider the focusing cubic Schr\"odinger equation in one space dimension
\begin{equation} \label{equ:cubic_NLS}
    \left\{ \begin{aligned}
        i\pt \psi + \px^2 \psi + |\psi|^2 \psi &= 0, \quad (t,x) \in \bbR \times \bbR, \\
        \psi(0) &= \psi_0.
    \end{aligned} \right.
\end{equation}
The Cauchy problem for \eqref{equ:cubic_NLS} is globally well-posed in $L^2(\bbR)$, see e.g. \cite{Tsutsumi87}.
The focusing cubic Schr\"odinger equation \eqref{equ:cubic_NLS} is a prime example of a completely integrable nonlinear dispersive equation on the line in the sense that it admits a Lax pair formulation as well as B\"acklund transformations.

We recall that \eqref{equ:cubic_NLS} is invariant with respect to the scaling
\begin{equation*}
 \psi(t,x) \mapsto \lambda \psi(\lambda^2 t, \lambda x), \quad \lambda > 0.
\end{equation*}
Moreover, \eqref{equ:cubic_NLS} is invariant under Galilei boosts, phase shifts, and spatial translations, which means that if $\psi(t,x)$ is a solution to \eqref{equ:cubic_NLS}, then for any $p, \gamma, \sigma \in \bbR$ the function
\begin{equation*}
    e^{i p x} e^{-i p^2 t} e^{i \gamma} \psi(t, x - 2pt -\sigma)
\end{equation*}
is also a solution to \eqref{equ:cubic_NLS}.

Two important features of the dynamics of the focusing cubic Schr\"odinger equation on the line are
the modified scattering behavior of dispersive solutions to \eqref{equ:cubic_NLS} 
and the existence of solitary wave solutions to \eqref{equ:cubic_NLS}.

The modified scattering behavior is related to the weak dispersion in one space dimension. More specifically, solutions to the linear Schr\"odinger equation on the line only decay in $L^\infty(\bbR)$ at the slow rate $t^{-\frac12}$ as $t \to \infty$. The cubic nonlinearity is therefore long-range and the asymptotics of dispersive solutions to \eqref{equ:cubic_NLS} are bound to feature corrections with respect to the free flow.

An example of a solitary wave solution to \eqref{equ:cubic_NLS} is given by
\begin{equation} \label{equ:intro_basic_solitary_wave}
 e^{it} \phi(x), \quad \phi(x) := \sqrt{2} \sech(x).
\end{equation}
The invariance of \eqref{equ:cubic_NLS} under scaling, Galilei boosts, phase shifts, and spatial translations then generates the following 4-parameter family of solitary waves
\begin{equation} \label{equ:intro_family_4parameter}
    e^{i p x} e^{i (\omega - p^2) t} e^{i\gamma} \phi_\omega(x-2pt-\sigma), \quad (\omega, p, \gamma, \sigma) \in (0,\infty) \times \bbR^3,
\end{equation}
where we use the notation
\begin{equation*}
 \phi_\omega(x) := \sqrt{\omega} \phi(\sqrt{\omega} x) = \sqrt{2\omega} \sech(\sqrt{\omega}x)
\end{equation*}
for the unique positive even ground state solution to the equation
\begin{equation} \label{equ:intro_phi_omega_equation}
    (-\px^2 + \omega) \phi_\omega = \phi_\omega^3, \quad x \in \bbR.
\end{equation}

A fundamental question concerning the family of solitary waves \eqref{equ:intro_family_4parameter} is their stability under the flow of \eqref{equ:cubic_NLS} with respect to small perturbations.
Orbital stability in $H^1(\bbR)$ follows from the classical works \cite{82CazenaveLions, 85ShatahStrauss, Weinstein85, Weinstein86, 87GrillakisShatahStrauss}. Exploiting the complete integrability of \eqref{equ:cubic_NLS}, one can in fact show that orbital stability holds with respect to weaker topologies, see for instance \cite{MizumachiPelinovsky12, KochTataru20}.
The asymptotic stability of the family of solitary waves \eqref{equ:intro_family_4parameter} with respect to perturbations that are small in weighted Sobolev norms was established in \cite{CuccPeli14, Saalmann17, BorghJenMcL18} using complete integrability techniques.
For perturbations in the energy space, asymptotic stability of \eqref{equ:intro_family_4parameter} cannot hold since the family of 2-solitons constructed by the inverse scattering transform in \cite{Olmedilla87, ZakhShab72} provides counter-examples, as pointed out in \cite{Martel22}.

\subsection{Main result}

In this work we study the asymptotic stability problem for the family of solitary waves \eqref{equ:intro_family_4parameter} using perturbative techniques that do not rely on the complete integrability of the focusing cubic Schr\"odinger equation on the line.
Specifically, we combine modulation techniques with a space-time resonances approach based on the distorted Fourier transform.
For technical reasons we restrict to even solutions to \eqref{equ:cubic_NLS}.
Our main result establishes the asymptotic stability of the 2-parameter family of (even) solitary waves
\begin{equation} \label{equ:intro_family_2parameter}
 e^{i \omega t} e^{i \gamma} \phi_\omega(x), \quad (\omega, \gamma) \in (0,\infty) \times \bbR,
\end{equation}
with respect to even perturbations that are small in a weighted Sobolev norm.

A common difficulty of asymptotic stability problems for solitary wave solutions to nonlinear Schr\"odinger equations on the line with low power nonlinearities is the weak dispersion in one space dimension.
In the case of the focusing cubic Schr\"odinger equation \eqref{equ:cubic_NLS}, an additional significant challenge is the slow local decay caused by the threshold resonances of the linearized operator around the solitary waves \eqref{equ:intro_family_2parameter}.
We uncover that in the case of the focusing cubic Schr\"odinger equation the worst effects of the threshold resonances are suppressed by two remarkable null structures of the quadratic nonlinearities.
Beyond the intrinsic interest in this problem, we also view it as a valuable and instructive step to further develop perturbative techniques to approach related asymptotic stability problems for solitary waves of (non-integrable) Schr\"odinger models on the line with threshold resonances or internal modes, where much remains to be understood in the absence of favorable structures.

The following is the main result of this work.

\begin{theorem} \label{thm:main_theorem}
For any $\omega_0 \in (0,\infty)$ there exist constants $0 < \varepsilon_0 \ll 1$, $0 < \delta \ll 1$, and $C_0 \geq 1$ with the following property:
Let $\gamma_0 \in \bbR$ and let $u_0 \in H^1_x(\bbR) \cap L^{2,1}_x(\bbR)$ be even with
\begin{equation} \label{equ:theorem_statement_smallness_initial_condition}
 \varepsilon := \| u_0 \|_{H^1_x(\bbR)} + \|\jx u_0\|_{L^2_x(\bbR)} \leq \varepsilon_0.
\end{equation}
Then the $H^1_x \cap L^{2,1}_x$--solution $\psi(t,x)$ to \eqref{equ:cubic_NLS} with initial condition
\begin{equation} \label{equ:theorem_statement_initial_condition}
 \psi_0(x) = e^{i \gamma_0} \bigl( \phi_{\omega_0}(x) + u_0(x) \bigr)
\end{equation}
exists globally in time and admits a decomposition
\begin{equation*}
 \psi(t,x) = e^{i \gamma(t)} \bigl( \phi_{\omega(t)}(x) + u(t,x) \bigr), \quad t \geq 0,
\end{equation*}
for some continuously differentiable paths $(\omega, \gamma) \colon [0,\infty) \to (0,\infty) \times \bbR$ so that 
\begin{equation} \label{equ:theorem_statement_decay_radiation}
 \|u(t)\|_{L^{\infty}_x(\bbR)} \leq C_0\varepsilon \jt^{-\frac12}, \quad t \geq 0.
\end{equation}
Moreover, there exists $\omega_\infty \in (0,\infty)$ such that
\begin{equation} \label{equ:theorem_statement_decay_modulation_parameters}
 |\omega(t)-\omega_\infty| \leq C_0 \varepsilon \jt^{-1+\delta}, \quad |\dot{\gamma}(t)-\omega_\infty| \leq C_0 \varepsilon \jt^{-1+\delta}, \quad t \geq 0.
\end{equation}
The asymptotics of the radiation term $u(t,x)$ exhibit logarithmic phase corrections with respect to the free flow:
there exist $W_-,W_+ \in L^\infty(\bbR)$ such that uniformly for all $x\in\bbR$ and all $t \geq 1$, 
\begin{equation} \label{eqn: theorem-u-asymptotics}
    \bigl|u(t,x) - u_\infty(t,x)\bigr| \leq C_0 \varepsilon t^{-\frac35+\delta},
\end{equation}
where 
\begin{equation} 
\begin{split}
u_\infty(t,x) &:= \frac{1}{\sqrt{2t}} e^{-it\omega_\infty}e^{i \frac{x^2}{4t}} e^{-i \frac{\pi}{4}}m_{1,\omega_\infty}\Big(x,\frac{x}{2t}\Big) W_+\Big(\frac{x}{2t}\Big)e^{-i\theta_\infty(t)}e^{-\frac{i}{2}\log(t)\vert W_+(\frac{x}{2t})\vert^2}\\
&\quad \, \, -\frac{1}{\sqrt{2t}} e^{it\omega_\infty}e^{-i \frac{x^2}{4t}} e^{i \frac{\pi}{4}}m_{2,\omega_\infty}\Big(x,-\frac{x}{2t}\Big)W_-\Big(-\frac{x}{2t}\Big)e^{i\theta_\infty(t)}e^{\frac{i}{2}\log(t)\vert W_-(-\frac{x}{2t})\vert^2}   
\end{split}
\end{equation}
and 
\begin{align*}
    \theta_\infty(t) &:= \int_0^t (\dot{\gamma}(s)-\omega_\infty) \,\ud s, \\
    m_{1,\omega_\infty}(x,\xi) &:= \frac{\bigl( \xi + i \sqrt{\omega_\infty} \tanh(\sqrt{\omega_\infty} x) \bigr)^2}{\bigl(|\xi|-i\sqrt{\omega_\infty}\bigr)^2}, \\
    m_{2,\omega_\infty}(x,\xi) &:= \frac{\bigl( \sqrt{\omega_\infty} \sech(\sqrt{\omega_\infty} x) \bigr)^2}{\bigl( |\xi|-i\sqrt{\omega_\infty} \bigr)^2}.
\end{align*}
Analogous statements hold for negative times $t \leq 0$.
\end{theorem}

\begin{remark}
 We compare the statement of Theorem~\ref{thm:main_theorem} with the asymptotic stability results obtained in \cite{BorghJenMcL18, CuccPeli14, Saalmann17} using complete integrability techniques. 
 While \cite{BorghJenMcL18, CuccPeli14, Saalmann17} do not impose parity restrictions, we consider even perturbations for technical reasons explained in Remark~\ref{rem:restriction_even_perturbations} below. As in \cite{BorghJenMcL18, CuccPeli14, Saalmann17} we prove that the radiation term decays at the rate \eqref{equ:theorem_statement_decay_radiation} of linear Schr\"odinger waves in one space dimension, 
 and as in \cite{BorghJenMcL18, Saalmann17} we obtain asymptotics \eqref{eqn: theorem-u-asymptotics} for the radiation term that uncover logarithmic phase corrections with respect to the corresponding free flow.
 We determine the scaling parameter $\omega_\infty$ for the final soliton, but unlike \cite{BorghJenMcL18, CuccPeli14, Saalmann17} a small uncertainty regarding the phase shift parameter of the final soliton remains due to the non-integrable decay rate \eqref{equ:theorem_statement_decay_modulation_parameters}. 
 Our assumptions on the initial perturbations are slightly stronger than in \cite{BorghJenMcL18, CuccPeli14, Saalmann17}.

 However, in contrast to \cite{BorghJenMcL18, CuccPeli14, Saalmann17} the strategy of our proof combining modulation techniques and a space-time resonances approach based on the distorted Fourier transform is applicable to many other asymptotic stability problems for solitary wave solutions to nonlinear Schr\"odinger equations.
 In particular, we view the analysis in this work for the focusing cubic Schr\"odinger equation as a valuable and instructive step towards studying asymptotic stability problems for Schr\"odinger models on the line with threshold resonances or internal modes, where much remains to be understood in the absence of favorable structures.
\end{remark}

\begin{remark} \label{rem:restriction_even_perturbations}
 The restriction to even perturbations in the statement of Theorem~\ref{thm:main_theorem} is for the following technical reason.
 If we considered arbitrary perturbations, we would also have to take into account the invariance of \eqref{equ:cubic_NLS} under Galilei boosts and under spatial translations.
 Correspondingly, we would have to track the evolution of two additional modulation parameters.
 We found that the modulation equation for the Galilei boost parameter $p(t)$ also exhibits a remarkable null structure, analogous to the crucial null structure in the modulation equation for the scaling parameter $\omega(t)$, see Lemma~\ref{lem:null_structure_modulation} and Section~\ref{sec:modulation_parameters}.
 Thus, we could seek to extend the strategy of the proof of Theorem~\ref{thm:main_theorem} to this more general setting.
 However, related to the fact that we do not obtain integrable decay rates in \eqref{equ:theorem_statement_decay_modulation_parameters} in the statement of Theorem~\ref{thm:main_theorem}, we would be facing a small uncertainty in the localization of the center of the modulated solitary wave.
 This would lead to small additional losses in the weighted energy estimates for the renormalized quadratic nonlinearities in Section~\ref{sec:energy_estimates}, and we would not be able to close the bootstrap estimates in the same manner as in Section~\ref{sec:energy_estimates}.
 Overcoming the additional challenges caused by the uncertainty in the localization of the center of the modulated solitary wave necessitates further new techniques beyond the scope of this paper.
\end{remark}

\subsection{References}

The literature on the asymptotic stability of solitary wave solutions to nonlinear Schr\"odinger equations and related models is vast.
Recently, this topic has received significant attention especially in one space dimension, where the combination of weak dispersion, low power nonlinearities, and intriguing spectral features leads to a rich class of dynamics.

For problems in one space dimension, it has become customary to distinguish two notions of asymptotic stability, which complement each other.
\emph{Local} asymptotic stability refers to convergence of the perturbed solution locally in space to a final soliton, usually under weak assumptions such as small finite energy perturbations. \emph{Full} (or detailed) asymptotic stability includes decay rates and asymptotics for the convergence to the final soliton at the expense of stronger assumptions on the initial perturbations such as smallness with respect to weighted Sobolev norms.

The analysis in this paper combines a circle of ideas and techniques related to \emph{modified scattering}, \emph{modulation}, and \emph{distorted Fourier theory}.
In what follows, we point the reader to the most relevant related articles, mostly focusing on problems in one space dimension. 
We refer to the surveys \cite{CuccMaeda_SurveyII, KMM17_Survey, Schl07, Tao_Survey} for a more exhaustive list of references, especially for results in higher space dimensions and for related models.

The modified scattering of small solutions to cubic Schr\"odinger equations on the line with or without potential has been studied in 
\cite{HN98, LS06, KatPus11, IT15, GermPusRou18, ChenPus19, ChenPus22, Del16, Naum16, Naum18, MasMurphSeg19, NaumWed22}.
For related results on the long-time behavior of small solutions to quadratic or cubic Klein-Gordon equations in one space dimension we refer to
\cite{Del01, LS05_1, LS05_2, HN08, HN10, HN12, Del16_KG, Stingo18, CL18, Sterb16, LS15, LLS1, LLS2, LLSS, GP20}.
Full asymptotic stability results for large solitary wave solutions to Schr\"odinger equations on the line have been obtained in
\cite{BusPerel92, BusSul03, KS06, CollotGermain23, MasakiMurphySegata},
while local asymptotic stability results include 
\cite{Martel22, Martel23, Rialland23, Rialland24, CuccMaeda2404, CuccMaeda2405}.
For small solitary wave solutions to Schr\"odinger equations in one space dimension we refer to \cite{Mizumachi08, Chen21} for full asymptotic stability results and to \cite{CuccMaeda19, CuccMaeda22} for local asymptotic stability results.
Closely related are the following full asymptotic stability results \cite{KK11_1, KK11_2, GP20, LS1, GermPusZhang22, KairzhanPusateri22, LS2} and the following local asymptotic stability results \cite{KMM17, KMM19, KMMV20, KM22, LL1, CuccMaeda_kink_23, CuccMaedaMurgScrob23, PalaciosPusateri24} for solitons in scalar field theories on the line.

References for the spectral and distorted Fourier theory as well as for linear dispersive estimates for scalar and matrix Schr\"odinger operators in one space dimension include \cite{00Weder, 04GoldbergSchlag, 07Goldberg, CGNT08, KS06, Li23, BusPerel92}. 

As pointed out at the beginning of the introduction, using complete integrability techniques the full asymptotic stability of the family of solitary wave solutions \eqref{equ:intro_family_4parameter} to the focusing cubic Schr\"odinger equation was established in \cite{CuccPeli14, Saalmann17, BorghJenMcL18}.

\smallskip 

\noindent {\it Closely related articles.} 
We conclude our literature overview with a brief discussion of the articles that are very closely related to our work.
Co-dimensional asymptotic stability for solitary wave solutions to super-critical focusing Schr\"odinger equations $i\pt \psi + \px^2 \psi + |\psi|^{p-1} \psi = 0$, $p > 5$, on the line was established in \cite{KS06}. Related earlier results were obtained in \cite{BusPerel92} under tailored assumptions on the nonlinearity and the spectrum of the linearized operator.
The analysis in \cite{KS06} combines modulation to take into account the symmetries of the equation with a dispersive setup largely based on Strichartz estimates, which suffices for the higher power nonlinearities considered in \cite{KS06}.
Moreover, \cite{KS06} systematically develops the spectral and distorted Fourier theory as well as linear dispersive decay estimates for matrix linearized Schr\"odinger operators around solitary waves with no threshold resonance (\emph{generic case}).
More recently, \cite{CollotGermain23} established full asymptotic stability of solitary waves for general nonlinear Schr\"odinger equations $i\pt \psi + \px^2 \psi + F'(|\psi|^2) \psi = 0$ allowing for cubic nonlinearities under the key assumption that the matrix linearized Schr\"odinger operator exhibits neither non-zero eigenvalues nor threshold resonances (generic case). 
To deal with modified scattering dynamics related to the long-range nature of the cubic nonlinearity, the analysis in \cite{CollotGermain23} combines modulation with a space-time resonances approach based on the distorted Fourier transform.
Previously, a similar strategy had been developed in \cite{Chen21} in the related context of proving the full asymptotic stability of small solitary wave solutions to cubic Schr\"odinger equations on the line with a trapping potential.

The analysis in this work further develops the framework from \cite{CollotGermain23} to prove full asymptotic stability of solitary waves of the focusing cubic Schr\"odinger equation on the line, where the linearized matrix operator additionally exhibits threshold resonances (\emph{non-generic case}). 
We also need to set up the distorted Fourier theory for this linearized operator building on the generic case treated in \cite{KS06}, 
and we establish linear dispersive decay estimates building on results from \cite{Li23}.
Moreover, we use a variable coefficient normal form reminiscent of the one employed in the proof of full asymptotic stability of the sine-Gordon kink under odd perturbations in \cite{LS1}, and we benefit from insights from the study of the long-time behavior of small solutions to nonlinear Klein-Gordon equations in one space dimension in the presence of a threshold resonance in the series of articles \cite{LLS1, LLS2, LLSS, LS1, LS2}.

\subsection{Overview of the proof}

In this subsection we provide an overview of the main ideas entering the proof of Theorem~\ref{thm:main_theorem}. 

\medskip 
\noindent {\it Linearized Operator.} 
Writing the focusing cubic Schr\"odinger equation \eqref{equ:cubic_NLS} as a system for the vector $(\psi, \bar{\psi})$ and linearizing around a solitary wave $e^{i\omega t} e^{i\gamma} \phi_\omega(x)$ for some fixed $\omega \in (0,\infty)$, $\gamma \in \bbR$, we arrive at the non-selfadjoint matrix Schr\"odinger operator
\begin{equation} \label{equ:intro_calH_definition}
    \begin{aligned}
        \calH(\omega) &:= \begin{bmatrix}
                    -\partial_x^2 + \omega & 0 \\ 0 & \partial_x^2 - \omega
        \end{bmatrix}
        +
        \begin{bmatrix}
                    -2\phi_\omega^2	 & - \phi_\omega^2 \\ \phi_\omega^2 & 2\phi_\omega^2
                \end{bmatrix}.
    \end{aligned}
\end{equation}
This operator is closed on $H^2_x(\bbR) \times H^2_x(\bbR)$ and its spectrum is well-known: 
$\calH(\omega)$ has essential spectrum $(-\infty, -\omega] \cup [\omega, \infty)$ with no embedded eigenvalues. 
The only eigenvalue is $0$ with algebraic multiplicity equal to $4$ and geometric multiplicity equal to $2$. Its generalized eigenfunctions are generated by the symmetries, i.e., the invariance of \eqref{equ:cubic_NLS} under phase shifts, scaling, spatial translations, and Galilei boosts. 
They are explicitly given by
\begin{equation*}
	Y_{1,\omega} := \begin{bmatrix}
		i\phi_\omega \\ -i\phi_\omega
		\end{bmatrix}, \quad Y_{2,\omega} := \begin{bmatrix}
				\partial_\omega \phi_\omega \\ \partial_\omega \phi_\omega
			\end{bmatrix}, \quad Y_{3,\omega} := \begin{bmatrix}
				\partial_x \phi_\omega \\ \partial_x \phi_\omega
			\end{bmatrix}, \quad Y_{4,\omega} := \begin{bmatrix}
				ix\phi_\omega \\ -ix\phi_\omega
			\end{bmatrix},
\end{equation*}
and it holds that
\begin{equation*} 
	\calH(\omega) Y_{1,\omega} = 0, \quad \calH(\omega) Y_{2,\omega} = i Y_{1,\omega}, \quad \calH(\omega) Y_{2,\omega} = 0, \quad \calH(\omega) Y_{4,\omega} = -2i Y_{3,\omega}.
\end{equation*}
The vector-valued functions $Y_{1,\omega}, Y_{2,\omega}$ are even, while $Y_{3,\omega}, Y_{4,\omega}$ are odd.
Since we only study even solutions to \eqref{equ:cubic_NLS}, we consider $\calH(\omega)$ restricted to the subspace of even functions. The eigenvalue~$0$ then has algebraic multiplicity $2$ and geometric multiplicity $1$.
Moreover, $\calH(\omega)$ exhibits an even threshold resonance at the edge of each part of the essential spectrum.
Indeed, the functions 
\begin{equation} \label{equ:intro_def_threshold_resonances}
	\Phi_{+,\omega}(x) := \begin{bmatrix} \tanh^2(\sqrt{\omega}x) \\ -\sech^2(\sqrt{\omega}x) \end{bmatrix}, \quad \Phi_{-,\omega}(x) := \sigma_1 \Phi_{+,\omega}(x) = \begin{bmatrix} -\sech^2(\sqrt{\omega}x) \\	\tanh^2(\sqrt{\omega}x) \end{bmatrix}
\end{equation}
belong to $L^\infty_x(\bbR) \times L^\infty_x(\bbR) \setminus L^2_x(\bbR) \times L^2_x(\bbR)$ and satisfy
\begin{equation*}
	\calH(\omega) \Phi_{+,\omega} = \omega \Phi_{+,\omega}, \quad \calH(\omega) \Phi_{-,\omega} = -\omega \Phi_{-,\omega}.
\end{equation*}

In Section~\ref{sec:spectral_and_dist_FT} we recall the spectral theory for $\calH(\omega)$ and we systematically develop the distorted Fourier theory for $\calH(\omega)$. 
We point out that the distorted Fourier theory established in \cite{BusPerel92, KS06} for matrix Schr\"odinger operators of the form~\eqref{equ:intro_calH_definition} in the generic case (no threshold resonances) does not directly apply to $\calH(\omega)$.
Analogously to \cite{BusPerel92, KS06}, our construction of the distorted Fourier transform associated with $\calH(\omega)$ is based on an explicit computation of the jump of the resolvent across the essential spectrum and on a representation of the linear evolution $e^{it\calH(\omega)}$ in terms of a contour integral of the resolvent (analogue of Stone's theorem in the non-selfadjoint case).
In Section~\ref{sec:linear_decay} we collect linear decay estimates for the linear evolution associated with the operator~$\calH(\omega)$. 

\begin{figure}
\centering
\begin{tikzpicture}[domain=-1.74:1.74,samples=50, scale=1.3]

\coordinate[label=below:{\footnotesize $-\omega$}] (pm) at ($(180:1)$);
\coordinate[label=below:{\footnotesize $+\omega$}] (pp) at ($(0:1)$);
\coordinate (px) at ($(0:0)$);		

\draw[->] ($(180:2)$) --++ ($(0:4)$) node[right] {$\Re$}; 
\draw[->] ($(270:2)$) --++ ($(90:4)$)node[above] {$\Im$};

\fill[blue] (px) circle (2pt);		
\draw[line width=5pt,color=orange,draw opacity=0.35] (pp) --++ ($(0:1)$);
\draw[line width=5pt,color=orange,draw opacity=0.35] (pm) --++ ($(0:-1)$);
\fill[red] (pm) circle (2pt);
\fill[red] (pp) circle (2pt);


\end{tikzpicture}
\caption{Spectrum of the operator $\calH(\omega)$ in \eqref{equ:intro_calH_definition}.}
\label{fig:enter-label}
\end{figure}

\medskip 

\noindent {\it Modulation.}
Our starting point for the analysis of the asymptotic stability of the family of solitary waves \eqref{equ:intro_family_2parameter} is to use standard modulation techniques to decompose the perturbed even solution $\psi(t,x)$ to \eqref{equ:cubic_NLS} into a modulated solitary wave and a radiation term that is orthogonal to directions related to the invariance of \eqref{equ:cubic_NLS} under scaling and phase shifts. Specifically, we write
\begin{equation*}
    \psi(t,x) = e^{i \gamma(t)} \bigl( \phi_{\omega(t)}(x) + u(t,x) \bigr) 
\end{equation*}
for continuously differentiable paths $\omega(t) \in (0,\infty)$ with $|\omega(t) - \omega_0| \lesssim \varepsilon$ and $\gamma(t) \in \bbR$ that are uniquely determined by the orthogonality conditions
\begin{equation} \label{equ:intro_orthogonality} 
    \langle U(t), \sigma_2 Y_{1, \omega(t)} \rangle = \langle U(t), \sigma_2 Y_{2, \omega(t)} \rangle = 0, \quad U(t) := \begin{bmatrix} u(t) \\ \baru(t) \end{bmatrix}, \quad t \geq 0.
\end{equation}
The resulting evolution equation for the radiation term $U(t)$ reads
\begin{equation} \label{equ:intro_evol_equ_U}
    i\partial_t U - \calH(\omega)U = (\dot{\gamma} - \omega) \sigma_3 U + \calMod + \calQ_\omega(U) + \calC(U).
\end{equation}
It is coupled to the following system of first-order differential equations for the modulation parameters
\begin{equation} \label{equ:intro_modulation_equations}
            \mathbb{M} \begin{bmatrix}
                \dot{\gamma} - \omega \\
                 \dot{\omega}
            \end{bmatrix}
            =
            \begin{bmatrix}
                \bigl\langle i \bigl( \calQ_\omega(U) + \calC(U) \bigr), \sigma_2 Y_{1,\omega} \bigr\rangle \\
                \bigl\langle i \bigl( \calQ_\omega(U) + \calC(U) \bigr), \sigma_2 Y_{2, \omega} \bigr\rangle
            \end{bmatrix}.
\end{equation}
In the preceding identities, $\sigma_2, \sigma_3$ are the Pauli matrices defined in \eqref{eqn:pauli_matrices}, $\calMod$ denotes relatively harmless terms that are a consequence of the time-dependence of the modulation parameters, and the quadratic and cubic nonlinearities on the right-hand sides of \eqref{equ:intro_evol_equ_U} and \eqref{equ:intro_modulation_equations} are given by
\begin{equation} \label{equ:intro_definition_quad_cubic}
    \begin{aligned}
        \calQ_\omega(U) := \begin{bmatrix}
                    -\phi_\omega(u^2 + 2u \bar{u}) \\ \phi_\omega(\baru^2 + 2u \baru)
                \end{bmatrix},
        \quad
        \calC(U) := \begin{bmatrix}
                     -u \baru u \\ \baru u \baru
                    \end{bmatrix}.
    \end{aligned}
\end{equation}
Moreover, the matrix $\mathbb{M}$ on the left-hand side of \eqref{equ:intro_modulation_equations} is given by
        \begin{equation*}
            \mathbb{M} := \begin{bmatrix}
                0 & c_\omega \\
                c_\omega & 0
            \end{bmatrix}
            +
            \begin{bmatrix}
                \langle U, \sigma_1 Y_{1, \omega} \rangle &  \langle U,  \sigma_2 \partial_\omega Y_{1, \omega} \rangle \\
                \langle U, \sigma_1 Y_{2, \omega} \rangle & \langle U, \sigma_2 \partial_\omega Y_{2, \omega} \rangle
            \end{bmatrix},
            \quad 
            c_\omega := \frac{2}{\sqrt{\omega}} > 0,
        \end{equation*}
which is clearly invertible for sufficiently small perturbations $U(t)$.     

\medskip 
\noindent {\it Schematic Setup and Strategy.}
Now our primary task is to simultaneously establish decay and asymptotics for the radiation term $U(t)$ and to prove convergence of the modulation parameter $\omega(t)$ to an asymptotic value $\omega_\infty \in (0,\infty) $ as $t\to\infty$, in other words to establish sufficiently fast decay of $\omega(t) - \omega_\infty$ to zero. 
In view of the weak dispersion in one space dimension in conjunction with the low power nonlinearities on the right-hand sides of \eqref{equ:intro_evol_equ_U} and \eqref{equ:intro_modulation_equations}, this is a formidable challenge that gets compounded by the subtle, but severe effects of the threshold resonances on the dynamics.

In order to prove decay for the radiation term by perturbative means, we naturally want to invoke dispersive estimates for the linear evolution $i\pt U - \calH(\omega) U = 0$ associated with \eqref{equ:intro_evol_equ_U}. However, the modulation parameter $\omega = \omega(t)$ is time-dependent. 
We therefore need to pass to a \emph{time-independent} reference operator $\calH(\ulomega)$ on the left-hand side of \eqref{equ:intro_evol_equ_U} for a suitably chosen fixed scaling parameter $\ulomega$ and then treat the difference perturbatively.
Heuristically, we take $\ulomega$ to be the asymptotic limit of $\omega(t)$ as $t \to \infty$. 
The proof of Theorem~\ref{thm:main_theorem} proceeds via an elaborate bootstrap argument, so that on a given finite time interval $[0,T]$ we actually choose\footnote{This is a slight simplification for the purposes of the discussion in this introduction. See the succinct proof of Theorem~\ref{thm:main_theorem} based on the statements of the two core bootstrap Propositions \ref{prop:modulation_parameters} and \ref{prop:profile_bounds}.} $\ulomega = \omega(T)$ to be the final value of the continuously differentiable path $\omega(t)$ on that time interval.
As part of our bootstrap argument we seek to show the decay estimate
\begin{equation} \label{equ:intro_decay_for_omega}
    |\omega(t) - \ulomega| \lesssim \varepsilon \jt^{-1+\delta}, \quad 0 \leq t \leq T,
\end{equation}
where $0 < \delta \ll 1$ is a small absolute constant.
Later, we will then extract the final value $\omega_\infty$ of the path $\omega(t)$ as $t \to \infty$ from \eqref{equ:intro_decay_for_omega} with $T = T_n$ for an increasing sequence of times $T_n \nearrow \infty$.
The decay rate~\eqref{equ:intro_decay_for_omega} turns out to be just about enough to treat perturbatively the error terms that arise from passing to a time-independent reference operator.
We point out that since $\omega(t)-\ulomega = -\int_t^T \dot{\omega}(s) \, \ud s$ by the fundamental theorem of calculus, an almost twice integrable in time decay rate $s^{-2+\delta}$ for $\dot{\omega}(s)$ as $s \to \infty$ would be necessary to crudely infer the decay estimate \eqref{equ:intro_decay_for_omega}.

For the discussion of the main aspects of the nonlinear analysis, in what follows we only consider schematic versions of \eqref{equ:intro_evol_equ_U} and \eqref{equ:intro_modulation_equations} that retain the most relevant terms.
Specifically, in the evolution equation \eqref{equ:intro_evol_equ_U} for the radiation term, we disregard milder modulation terms and we pass to the reference operator $\calH(\ulomega)$, which produces the delicate difference term $(\omega - \ulomega) \sigma_3 U$ on the right-hand side. Moreover, in the quadratic nonlinearity we replace $\phi_\omega(x)$ by $\phi_\ulomega(x)$ and ignore the resulting milder error terms.
In the modulation equations \eqref{equ:intro_modulation_equations}, we focus on the equation for~$\dot{\omega}$.
To unveil its leading order structure, we only keep the quadratic nonlinearity on the right-hand side and we again replace the time-dependent parameter $\omega$ by $\ulomega$, disregarding the resulting milder error terms.
Then we are facing the following schematic evolution equation for the radiation term
\begin{equation} \label{equ:intro_evol_equ_U_essence}
    i\partial_t U - \calH(\ulomega)U = (\omega - \ulomega) \sigma_3 U + \calQ_\ulomega(U) + \calC(U) + \ldots 
\end{equation}
coupled to a differential equation for the modulation parameter $\omega(t)$ of the schematic form
\begin{equation} \label{equ:intro_dot_omega_essence}
    \dot{\omega} = i c_\ulomega^{-1} \langle \calQ_\ulomega(U), \sigma_2 Y_{1,\ulomega} \rangle + \ldots.
\end{equation}

Our overall strategy is to take a space-time resonances approach based on the distorted Fourier transform associated with the reference operator $\calH(\ulomega)$ to infer decay and asymptotics for $U(t)$, while we exploit the oscillations in the equation for $\dot{\omega}$ to arrive at the decay estimate \eqref{equ:intro_decay_for_omega} for $\omega(t) - \ulomega$.
As we explain next, in comparison with the generic case (no threshold resonances) considered in \cite{CollotGermain23}, this strategy is significantly compounded by the slow local decay of the radiation term caused by the threshold resonances of the linearized operator $\calH(\ulomega)$. 

\medskip 
\noindent {\it Main Enemy: Slow Local Decay.}
Since the quadratic nonlinearity defined in \eqref{equ:intro_definition_quad_cubic} is spatially localized due to the rapid decay of the coefficient $\phi_\ulomega(x)$, the leading order behavior of $\calQ_\ulomega(U(t))$ is entirely governed by the \emph{local decay} of the radiation term $U(t)$. 
In view of \eqref{equ:intro_dot_omega_essence}, the same observation holds for the leading order behavior of $\dot{\omega}(t)$. 

In the absence of threshold resonances (generic case), the Schr\"odinger waves enjoy improved local decay, which simplifies the treatment of the terms $(\omega-\ulomega) \sigma_3 U$ and $\calQ_\ulomega(U)$ on the right-hand side of \eqref{equ:intro_evol_equ_U_essence} and the derivation of sufficient decay for $\omega(t) - \ulomega$.
However, in our setting the threshold resonances of $\calH(\ulomega)$ lead to slow local decay.
In fact, \cite[Theorem 1.1]{Li23} by the first author quantifies how the slow local decay can be pinned down exactly to the contributions of the threshold resonances, while the bulk of the Schr\"odinger waves still enjoy improved local decay.
Specifically, for the linear evolution of a (vector-valued) Schwartz function~$F$ we have for all $t \geq 1$ that
\begin{equation*}
    \biggl\| \jx^{-2} \biggl( e^{-it\calH(\ulomega)} \ulPe F - \frac{e^{-it\ulomega}}{t^{\frac12}} \Phi_{+,\ulomega}(x) c_+(F) - \frac{e^{it\ulomega}}{t^{\frac12}} \Phi_{-,\ulomega}(x) c_-(F) \biggr) \biggr\|_{L^\infty_x \times L^\infty_x} \lesssim \frac{1}{t^{\frac32}} \bigl\| \jx^2 F \bigr\|_{L^1_x \times L^1_x},
\end{equation*}
where $\ulPe$ denotes the Riesz projection to the essential spectrum relative to $\calH(\ulomega)$, $\Phi_{\pm,\ulomega}(x)$ are the threshold resonances defined in \eqref{equ:intro_def_threshold_resonances}, and $c_{\pm, \ulomega}(F)$ should be thought of as projections of $F$ onto these 
\begin{equation*}
    c_{+,\ulomega}(F) := \frac{e^{-i\frac{\pi}{4}}}{\sqrt{4\pi}}\langle F,\sigma_3 \Phi_{+,\ulomega}\rangle, \quad c_{-, \ulomega}(F) := -\frac{e^{i\frac{\pi}{4}}}{\sqrt{4\pi}}\langle F,\sigma_3 \Phi_{-,\ulomega}\rangle.  
\end{equation*}
Thus, the leading order behavior of $\calQ_\ulomega(U(t))$ can be thought of as a linear combination of slowly decaying quadratic source terms of the schematic form 
\begin{equation} \label{equ:intro_source_terms_schematic}
    \frac{1}{t} \phi_\ulomega(x) \bigl( e^{\mp i t \ulomega} \Phi_{\pm, \ulomega}(x) \bigr)^2 \varepsilon^2,
\end{equation}
see the expansion \eqref{equ:setup_expansion_quadratic_nonlinearity} for the precise expression.
These source terms create resonances both in the evolution equation \eqref{equ:intro_evol_equ_U_essence} for the radiation term and in the differential equation \eqref{equ:intro_dot_omega_essence} for $\dot{\omega}$. For more detailed explanations we refer to the discussions after \eqref{equ:wtilcalF_of_Q1}, respectively after \eqref{equ:modulation_quadratic_leading_order_expansion}.
In particular, the slow $t^{-1}$ decay of the schematic source terms \eqref{equ:intro_source_terms_schematic} is far away from the twice integrability in time of $\dot{\omega}(t)$ pointed out earlier, which would be necessary to crudely infer the decay estimate \eqref{equ:intro_decay_for_omega}.

\medskip 
\noindent {\it Null Structures.} 
It turns out that we do not have to address the worst effects of the resonant quadratic source terms \eqref{equ:intro_source_terms_schematic} in our analysis thanks to two remarkable null structures in the quadratic nonlinearities of the evolution equation for the radiation term as well as of the modulation equations.

We refer to Lemma~\ref{lem:null_structure_radiation} and the discussion in Subsection~\ref{subsec:normal_form} for a quadratic null structure in the evolution equation for the radiation term. It was already observed by the first author in \cite[Lemma 1.6]{Li23} and it is reminiscent of a similar null structure for perturbations of the sine-Gordon kink, see \cite{LS1, LLSS}. 
Using a variable coefficient normal form inspired by \cite{LLS2, LS1}, we can then recast the quadratic nonlinearity $\calQ_\ulomega(U)$ on the right-hand side of \eqref{equ:intro_evol_equ_U_essence} into a spatially localized cubic term, which is slightly better behaved than the constant coefficient cubic term $\calC(U)$. 

For the details of the null structure in the modulation equations we refer to Lemma~\ref{lem:null_structure_modulation} and the discussion in Section~\ref{sec:modulation_parameters}, specifically the treatment of the term $I(s)$ in \eqref{equ:modulation_proof_expansion_of_RHS}.
Thanks to this remarkable null structure, in the proof of the decay estimate \eqref{equ:intro_decay_for_omega} we can integrate by parts in the time integral of the leading order quadratic contribution
\begin{equation*} 
    \omega(t) - \ulomega = - \int_t^T \dot{\omega}(s) \, \ud s = -i c_\ulomega^{-1} \int_t^T \bigl\langle \calQ_\ulomega\bigl(U(s)\bigr), \sigma_2 Y_{1,\ulomega} \bigr\rangle \, \ud s + \ldots,
\end{equation*} 
which effectively turns it into the contribution of a cubic term 
\begin{equation*}
    \int_t^T \bigl\langle \calC\bigl(U(s)\bigr), \sigma_2 Y_{1,\ulomega} \bigr\rangle \, \ud s.
\end{equation*}
We show that the latter does not exhibit any time resonances, whence we can integrate by parts in time once more and turn it into the time integral of a quartic term, which enjoys the desired crude twice integrability to conclude \eqref{equ:intro_decay_for_omega}. 

\medskip 
\noindent {\it Modified Scattering Analysis.}
Finally, we discuss the analysis of the modified scattering behavior of the radiation term.
As explained above, we can use a variable coefficient normal form to recast the quadratic nonlinearity $\calQ_\ulomega(U)$ on the right-hand side of \eqref{equ:intro_evol_equ_U_essence} into a better behaved spatially localized cubic term.
Moreover, thanks to the sufficiently fast decay \eqref{equ:intro_decay_for_omega} of $\omega(t) - \ulomega$, we can use an integrating factor argument inspired by \cite[Proposition 9.5]{CollotGermain23} to effectively remove the term $(\omega-\ulomega) \sigma_3 U$ from \eqref{equ:intro_evol_equ_U_essence}, see \eqref{equ:setup_definition_theta} and \eqref{equ:setup_evol_equ_ftilplus1}.
Then we are essentially reduced to studying the asymptotic behavior of small solutions to the cubic matrix Schr\"odinger equation
\begin{equation} \label{equ:intro_evol_equ_U_essence_cubic}
    i\partial_t U - \calH(\ulomega)U = \calC(U).
\end{equation}
In analogy to the study of the long-time dynamics of small solutions to the scalar cubic Schr\"odinger equation with (or without) potential in one space dimension
\begin{equation} \label{equ:intro_scalar_cubic}
    i\pt w - H w = |w|^2 w, \quad H := -\px^2 + V(x),
\end{equation}
we expect logarithmic phase corrections with respect to the free flow.
As in the generic case considered in \cite{CollotGermain23}, we proceed in the spirit of the space-time resonances method for the scalar equation \eqref{equ:intro_scalar_cubic}, see \cite{KatPus11, GermPusRou18, ChenPus19, ChenPus22}, and we introduce the profile of the radiation term $U(t)$ relative to the reference operator $\calH(\ulomega)$ by filtering out the linear evolution
\begin{equation}
    F_\ulomega(t) := e^{it\calH(\ulomega)} \ulPe U(t),
\end{equation}
where we recall that $\ulPe$ denotes the projection to the essential spectrum for $\calH(\ulomega)$.

Then the idea is to consider the evolution equation for the (vector-valued) distorted Fourier transform relative to $\calH(\ulomega)$ of the profile $\wtilcalF_\ulomega[F_\ulomega(t)](\xi) =: ( \tilfplusulo(t,\xi), \tilfminusulo(t,\xi) )$, which reads 
\begin{equation} \label{eqn: model problem}
    \partial_t \tilf_{\pm,\ulomega}(t,\xi) = -ie^{\pm i t(\xi^2+\ulomega)}\wtilcalF_{\pm,\ulomega}\big[ \calC(U(t))\big](\xi).
\end{equation}
We uncover the fine structure of the contributions of the cubic terms on the right-hand side of \eqref{eqn: model problem} by inserting the representation formula \eqref{equ:setup_ulPeU_representation_formula} for the radiation term in terms of the distorted Fourier transform of its profile.
Analyzing the interactions of the corresponding distorted Fourier basis elements (cubic spectral distributions), we arrive at the decomposition 
\begin{equation} \label{equ:intro_cubic_decomposition}
\wtilcalF_{\pm, \ulomega}\bigl[\calC(U(t))\bigr](\xi) = \wtilcalC_{\pm,\delta_0,\ulomega}(t,\xi) + \wtilcalC_{\pm,\pvdots,\ulomega}(t,\xi)+\wtilcalC_{\pm,\mathrm{reg},\ulomega}(t,\xi) + \ldots
\end{equation}
The three terms on the right-hand side of \eqref{equ:intro_cubic_decomposition} are trilinear expressions in terms of the components of the distorted Fourier transform of the profile involving a Dirac kernel, a Hilbert-type kernel, and a regular kernel, respectively.
See Section~\ref{subsec:cubic_spectral_distributions} for the details.

In view of the linear dispersive decay estimate from Lemma~\ref{lem:linear_dispersive_decay}, in order to recover the linear decay rate $t^{-\frac12}$ for the radiation term, it suffices to establish slowly growing weighted energy estimates
\begin{equation} \label{equ:intro_weighted_energies}
    \bigl\| \bigl( \pxi \tilfplusulo(t,\xi), \pxi \tilfminusulo(t,\xi) \bigr) \bigr\|_{L^2_\xi \times L^2_\xi} \lesssim \varepsilon \jt^\delta, \quad t \geq 0,
\end{equation}
and uniform-in-time pointwise bounds
\begin{equation} \label{equ:intro_pointwise_bounds}
    \bigl\| \bigl( \tilfplusulo(t,\xi), \tilfminusulo(t,\xi) \bigr) \bigr\|_{L^\infty_\xi \times L^\infty_\xi} \lesssim \varepsilon, \quad t \geq 0.
\end{equation}
Thanks to a simple relation between the two components $\tilfplusulo(t,\xi)$ and $\tilfminusulo(t,\xi)$, see Lemma~\ref{lem:distFT_components_relation}, it in fact suffices to derive the bounds \eqref{equ:intro_weighted_energies}, \eqref{equ:intro_pointwise_bounds} for the component $\tilfplusulo(t,\xi)$.

In Section~\ref{sec:energy_estimates} we establish the weighted energy estimates \eqref{equ:intro_weighted_energies}, where $0 < \delta \ll 1$ is the same small absolute constant as in \eqref{equ:intro_decay_for_omega}. In Section~\ref{sec:pointwise_profile} we infer the pointwise bounds \eqref{equ:intro_pointwise_bounds} for the distorted Fourier transform of the profile via an ODE argument, in the course of which we capture the logarithmic phase corrections in the asymptotics of the radiation term.
More specifically, we carry out a stationary phase analysis of the contributions of the cubic terms in \eqref{eqn: model problem}, where we rely on their fine structure determined in \eqref{equ:intro_cubic_decomposition}. 
We find that the evolution equation for $\tilde{f}_{+,\ulomega}(t,\xi)$ is effectively given by 
\begin{equation}
    \partial_t \tilfplusulo(t,\xi) = \frac{i}{2t}\vert \tilfplusulo(t,\xi)\vert^2\tilfplusulo(t,\xi) + r(t,\xi),
\end{equation}
where the leading order term stems from the cubic interactions with a Dirac kernel, and where the remainder term $r(t,\xi)$ enjoys faster decay. 
Using a standard integrating factor argument, we can then deduce the pointwise bounds \eqref{equ:intro_pointwise_bounds}.


\medskip 

\noindent {\it Further Comments.} 
We conclude the discussion of the main ideas of the proof of Theorem~\ref{thm:main_theorem} with several more technical comments. 
\begin{enumerate}[label=(\roman*)]

    \item Owing to the quadratic null structures uncovered in Lemma~\ref{lem:null_structure_radiation} and Lemma~\ref{lem:null_structure_modulation}, the quadratic nonlinearity neither affects the asymptotics of the radiation term nor the asymptotic behavior of the modulation parameters. In the absence of favorable structures for related models with threshold resonances, this is a regime where much remains to be understood.

    \item In contrast to the generic case treated in \cite{CollotGermain23}, we do not obtain integrable decay rates for $\omega(t)-\omega(T)$ and for $\dot{\gamma}(t)-\omega(t)$. As a consequence, we have to handle a slight uncertainty in the phase shift in the proof of Theorem~\ref{thm:main_theorem} and we do not extract the phase shift parameter for the final soliton.  
    It is an interesting question how one could possibly further enhance the analysis in this paper to deduce integrable decay rates.

    \item Interestingly, in the derivation of the differential equation to extract the asymptotics of the distorted Fourier transform of the profile in Lemma~\ref{lemma: effective-ODE-profile}, the (critically decaying) cubic interactions with a Hilbert-type kernel do not contribute due to the vanishing property~\eqref{eqn:cubic-diagonal-property}, see the proof of \eqref{eqn:proof-stationary-phase-pv}.

    \item The null structure in the modulation equations discussed above is reminiscent of a similar favorable structure exploited in the modulation equations in \cite[Lemma 5.12]{KS09_quintic_NLS} in the context of constructing blow-up solutions to the critical focusing nonlinear Schr\"odinger equation on the line. The authors are grateful to Joachim Krieger for pointing out this connection.
  
    \item The tensorized product structure of the distorted Fourier basis elements is convenient in our analysis, see for instance Lemma~\ref{lemma: PDO on m12} and the study of the fine structure of the cubic spectral distributions in Subsection~\ref{subsec:cubic_spectral_distributions}.

    \item We exploit the simple relationship \eqref{equ:setup_distFT_components_relation} between the two components $\tilfplusulo(t,\xi)$ and $\tilfminusulo(t,\xi)$ of the distorted Fourier transform of the profile in the nonlinear analysis. It should be compared with the corresponding relation in the generic case, see e.g. \cite[Lemma 5.12]{CollotGermain23}.

    \item In the derivation of the slowly growing weighted energy estimates for the profile in Proposition~\ref{prop: weighted-energy-estimate}, the dual local smoothing estimates for the Schr\"odinger evolution from Lemma~\ref{lemma:local_smoothing} play a key role to handle all contributions of spatially localized nonlinearities with cubic-type time decay. This idea was first used in \cite{ChenPus22}. See also \cite{LLS1, LLS2, LS1} for a related approach.

    \item We made use of the Wolfram Mathematica software system in the computation of the resolvent kernel in Subsection~\ref{subsec:resolvent_kernel} as well as in the computation of the quadratic null structures in Lemma~\ref{lem:null_structure_radiation} and in Lemma~\ref{lem:null_structure_modulation}.
\end{enumerate}

\medskip 

\noindent {\it Acknowledgements:   
The authors are grateful to Gong Chen and Wilhelm Schlag for helpful discussions and for valuable comments on the manuscript.
The second author warmly thanks the Erwin-Schr\"odinger International Institute for Mathematics and Physics for its hospitality and support, where part of this work was conducted during the thematic programme ``Nonlinear Waves and Relativity''.
}

\section{Preliminaries} \label{sec:preliminaries}

\noindent {\it Notation and conventions.} 
Throughout this paper implicit constants in estimates may depend on the initial scaling parameter $\omega_0 \in (0,\infty)$ fixed in the statement of Theorem~\ref{thm:main_theorem}.
We denote by $C > 0$ an absolute constant whose value may change from line to line. 
For non-negative $X, Y$ we use the notation $X \lesssim Y$ if $X \leq C Y$ and we write $X \ll Y$ to indicate that the implicit constant should be regarded as small.
Moreover, we adopt the Japanese bracket notation
\begin{equation*}
    \jap{t} := (1+t^2)^{\frac12}, \quad \jap{x} := (1+x^2)^{\frac12}, \quad \jap{\xi} := (1+\xi^2)^{\frac12}.
\end{equation*}
We use the standard notation for the Lebesgue spaces $L^p$ and for the Sobolev spaces $H^k$, $W^{k,p}$. The space $L^{2,1}$ is defined by the weighted norm $\|f\|_{L^{2,1}} := \|\jx f\|_{L^2}$.

\medskip 

\noindent {\it Pauli matrices.}
We recall the standard definitions of the Pauli matrices
\begin{align} \label{eqn:pauli_matrices}
 \sigma_1 = \begin{bmatrix} 0 & 1 \\ 1 & 0 \end{bmatrix}, \quad \sigma_2 = \begin{bmatrix} 0 & -i \\ i & 0 \end{bmatrix}, \quad \sigma_3 = \begin{bmatrix} 1 & 0 \\ 0 & -1 \end{bmatrix}.
\end{align}

\medskip 

\noindent {\it Inner products.}
We define the $L^2$-inner product of two scalar-valued functions $f, g \colon \bbR \to \bbC$ by
\begin{equation} \label{equ:inner_product_scalar}
 \langle f, g \rangle := \int_\bbR f(x) \barg(x) \, \ud x,
\end{equation}
and the $L^2$-inner product of two vector-valued functions $U, V \colon \bbR \to \bbC^2$ by
\begin{equation} \label{equ:inner_product_vector}
 \langle U, V \rangle := \int_{\bbR} \bigl( u_1(x) \bar{v}_1(x) + u_2(x) \barv_2(x) \bigr) \, \ud x, \quad U := \begin{bmatrix} u_1 \\ u_2 \end{bmatrix}, \quad V := \begin{bmatrix} v_1 \\ v_2 \end{bmatrix}.
\end{equation}

\medskip 
\noindent {\it Fourier transform.} 
Our conventions for the (flat) Fourier transform of a function $f \in \calS(\bbR)$ are
\begin{equation*}
\begin{aligned}
 \widehat{\calF}[f](\xi) = \hatf(\xi) &:= \frac{1}{\sqrt{2\pi}} \int_\bbR e^{-ix\xi} f(x) \, \ud x, \\
 \widehat{\calF}^{-1}[f](\xi) = \check{f}(\xi) &:= \frac{1}{\sqrt{2\pi}} \int_\bbR e^{ix\xi} f(x) \, \ud x.
\end{aligned}
\end{equation*}
Then the convolution laws for $g,h \in \calS(\bbR)$ read 
\begin{equation*}
\widehat{\calF}\bigl[g \ast h\bigr] = \sqrt{2\pi} \hatg \hath, \quad \widehat{\calF}\bigl[g h\bigr] = \frac{1}{\sqrt{2\pi}} \hatg \ast \hath.
\end{equation*}
We recall that in the sense of tempered distributions
\begin{align}
    \widehat{\calF}[1](\xi) &= \sqrt{2\pi} \delta_0(\xi), \label{equ:preliminaries_FT_one} \\
    \widehat{\calF}[\tanh(\cdot)] &= -i \sqrt{\frac{\pi}{2}} \pvdots \cosech \Bigl(\frac{\pi}{2} \xi\Bigr). \label{equ:preliminaries_FT_tanh}
\end{align}
While \eqref{equ:preliminaries_FT_one} is standard, we refer to \cite[Lemma 5.6]{LS1} for a proof of \eqref{equ:preliminaries_FT_tanh}.

\medskip 
\noindent {\it $\calJ$-invariance.} 
Following the terminology in \cite{KS06}, we say that a vector $U \in \bbC^2$ is $\calJ$-invariant if
\begin{equation*}
	\overline{\sigma_1 U} = U.
\end{equation*}
The inner product \eqref{equ:inner_product_vector} of two $\calJ$-invariant vectors is real-valued.
The $\calJ$-invariance of the evolution equation \eqref{equ:setup_perturbation_equ} for the vectorial perturbation will allow us to go back from the system to the scalar cubic Schr\"odinger equation \eqref{equ:cubic_NLS}.

\section{Spectral Theory and Distorted Fourier Transform} \label{sec:spectral_and_dist_FT}

Linearizing the focusing cubic Schr\"odinger equation \eqref{equ:cubic_NLS} around the solitary wave solutions~\eqref{equ:intro_family_4parameter} leads to the following non-selfadjoint matrix Schr\"odinger operator
\begin{equation}\label{eqn:calH}
	\calH(\omega) := \calH_0(\omega) + \calV(\omega) := \begin{bmatrix}- \partial_x^2 + \omega & 0 \\ 0 & \partial_x^2 - \omega \end{bmatrix}  + \begin{bmatrix}
		-2\phi_\omega^2(x) &- \phi_\omega^2(x)\\
		\phi_\omega^2(x) & 2\phi_\omega^2(x)
	\end{bmatrix}, \quad \omega \in (0,\infty).
\end{equation}
After recalling its well-known spectrum, the main purpose of this section is to systematically develop the distorted Fourier theory for $\calH(\omega)$. 

We note that for matrix Schr\"odinger operators of the form \eqref{eqn:calH} in the generic case (no threshold resonances) the distorted Fourier theory was established in \cite{KS06, BusPerel92}, but the results from \cite{KS06, BusPerel92} do not directly apply to $\calH(\omega)$.
Proceeding similarly to \cite{KS06}, we first determine the kernel of the resolvent of $\calH(\omega)$, which yields an explicit formula for the jump of the resolvent across the essential spectrum. Inserting the latter into a representation of the linear evolution $e^{it\calH(\omega)}$ in terms of a contour integral of the resolvent (analogue of Stone's theorem in the non-selfadjoint case) then naturally leads to the definition of the distorted Fourier transform associated with $\calH(\omega)$. 

Throughout this section the following identities for the operator $\calH(\omega)$ will be used.
We have the symmetry relations
\begin{equation}\label{eqn:symmetry of calH}
	\sigma_1 \calH(\omega) = -\calH(\omega)\sigma_1, \quad \sigma_2 \calH(\omega) = - \calH(\omega)^*\sigma_2,\quad \sigma_3 \calH(\omega) = \calH(\omega)^*\sigma_3,
\end{equation}
where $\sigma_1,\sigma_2,\sigma_3$ are the Pauli matrices defined in \eqref{eqn:pauli_matrices}
and where $\calH(\omega)^\ast$ denotes the adjoint of $\calH(\omega)$ with respect to the inner product \eqref{equ:inner_product_vector}.
Moreover, we point out a scaling relation specific to the cubic Schr\"odinger equation. 
If a function $G(x) = (g_1(x),g_2(x))^\top$ solves the equation
\begin{equation} \label{eqn:scaling1}
	\calH(1)G(x) = \lambda G(x)
\end{equation}
for some $\lambda \in \bbC$, then the function $G(\sqrt{\omega}x)=(g_1(\sqrt{\omega}x,g_2(\sqrt{\omega}x))^\top$ solves the equation
\begin{equation}\label{eqn:scaling2}
	\calH(\omega)G(\sqrt{\omega}x) = \omega \lambda G(\sqrt{\omega}x).
\end{equation}

\subsection{Spectrum of linearized operator}

The spectrum of the matrix Schr\"odinger operator $\calH(\omega)$ is well-known. A full description is provided in the next proposition.

\begin{proposition} \label{prop:spectrum_of_calH}
	Fix $\omega \in (0,\infty)$.
    Then the following holds for the operator $\calH(\omega)$ defined in \eqref{eqn:calH} on $L^2(\bbR) \times L^2(\bbR)$ with domain $H^2(\bbR) \times H^2(\bbR)$:
	\begin{enumerate}[leftmargin=1.8em]
 
		\item Discrete spectrum: $0$ is the only eigenvalue with algebraic multiplicity equal to 4 and geometric multiplicity equal to 2.
  
		\item Generalized nullspace: $\calN_\mathrm{g}(\calH(\omega)) = \ker(\calH(\omega)^2)$ is spanned by the generalized eigenfunctions
		\begin{equation} \label{eqn:nullspace of calH-omega}
		Y_{1,\omega} := \begin{bmatrix}
				i\phi_\omega \\ -i\phi_\omega
			\end{bmatrix}, \quad Y_{2,\omega} := \begin{bmatrix}
				\partial_\omega \phi_\omega \\ \partial_\omega \phi_\omega
			\end{bmatrix}, \quad Y_{3,\omega} := \begin{bmatrix}
				\partial_x \phi_\omega \\ \partial_x \phi_\omega
			\end{bmatrix}, \quad Y_{4,\omega} := \begin{bmatrix}
				ix\phi_\omega \\ -ix\phi_\omega
			\end{bmatrix},
		\end{equation}
        which satisfy
        \begin{equation} \label{eqn:eig-eqn for nullspace of calH-omega}
	       \calH(\omega) Y_{1,\omega} = 0,\quad \calH(\omega) Y_{2,\omega} = i Y_{1,\omega}, \quad \calH(\omega) Y_{2,\omega} = 0, \quad \calH(\omega) Y_{4,\omega} = -2i Y_{3,\omega}.
        \end{equation}
        Moreover, we have $\calN_\mathrm{g}(\calH(\omega)^*) = \ker((\calH(\omega)^*)^2)= \mathrm{span}(\sigma_{2}Y_{k,\omega}:1\leq k\leq4)$, and the vectors $Y_{k,\omega}$ and $\sigma_2 Y_{k,\omega}$, $1 \leq k \leq 4$, are $\calJ$-invariant.
        The functions $Y_{1,\omega}$, $Y_{2,\omega}$ are even, while the functions $Y_{3,\omega}$, $Y_{4,\omega}$ are odd.
        \item The essential spectrum of $\calH(\omega)$ is $(-\infty,-\omega] \cup [\omega,\infty)$ and there are no embedded eigenvalues.
        \item Threshold resonances: The even vector-valued functions 
        \begin{equation}
	           \Phi_{+,\omega}(x) := \begin{bmatrix} \tanh^2(\sqrt{\omega}x) \\ -\sech^2(\sqrt{\omega}x) \end{bmatrix}, \quad \Phi_{-,\omega}(x) := \sigma_1 \Phi_{+,\omega}(x) = \begin{bmatrix} -\sech^2(\sqrt{\omega}x) \\	\tanh^2(\sqrt{\omega}x) \end{bmatrix}
        \end{equation}
         belong to $L^\infty(\bbR) \times L^\infty(\bbR) \setminus L^2(\bbR) \times L^2(\bbR)$ and satisfy
        \begin{equation}
	           \calH(\omega) \Phi_{+,\omega} = \omega \Phi_{+,\omega}, \quad \calH(\omega) \Phi_{-,\omega} = -\omega \Phi_{-,\omega}.
        \end{equation}
	\end{enumerate}
\end{proposition}

\begin{remark}
    In the proof of Theorem~\ref{thm:main_theorem} we only consider even solutions to the focusing cubic Schr\"odinger equation. For the restriction of $\calH(\omega)$ to the subspace of even functions in $H^2(\bbR) \times H^2(\bbR)$ the eigenvalue $0$ has algebraic multiplicity $2$ and geometric multiplicity $1$.
\end{remark}

\begin{proof}[Proof of Proposition~\ref{prop:spectrum_of_calH}]
For (2), the full description of the generalized nullspace can be found in, e.g., \cite[Proposition 1.2.2.]{BuslaevPerelman95} and \cite[Appendix B]{Weinstein85}, and the statement for the generalized nullspace for the adjoint $\calH^*$ follows from the symmetry relations \eqref{eqn:symmetry of calH} and \eqref{eqn:nullspace of calH-omega}. Note that a vector $Y$ is $\calJ$-invariant if and only if $\sigma_2 Y$ is $\calJ$-invariant. This does not necessarily hold true for $\sigma_3 Y$. Furthermore, the asserted geometric and algebraic multiplicity of the zero eigenvalue in item (1) is obvious from \eqref{eqn:eig-eqn for nullspace of calH-omega}. To show that zero is the only eigenvalue, we note that since by direct computation
\begin{equation*}
	c_\omega := \frac{\ud}{\ud \omega} \int_\bbR  (\phi_\omega(x))^2\,\ud x  = \frac{2}{\sqrt{\omega}} >0,
\end{equation*}
the statement from \cite[Proposition 1.2.2.]{BuslaevPerelman95} implies that $\mathrm{spec}(\calH) \subset \bbR$, and by \cite[page~130]{PKA98} there are no non-zero real eigenvalues (internal modes) for $\calH$. For (3), the statement about the essential spectrum follows from Weyl's criterion (see, e.g., \cite{HundertmarkLee07}), and the absence of embedded eigenvalues was proven in \cite[Proposition~9.2]{KS06}. Lastly, for (4), the formulas for $\Phi_{\pm,\omega}(x)$ are obtained from \cite[(3.54)]{CGNT08} using the conjugation identity \eqref{eqn:conjugation-with-calH} and the rescaling identity \eqref{eqn:scaling2}.
\end{proof}

\subsection{Resolvent kernel} \label{subsec:resolvent_kernel}

The goal of this subsection is to derive an explicit formula for the kernel of the resolvent of the operator $\calH(\omega)$ when $\omega = 1$.
Afterwards, we deduce the corresponding formulas for the resolvent of $\calH(\omega)$ for arbitrary $\omega > 0$.

Adopting the same notation as in \cite{CGNT08}, we define the scalar Schr\"odinger operators
\begin{equation*}
	\begin{split}
		L_0 &:= - \partial_x^2 - 6 \sech^2(x) + 1,\\
		L_1 &:= - \partial_x^2 - 2 \sech^2(x) + 1,\\
		L_3 &:= L_2 = -\partial_x^2 + 1.
	\end{split}
\end{equation*}
Then the following factorization identities hold
\begin{equation}\label{eqn:factorization-L1L2}
    L_1 = S^\ast S, \quad L_2 = S S^\ast, \quad S := \partial_x + \tanh(x).
\end{equation}
By \cite[(3.24)]{CGNT08}, we have
\begin{equation*}
    S L_0 S^* = S^* L_3 S,
\end{equation*}
and we infer from \eqref{eqn:factorization-L1L2} that 
\begin{equation}\label{eqn:SSL_0L_1}
S S L_0 L_1 = S S L_0 S^* S = S S^* L_3 S S = L_2 L_3 S S.
\end{equation}
Next, we introduce the matrix operators
\begin{equation*}
	\begin{split}
	\calL_1 := \begin{bmatrix}
			0 & L_1 \\  L_0 & 0
		\end{bmatrix},\quad
	\calL_0 := \begin{bmatrix}
		0 & L_3 \\  L_2 & 0
	\end{bmatrix},\quad
	\calD := \begin{bmatrix}
		0 & D_1\\
		D_2 & 0
	\end{bmatrix},
	\end{split}
\end{equation*}
where
\begin{equation*}
	\begin{split}
		D_1 &:= L_3SS = \big(-\partial_x^2+1\big)\big(\partial_x + \tanh(x)\big)\big(\partial_x + \tanh(x)\big),\\
		D_2 &:= SSL_0 = \big(\partial_x + \tanh(x)\big)\big(\partial_x + \tanh(x)\big)\big(-\partial_x^2 - 6 \sech^2(x) + 1\big).
	\end{split}
\end{equation*}
The identity \eqref{eqn:SSL_0L_1} implies that
\begin{equation} \label{equ:calD_calL1}
	\calD \calL_1 = \calL_0 \calD.
\end{equation}
Recall that the operators $\calL_0$ and $\calL_1$ are related to the operators $\calH_0 = \calH_0(1)$ and $\calH_1 = \calH(1)$ respectively by the conjugation identities
\begin{equation} \label{eqn:conjugation-with-calH}
	\calL_1 = \calC^{-1} \calH_1 \calC, \quad \calL_0 = \calC^{-1}\calH_0 \calC,\quad \calC := \frac{1}{\sqrt{2}} \begin{bmatrix}
		1 & 1 \\ 1 & -1
	\end{bmatrix}. 
\end{equation}
Upon defining 
\begin{equation}\label{eqn:def-wtilcalD}
    \wtilcalD := \calC \calD \calC^{-1} = \frac{1}{2} \begin{bmatrix}
        D_1 + D_2 & -D_1 + D_2 \\ D_1 - D_2 & - D_1 - D_2
    \end{bmatrix},
\end{equation}
we deduce from \eqref{equ:calD_calL1} and \eqref{eqn:conjugation-with-calH} the identity
\begin{equation*}
	\wtilcalD \calH_1 = \calH_0 \wtilcalD.
\end{equation*}
Then, using the self-adjointness of $\calH_0$ and the symmetry relations \eqref{eqn:symmetry of calH}, we infer the adjoint identity 
\begin{equation} \label{eqn:adjoint-identity}
	\calH_1 \sigma_3 \wtilcalD^* = \sigma_3 \wtilcalD^* \calH_0.
\end{equation}
The identity \eqref{eqn:adjoint-identity} is our starting point for the derivation of the kernel of the resolvent of $\calH_1$. We begin by computing a fundamental set of solutions to the equation $\calH_1 \bff = z\bff$, where $z \in \bbC_+$ is close to the positive part $[1,\infty)$ of the essential spectrum.

\begin{lemma} \label{lem:spectral_sec_definitions_bff_bfg}
Let $z = \xi^2 + 1 + i \varepsilon$ with $\xi\in \bbR$ and $\varepsilon >0$, and let $x \in \bbR$.
Denote by $\sqrt{\cdot}$ a branch of the square root that is analytic on $\bbC \backslash [0,\infty)$ and such that $\sqrt{-1} = i$.
Define the functions
	\begin{align}
		\bff_1(x,z) &:= \begin{bmatrix}
			\big(\sqrt{1-z}-\tanh(x)\big)^2 e^{\sqrt{1-z}x}\\
			- \sech^2(x)e^{\sqrt{1-z}x}
		\end{bmatrix}, \label{eqn: f1}\\
		\bff_3(x,z) &:= \begin{bmatrix}
			-\sech^2(x)e^{-\sqrt{1+z}x}\\
			\big(\sqrt{1+z}+\tanh(x)\big)^2e^{-\sqrt{1+z}x}
		\end{bmatrix} ,\\
		\bfg_2(x,z) &:=  \begin{bmatrix}
			\big(\sqrt{1-z}+\tanh(x)\big)^2e^{-\sqrt{1-z}x}\\
			-\sech^2(x)e^{-\sqrt{1-z}x}
		\end{bmatrix} ,\\
		\bfg_4(x,z) &:= \begin{bmatrix}
			-\sech^2(x)e^{\sqrt{1+z}x}\\
			\big(\sqrt{1+z}-\tanh(x)\big)^2e^{\sqrt{1+z}x}
		\end{bmatrix} \label{eqn: f4}.
	\end{align}
Then $\{ \bff_1(\cdot,z), \bff_3(\cdot,z), \bfg_2(\cdot,z), \bfg_3(\cdot,z) \}$ form a fundamental set of solutions to $(\calH_1 - z)\bff = 0$. 
Moreover, they satisfy the asymptotics
	\begin{equation} \label{eqn: limits of f1f3g2g4}
		\lim_{x \rightarrow +\infty} \bff_1(x,z) = \lim_{x \rightarrow +\infty} \bff_3(x,z) = 	\lim_{x \rightarrow -\infty} \bfg_2(x,z) = 	\lim_{x \rightarrow -\infty} \bfg_4(x,z) = \begin{bmatrix}0 \\ 0 \end{bmatrix},
	\end{equation}
	and their Wronskians\footnote{For two $\bbC^2$-valued solutions $\bff,\bfg$ to the equation $\calH_1 \bff = z\bff$, their Wronskian is defined as $W[\bff,\bfg] = \partial_x \bff \cdot \bfg - \bff \cdot \partial_x \bfg$, where $\cdot$ denotes the real scalar product.} are given by
	\begin{align}
		&W[\bff_1,\bfg_2](z) = 2 z^2\sqrt{1-z} , \label{eqn:wronskian1}\\
		&W[\bff_3,\bfg_4](z) = -2z^2\sqrt{1+z} , \label{eqn:wronskian2} \\
		&W[\bff_1,\bfg_4](z) = W[\bff_3,\bfg_2](z) = 0.\label{eqn:wronskian3}
	\end{align}
\end{lemma}
\begin{proof}
From the definition of $\wtilcalD$ in \eqref{eqn:def-wtilcalD} we obtain that
	\begin{equation*}
		\sigma_3 \wtilcalD^* = \frac{1}{2} \begin{bmatrix}
			(D_1^* + D_2^*) & (D_1^* - D_2^*) \\ (D_1^* - D_2^*) & (D_1^* + D_2^*)
		\end{bmatrix},
	\end{equation*}
	where
	\begin{equation*}
		\begin{split}
			&D_1^* = S^*S^*L_3 = \big(-\partial_x+\tanh(x)\big)\big(-\partial_x+\tanh(x)\big)\big(-\partial_x^2+1\big),\\
			&D_2^* = L_0S^*S^* =\big(-\partial_x^2 - 6 \sech^2(x)+1\big)\big(-\partial_x+\tanh(x)\big)\big(-\partial_x+\tanh(x)\big).
		\end{split}
	\end{equation*}
	For any function $h(z)$, we have by direct computation 
	\begin{align*}
		&D_1^\ast \bigl( e^{h(z)x} \bigr) = (1-h(z)^2) \big( (h(z)-\tanh(x))^2-\sech^2(x) \big) e^{h(z)x},\\
		&D_2^\ast \bigl( e^{h(z)x} \bigr) = (1-h(z)^2) \big( (h(z)-\tanh(x))^2+\sech^2(x) \big) e^{h(z)x}.
	\end{align*}
	In addition, it holds that
	\begin{align*}
		&(D_1^*+D_2^*)\bigl( e^{h(z)x} \bigr) = 2(1-h(z)^2) \big( h(z) - \tanh(x) \big)^2 e^{h(z)x},\\
		&(D_1^*-D_2^*)\bigl( e^{h(z)x} \bigr) = -2(1-h(z)^2) \big( \sech(x) \big)^2 e^{h(z)x}.
	\end{align*}
	It is straightforward to check that the following vector-valued functions
	\begin{equation*}
		\begin{bmatrix}
			e^{\sqrt{1-z}x} \\ 0
		\end{bmatrix}, \quad \begin{bmatrix}
		e^{-\sqrt{1-z}x} \\ 0
	\end{bmatrix}, \quad \begin{bmatrix}
	0 \\	e^{ \sqrt{1+z}x}
\end{bmatrix}, \quad \begin{bmatrix}
0 \\	e^{-\sqrt{1+z}x}
\end{bmatrix}
	\end{equation*}
	form a fundamental set of solutions to $\calH_0 \bff = z\bff$. Hence, in view of the above identities and \eqref{eqn:adjoint-identity}, we can build four linearly independent solutions to $\calH_1 \bff = z\bff$ as follows,
	\begin{align*}
		\sigma_3\wtilcalD^*\begin{bmatrix}
			e^{\sqrt{1-z}x} \\ 0
		\end{bmatrix} 
        &= \begin{bmatrix}
		      z(\sqrt{1-z}-\tanh(x))^2e^{\sqrt{1-z}x}\\
		      -z \sech^2(x) e^{\sqrt{1-z}x}
	       \end{bmatrix}, \\
		\sigma_3\wtilcalD^*\begin{bmatrix}
	       0 \\	e^{- \sqrt{1+z}x}
        \end{bmatrix} 
        &= \begin{bmatrix}
	           z\sech^2(x)e^{-\sqrt{1+z}x}\\-z(\sqrt{1+z}+\tanh(x))^2e^{-\sqrt{1+z}x}
            \end{bmatrix}, \\
		\sigma_3\wtilcalD^*\begin{bmatrix}
			e^{-\sqrt{1-z}x} \\ 0
		\end{bmatrix} 
        &= \begin{bmatrix}
		z(\sqrt{1-z}+\tanh(x))^2e^{-\sqrt{1-z}x}\\
		-z \sech^2(x) e^{-\sqrt{1-z}x}
	    \end{bmatrix}, \\
		      \sigma_3\wtilcalD^*\begin{bmatrix}
	        0 \\	e^{ \sqrt{1+z}x}
        \end{bmatrix} 
        &=
        \begin{bmatrix}
	       z\sech^2(x)e^{\sqrt{1+z}x}\\-z(\sqrt{1+z}-\tanh(x))^2e^{\sqrt{1+z}x}
        \end{bmatrix}.
\end{align*}
    Dividing by the factor $z$ (or by $-z$) gives the expressions \eqref{eqn: f1}--\eqref{eqn: f4}. 
    By our choice of branch for the square root, we note that $\Re (\sqrt{1-z}) < 0$ as well as  $\Re ( \sqrt{1+z} ) > 0$ for $z = \xi^2 + 1 + i\varepsilon$ with $\xi \in \bbR$ and $\varepsilon > 0$.
	Thus, the asserted limits \eqref{eqn: limits of f1f3g2g4} follow from the expressions \eqref{eqn: f1}--\eqref{eqn: f4}. The formulas \eqref{eqn:wronskian1}, \eqref{eqn:wronskian2}, \eqref{eqn:wronskian3} for the Wronskians follow by direct computation.
\end{proof}

\begin{remark}
    The explicit formulas for the Jost functions $\bff_1$ and $\bfg_4$ for $z=\xi^2+1+i0$ are already contained in \cite{Kaup90}. An analogous derivation of these formulas using the identity \eqref{eqn:adjoint-identity} as in the proof of the preceding Lemma~\ref{lem:spectral_sec_definitions_bff_bfg} was recently given in \cite[Section 10]{CuccMaeda2404}.
\end{remark}

Next, we define for $x\in\bbR$ and $z = \xi^2 + 1 + i \varepsilon$ with $\xi \in \bbR$ and $\varepsilon > 0$, the $2 \times 2$ matrix-valued functions
\begin{align}
\bbF(x,z) &:= \begin{bmatrix}
		\bff_1(x,z) & \bff_3(x,z)
	\end{bmatrix},\\
\bbG(x,z) &:= \begin{bmatrix}
	\bfg_2(x,z) & \bfg_4(x,z)
\end{bmatrix},\\
\bbD(z) &:=  \begin{bmatrix}
	W[\bff_1,\bfg_2](z) & 	W[\bff_1,\bfg_4](z)\\
	W[\bff_3,\bfg_2](z) & 	W[\bff_3,\bfg_4](z)
\end{bmatrix} = \begin{bmatrix}
2z^2\sqrt{1-z} & 0 \\ 0 & -2z^2\sqrt{1+z}
\end{bmatrix},
\end{align}
where the functions $\bff_1$, $\bff_3$, $\bfg_2$, and $\bfg_4$ on the right-hand sides are defined as in the statement of the preceding Lemma~\ref{lem:spectral_sec_definitions_bff_bfg}. 
Proceeding as in the proof of \cite[Lemma 6.5]{KS06}, we now express the kernel of the resolvent of $\calH_1$ in terms of the matrices $\bbF$, $\bbG$, and $\bbD$ defined above.

\begin{lemma}
Let $z = \xi^2+1+i\varepsilon$ with $\xi \in \bbR$ and $\varepsilon>0$. We have
\begin{equation}
\left(\calH_1 - z\right)^{-1}(x,y) = \begin{cases}
	-\bbF(x,z)\bbD(z)^{-1} \bbG(y,z)^\top \sigma_3,\quad \text{if $x\geq y$},\\
	-\bbG(x,z)\bbD(z)^{-1} \bbF(y,z)^\top \sigma_3,\quad \text{if $x\leq y$}.
\end{cases}
\end{equation}
In particular, for $x \geq y$ it holds that
\begin{equation}\label{eqn:(caLH-z)1}
	\begin{split}
		&(\calH-z)^{-1}(x,y)\\
		&=\frac{e^{\sqrt{1-z}(x-y)}}{2z^2\sqrt{1-z}} \cdot \begin{bsmallmatrix}
			- (\sqrt{1-z}-\tanh(x))^2(\sqrt{1-z}+\tanh(y))^2 & -(\sqrt{1-z}-\tanh(x))^2\sech^2(y) \\ \sech^2(x)(\sqrt{1-z}+\tanh(y))^2 & \sech^2(x)\sech^2(y)
		\end{bsmallmatrix}\\
		&\quad + \frac{e^{\sqrt{1+z}(y-x)}}{2z^2\sqrt{1+z}} \cdot \begin{bsmallmatrix}
			\sech^2(x)\sech^2(y) & \sech^2(x)(\sqrt{1+z}-\tanh(y))^2 \\ -(\sqrt{1+z}+\tanh(x))^2\sech^2(y) & -{(\sqrt{1+z}+\tanh(x))^2(\sqrt{1+z}-\tanh(y))^2}
		\end{bsmallmatrix},
	\end{split}
\end{equation}
while for $x \leq y$ it holds that
\begin{equation}\label{eqn:(caLH-z)2}
	\begin{split}
		&(\calH - z)^{-1}(x,y) \\
		&=\frac{e^{-\sqrt{1-z}(x-y)}}{2z^2\sqrt{1-z}} \cdot \begin{bsmallmatrix}
			- (\sqrt{1-z}+\tanh(x))^2(\sqrt{1-z}-\tanh(y))^2 & -(\sqrt{1-z}+\tanh(x))^2\sech^2(y) \\ \sech^2(x)(\sqrt{1-z}-\tanh(y))^2 & \sech^2(x)\sech^2(y)
		\end{bsmallmatrix}\\
		&\quad + \frac{e^{-\sqrt{1+z}(y-x)}}{2z^2\sqrt{1+z}} \cdot \begin{bsmallmatrix}
			\sech^2(x)\sech^2(y) & \sech^2(x)(\sqrt{1+z}+\tanh(y))^2 \\ -(\sqrt{1+z}-\tanh(x))^2\sech^2(y) & -{(\sqrt{1+z}-\tanh(x))^2(\sqrt{1+z}+\tanh(y))^2}
		\end{bsmallmatrix}.
	\end{split}
\end{equation}
\end{lemma}
\begin{proof}
We want to find the Green's function for the system of equations
\begin{equation} \label{equ:green_function_computation}
	(\calH_1-z)\bff = \bfh,
\end{equation}
where $\bfh \in L^2(\bbR) \times L^2(\bbR)$ is given.
In view of the asymptotics \eqref{eqn: limits of f1f3g2g4}, an $L^2(\bbR)$-integrable solution to \eqref{equ:green_function_computation} takes the form 
\begin{equation} \label{equ:spectral_variation_param_formula_bff}
	\bff(x) = \bbG(x,z)\int_x^\infty \bbM_2(y,z)\bfh(y) \,\ud y + \bbF(x,z)\int_{-\infty}^x \bbM_1(y,z) \bfh(y) \,\ud y,
\end{equation}
where the matrices $\bbM_1$ and $\bbM_2$ are required to satisfy the compatibility conditions
\begin{equation*}
\begin{split}
	\bbF(x,z)\bbM_1(x,z)  - \bbG(x,z)\bbM_2(x,z) &= 0_{2\times2},\\
	\partial_x \bbF(x,z)\bbM_1(x,z)  - \partial_x \bbG(x,z)\bbM_2(x,z) &= -\sigma_3.\\
\end{split}
\end{equation*}
This system of equations can be written in matrix form as
\begin{equation*}
\begin{bmatrix}
	\bbF & \bbG \\ \bbF' & \bbG'
\end{bmatrix}	\begin{bmatrix}
\bbM_1 \\ - \bbM_2
\end{bmatrix} = \begin{bmatrix}
0_{2\times 2} \\ -\sigma_3
\end{bmatrix}.
\end{equation*}
Using that $\bbD = \bbD^\top$ and applying \cite[(5.22)]{KS06}, we have
\begin{equation*}
\begin{bmatrix}
	\bbF & \bbG \\ \bbF' & \bbG'
\end{bmatrix}^{-1} = \begin{bmatrix}
-\bbD(z)^{-1}\bbG'^\top & \bbD(z)^{-1} \bbG^\top \\
\bbD(z)^{-1} \bbF'^\top & - \bbD(z)^{-1} \bbF^\top
\end{bmatrix},
\end{equation*}
whence
\begin{equation} \label{equ:spectral_bbM12_formulas}
	\begin{bmatrix}
		\bbM_1(x,z) \\
		\bbM_2(x,z)
	\end{bmatrix} = \begin{bmatrix}
	-\bbD(z)^{-1} \bbG(x,z)^\top\sigma_3\\
	-\bbD(z)^{-1} \bbF(x,z)^\top\sigma_3
\end{bmatrix}.
\end{equation}
Finally, we obtain the expressions \eqref{eqn:(caLH-z)1}--\eqref{eqn:(caLH-z)2} for the resolvent kernels by direct computation from \eqref{equ:spectral_bbM12_formulas}, \eqref{equ:spectral_variation_param_formula_bff} and from the explicit expressions \eqref{eqn: f1}--\eqref{eqn: f4} as well as \eqref{eqn:wronskian1}--\eqref{eqn:wronskian3}.
\end{proof}

Next, we compute the formula for the jump of the resolvent across the positive part $[1,\infty)$ of the essential spectrum.

\begin{lemma} \label{lem:jump_resolvent}
For $x, \xi \in \bbR$ set
\begin{equation} \label{equ:definition_bfF_bfG}
	\bfF(x,\xi) := \frac{e^{ix \xi}}{(\xi-i)^2}\begin{bmatrix}
		(\xi+i\tanh(x))^2 \\ \sech^2(x)
	\end{bmatrix}, \quad \bfG(x,\xi) := \frac{e^{-ix \xi}}{(\xi-i)^2}\begin{bmatrix}
		(\xi - i \tanh(x))^2 \\ \sech^2(x)
	\end{bmatrix},
\end{equation}
and
\begin{equation}
	\bbE(x,\xi) := \begin{bmatrix}
		\bfF(x,\xi) & \bfG(x,\xi)
	\end{bmatrix}.
\end{equation}
Then for any $\xi>0$,
\begin{equation} \label{equ:jump_resolvent_omega_equal1}
	\left(\calH_1-(\xi^2+1+i0)\right)^{-1}(x,y) - \left(\calH_1-(\xi^2+1-i0)\right)^{-1}(x,y) = -\frac{1}{2i \xi}\bbE(x,\xi)\bbE(y,\xi)^* \sigma_3.
\end{equation}
\end{lemma}
\begin{proof}
	Let $z_\varepsilon := \xi^2 + 1 + i \varepsilon$ with $\xi > 0$ and $\varepsilon >0$. Note that
	\begin{equation*}
		\lim_{\varepsilon \rightarrow 0} \sqrt{1-z_\varepsilon} = i \xi, \qquad \lim_{\varepsilon \rightarrow 0} \sqrt{1+z_\varepsilon} = \sqrt{\xi^2+2}.
	\end{equation*}
	We first consider the case $x \geq y$. From \eqref{eqn:(caLH-z)1}, we have
	\begin{equation*}
		\begin{split}
			&\big(\calH_1 - (\xi^2+1+i0)\big)^{-1}(x,y) \\
            &= \lim_{\varepsilon \rightarrow 0}\, \big(\calH_1 - (\xi^2+1+i\varepsilon)\big)^{-1}(x,y)\\
			&= \frac{e^{i\xi(x-y)}}{2i\xi(\xi^2+1)^2}\begin{bsmallmatrix}
				-(i\xi-\tanh(x))^2(i\xi+\tanh(y))^2 & -(i\xi-\tanh(x))^2\sech^2(y)\\
				\sech^2(x)(i\xi+\tanh(y))^2 & \sech^2(x) \sech^2(y)
			\end{bsmallmatrix}\\
			&\quad +\frac{e^{\sqrt{\xi^2+2}(y-x)}}{2\sqrt{\xi^2+2}(\xi^2+1)^2} \begin{bsmallmatrix}
				\sech^2(x)\sech^2(y) & \sech^2(x) \big(\sqrt{\xi^2+2}-\tanh(y)\big)^2\\
				-\big(\sqrt{\xi^2+2}+\tanh(x)\big)^2\sech^2(y) & -\big(\sqrt{\xi^2+2}+\tanh(x)\big)^2\big(\sqrt{\xi^2+2}-\tanh(y)\big)^2
			\end{bsmallmatrix}.
		\end{split}
	\end{equation*}
	Thus, we also obtain that
	\begin{equation*}
		\begin{split}
			&\big(\calH_1 - (\xi^2+1-i0)\big)^{-1}(x,y) \\
            &= \overline{\big(\calH_1 - (\xi^2+1+i0)\big)^{-1}(x,y)} \\
			&= -\frac{e^{-i\xi(x-y)}}{2i\xi(\xi^2+1)^2}\begin{bsmallmatrix}
				-(i\xi+\tanh(x))^2(i\xi-\tanh(y))^2 & -(i\xi+\tanh(x))^2\sech^2(y)\\
				\sech^2(x)(i\xi-\tanh(y))^2 & \sech^2(x) \sech^2(y)
			\end{bsmallmatrix}\\
			&\quad +\frac{e^{\sqrt{\xi^2+2}(y-x)}}{2\sqrt{\xi^2+2}(\xi^2+1)^2} \begin{bsmallmatrix}
				\sech^2(x)\sech^2(y) & \sech^2(x) \big(\sqrt{\xi^2+2}-\tanh(y)\big)^2\\
				-\big(\sqrt{\xi^2+2}+\tanh(x)\big)^2\sech^2(y) & -\big(\sqrt{\xi^2+2}+\tanh(x)\big)^2\big(\sqrt{\xi^2+2}-\tanh(y)\big)^2
			\end{bsmallmatrix}.
		\end{split}
	\end{equation*}
	Then using that $(\xi^2+1)^2 = (\xi-i)^2(\xi+i)^2$, we find
	\begin{equation*}
		\begin{split}
        &\left(\calH_1-(\xi^2+1+i0)\right)^{-1}(x,y) - \left(\calH_1-(\xi^2+1-i0)\right)^{-1}(x,y)\\
        &=- \frac{1}{2i\xi} \frac{1}{(\xi^2+1)^2}\left(e^{i\xi(x-y)}\begin{bsmallmatrix}
            (\xi+i\tanh(x))^2(\xi-i\tanh(y))^2 & -	(\xi+i\tanh(x))^2\sech^2(y) \\
            \sech^2(x)(\xi-i\tanh(y))^2 & - \sech^2(x)\sech^2(y)
        \end{bsmallmatrix}\right.\\
        &\quad + \left. e^{-i\xi(x-y)}\begin{bsmallmatrix}
	       (\xi-i\tanh(x))^2(\xi+i\tanh(y))^2 & -	(\xi-i\tanh(x))^2\sech^2(y) \\
	       \sech^2(x)(\xi+i\tanh(y))^2 & - \sech^2(x)\sech^2(y)
        \end{bsmallmatrix}\right)			\\
        &= -\frac{1}{2i \xi}\bbE(x,\xi)\bbE(y,\xi)^* \sigma_3.
		\end{split}
	\end{equation*}
    This proves \eqref{equ:jump_resolvent_omega_equal1} for $x \geq y$.
    For the case $x \leq y$ we repeat the preceding computations using \eqref{eqn:(caLH-z)2}, which gives the same formula.
\end{proof}

The functions $\bfF(x,\xi)$ and $\bfG(x,\xi)$ form a pair of bounded solutions to the equation $\calH_1\bff=(\xi^2+1)\bff$. 
The identity \eqref{equ:jump_resolvent_omega_equal1} expresses a product-type formula for the jump of the resolvent across $[1,\infty)$. 
The scaling relations \eqref{eqn:scaling1}--\eqref{eqn:scaling2} then imply that for arbitrary $\omega>0$, the functions $\bfF(\sqrt{\omega}x, \frac{\xi}{\sqrt{\omega}})$, $\bfG(\sqrt{\omega}x, \frac{\xi}{\sqrt{\omega}})$ are a pair of bounded solutions to the equation $\calH(\omega)\bff = (\xi^2+\omega)\bff$. The next corollary describes the jump of the resolvent for arbitrary $\omega > 0$ in terms of these rescaled functions.
The proof proceeds along the lines of the preceding Lemma~\ref{lem:jump_resolvent}.

\begin{corollary} \label{corollary: jump formula}
	Fix $\omega \in (0,\infty)$. For $x, \xi \in \bbR$ set
    \begin{equation}
		\bbE_\omega(x,\xi) := \begin{bmatrix}
			\bfF(\sqrt{\omega}x,\frac{\xi}{\sqrt{\omega}}) & \bfG(\sqrt{\omega}x,\frac{\xi}{\sqrt{\omega}})
		\end{bmatrix}.
	\end{equation}
	Then for any $\xi>0$,
	\begin{equation} \label{equ:jump_formula}
		\left(\calH(\omega)-(\xi^2+\omega+i0)\right)^{-1}(x,y) - \left(\calH(\omega)-(\xi^2+\omega-i0)\right)^{-1}(x,y) = -\frac{1}{2i \xi}\bbE_\omega(x,\xi)\bbE_\omega(y,\xi)^* \sigma_3.
	\end{equation}
\end{corollary}

We denote by $P_\mathrm{d}$ the Riesz projection onto the discrete spectrum of $\calH(\omega)$ defined as 
\begin{equation} \label{equ:definition_Pd}
	P_\mathrm{d} := \frac{1}{2\pi i} \oint_\gamma \bigl(\calH(\omega) - zI\bigr)^{-1} \, \ud z,
\end{equation}
where $\gamma$ is a simple closed curve around the origin that lies within the resolvent set of $\calH(\omega)$.
Then we define the projection onto the essential spectrum by
\begin{equation*}
    P_{\mathrm{e}} := I - P_{\mathrm{d}}.
\end{equation*}
We refer the reader to \cite[Remark~9.5]{KS06} for a proof of the fact that the Riesz projections $P_\mathrm{d}$ and $P_\mathrm{e}$ preserve the space of $\calJ$-invariant functions on $L^2(\bbR) \times L^2(\bbR)$. 
The next lemma provides an explicit decomposition of $L^2(\bbR) \times L^2(\bbR)$ into the spectral subspaces associated with the essential spectrum and with the discrete spectrum of $\calH(\omega)$.

\begin{lemma} \label{lemma: L2 decomposition}
There is the decomposition
\begin{equation} \label{eqn: L2 decomposition}
	L^2(\bbR) \times L^2(\bbR) = \big(\calN_\mathrm{g}(\calH(\omega)^*) \big)^\perp + \calN_\mathrm{g}(\calH(\omega)),
\end{equation}
where the individual summands are linearly independent, but not necessarily orthogonal. The decomposition \eqref{eqn: L2 decomposition} is invariant under the flow of $\calH(\omega)$, and the projection $P_\mathrm{e}$ is the projection onto the orthogonal complement in \eqref{eqn: L2 decomposition}.
Explicitly, for $U \in L^2(\bbR) \times L^2(\bbR)$ it holds that
\begin{equation} \label{equ:spectral_decomposition_L2L2}
    U = P_{\mathrm{e}} U + \sum_{j=1}^4 d_{j,\omega} Y_{j,\omega}
\end{equation}
with 
\begin{equation*}
    d_{1,\omega} = \frac{\langle U, \sigma_2 Y_{2, \omega}\rangle}{\langle Y_{1,\omega}, \sigma_2 Y_{2,\omega}\rangle}, \quad d_{2,\omega} = \frac{\langle U, \sigma_2 Y_{1,\omega} \rangle}{\langle Y_{2,\omega}, \sigma_2 Y_{1,\omega}\rangle}, \quad d_{3,\omega} = \frac{\langle U, \sigma_2 Y_{4,\omega}\rangle}{\langle Y_{3, \omega}, \sigma_2 Y_{4,\omega}\rangle}, \quad d_{4,\omega} = \frac{\langle U, \sigma_2 Y_{3,\omega}\rangle}{\langle Y_{4, \omega}, \sigma_2 Y_{3,\omega}\rangle}.
\end{equation*}
By direct computation, we have 
\begin{equation}\label{eqn: inner-product_Y_j}
    \langle Y_{1,\omega}, \sigma_2 Y_{2,\omega}\rangle = - \frac{\ud}{\ud \omega} \|\phi_\omega\|_{L^2_x}^2 = - \frac{2}{\sqrt{\omega}}, \quad \langle Y_{3,\omega}, \sigma_2 Y_{4,\omega}\rangle = \|\phi_\omega\|_{L^2_x}^2 = 4 \sqrt{\omega}.
\end{equation}
For even $U \in L^2(\bbR) \times L^2(\bbR)$ it follows that
\begin{equation} \label{equ:spectral_decomposition_L2L2_even}
    U = P_{\mathrm{e}} U + d_{1,\omega} Y_{1,\omega} +  d_{2,\omega} Y_{2,\omega}.
\end{equation}
The projection $P_{\mathrm{e}}$ is bounded as an operator $H^1(\bbR) \to H^1(\bbR)$, $L^{2,1}(\bbR) \to L^{2,1}(\bbR)$, and as an operator $L^p(\bbR) \to L^p(\bbR)$ for any $1 \leq p \leq \infty$.
\end{lemma}
\begin{proof}
    We only prove \eqref{equ:spectral_decomposition_L2L2} and \eqref{equ:spectral_decomposition_L2L2_even}, and refer to the proof of Lemma~9.4 in \cite{KS06} for the first part of the statement of the lemma.
    The identity \eqref{equ:spectral_decomposition_L2L2} follows from \eqref{eqn: L2 decomposition} and the fact that $\calN_g(\calH(\omega))$ is spanned by $Y_{j,\omega}$, $1 \leq j \leq 4$, see Proposition~\ref{prop:spectrum_of_calH}. 
    Since $P_{\mathrm{e}} U$ is orthogonal to $\calN_g(\calH(\omega)^\ast)$ and since $\calN_g(\calH(\omega)^\ast)$ is spanned by $\sigma_2 Y_{k,\omega}$, $1 \leq k \leq 4$, see Proposition~\ref{prop:spectrum_of_calH}, we obtain from \eqref{equ:spectral_decomposition_L2L2} for $1 \leq k \leq 4$ that
    \begin{equation*}
        \langle U, \sigma_2 Y_{k,\omega} \rangle = \sum_{j=1}^4 d_{j,\omega} \langle Y_{j,\omega}, \sigma_2 Y_{k,\omega} \rangle.
    \end{equation*}
    The asserted formulas for $d_{j,\omega}$ now follow from the facts that $Y_{1,\omega}$, $Y_{2,\omega}$ are even, while $Y_{3,\omega}, Y_{4,\omega}$ are odd, and that $\langle Y_{j,\omega}, \sigma_2 Y_{j,\omega} \rangle = 0$ for $1 \leq j \leq 4$. 
    The identity \eqref{equ:spectral_decomposition_L2L2_even} for even $U \in L^2(\bbR) \times L^2(\bbR)$ follows immediately from \eqref{equ:spectral_decomposition_L2L2} because $Y_{3,\omega}, Y_{4,\omega}$ are odd.    
    Finally, the asserted boundedness properties of the projection $P_{\mathrm{e}}$ are a consequence of \eqref{equ:spectral_decomposition_L2L2} and the fact that the generalized eigenfunctions $Y_{j,\omega}$, $1 \leq j \leq 4$, are Schwartz functions.
\end{proof}

Finally, we recall from Lemma~12 and Remark~1 in \cite{ES06} a representation of the (non-unitary) flow $e^{it\calH(\omega)}$ in terms of the jump of the resolvent across the essential spectrum. 
The statement is analogous to Stone's formula for self-adjoint Schr\"odinger operators, where it is a consequence of the spectral theorem.

\begin{proposition}
There is the representation formula
\begin{equation} \label{eqn:representation formula for e-itH}
	\begin{split}
	e^{it\calH(\omega)} &= 	e^{it\calH(\omega)}\Pe + 	e^{it\calH(\omega)}P_\mathrm{d}\\
	&=	\frac{1}{2\pi i}\int_{\vert \lambda \vert \geq \omega}e^{it\lambda}\left[\big(\calH(\omega) - (\lambda+i0)\big)^{-1}-\big(\calH(\omega) - (\lambda-i0)\big)^{-1}\right] \, \ud \lambda + \sum_{k=0}^1 \frac{(it)^k}{k!}\calH(\omega)^k P_\mathrm{d}.
	\end{split}
\end{equation}
\end{proposition}

\subsection{Distorted Fourier theory}

Now we are in the position to introduce the distorted Fourier transform relative to $\calH(\omega)$.
We essentially read off its definition from the representation formula~\eqref{eqn:representation formula for e-itH} for the flow $e^{it\calH(\omega)}$ upon inserting the (explicit) expression for the kernel of the jump of the resolvent across the essential spectrum obtained in Corollary~\ref{corollary: jump formula}.

\begin{proposition} \label{prop: representation formula}
Fix $\omega \in (0,\infty)$. For $x, \xi \in \bbR$ set 
\begin{equation} \label{eqn: Psi_pm_omega}
    \begin{aligned}
	\Psi_{+,\omega}(x,\xi) &:= \frac{1}{\sqrt{2\pi}} 
    \left\{
    \begin{aligned}
	   &\bfF\big(\sqrt{\omega} x, \tfrac{\xi}{\sqrt{\omega}}\big), \quad \, \, \, \, \, \, \text{if $\xi \geq 0$} \\
	   &\bfG\big(\sqrt{\omega} x, -\tfrac{\xi}{\sqrt{\omega}}\big), \quad \text{if $\xi < 0$} \\
    \end{aligned} 
    \right\}
    = 
    \begin{bmatrix}
	   \Psi_{1,\omega}(x,\xi)\\
	   \Psi_{2,\omega}(x,\xi)
	\end{bmatrix}, \\
    \Psi_{-,\omega}(x,\xi) &:= \sigma_1 \Psi_{+,\omega}(x,\xi) = \begin{bmatrix}
        \Psi_{2,\omega}(x,\xi) \\
        \Psi_{1,\omega}(x,\xi)
    \end{bmatrix},
    \end{aligned}
\end{equation}
where 
\begin{align}
	\Psi_{1,\omega}(x,\xi) &:= \frac{1}{\sqrt{2\pi}}  \frac{\bigl(\xi + i \sqrt{\omega}\tanh(\sqrt{\omega}x)\bigr)^2}{(\vert \xi \vert - i \sqrt{\omega})^2} e^{i x \xi} \equiv \frac{1}{\sqrt{2\pi}} m_{1,\omega}(x,\xi) e^{ix\xi}, \label{eqn:m-1,omega} \\
	\Psi_{2,\omega}(x,\xi) &:= \frac{1}{\sqrt{2\pi}}  \frac{\big(\sqrt{\omega}\sech(\sqrt{\omega}x))^2}{(\vert \xi \vert - i \sqrt{\omega}\big)^2} e^{i x \xi} \equiv \frac{1}{\sqrt{2\pi}} m_{2,\omega}(x,\xi) e^{i x \xi} \label{eqn:m-2,omega}.
\end{align}
Then we have for every $F, G \in \calS(\bbR) \times \calS(\bbR)$ that 
\begin{equation}
	\begin{split} \label{eqn: e-itH Pe pairing formula}
		\langle e^{it\calH(\omega)}\Pe F,G\rangle &= \int_\bbR e^{it(\xi^2+\omega)}\langle F,\sigma_3 \Psi_{+,\omega}(\cdot,\xi)\rangle \overline{\langle G, \Psi_{+,\omega}(\cdot,\xi)\rangle} \,\ud \xi\\
		&\quad - \int_\bbR e^{-it(\xi^2+\omega)}\langle F,\sigma_3 \Psi_{-,\omega}(\cdot,\xi)\rangle \overline{\langle G, \Psi_{-,\omega}(\cdot,\xi)\rangle} \,\ud \xi,
	\end{split}
\end{equation}
and
\begin{equation} \label{eqn: Pe pairing formula}
	\langle \Pe F,G \rangle = \int_\bbR \langle F,\sigma_3 \Psi_{+,\omega}(\cdot,\xi)\rangle \overline{\langle G, \Psi_{+,\omega}(\cdot,\xi)\rangle} \,\ud \xi - \int_\bbR \langle F,\sigma_3 \Psi_{-,\omega}(\cdot,\xi)\rangle \overline{\langle G, \Psi_{-,\omega}(\cdot,\xi)\rangle} \,\ud \xi.
\end{equation}
These integrals are absolutely convergent because the integrands are rapidly decaying.
\end{proposition}
\begin{proof}
	We proceed along the lines of the proofs of Proposition~6.9 and Corollary~6.10 in \cite{KS06}. 
    From the representation formula \eqref{eqn:representation formula for e-itH} we obtain after a change of variables
	\begin{equation*}
		\begin{aligned}
		&\langle e^{it\calH(\omega)} \Pe F,G \rangle \\ 
        &= \frac{1}{2\pi i} \left(\int_{\omega}^\infty + \int_{-\infty}^{-\omega}\right) e^{it\lambda} \bigl\langle \bigl[\big(\calH - (\lambda+i0)\big)^{-1} - \big(\calH - (\lambda-i0)\big)^{-1} \bigr] F, G \bigr\rangle \, \ud \lambda \\
		&= \frac{1}{2\pi i} \int_0^\infty e^{it(\xi^2+\omega)} 2 \xi \bigl\langle \bigl[ \big(\calH-(\xi^2+\omega+i0)\big)^{-1}-\big(\calH-(\xi^2+\omega-i0) \big)^{-1}\bigr] F, G \bigr\rangle \, \ud \xi \\
		&\quad+ \frac{1}{2\pi i} \int_{0}^\infty e^{-it(\xi^2+\omega)}2\xi \bigl\langle \bigl[ \big(\calH-(-\xi^2-\omega+i0)\big)^{-1}-\big(\calH-(-\xi^2-\omega-i0)\big)^{-1} \bigr] F, G \bigr\rangle \, \ud \xi.
		\end{aligned}
	\end{equation*}
    Inserting the identity \eqref{equ:jump_formula} from Corollary~\ref{corollary: jump formula} for the jump of the resolvent and observing how the definition \eqref{eqn: Psi_pm_omega} of $\Psi_{\pm,\omega}(x,\xi)$ is related to the entries of $\bbE_\omega(x,\xi)$, we find for $\xi \geq 0$ that
	\begin{equation*}
		\begin{split}
			&\frac{1}{2\pi i} \int_0^\infty e^{it(\xi^2+\omega)} 2\xi \langle [\big(\calH-(\xi^2+\omega+i0)\big)^{-1}-\big(\calH-(\xi^2+\omega-i0)\big)^{-1}]F,G\rangle \, \ud \xi\\
			&= \int_0^\infty e^{it(\xi^2+\omega)}\int_\bbR \barG(x) \cdot \int_\bbR \begin{bmatrix}\Psi_{+,\omega}(x,\xi) & \Psi_{+,\omega}(x,-\xi)\end{bmatrix} \\ 
            &\qquad \qquad \qquad \qquad \qquad \qquad \qquad \times \begin{bmatrix} \Psi_{+,\omega}(y,\xi) & \Psi_{+,\omega}(y,-\xi) \end{bmatrix}^* \sigma_3 F(y) \, \ud y \, \ud x \, \ud \xi \\
			&= \int_0^\infty e^{it(\xi^2+\omega)}\overline{\begin{bmatrix}
				\langle G, \Psi_{+,\omega}(\cdot,\xi)\rangle\\
				\langle G, \Psi_{+,\omega}(\cdot,-\xi)\rangle
			\end{bmatrix}} \cdot \begin{bmatrix}
				\langle F,\sigma_3 \Psi_{+,\omega}(\cdot,\xi)\rangle \\
				\langle F,\sigma_3 \Psi_{+,\omega}(\cdot,-\xi)\rangle
			\end{bmatrix} \, \ud \xi \\
		      &= \int_\bbR e^{it(\xi^2+\omega)}\langle F,\sigma_3 \Psi_{+,\omega}(\cdot,\xi)\rangle \overline{\langle G, \Psi_{+,\omega}(\cdot,\xi)\rangle} \,\ud \xi.
		\end{split}
	\end{equation*}
    On the other hand, using the symmetries \eqref{eqn:symmetry of calH} and the anti-commutation identity $\sigma_1\sigma_3 = -\sigma_3\sigma_1$, we find that the jump of the resolvent across the negative part $(-\infty,-\omega]$ of the essential spectrum is given by 
    \begin{equation*}
	   \big(\calH-(-\xi^2-\omega+i0)\big)^{-1}(x,y)-\big(\calH-(-\xi^2-\omega-i0)\big)^{-1}(x,y) = \frac{1}{2i \xi} \sigma_1 \bbE_\omega(x,\xi) \bbE_{\omega}^*(y,\xi)\sigma_1\sigma_3,
    \end{equation*}
    whence
    \begin{equation*}
	   \begin{split}
        &\frac{1}{2\pi i} \int_{0}^\infty e^{-it(\xi^2+\omega)} 2\xi \bigl\langle \bigl[ \big(\calH-(-\xi^2-\omega+i0)\big)^{-1} - \big(\calH-(-\xi^2-\omega-i0)\big)^{-1} \bigr] F, G \bigr\rangle \, \ud \xi \\
        &= - \int_0^\infty e^{-it(\xi^2+\omega)} \int_\bbR \barG(x) \cdot \int_\bbR \sigma_1 \begin{bmatrix} \Psi_{+,\omega}(x,\xi) & \Psi_{+,\omega}(x,-\xi)\end{bmatrix} \\ 
        &\qquad \qquad \qquad \qquad \qquad \qquad \qquad \qquad \times \begin{bmatrix} \Psi_{+,\omega}(y,\xi) & \Psi_{+,\omega}(y,-\xi)\end{bmatrix}^* \sigma_1\sigma_3 F(y) \, \ud y \, \ud x \, \ud \xi \\
        &= -\int_\bbR e^{-it(\xi^2+\omega)} \langle F,\sigma_3 \sigma_1\Psi_{+,\omega}(\cdot,\xi)\rangle \overline{\langle G, \sigma_1\Psi_{+,\omega}(\cdot,\xi)\rangle} \,\ud \xi \\
        &=- \int_\bbR e^{-it(\xi^2+\omega)}\langle F,\sigma_3 \Psi_{-,\omega}(\cdot,\xi)\rangle \overline{\langle G, \Psi_{-,\omega}(\cdot,\xi)\rangle} \,\ud \xi.
	   \end{split}
    \end{equation*}
    Combining the preceding identities yields \eqref{eqn: e-itH Pe pairing formula}. The identity \eqref{eqn: Pe pairing formula} follows from \eqref{eqn: e-itH Pe pairing formula} by evaluating at time $t=0$. 
    Using that $\calH^m \Psi_{\pm, \omega}(\cdot, \xi) = (\pm 1)^m (\xi^2+\omega)^m \Psi_{\pm,\omega}(\cdot, \xi)$,
    we conclude the rapid decay of the integrands in $\xi$.
\end{proof}

We view the functions $\Psi_{\pm, \omega}(x,\xi)$ as the distorted Fourier basis for the operator $\calH(\omega)$. 
They are generalized eigenfunctions for the essential spectrum in the sense that 
\begin{equation*}
    \calH(\omega) \Psi_{\pm,\omega}(x,\xi) = \pm (\xi^2+\omega)\Psi_{\pm,\omega}(x,\xi).
\end{equation*}
In view of the representation formulas \eqref{eqn: e-itH Pe pairing formula} and \eqref{eqn: Pe pairing formula}, we define the distorted Fourier transform relative to $\calH(\omega)$ of a vector-valued function $F \in \calS(\bbR) \times \calS(\bbR)$ by
\begin{equation} \label{eqn:def-dFT}
	\wtilcalF_\omega[F](\xi) := \begin{bmatrix}\wtilcalF_{+, \omega}[F](\xi)\\ \wtilcalF_{-, \omega}[F](\xi)\end{bmatrix} := \begin{bmatrix}
		\langle F,\sigma_3 \Psi_{+,\omega}(\cdot,\xi)\rangle \\ \langle F,\sigma_3 \Psi_{-,\omega}(\cdot,\xi)\rangle
	\end{bmatrix}.
\end{equation}
By duality, we infer from Proposition~\ref{prop: representation formula} that for any $F \in \calS(\bbR)\times\calS(\bbR)$
\begin{equation} \label{eqn: representation e-itH}
    \begin{aligned}
        (e^{it\calH(\omega)}\Pe F)(x) &= \int_\bbR e^{it(\xi^2+\omega)}\wtilcalF_{+,\omega}[F](\xi)\Psi_{+,\omega}(x,\xi)\,\ud\xi \\
        &\quad - \int_\bbR e^{-it(\xi^2+\omega)}\wtilcalF_{-,\omega}[F](\xi)\Psi_{-,\omega}(x,\xi)\,\ud\xi,
    \end{aligned}
\end{equation}
and
\begin{equation} \label{eqn: representation Pe}
	(\Pe F)(x) =   \int_\bbR \wtilcalF_{+,\omega}[F](\xi)\Psi_{+,\omega}(x,\xi)\,\ud\xi -\int_\bbR \wtilcalF_{-,\omega}[F](\xi)\Psi_{-,\omega}(x,\xi)\,\ud\xi.
\end{equation}
In view of \eqref{eqn: representation e-itH} and \eqref{eqn: representation Pe}, we define the inverse of the distorted Fourier transform relative to $\calH(\omega)$ of a vector-valued function $G = (G_1, G_2)^{\top} \in \calS(\bbR) \times \calS(\bbR)$ by
\begin{equation} \label{equ:def-inverse-dFT}
    \wtilcalF_\omega^{-1}\bigl[G\bigr](x) := \int_\bbR G_1(\xi) \Psi_{+,\omega}(x,\xi) \, \ud \xi - \int_\bbR G_2(\xi) \Psi_{-,\omega}(x,\xi) \, \ud \xi.
\end{equation}
Then we can write \eqref{eqn: representation Pe} more succinctly as $P_{\mathrm{e}} = \wtilcalF_\omega^{-1} \wtilcalF_\omega$.

Next, we collect several basic properties of the distorted Fourier transform and its inverse.
We begin with the following orthogonality relations for the distorted Fourier basis elements.

\begin{lemma}\label{lemma:orthogonality of basis} 
Fix $\omega \in (0,\infty)$. Then we have in the sense of distributions,
	\begin{align}
		&\langle \Psi_{+,\omega}(\cdot,\xi),\sigma_3 \Psi_{+,\omega}(\cdot,\eta)\rangle = \delta_0(\xi - \eta),\label{eqn: ++ orthogonality}\\
		&\langle \Psi_{-,\omega}(\cdot,\xi),\sigma_3 \Psi_{-,\omega}(\cdot,\eta)\rangle = -\delta_0(\xi - \eta),\label{eqn: -- orthogonality}\\
		&\langle \Psi_{+,\omega}(\cdot,\xi),\sigma_3 \Psi_{-,\omega}(\cdot,\eta)\rangle = 	0,\label{eqn: +- orthogonality}\\	&\langle \Psi_{-,\omega}(\cdot,\xi),\sigma_3 \Psi_{+,\omega}(\cdot,\eta)\rangle = 0.\label{eqn: -+ orthogonality}
	\end{align}
\end{lemma}
\begin{proof}
It suffices to prove these identities for $\omega\equiv1$ thanks to the rescaling identity $\Psi_{\pm,\omega}(x,\xi) = \Psi_{\pm, 1}(\omega^{\frac{1}{2}}x,\omega^{-\frac{1}{2}}\xi)$.
Using \eqref{eqn:m-1,omega}--\eqref{eqn:m-2,omega} we compute in the sense of distributions 
\begin{equation*}
	\langle \Psi_{+, 1}(\cdot,\xi),\sigma_3 \Psi_{+, 1}(\cdot,\eta)\rangle = \frac{1}{2\pi}\int_\bbR e^{ix(\xi - \eta)}\big(m_{1}(x,\xi)\overline{m_{1}(x,\eta)}-m_{2}(x,\xi)\overline{m_{2}(x,\eta)}\big)  \ud x,
\end{equation*}
where we use the short-hand notation $m_1 \equiv m_{1,1}$ and $m_2 \equiv m_{2,1}$. 
Using that $\sech^2 + \tanh^2 = 1$, we may further simplify the integrand by writing
\begin{equation*}
	\begin{split}
	&m_{1}(x,\xi)\overline{m_{1}(x,\eta)}-m_{2}(x,\xi)\overline{m_{2}(x,\eta)}\\
	&=(\vert \xi \vert - i)^{-2}(\vert \eta \vert+i)^{-2}\left( (\xi+i\tanh(x))^2(\eta-i\tanh(x))^2 - \sech^4(x) \right)\\
	&=(\vert \xi \vert - i)^{-2}(\vert \eta \vert+i)^{-2} (\xi^2\eta^2+4\xi\eta - \xi^2 - \eta^2+1) \\
	&\quad -2i(\vert \xi \vert - i)^{-2}(\vert \eta \vert+i)^{-2}(\xi-\eta)(\xi\eta+1) \tanh(x)\\
	&\quad + 2i(\vert \xi \vert - i)^{-2}(\vert \eta \vert+i)^{-2}(\xi-\eta)\sech^2(x)\tanh(x)\\
	&\quad + (\vert \xi \vert - i)^{-2}(\vert \eta \vert+i)^{-2}(\xi^2+\eta^2-4\xi\eta-2) \sech^2(x).
	\end{split}
\end{equation*}
Invoking the Fourier transform identities \eqref{equ:preliminaries_FT_one}, \eqref{equ:preliminaries_FT_tanh} as well as \eqref{eqn: FT-sechsech}, \eqref{equ:FT_sech2tanh} from Appendix~\ref{appendix-FT-formulas}, we find
\begin{equation*}
	\begin{split}
		&\langle \Psi_{+, 1}(\cdot,\xi),\sigma_3 \Psi_{+, 1}(\cdot,\eta)\rangle \\
		&= (\vert \xi \vert - i)^{-2}(\vert \eta \vert+i)^{-2} (\xi^2\eta^2+4\xi\eta - \xi^2 - \eta^2+1) \delta_0(\xi - \eta)\\
		&\quad +\frac{1}{2} (\vert \xi \vert - i)^{-2}(\vert \eta \vert+i)^{-2}\underbrace{\left(2\xi\eta+2-(\xi-\eta)^2+(\xi^2+\eta^2-4\xi\eta-2)\right)}_{= \, 0} (\xi-\eta) \cosech\Bigl(\frac{\pi}{2} (\xi-\eta)\Bigr) \\
		&= \delta_0(\xi-\eta).
	\end{split}
\end{equation*}
This completes the proof of \eqref{eqn: ++ orthogonality}.

We proceed similarly to prove \eqref{eqn: +- orthogonality}. Here we have
\begin{equation*}
\langle \Psi_{+, 1}(\cdot,\xi),\sigma_3 \Psi_{-, 1}(\cdot,\eta)\rangle = \frac{1}{2\pi}\int_\bbR e^{ix(\xi - \eta)}\big(m_{1}(x,\xi)\overline{m_{2}(x,\eta)}-m_{2}(x,\xi)\overline{m_{1}(x,\eta)}\big)  \ud x,
\end{equation*}
where
\begin{equation*}
	\begin{split}
        &m_{1}(x,\xi)\overline{m_{2}(x,\eta)}-m_{2}(x,\xi)\overline{m_{1}(x,\eta)} \\
        &= (\vert \xi \vert - i)^{-2}(\vert \eta \vert+i)^{-2}\left(\sech^2(x)(\xi+i\tanh(x))^2 - \sech^2(x)(\eta - i \tanh(x))^2 \right) \\
        &=(\vert \xi \vert - i)^{-2}(\vert \eta \vert+i)^{-2}(\xi^2-\eta^2)\sech^2(x) \\
        &\quad +2i(\vert \xi \vert - i)^{-2}(\vert \eta \vert+i)^{-2}(\xi+\eta)\sech^2(x)\tanh(x).
	\end{split}
\end{equation*}
Invoking \eqref{eqn: FT-sechsech} and \eqref{equ:FT_sech2tanh} from  Appendix~\ref{appendix-FT-formulas}, we conclude
\begin{equation*}
	\begin{split}
        \langle \Psi_{+, 1}(\cdot,\xi),\sigma_3 \Psi_{-, 1}(\cdot,\eta)\rangle
        &=\frac{1}{2}(\vert \xi \vert - i)^{-2}(\vert \eta \vert+i)^{-2} (\xi^2-\eta^2)(\xi-\eta) \cosech\Bigl(\frac{\pi}{2} (\xi-\eta)\Bigr) \\
        &\quad + \frac{1}{2}(\vert \xi \vert - i)^{-2}(\vert \eta \vert+i)^{-2} (\xi + \eta)(\eta-\xi)(\xi-\eta) \cosech\Bigl(\frac{\pi}{2} (\xi-\eta)\Bigr) \\
        &\equiv 0.
	\end{split}
\end{equation*}
This completes the proof of \eqref{eqn: +- orthogonality}.

Finally, the identities \eqref{eqn: -- orthogonality} and \eqref{eqn: -+ orthogonality} follow from \eqref{eqn: ++ orthogonality} and \eqref{eqn: +- orthogonality} respectively by using the definition $\Psi_{-,\omega} = \sigma_1 \Psi_{+,\omega}$ and the anti-commutation identity $\sigma_3 \sigma_1 = -\sigma_1 \sigma_3$.
\end{proof}

Using Lemma~\ref{lemma:orthogonality of basis}, we determine how the distorted Fourier transform acts on the projections $P_{\mathrm{e}}$ and $P_{\mathrm{d}}$. Additionally, we obtain representation formulas for the evolution $e^{it\calH(\omega)}$ on the distorted Fourier side.

\begin{corollary} \label{cor:distFT_of_propagator}
Fix $\omega \in (0,\infty)$. We have 
\begin{equation} \label{equ:wtilcalF_applied_to_P}
    \wtilcalF_{\pm, \omega} P_{\mathrm{e}} = \wtilcalF_{\pm, \omega}, \quad \wtilcalF_{\pm, \omega} P_{\mathrm{d}} = 0,
\end{equation}
and the evolution $e^{it\calH(\omega)}$ satisfies
\begin{align}
\wtilcalF_{+, \omega}[e^{it\calH(\omega)}\Pe F](\xi) &= e^{ it(\xi^2+\omega)}\wtilcalF_{+, \omega}[F](\xi),\label{eqn: diag+}\\
\wtilcalF_{-, \omega}[e^{it\calH(\omega)}\Pe F](\xi) &= e^{- it(\xi^2+\omega)}\wtilcalF_{-, \omega}[F](\xi).	\label{eqn: diag-}
\end{align}
\end{corollary}
\begin{proof}
We first prove \eqref{eqn: diag+} and \eqref{eqn: diag-}.
By taking the inner product of \eqref{eqn: representation e-itH} with $\sigma_3\Psi_{+,\omega}(\cdot,\xi)$, we infer \eqref{eqn: diag+} from \eqref{eqn: ++ orthogonality} and \eqref{eqn: +- orthogonality}. We derive \eqref{eqn: diag-} analogously using \eqref{eqn: -- orthogonality} and \eqref{eqn: -+ orthogonality}.
Finally, evaluating \eqref{eqn: diag+} and \eqref{eqn: diag-} at $t=0$ gives $\wtilcalF_{\pm,\omega} P_{\mathrm{e}} = \wtilcalF_{\pm, \omega}$, which then implies $\wtilcalF_{\pm,\omega} P_{\mathrm{d}} = 0$.
\end{proof}

Next, we derive expressions for the action of $\sigma_3$ and of $\px$ on the distorted Fourier side.

\begin{lemma} \label{lem:distFT_applied_to_sigmathree_F}
Fix $\omega \in (0,\infty)$. 
Let $F= (f_1,f_2)^\top \in \calS(\bbR) \times \calS(\bbR)$.	We have
\begin{align*}
\wtilcalF_{+, \omega}[\sigma_3 F] &=\wtilcalF_{+, \omega}[ F] + \calL_{f_2,\omega},	\\
\wtilcalF_{-, \omega}[\sigma_3 F] &= -\wtilcalF_{-, \omega}[ F] + \calL_{f_1,\omega},
\end{align*}
where
\begin{equation*}
 \calL_{f,\omega}(\xi) := 2\langle f,\Psi_{2,\omega}(\cdot,\xi)\rangle.
\end{equation*}
Moreover, we have 
\begin{align*}
    \wtilcalF_{+,\omega}[\px F] &= i \xi \wtilcalF_{+,\omega}[F] + \calK_{+,\omega}[F],  \\
    \wtilcalF_{-,\omega}[\px F] &= i \xi \wtilcalF_{-,\omega}[F] + \calK_{-,\omega}[F],
\end{align*}
where 
\begin{align*}
    \calK_{+,\omega}[F](\xi) &:= \frac{1}{\sqrt{2\pi}}\langle f_1,e^{ix\xi}\px m_{1,\omega}(x,\xi)\rangle -  \frac{1}{\sqrt{2\pi}}\langle f_2,e^{ix\xi}\px m_{2,\omega}(x,\xi)\rangle,\\
    \calK_{-,\omega}[F](\xi) &:= \frac{1}{\sqrt{2\pi}}\langle f_1,e^{ix\xi}\px m_{2,\omega}(x,\xi)\rangle -  \frac{1}{\sqrt{2\pi}}\langle f_2,e^{ix\xi}\px m_{1,\omega}(x,\xi)\rangle.
\end{align*}
\end{lemma}
\begin{proof}
    The asserted identities follow directly from the definitions \eqref{eqn:m-1,omega}, \eqref{eqn:m-2,omega}, and \eqref{eqn:def-dFT}. To compute the action of $\px$ we integrate by parts.
\end{proof}

Finally, we relate $\wtilcalF_{+,\omega}$ and $\wtilcalF_{-,\omega}$ under complex conjugation. 

\begin{lemma} \label{lem:distFT_components_relation}
Fix $\omega \in (0,\infty)$. 
Let $F= (f,\barf)^\top$ for $f \in \calS(\bbR)$. Then we have
\begin{equation} \label{equ:distFT_components_relation}
	\wtilcalF_{+,\omega}[F](\xi) = - \frac{(\vert \xi \vert - i \sqrt{\omega})^2}{(\vert \xi \vert + i \sqrt{\omega})^2} \overline{\wtilcalF_{-, \omega}[F](-\xi)}.
\end{equation}
\end{lemma}
\begin{proof}
We start from the right hand side, viz.,
\begin{equation*}
	\begin{split}
        - \frac{(\vert \xi \vert - i \sqrt{\omega})^2}{(\vert \xi \vert + i \sqrt{\omega})^2} \overline{\wtilcalF_{-, \omega}[F](-\xi)} &= - \frac{(\vert \xi \vert - i \sqrt{\omega})^2}{(\vert \xi \vert + i \sqrt{\omega})^2} \overline{\langle F,\sigma_3 \Psi_{-,\omega}(\cdot,-\xi)\rangle } \\
        &= - \frac{(\vert \xi \vert - i \sqrt{\omega})^2}{(\vert \xi \vert + i \sqrt{\omega})^2} \int_\bbR\left( \barf(x)\Psi_{2,\omega}(x,-\xi)-f(x)\Psi_{1,\omega}(x,-\xi) \right)\ud x\\
        &= \int_\bbR \left(f(x)\overline{\Psi_{1,\omega}(x,\xi)}- \barf(x)\overline{\Psi_{2,\omega}(x,\xi)}\right) \,\ud x\\
        &= \wtilcalF_{+, \omega}[F](\xi).
	\end{split}
\end{equation*}
Note that the penultimate line follows from \eqref{eqn:m-1,omega}--\eqref{eqn:m-2,omega}.
\end{proof}

\begin{remark}
    We exploit the simple relationship \eqref{equ:distFT_components_relation} between the components of the distorted Fourier transform in the nonlinear analysis in this paper. In the generic case (no threshold resonances) the corresponding relation between the components of the distorted Fourier transform is more involved, see Lemma~5.12 in \cite{CollotGermain23}.
\end{remark}

\subsection{Mapping properties}

In this subsection, we record some bounds for $m_{j,\omega}(x,\xi)$, $j \in \{1,2\}$, and we derive mapping properties for the distorted Fourier transform.

\begin{lemma} \label{lemma: PDO on m12}
For a fixed $\omega \in (0,\infty)$, we have for any $j \in \bbN_0$ and any $k\in \{0,1\}$ that
\begin{equation}\label{equ:PDO_on_m12}
	\sup_{x,\xi \in \bbR} \, \bigl|\partial_x^j \partial_\xi^k m_{1,\omega}(x,\xi)\bigr| + \sup_{x,\xi \in \bbR} \, \bigl|\partial_x^j \partial_\xi^k m_{2,\omega}(x,\xi)\bigr| \lesssim_{\omega,j,k} 1.
\end{equation}
Moreover, both $m_{1,\omega}(x,\xi)$ and $m_{2,\omega}(x,\xi)$ can be written as sums of tensorized terms of the form $a(x)b(\xi)$, where $a(x)$ is a smooth bounded function and $b(\xi)$ is a bounded Lipschitz function.
\end{lemma}
\begin{proof}
These observations can be inferred directly from the explicit expressions \eqref{eqn:m-1,omega}--\eqref{eqn:m-2,omega}.
\end{proof}

We obtain the following mapping properties of the distorted Fourier transform.

\begin{proposition} \label{prop:mapping_properties_dist_FT}
	The distorted Fourier transform \eqref{eqn:def-dFT} and its inverse \eqref{equ:def-inverse-dFT} are bounded from $L^p$ to $L^{p'}$ for $1 \leq p \leq 2$, and from $H^1$ to $L^{2,1}$ and from $L^{2,1}$ to $H^1$.
\end{proposition}
\begin{proof}
    In view of the tensorized structure of $m_{1,\omega}(x,\xi)$ and $m_{2,\omega}(x,\xi)$ observed in Lemma~\ref{lemma: PDO on m12}, it suffices to study the mapping properties of an operator of the form
	\begin{equation*}
		T_{a,b}[f](\xi) := \int_\bbR e^{ix\xi}a(x)b(\xi)f(x)\,\ud x,
	\end{equation*}
	where $a,b \in W^{1,\infty}(\bbR)$. Indeed, the expressions  \eqref{eqn:m-1,omega}--\eqref{eqn:m-2,omega} indicate that the distorted Fourier transform \eqref{eqn:def-dFT} can be written as sums of such operators. The $L_x^1$ to $L_\xi^\infty$ bound then follows from taking absolute values, while the $L_x^2$ to $L_\xi^2$ follows from an application of Plancherel's identity
	\begin{equation*}
		\left \Vert \int_\bbR e^{ix\xi}a(x)b(\xi)f(x)\,\ud x \right \Vert_{L_\xi^2} \lesssim \Vert b \Vert_{L_\xi^\infty}\Vert a(x)f(x)\Vert_{L_x^2} \lesssim \Vert f \Vert_{L_x^2}.
	\end{equation*}
	Hence, we obtain $L_x^p$ to $L_\xi^{p'}$ boundedness by interpolation. Next, we consider the $H_x^1$ to $L_\xi^{2,1}$ bound. Using integration by parts, we have
	\begin{equation*}
		\xi \int_\bbR e^{ix\xi}a(x)b(\xi)f(x)\,\ud x = ib(\xi) \int_\bbR e^{ix \xi} \bigl( a(x)\partial_xf(x) + \partial_x a(x) f(x) \bigr) \, \ud x.
	\end{equation*}
	By Plancherel's identity, we find that
	\begin{equation*}
		\begin{split}
			\Vert \xi T_{a,b}[f](\xi)\Vert_{L_\xi^2} \lesssim \Vert b \Vert_{L_\xi^\infty} \left(\Vert a(x) \partial_x f(x)\Vert_{L_x^2} + \Vert \partial_x a(x) f(x) \Vert_{L_x^2} \right)\lesssim \Vert f \Vert_{H_x^1},
		\end{split}
	\end{equation*}
	and we conclude that
	\begin{equation*}
		\Vert T_{a,b}[f]\Vert_{L_\xi^{2,1}}\lesssim \Vert T_{a,b}[f]\Vert_{L_\xi^{2}} + \Vert \xi T_{a,b}[f]\Vert_{L_\xi^{2}} \lesssim \Vert f \Vert_{H_x^1}.
	\end{equation*}
	Finally, we look at the $L_x^{2,1}$ to $H_\xi^1$ bound. Taking a derivative, we have
	\begin{equation*}
		\partial_\xi T_{a,b}[f](\xi) =  \int_\bbR e^{ix\xi}\big(ixb(\xi)+\partial_\xi b(\xi)\big)a(x)f(x) \,\ud x,
	\end{equation*}
	which implies that
	\begin{equation*}
		\Vert \partial_\xi T_{a,b}[f](\xi) \Vert_{L_\xi^2} \lesssim \Vert b \Vert_{L_\xi^\infty}\Vert x a(x)f(x) \Vert_{L_x^2} + \Vert \partial_\xi b \Vert_{L_\xi^\infty} \Vert a(x)f(x)\Vert_{L_x^2} \lesssim \Vert f \Vert_{L_x^{2,1}}.
	\end{equation*}
	From here, we can conclude the $L_x^{2,1}$ to $H_\xi^1$ bound. 
    
    The corresponding mapping properties of the inverse distorted Fourier transform \eqref{equ:def-inverse-dFT} can be proved similarly.
\end{proof}

Lastly, we record below the explicit expressions for the components of the rescaled vector-valued functions
$\bfF(\sqrt{\omega}x, \frac{\xi}{\sqrt{\omega}})$ and $\bfG(\sqrt{\omega}x, \frac{\xi}{\sqrt{\omega}})$ with 
$\bfF(x,\xi)$ and $\bfG(x,\xi)$ defined in \eqref{equ:definition_bfF_bfG},
\begin{align}
\bfF_{1,\omega}(x,\xi) &:= \frac{1}{\sqrt{2\pi}} \frac{1}{(\xi - i\sqrt{\omega})^2}\big(\xi+i\sqrt{\omega}\tanh(\sqrt{\omega}x)\big)^2e^{ix\xi} \equiv \frakm_{1,\omega}(x,\xi)e^{ix\xi},\label{eqn:bfF_1omega}\\
\bfF_{2,\omega}(x,\xi) &:= \frac{1}{\sqrt{2\pi}} \frac{\omega}{(\xi - i\sqrt{\omega})^2}\sech^2(\sqrt{\omega}x)e^{ix\xi} \equiv \frakm_{2,\omega}(x,\xi)e^{ix\xi},\label{eqn:bfF_2omega}\\
\bfG_{1,\omega}(x,\xi) &:= \frac{1}{\sqrt{2\pi}} \frac{1}{(\xi - i\sqrt{\omega})^2}\big(\xi-i\sqrt{\omega}\tanh(\sqrt{\omega}x)\big)^2e^{-ix\xi} \equiv \frakm_{1,\omega}(-x,\xi)e^{-ix\xi},\label{eqn:bfG_1omega}\\
\bfG_{2,\omega}(x,\xi) &:= \frac{1}{\sqrt{2\pi}} \frac{\omega}{(\xi - i\sqrt{\omega})^2}\sech^2(\sqrt{\omega}x)e^{-ix\xi} \equiv \frakm_{2,\omega}(-x,\xi)e^{-ix\xi}.\label{eqn:bfG_2omega}
\end{align}

The following lemma will be used to prove improved local decay estimates in Lemma~\ref{lem:improved_local_decay_difference_Psis}.

\begin{lemma} \label{lemma: PDO bounds on F and G}
	Fix $\omega \in (0,\infty)$. For $j = 1, 2$ and for any integer $\ell \geq 0$ we have uniformly for all $x \in \bbR$ and all $\xi \in \bbR$ that
	\begin{align}
		\vert \partial_\xi^\ell  \bfF_{j,\omega} (x,\xi) \vert + \vert  \partial_\xi^\ell  \bfG_{j,\omega} (x,\xi) \vert &\lesssim \jx^\ell, \label{eqn: bound on partial_xi calF} \\
		\vert \partial_\xi^\ell \px \bfF_{j,\omega} (x,\xi) \vert + \vert  \partial_\xi^\ell \px  \bfG_{j,\omega} (x,\xi) \vert &\lesssim \jxi \jx^\ell, \label{eqn: bound on partial_xi partial_x calF} \\
		\left\vert \left(\frac{\bfF_{j,\omega} (x,\xi)- \bfF_{j,\omega} (x,0)}{\xi}\right) \right\vert + \left\vert \left(\frac{\bfG_{j,\omega}(x,\xi) - \bfG_{j,\omega}(x,0)}{\xi}\right)\right\vert &\lesssim \jx, \label{eqn: bound on calF difference quotient} \\
		\left\vert  \partial_\xi \left(\frac{\bfF_{j,\omega} (x,\xi)- \bfF_{j,\omega} (x,0)}{\xi}\right) \right\vert + \left\vert \partial_\xi \left(\frac{\bfG_{j,\omega}(x,\xi) - \bfG_{j,\omega}(x,0)}{\xi}\right)\right\vert &\lesssim \jx^{2}, \label{eqn: bound on partial_xi calF difference quotient} \\
		\left\vert  \partial_\xi  \partial_x \left(\frac{\bfF_{j,\omega} (x,\xi)- \bfF_{j,\omega} (x,0)}{\xi}\right) \right\vert+	\left\vert \partial_\xi  \partial_x  \left(\frac{\bfG_{j,\omega}(x,\xi) - \bfG_{j,\omega}(x,0)}{\xi}\right)\right\vert &\lesssim \jx^2. \label{eqn: bound on partial_xi partial_x calF difference quotient}
	\end{align}
\end{lemma}
\begin{proof} 
By inspection, the symbol $\frakm_{j,\omega}$, $j=1,2$ are smooth in both $x$ and $\xi$ variables and satisfy the uniform bound
	\begin{equation}\label{eqn:bound on frakm-jomega}
		\sup_{x, \xi \in \bbR} \vert \partial_\xi^k \partial_x^\ell \frakm_{1,\omega}(x,\xi)\vert +		\sup_{x, \xi \in \bbR} \vert \partial_\xi^k \partial_x^\ell \frakm_{2,\omega}(x,\xi)\vert \lesssim_{\omega,k,\ell} 1,
	\end{equation}
for any $k,\ell \in \bbN_0$. The estimates \eqref{eqn: bound on partial_xi calF} and \eqref{eqn: bound on partial_xi partial_x calF} hence follows. We prove the remaining asserted bounds \eqref{eqn: bound on calF difference quotient}, \eqref{eqn: bound on partial_xi calF difference quotient}, \eqref{eqn: bound on partial_xi partial_x calF difference quotient} for the terms involving $\bfF_{j,\omega}$, $j = 1,2$. The corresponding proofs for the terms with $\bfG_{j,\omega}$, $j = 1,2$, essentially proceed identically.
	In order to establish \eqref{eqn: bound on calF difference quotient} we use the fundamental theorem of calculus to write for $\xi \neq 0$,
	\begin{equation*}
		\frac{1}{\xi} \bigl( \bfF_{j,\omega}(x,\xi) - \bfF_{j,\omega}(x,0) \bigr) = \frac{1}{\xi} \int_0^\xi \partial_\eta \bfF_{j,\omega}(x,\eta) \, \ud \eta.
	\end{equation*}
	Then \eqref{eqn: bound on calF difference quotient} follows using \eqref{eqn: bound on partial_xi calF}.
	Similarly, we write for $\xi \neq 0$,
	\begin{equation*}
		\begin{aligned}
			\partial_\xi \left(\frac{\bfF_{j,\omega}(x,\xi)- \bfF_{j,\omega}(x,0)}{\xi}\right) &= \frac{1}{\xi^2} \Bigl( \xi (\pxi \bfF_{j,\omega})(x,\xi) - \bigl( \bfF_{j,\omega} (x,\xi)- \bfF_{j,\omega} (x,0) \bigr) \Bigr) \\
			&= \frac{1}{\xi^2} \int_0^\xi \eta (\partial_\eta^2 \bfF_{j,\omega})(x,\eta) \, \ud \eta.
		\end{aligned}
	\end{equation*}
	Hence, \eqref{eqn: bound on partial_xi calF difference quotient} follows using \eqref{eqn: bound on partial_xi calF}.
	Finally, for the proof of \eqref{eqn: bound on partial_xi partial_x calF difference quotient} we compute that
	\begin{equation*}
		\begin{aligned}
			&\partial_\xi  \partial_x \left(\frac{\bfF_{j,\omega} (x,\xi)- \bfF_{j,\omega} (x,0)}{\xi}\right) \\
			&= - x \frakm_{j,\omega}(x,\xi) e^{ix\xi} + i (\pxi \frakm_{j,\omega})(x,\xi) e^{ix\xi} + \pxi \biggl( \frac{(\px \frakm_{j, \omega})(x,\xi) e^{ix\xi} - (\px \frakm_{j, \omega})(x,0)}{\xi} \biggr) \\
			&= - x \frakm_{j,\omega}(x,\xi) e^{ix\xi} + i (\pxi \frakm_{j,\omega})(x,\xi) e^{ix\xi} + \frac{1}{\xi^2} \int_0^\xi \eta (\partial_\eta^2 \frakn_{j,\omega})(x,\eta) \, \ud \eta
		\end{aligned}
	\end{equation*}
	with
	\begin{equation*}
		\frakn_{j,\omega}(x,\eta) := (\px \frakm_{j, \omega})(x,\eta) e^{ix\eta}.
	\end{equation*}
	The asserted bound \eqref{eqn: bound on partial_xi partial_x calF difference quotient} now follows using \eqref{eqn:bound on frakm-jomega}.
\end{proof}

\section{Linear Decay Estimates} \label{sec:linear_decay}

In this section we recall several decay estimates for the linear Schr\"odinger evolution. We provide complete proofs for the convenience of the reader.

\subsection{Dispersive decay}

We begin with a dispersive decay estimate.

\begin{lemma} \label{lem:linear_dispersive_decay}
	Suppose $a \colon \bbR^2 \to \bbC$ satisfies
	\begin{equation}\label{eqn: lem4.1-assumption-a}
		\sup_{x,\xi \in \bbR} \, \bigl( |a(x,\xi)| + |\pxi a(x,\xi)| \bigr) \lesssim 1.
	\end{equation}
	Then there exists a constant $C \geq 1$ such that uniformly for all $t \geq 1$,
	\begin{equation}\label{eqn:linear_dispersive_decay}
		\biggl\| \int_\bbR e^{\pm i t \xi^2} e^{ix\xi} a(x,\xi) g(\xi) \, \ud \xi - \frac{\sqrt{\pi}}{t^{\frac12}} e^{\mp i \frac{x^2}{4t}} e^{\pm i \frac{\pi}{4}} a\Bigl(x, \mp \frac{x}{2t} \Bigr) g\Bigl(\mp\frac{x}{2t}\Bigr) \biggr\|_{L^\infty_x} \leq \frac{C}{t^{\frac34}} \| g\|_{H^1_\xi}.
	\end{equation}
	In particular, it follows that for all $t \geq 1$,
	\begin{equation*}
		\biggl\| \int_\bbR e^{\pm i t \xi^2} e^{ix\xi} a(x,\xi) g(\xi) \, \ud \xi \biggr\|_{L^\infty_x} \lesssim \frac{1}{t^{\frac12}} \|g\|_{L^\infty_\xi} + \frac{1}{t^{\frac34}} \| g(\xi)\|_{H^1_\xi}.
	\end{equation*}
\end{lemma}
\begin{proof}
In what follows we prove \eqref{eqn:linear_dispersive_decay} for the case of the phase $e^{it\xi^2}$. By applying the change of variables $\xi \mapsto \xi - \frac{x}{2t}$, we have  
\begin{equation*}
\int_\bbR e^{it\xi^2}e^{ix\xi} a(x,\xi)g(\xi)\, \ud \xi = e^{-i\frac{x^2}{4t}}\int_\bbR e^{it\xi^2}a(x,\xi - \tfrac{x}{2t})g(\xi - \tfrac{x}{2t})\,\ud \xi.    
\end{equation*}
Using the identity
\begin{equation*}
\int_\bbR e^{it\xi^2}\,\ud \xi = \frac{\sqrt{\pi}}{t^{\frac{1}{2}}}e^{i \frac{\pi}{4}},
\end{equation*}
we find that 
\begin{equation}\label{eqn: proof-lem4.1-1}
\int_\bbR e^{it\xi^2}e^{ix\xi} a(x,\xi)g(\xi)\, \ud \xi -\frac{\sqrt{\pi}}{t^{\frac{1}{2}}} e^{- i \frac{x^2}{4t}} e^{ i \frac{\pi}{4}} a\Bigl(x,- \frac{x}{2t} \Bigr) g\Bigl(-\frac{x}{2t}\Bigr)
= e^{-i\frac{x^2}{4t}}\int_\bbR e^{it\xi^2}h(t,x,\xi)\,\ud \xi ,
\end{equation}
where 
\begin{equation*}
h(t,x,\xi) := a(x,\xi - \tfrac{x}{2t})g(\xi - \tfrac{x}{2t})-a(x, - \tfrac{x}{2t})g(- \tfrac{x}{2t}).
\end{equation*}
We now establish \eqref{eqn:linear_dispersive_decay} by estimating the right-hand side of \eqref{eqn: proof-lem4.1-1}. We split the integral in \eqref{eqn: proof-lem4.1-1} onto two regions by writing
\begin{equation*}
\int_\bbR e^{it\xi^2}h(t,x,\xi) \,\ud \xi = \int_{[\vert \xi \vert \leq t^{-\frac{1}{2}}]} e^{it\xi^2}h(t,x,\xi) \,\ud \xi+\int_{[\vert \xi \vert > t^{-\frac{1}{2}}]} e^{it\xi^2}h(t,x,\xi) \,\ud \xi =:  I(t,x) + II(t,x).   
\end{equation*}
Using $h(t,x,0) = 0$, the fundamental theorem of calculus, and the Cauchy-Schwarz inequality, we obtain the uniform bound 
\begin{equation}\label{eqn: proof-lem4.1-2}
\vert h(t,x,\xi) \vert = \left \vert \int_0^\xi \partial_\eta h(t,x,\eta)\,\ud \eta \right \vert  \leq \vert \xi \vert^{\frac{1}{2}}\Vert \partial_\xi h(t,x,\cdot)\Vert_{L_\xi^2}.
\end{equation}
Then, by \eqref{eqn: lem4.1-assumption-a}, we find that 
\begin{equation*}
\begin{split}
\Vert I(t,x)\Vert_{L_x^\infty} \leq \int_{[\vert \xi \vert \leq t^{-\frac{1}{2}}]} \vert \xi \vert^{\frac{1}{2}} \cdot \Vert \partial_\xi h(t,x,\cdot)\Vert_{L_\xi^2} \,\ud \xi \lesssim t^{-\frac{3}{4}} \Vert \partial_\xi h(t,x,\cdot)\Vert_{L_\xi^2} \lesssim t^{-\frac{3}{4}}\Vert g \Vert_{H_\xi^1}.
\end{split}    
\end{equation*}
On $II(t,x)$, we integrate by parts and write 
\begin{equation*}
\begin{split}
    II(t,x) = \left[\frac{1}{2it} e^{it\xi^2} \frac{h(t,x,\xi)}{\xi}\right]_{\xi = - t^{-\frac{1}{2}}}^{\xi =  t^{-\frac{1}{2}}} - \frac{1}{2it}\int_{[\vert \xi \vert \geq t^{-\frac{1}{2}}]}e^{it\xi^2} \partial_\xi \left(\frac{h(t,x,\xi)}{\xi}\right) \, \ud \xi.
\end{split}
\end{equation*}
Using \eqref{eqn: proof-lem4.1-2} and \eqref{eqn: lem4.1-assumption-a}, we obtain that 
\begin{equation*}
    \left \vert \left[\frac{1}{2it} e^{it\xi^2} \frac{h(t,x,\xi)}{\xi}\right]_{\xi = - t^{-\frac{1}{2}}}^{\xi =  t^{-\frac{1}{2}}} \right \vert \lesssim t^{-\frac{3}{4}} \Vert \partial_\xi h(t,x,\cdot)\Vert_{L_\xi^2} \lesssim t^{-\frac{3}{4}}\Vert g \Vert_{H_\xi^1}.
\end{equation*}
Moreover, using \eqref{eqn: proof-lem4.1-2}, the Cauchy-Schwarz inequality, and \eqref{eqn: lem4.1-assumption-a}, we conclude that 
\begin{equation*}
\begin{split}
&\left \vert \frac{1}{2it}\int_{[\vert \xi \vert \geq t^{-\frac{1}{2}}]}e^{it\xi^2} \partial_\xi \left(\frac{h(t,x,\xi)}{\xi}\right) \, \ud \xi \right \vert \lesssim   t^{-1}\int_{[\vert \xi \vert \geq t^{-\frac{1}{2}}]} \left(\left \vert   \frac{h(t,x,\xi)}{\xi^2}\right\vert + \left \vert   \frac{\partial_\xi h(t,x,\xi)}{\xi}\right\vert \right) \, \ud \xi \\
&\lesssim t^{-1} \Vert \partial_\xi h(t,x,\cdot)\Vert_{L_\xi^2} \cdot \left(\int_{[\vert \xi \vert \geq t^{-\frac{1}{2}}]} \left \vert   \frac{\vert \xi \vert^{\frac{1}{2}}}{\xi^2}\right\vert \, \ud \xi + \left(\int_{[\vert \xi \vert \geq t^{-\frac{1}{2}}]} \left \vert   \frac{1}{\xi}\right\vert^2 \, \ud \xi \right)^{1/2} \right)\\
&\lesssim t^{-\frac{3}{4}}\Vert \partial_\xi h(x,\cdot)\Vert_{L_\xi^2}\lesssim t^{-\frac{3}{4}} \Vert g \Vert_{H_\xi^1}.
\end{split}
\end{equation*}
This finishes the proof of \eqref{eqn:linear_dispersive_decay}.
\end{proof}

\subsection{Improved local decay}

A sharp understanding of the leading order local decay behavior of the radiation is key at several places in the proof of Theorem~\ref{thm:main_theorem}. 
The next lemma effectively asserts that the Schr\"odinger waves enjoy improved local decay upon subtracting off the threshold resonances.

\begin{lemma} \label{lem:improved_local_decay_difference_Psis}
	Fix $\omega \in (0, \infty)$. Then there exists $C \geq 1$ such that for $j = 1,2$ we have uniformly for all $t \geq 0$ that
	\begin{equation} \label{equ:improved_local_decay_difference_Psis_no_px}
		\begin{aligned}
			\biggl\| \jx^{-2} \int_\bbR e^{\pm i t \xi^2} \bigl( \Psi_{j,\omega}(x,\xi) - \Psi_{j,\omega}(x,0) \bigr) g(\xi) \, \ud \xi \biggr\|_{L^\infty_x} \leq \frac{C}{\jt} \Bigl( \|g\|_{L^\infty_\xi} + \|\pxi g\|_{L^2_\xi} + \|\jxi g\|_{L^2_\xi} \Bigr)
		\end{aligned}
	\end{equation}
	and
	\begin{equation} \label{equ:improved_local_decay_difference_Psis_with_px}
		\begin{aligned}
			\biggl\| \jx^{-3} \int_\bbR e^{\pm i t \xi^2} \px \bigl( \Psi_{j,\omega}(x,\xi) - \Psi_{j,\omega}(x,0) \bigr) g(\xi) \, \ud \xi \biggr\|_{L^2_x} \leq \frac{C}{\jt} \Bigl( \|g\|_{L^\infty_\xi} + \|\pxi g\|_{L^2_\xi} + \|\jxi g\|_{L^2_\xi} \Bigr).
		\end{aligned}
	\end{equation}
\end{lemma}
\begin{proof}
	In what follows it suffices to consider the case of the phase $e^{it\xi^2}$, the other case being analogous.
	We begin with the proof of \eqref{equ:improved_local_decay_difference_Psis_no_px}.
	For short times $0 \leq t \leq 1$ we use \eqref{equ:PDO_on_m12} to obtain the crude bound
	\begin{equation*}
		\begin{split}
			\biggl\| \jx^{-2} \int_\bbR e^{\pm i t \xi^2} \bigl( \Psi_{j,\omega}(x,\xi) - \Psi_{j,\omega}(x,0) \bigr) g(\xi) \, \ud \xi \biggr\|_{L^\infty_x} \lesssim \Bigl( \sup_{x, \xi \in \bbR} \, \bigl| m_{j,\omega}(x,\xi)\bigr| \Bigr) \|g\|_{L_\xi^1} \lesssim \|\jxi g\|_{L_\xi^2}.
		\end{split}
	\end{equation*}
	For times $t \geq 1$, we split the integral into two parts
	\begin{equation*}
		\begin{aligned}
			\int_\bbR e^{it\xi^2} \bigl( \Psi_{j,\omega}(x,\xi) - \Psi_{j,\omega}(x,0) \bigr) g(\xi) \, \ud \xi &= \int_0^\infty e^{it\xi^2} \bigl( \Psi_{j,\omega}(x,\xi) - \Psi_{j,\omega}(x,0) \bigr) g(\xi) \, \ud \xi \\
			&\quad + \int_{-\infty}^0 e^{it\xi^2} \bigl( \Psi_{j,\omega}(x,\xi) - \Psi_{j,\omega}(x,0) \bigr) g(\xi) \, \ud \xi.
		\end{aligned}
	\end{equation*}
	Then we integrate by parts in the frequency variable $\xi$ in each integral.
	Recalling the definitions of $\bfF_{j,\omega}(x,\xi)$ and $\bfG_{j,\omega}(x,\xi)$ from \eqref{eqn:bfF_1omega}--\eqref{eqn:bfG_2omega},
	we find
	\begin{equation} \label{equ:improved_local_decay_differed_Psis_no_px_integral_pos}
		\begin{split}
			&\int_0^\infty e^{it\xi^2}  \bigl( \Psi_{j,\omega}(x,\xi) - \Psi_{j,\omega}(x,0) \bigr) g(\xi) \, \ud \xi \\
			&= \int_0^\infty e^{it\xi^2}  \bigl( \bfF_{j,\omega}(x,\xi) - \bfF_{j,\omega}(x,0) \bigr) g(\xi) \, \ud \xi \\
			&= \frac{1}{2it} \biggl[ e^{it\xi^2} \frac{\bfF_{j,\omega}(x,\xi)-\bfF_{j,\omega}(x,0)}{\xi}g(\xi) \biggr]_{\xi=0}^{\xi=\infty} \\
			&\quad - \frac{1}{2it}\int_0^\infty e^{it\xi^2} \partial_\xi \left(\frac{ \bfF_{j,\omega}(x,\xi) - \bfF_{j,\omega}(x,0) }{\xi}\right)g(\xi) \, \ud \xi \\
			&\quad - \frac{1}{2it}\int_0^\infty e^{it\xi^2}  \left(\frac{\bfF_{j,\omega}(x,\xi) - \bfF_{j,\omega}(x,0) }{\xi}\right) \partial_\xi g(\xi) \, \ud \xi,
		\end{split}
	\end{equation}
	and
	\begin{equation} \label{equ:improved_local_decay_differed_Psis_no_px_integral_neg}
		\begin{split}
			&\int_{-\infty}^0 e^{it\xi^2}  \bigl( \Psi_{j,\omega}(x,\xi) - \Psi_{j,\omega}(x,0) \bigr) g(\xi) \, \ud \xi \\
			&= \int_{-\infty}^0 e^{it\xi^2}  \bigl( \bfG_{j,\omega}(x,-\xi) - \bfG_{j,\omega}(x,0) \bigr) g(\xi) \, \ud \xi \\
			&= \frac{1}{2it}\left[ e^{it\xi^2} \frac{\bfG_{j,\omega}(x,-\xi) - \bfG_{j,\omega}(x,0)}{\xi}g(\xi) \right]_{\xi=-\infty}^{\xi=0} \\
			&\quad - \frac{1}{2it}\int_{-\infty}^0 e^{it\xi^2} \partial_\xi \left(\frac{ \bfG_{j,\omega}(x,-\xi) - \bfG_{j,\omega}(x,0) }{\xi}\right)g(\xi) \, \ud \xi \\
			&\quad -\frac{1}{2it}\int_{-\infty}^0 e^{it\xi^2} \biggl( \frac{\bfG_{j,\omega}(x,-\xi) - \bfG_{j,\omega}(x,0)}{\xi} \biggr) \partial_\xi g(\xi) \, \ud \xi.
		\end{split}
	\end{equation}
	Both boundary terms on the right-hand sides of \eqref{equ:improved_local_decay_differed_Psis_no_px_integral_pos} and \eqref{equ:improved_local_decay_differed_Psis_no_px_integral_neg} vanish at $\xi = \pm \infty$, but they take different limits at $\xi = 0$. Indeed, we have
	\begin{equation*}
		\begin{split}
			\lim_{\xi \rightarrow 0^+} \left(e^{it\xi^2} \frac{\Psi_{j,\omega}(x,\xi)-\Psi_{j,\omega}(x,0)}{\xi}g(\xi) \right)&= (\partial_\xi\bfF_{j,\omega})(x,0) g(0),\\
			\lim_{\xi \rightarrow 0^-} \left(e^{it\xi^2} \frac{\Psi_{j,\omega}(x,\xi)-\Psi_{j,\omega}(x,0)}{\xi}g(\xi) \right)&= -(\partial_\xi\bfG_{j,\omega})(x,0) g(0),
		\end{split}
	\end{equation*}
	and by direct computation one sees that $(\partial_\xi\bfF_{j,\omega})(x,0) \neq -(\partial_\xi\bfG_{j,\omega})(x,0)$.
	Since $\vert \partial_\xi \bfF_{j,\omega}(x,0) \vert + \vert \partial_\xi \bfG_{j,\omega}(x,0) \vert \lesssim \jx$ by \eqref{eqn: bound on partial_xi calF}, we can estimate both boundary terms by $\jx \Vert g \Vert_{L_\xi^\infty}$, which leads to an acceptable bound when paired with the inverse weights.

	Next, we analyze the remaining two terms on the right-hand side of \eqref{equ:improved_local_decay_differed_Psis_no_px_integral_pos} involving $\bfF_{j,\omega}$.
	The case of \eqref{equ:improved_local_decay_differed_Psis_no_px_integral_neg} with $\bfG_{j,\omega}$ can be treated analogously.
	Using the bound \eqref{eqn: bound on partial_xi calF difference quotient}, we can estimate the second term on the right-hand side of \eqref{equ:improved_local_decay_differed_Psis_no_px_integral_pos} by
	\begin{equation*}
		\begin{split}
			&\biggl\| \frac{1}{2it} \jx^{-2} \int_0^\infty e^{it\xi^2} \partial_\xi \left(\frac{ \bfF_{j,\omega}(x,\xi) - \bfF_{j,\omega}(x,0) }{\xi}\right) g(\xi) \, \ud \xi \biggr\|_{L^\infty_x} \\
			&\quad \lesssim \frac{1}{t} \Biggl( \sup_{x, \xi \in \bbR} \, \jx^{-2} \biggl| \partial_\xi \left(\frac{ \bfF_{j,\omega}(x,\xi) - \bfF_{j,\omega}(x,0) }{\xi}\right) \biggr| \Biggr) \, \|g\|_{L_\xi^1} \lesssim t^{-1} \|\jxi g\|_{L_\xi^2}.
		\end{split}
	\end{equation*}
	For the third term on the right-hand side of \eqref{equ:improved_local_decay_differed_Psis_no_px_integral_pos}, we further split the integral into a low frequency and a high frequency part via a smooth cut-off function $\chi_0(\xi)$ localized around the origin,
	\begin{equation*}
		\begin{split}
			&\frac{1}{2it} \jx^{-2} \int_0^\infty e^{it\xi^2}  \left(\frac{\bfF_{j,\omega}(x,\xi) - \bfF_{j,\omega}(x,0) }{\xi}\right) \partial_\xi g(\xi) \, \ud \xi \\
			&= \frac{1}{2it} \jx^{-2} \int_0^\infty e^{it\xi^2} \left(\frac{\bfF_{j,\omega}(x,\xi) - \bfF_{j,\omega}(x,0) }{\xi}\right) \chi_0(\xi) \partial_\xi g(\xi) \, \ud \xi \\
			&\quad + \frac{1}{2it} \jx^{-2} \int_0^\infty e^{it\xi^2}  \left(\frac{\bfF_{j,\omega}(x,\xi) - \bfF_{j,\omega}(x,0) }{\xi}\right) \bigl( 1 - \chi_0(\xi) \bigr) \partial_\xi g(\xi) \, \ud \xi \\
			&=: I(t,x) + II(t,x).
		\end{split}
	\end{equation*}
	Using the bound \eqref{eqn: bound on calF difference quotient} and the Cauchy-Schwarz inequality, we estimate the first term $I(t,x)$ by
	\begin{equation*}
		\begin{aligned}
			\|I(t,x)\|_{L^\infty_x} &\lesssim t^{-1} \Biggl( \sup_{x,\xi \in \bbR} \, \jx^{-1} \biggl| \left(\frac{\bfF_{j,\omega}(x,\xi) - \bfF_{j,\omega}(x,0) }{\xi}\right) \biggr| \Biggr) \|\chi_0\|_{L^2_\xi} \|\pxi g\|_{L^2_\xi} \lesssim t^{-1} \|\pxi g\|_{L^2_\xi},
		\end{aligned}
	\end{equation*}
	while we use \eqref{eqn: bound on partial_xi calF} and the Cauchy-Schwarz inequality to bound the second term $II(t,x)$ by
	\begin{equation*}
		\begin{aligned}
			\|II(t,x)\|_{L^\infty_x} &\lesssim t^{-1} \biggl( \sup_{x,\xi \in \bbR} \, \bigl| \bfF_{j,\omega}(x,\xi) \bigr| \biggr) \bigl\| \xi^{-1} \bigl( 1 - \chi_0(\xi) \bigr) \bigr\|_{L^2_\xi} \|\pxi g\|_{L^2_\xi} \lesssim t^{-1} \|\pxi g\|_{L^2_\xi}.
		\end{aligned}
	\end{equation*}
	This finishes the proof of the bound \eqref{equ:improved_local_decay_difference_Psis_no_px}.

	Now we turn to the second asserted estimate \eqref{equ:improved_local_decay_difference_Psis_with_px}.
	In what follows, in view of Lemma~\ref{lemma: PDO on m12} we may assume without loss of generality that both $m_{1,\omega}(x,\xi)$ and $m_{2,\omega}(x,\xi)$ can be written as a product $a(x) b(\xi)$ with $a,b \in W^{1,\infty}(\bbR)$.
	Moreover, we record the identity
	\begin{equation} \label{equ:improved_local_decay_px_of_difference_Psis_identity}
		\partial_x \bigl( \Psi_{j,\omega}(x,\xi)-\Psi_{j,\omega}(x,0) \bigr) = i\xi e^{ix\xi} m_{j,\omega}(x,\xi) + e^{ix\xi} \partial_x m_{j,\omega}(x,\xi) - \partial_x m_{j,\omega}(x,0).
	\end{equation}
	For short times $0 \leq t \leq 1$, we use \eqref{equ:improved_local_decay_px_of_difference_Psis_identity} to write
	\begin{equation} \label{equ:improved_local_decay_px_of_difference_short_time_split}
		\begin{aligned}
			&\jx^{-3} \int_\bbR e^{it\xi^2} \px \bigl( \Psi_{j,\omega}(x,\xi)-\Psi_{j,\omega}(x,0) \bigr) g(\xi) \, \ud \xi \\
			&= \jx^{-3} \int_\bbR e^{it\xi^2} i\xi e^{ix\xi} m_{j,\omega}(x,\xi) g(\xi) \, \ud \xi \\
			&\quad + \jx^{-3} \int_\bbR e^{it\xi^2} \bigl( e^{ix\xi} \partial_x m_{j,\omega}(x,\xi) - \partial_x m_{j,\omega}(x,0) \bigr) g(\xi) \, \ud \xi.
		\end{aligned}
	\end{equation}
	Then the first term on the right-hand side of \eqref{equ:improved_local_decay_px_of_difference_short_time_split} is of the form
	\begin{equation*}
		\jx^{-3} \int_\bbR e^{it\xi^2}e^{ix\xi} i\xi m_{j,\omega}(x,\xi)g(\xi) \, \ud \xi = \jx^{-3} a(x) \int_\bbR e^{it\xi^2}e^{ix \xi} b(\xi) i \xi g(\xi) \, \ud \xi.
	\end{equation*}
	It can therefore be bounded using Plancherel's identity by
	\begin{equation*}
		\begin{split}
			&\biggl\| \jx^{-3} a(x) \int_\bbR e^{it\xi^2}e^{ix \xi} b(\xi) i \xi g(\xi) \, \ud \xi \biggr\|_{L^2_x} \lesssim \bigl\| \jx^{-3} a(x) \bigr\|_{L_x^\infty} \Bigl\| e^{it\xi^2} b(\xi) i \xi g(\xi) \Bigr\|_{L^2_\xi} \lesssim \|\jxi g\|_{L^2_\xi}.
		\end{split}
	\end{equation*}
	Instead, for the second term on the right-hand side of \eqref{equ:improved_local_decay_px_of_difference_short_time_split}, we use \eqref{lemma: PDO on m12} and the Cauchy-Schwarz inequality to conclude
	\begin{equation*}
		\begin{aligned}
			&\biggl\| \jx^{-3} \int_\bbR e^{it\xi^2} \bigl( e^{ix\xi} \partial_x m_{j,\omega}(x,\xi) - \partial_x m_{j,\omega}(x,0) \bigr) g(\xi) \, \ud \xi \biggr\|_{L^2_x} \\
			&\lesssim \bigl\| \jx^{-3} \bigr\|_{L^2_x} \biggl( \sup_{x,\xi \in \bbR} \, \bigl| \partial_x m_{j,\omega}(x,\xi) \bigr| \biggr) \bigl\| \jxi^{-1} \bigr\|_{L^2_\xi} \|\jxi g\|_{L^2_\xi} \lesssim \|\jxi g\|_{L^2_\xi}.
		\end{aligned}
	\end{equation*}
	For longer times $t \geq 1$, we split the integral into two parts as before,
	\begin{equation*}
		\begin{split}
			&\int_\bbR e^{it\xi^2}  \px\bigl( \Psi_{j,\omega}(x,\xi) - \Psi_{j,\omega}(x,0) \bigr) g(\xi) \, \ud \xi \\
			&= \int_0^\infty e^{it\xi^2}  \px\bigl( \bfF_{j,\omega}(x,\xi) - \bfF_{j,\omega}(x,0) \bigr) g(\xi) \, \ud \xi 	+ \int_{-\infty}^0 e^{it\xi^2}  \px\bigl( \bfG_{j,\omega}(x,-\xi) - \bfG_{j,\omega}(x,0) \bigr) g(\xi) \, \ud \xi,
		\end{split}
	\end{equation*}
	and then integrate by parts in the frequency variable in each integral. Both integrals can be analyzed in the same way, so we focus on the first integral. We find
	\begin{equation} \label{equ:improved_local_decay_with_px_integral_split}
		\begin{aligned}
			&\int_0^\infty e^{it\xi^2}  \px\bigl( \bfF_{j,\omega}(x,\xi) - \bfF_{j,\omega}(x,0) \bigr) g(\xi) \, \ud \xi		\\
			&= \frac{1}{2it}\left[ e^{it\xi^2} \px\left(\frac{\bfF_{j,\omega}(x,\xi) - \bfF_{j,\omega}(x,0) }{\xi}\right)g(\xi) \right]_{\xi=0}^{\xi = \infty} \\
			&\quad - \frac{1}{2it}\int_0^\infty e^{it\xi^2} \pxi \px \left(\frac{ \bfF_{j,\omega}(x,\xi) - \bfF_{j,\omega}(x,0) }{\xi}\right)g(\xi) \, \ud \xi \\
			&\quad -\frac{1}{2it}\int_0^\infty e^{it\xi^2}   \px\left(\frac{\bfF_{j,\omega}(x,\xi) - \bfF_{j,\omega}(x,0) }{\xi}\right) \partial_\xi g(\xi) \, \ud \xi.
		\end{aligned}
	\end{equation}
	Using \eqref{eqn: bound on partial_xi partial_x calF}, the first term on the right-hand side of \eqref{equ:improved_local_decay_with_px_integral_split} can be estimated by
	\begin{equation*}
		\begin{aligned}
			\biggl\| \jx^{-3} \frac{1}{2it}\left[ e^{it\xi^2} \px\left(\frac{\bfF_{j,\omega}(x,\xi) - \bfF_{j,\omega}(x,0) }{\xi}\right)g(\xi) \right]_{\xi=0}^{\xi = \infty} \biggr\|_{L^2_x} &= \biggl\| \jx^{-3} \frac{1}{2it} (\pxi \px \bfF_{j,\omega})(x,0) g(0) \biggr\|_{L^2_x} \\
			&\lesssim t^{-1} \|g\|_{L^\infty_\xi}.
		\end{aligned}
	\end{equation*}
	For the second term on the right-hand side of \eqref{equ:improved_local_decay_with_px_integral_split}, we use \eqref{eqn: bound on partial_xi partial_x calF difference quotient} to conclude
	\begin{equation}
		\begin{aligned}
			&\left \Vert \frac{1}{2it} \jx^{-3} \int_0^\infty e^{it\xi^2} \pxi \px \left(\frac{ \bfF_{j,\omega}(x,\xi) - \bfF_{j,\omega}(x,0) }{\xi} \right)g(\xi)\ud \xi \right \Vert_{L_x^2} \\
			&\lesssim t^{-1} \bigl\| \jx^{-1} \bigr\|_{L^2_x} \Biggl( \sup_{x,\xi \in \bbR} \, \jx^{-2} \, \biggl| \pxi \px \biggl( \frac{ \bfF_{j,\omega}(x,\xi) - \bfF_{j,\omega}(x,0) }{\xi} \biggr) \biggr| \Biggr) \|g\|_{L_\xi^1} \lesssim t^{-1}\Vert \jxi g \Vert_{L_\xi^2}.
		\end{aligned}
	\end{equation}
	Finally, to treat the third term on the right-hand side of \eqref{equ:improved_local_decay_with_px_integral_split}, we insert the following identity
	\begin{equation} \label{equ:improved_local_decay_with_px_splitting}
		\begin{aligned}
			&\px\left(\frac{\bfF_{j,\omega}(x,\xi) - \bfF_{j,\omega}(x,0) }{\xi}\right) \\
			&= i e^{ix\xi}m_{j,\omega}(x,\xi) + e^{ix\xi} \left(\frac{\partial_x m_{j,\omega}(x,\xi) - \partial_x m_{j,\omega}(x,0)}{\xi}\right)+ \left(\frac{e^{ix\xi}-1}{\xi}\right) \px m_{j,\omega}(x,\xi).
		\end{aligned}
	\end{equation}
	Then the first term in \eqref{equ:improved_local_decay_with_px_splitting} leads to estimating
	\begin{equation*}
		\frac{1}{2t} \jx^{-3} \int_0^\infty e^{it\xi^2} e^{ix\xi} i m_{j,\omega}(x,\xi)\partial_\xi g(\xi) \, \ud \xi.
	\end{equation*}
	Using the tensorized structure $a(x) b(\xi)$ of $m_{j,\omega}(x,\xi)$ along with Plancherel's identity, this term can be bounded in $L^2_x$ by $t^{-1} \|\pxi g\|_{L^2_\xi}$, as desired.
	On the other hand, by the fundamental theorem of calculus and \eqref{lemma: PDO on m12}, we have uniformly for all $\xi \in \bbR$ that
	\begin{equation*}
		\sup_{x \in \bbR} \, \left \vert \frac{\partial_x m_{j,\omega}(x,\xi) - \partial_x m_{j,\omega}(x,0)}{\xi} \right \vert + \sup_{x \in \bbR} \, \jx^{-1} \left \vert \frac{e^{ix\xi}-1}{\xi} \right \vert \lesssim \min(1,\vert \xi \vert^{-1}) \lesssim \jxi^{-1}.
	\end{equation*}
	Hence, for the contributions of the other two terms in \eqref{equ:improved_local_decay_with_px_splitting} we obtain by the Cauchy-Schwarz inequality,
	\begin{equation*}
		\begin{aligned}
			&\left \Vert \frac{1}{2it} \jx^{-3} \int_0^\infty e^{it\xi^2}e^{ix\xi} \left(\frac{\partial_x m_{j,\omega}(x,\xi) - \partial_x m_{j,\omega}(x,0)}{\xi}\right) \partial_\xi g(\xi) \, \ud \xi \right \Vert_{L_x^2}\\
			&\lesssim t^{-1} \bigl\| \jx^{-3} \bigr\|_{L_x^2} \int_0^\infty \jxi^{-1} |\partial_\xi g(\xi)| \, \ud \xi  \lesssim t^{-1}\Vert \partial_\xi g(\xi)\Vert_{L_\xi^2},
		\end{aligned}
	\end{equation*}
	and
	\begin{equation*}
		\begin{split}
			&\biggl\| \frac{1}{2it} \jx^{-3} \int_0^\infty e^{it\xi^2}  \left(\frac{e^{ix\xi}-1}{\xi}\right) \px m_{j,\omega}(x,\xi) \partial_\xi g(\xi) \, \ud \xi \biggr\|_{L_x^2}		\\
			&\lesssim t^{-1} \bigl\| \jx^{-2} \bigr\|_{L_x^2} \biggl( \sup_{x,\xi \in \bbR} \, |\px m_{j,\omega}(x,\xi)| \biggr)	\int_0^\infty \jxi^{-1} |\partial_\xi g(\xi)| \, \ud \xi  \lesssim t^{-1}\Vert \partial_\xi g(\xi)\Vert_{L_\xi^2}.
		\end{split}
	\end{equation*}
	This finishes the proof of the second asserted bound \eqref{equ:improved_local_decay_difference_Psis_with_px}.
\end{proof}

The main purpose of the next lemma is to conclude that high frequencies enjoy improved local decay.

\begin{lemma} \label{lem:improved_local_decay_simple} 
 Let $\fraka \in W^{1,\infty}(\bbR)$.
 Denote by $\chi_0(\xi)$ a smooth even non-negative cut-off function with $\chi_0(\xi) = 1$ for $|\xi| \leq 1$ and $\chi_0(\xi) = 0$ for $|\xi| \geq 2$. There exists $C \geq 1$ such that uniformly for all $t \geq 0$,
 \begin{align}
  \biggl| \int_\bbR e^{\pm i t \xi^2} \fraka(\xi) \bigl( 1 - \chi_0(\xi) \bigr) g(\xi) \, \ud \xi \biggr| &\leq \frac{C}{\jt} \Bigl( \|\pxi g(\xi)\|_{L^2_\xi} + \|\jxi g(\xi)\|_{L^2_\xi} \Bigr), \label{equ:improved_local_decay_simple_bound1} \\
  \biggl\| \jx^{-2} \int_\bbR \bigl( e^{ix\xi} - 1 \bigr) e^{\pm i t \xi^2} \fraka(\xi) g(\xi) \, \ud \xi \biggr\|_{L^\infty_x} &\leq \frac{C}{\jt} \Bigl( \|\pxi g(\xi)\|_{L^2_\xi} + \|\jxi g(\xi)\|_{L^2_\xi} \Bigr). \label{equ:improved_local_decay_simple_bound2}
 \end{align}

\end{lemma}
\begin{proof}
    In what follows we treat the case of the phase $e^{it\xi^2}$.
    We begin with the proof of \eqref{equ:improved_local_decay_simple_bound1}.
	For short times $0 \leq t \leq 1$, we just use the crude bound
    \begin{equation*}
     \biggl| \int_\bbR e^{i t \xi^2} \fraka(\xi) \bigl( 1 - \chi_0(\xi) \bigr) g(\xi) \, \ud \xi \biggr| \leq \|\fraka\|_{L^\infty_\xi} \|g\|_{L^1_\xi} \lesssim \|\jxi g\|_{L^2_\xi}.
    \end{equation*}
	For longer times $t \geq 1$, we integrate by parts in the variable $\xi$ and then use the Cauchy-Schwarz inequality to conclude
	\begin{equation*}
	 \begin{aligned}
		\biggl| \int_\bbR e^{ i t \xi^2} \fraka(\xi) \bigl( 1 - \chi_0(\xi) \bigr) g(\xi) \, \ud \xi \biggr| &\leq \biggl| -\frac{1}{2it}\int_\bbR e^{it\xi^2} \partial_\xi \left( \fraka(\xi) \frac{1 - \chi_0(\xi)}{\xi} \right) g(\xi) \,\ud \xi \biggr| \\
		&\quad + \biggl| - \frac{1}{2it}\int_\bbR e^{it\xi^2} \left( \fraka(\xi) \frac{1 - \chi_0(\xi)}{\xi} \right) \partial_\xi g(\xi) \,\ud \xi \biggr| \\
		&\lesssim \frac{1}{t} \Bigl\| \pxi \Bigl( \fraka(\xi) \xi^{-1} \bigl( 1 - \chi_0(\xi) \bigr) \Bigr) \Bigr\|_{L^2_\xi} \|g\|_{L^2_\xi} \\
		&\quad + \frac{1}{t} \bigl\| \fraka(\xi) \xi^{-1} \bigl( 1 - \chi_0(\xi) \bigr) \bigr\|_{L^2_\xi} \|\pxi g\|_{L^2_\xi} \\
		&\lesssim t^{-1} \bigl( \|g\|_{L^2_\xi} + \|\pxi g\|_{L^2_\xi} \bigr),
     \end{aligned}
	\end{equation*}
	as desired. Next, we turn to the proof of \eqref{equ:improved_local_decay_simple_bound2}. For short times $0 \leq t \leq 1$ we trivially estimate
	\begin{equation*}
	 \biggl\| \jx^{-2} \int_\bbR \bigl( e^{ix\xi} - 1 \bigr) e^{i t \xi^2} \fraka(\xi) g(\xi) \, \ud \xi \biggr\|_{L^2_x} \lesssim \bigl\| \jx^{-2} \bigr\|_{L^2_x} \|\fraka\|_{L^\infty_\xi} \|g\|_{L^1_\xi} \lesssim \|\jxi g\|_{L^2_\xi}.
	\end{equation*}
    Instead for times $t \geq 1$ we integrate by parts in $\xi$,
    \begin{equation*}
     \begin{aligned}
      \int_\bbR \bigl( e^{ix\xi} - 1 \bigr) e^{i t \xi^2} \fraka(\xi) g(\xi) \, \ud \xi &= - \frac{1}{2it} \int_\bbR e^{it\xi^2} \biggl( \frac{e^{ix\xi} - 1}{\xi} \biggr) \pxi \bigl( \fraka(\xi) g(\xi) \bigr) \, \ud \xi \\
      &\quad - \frac{1}{2it} \int_\bbR e^{it\xi^2} \pxi \biggl( \frac{e^{ix\xi} - 1}{\xi} \biggr) \fraka(\xi) g(\xi) \, \ud \xi.
     \end{aligned}
    \end{equation*}
    Using the bounds
    \begin{equation*}
     \sup_{x \in \bbR} \, \jx^{-1} \biggl| \frac{e^{ix\xi}-1}{\xi} \biggr| \lesssim \jxi^{-1}, \quad
     \sup_{x \in \bbR} \, \jx^{-2} \biggl| \pxi \biggl( \frac{e^{ix\xi}-1}{\xi} \biggr) \biggr| \lesssim 1,
    \end{equation*}
    the asserted improved local decay estimate \eqref{equ:improved_local_decay_simple_bound2} then follows by the Cauchy-Schwarz inequality.
\end{proof}

\subsection{Local smoothing estimate}

We recall a local smoothing estimate from \cite[Lemma~3.5]{ChenPus22} that will be used throughout Section~\ref{sec:energy_estimates} in the derivation of the weighted energy estimates to handle the contributions of spatially localized nonlinearities with cubic-type decay. The proof is provided for the reader's convenience.

\begin{lemma} \label{lemma:local_smoothing}
Let $a:\bbR^2 \rightarrow \bbC$ satisfy 
\begin{equation}
\sup_{x,\xi \in \bbR} \vert a(x,\xi) \vert \lesssim 1,
\end{equation}
and let $\varphi:\bbR \rightarrow \bbC$ satisfy $\vert \varphi(\xi)\vert \lesssim \vert \xi \vert^{\frac{1}{2}}$. Then, for all $t > 0$,
\begin{equation}\label{eqn:local_smoothing_estimate}
\left \Vert \int_\bbR e^{is\xi^2 } \varphi(\xi) a(x,\xi) h(\xi) \, \ud \xi \right \Vert_{L_x^\infty L_s^2([0,t])} \lesssim \Vert h \Vert_{L_\xi^2},
\end{equation}
and by duality,
\begin{equation}\label{eqn:local_smoothing_dual_estimate}
\left \Vert \int_0^t \int_\bbR e^{is\xi^2} \varphi(\xi)a(x,\xi)F(s,x) \,\ud x \,\ud s \right \Vert_{L_\xi^2} \lesssim \Vert F \Vert_{L_x^1L_s^2([0,t])}.
\end{equation}
\end{lemma}
\begin{proof}
Since we may decompose $h = h\mathbbm{1}_{\xi \geq 0 }+h\mathbbm{1}_{\xi < 0 }$, it suffices to show \eqref{eqn:local_smoothing_estimate} with $h\mathbbm{1}_{\xi \geq 0 }$ replacing $h$. By a combination of a change of variables $\xi^2 = \lambda$, the Plancherel's identity (in $s$), and an inverse change of variables $\xi = + \sqrt{\lambda}$, we obtain uniformly for all $x\in \bbR$ that 
\begin{equation*}
\begin{split}
\left \Vert \int_\bbR e^{is\xi^2} \varphi(\xi) a(x,\xi) h(\xi)\mathbbm{1}_{\xi \geq 0} \, \ud \xi \right \Vert_{L_s^2(\bbR)}^2&=\left \Vert \int_0^\infty e^{is\lambda} \varphi(\sqrt{\lambda}) a(x,\sqrt{\lambda}) h(\sqrt{\lambda}) \frac{\ud \lambda}{2\sqrt{\lambda}} \right \Vert_{L_s^2(\bbR)}^2\\
&=\int_0^\infty \left \vert  \frac{\varphi(\sqrt{\lambda})}{2\sqrt{\lambda}} a(x,\sqrt{\lambda}) h(\sqrt{\lambda}) \right \vert^2 \ud \lambda\\
&\lesssim \int_0^\infty \vert h(\sqrt{\lambda})\vert^2 \frac{\,\ud \lambda}{\sqrt{\lambda}} \lesssim \Vert h \Vert_{L_\xi^2}^2.
\end{split}
\end{equation*}
This concludes \eqref{eqn:local_smoothing_estimate}. For the inhomogeneous estimate, we take an inner product with an arbitrary function $h \in L_\xi^2(\bbR)$ and the expression in the left-hand side of \eqref{eqn:local_smoothing_dual_estimate}. By the Fubini's theorem and H\"older's inequality, we find that
\begin{equation*}
\begin{split}
\left \vert \left \langle h, \int_0^t \int_\bbR e^{is\xi^2} \varphi(\xi)a(x,\xi)F(s,x) \,\ud x \,\ud s \right \rangle \right \vert &\lesssim \left \Vert \int_\bbR e^{is\xi^2 } \varphi(\xi) a(x,\xi) h(\xi) \, \ud \xi \right \Vert_{L_x^\infty L_s^2([0,t])} \Vert F \Vert_{L_x^1L_s^2([0,t])}\\
&\lesssim \Vert h \Vert_{L_\xi^2}\Vert F \Vert_{L_x^1L_s^2([0,t])}.
\end{split}
\end{equation*}
This bound implies \eqref{eqn:local_smoothing_dual_estimate} by duality.
\end{proof}
Next, we state a corollary for the special case when the functions $a(x,\xi)$ and $F(s,x)$ in \eqref{eqn:local_smoothing_dual_estimate} do not depend on $x$.

\begin{corollary}
Let $a:\bbR \rightarrow \bbC$ satisfy 
\begin{equation}
\sup_{\xi \in \bbR} \vert a(\xi) \vert \lesssim 1,
\end{equation}
and let $\varphi:\bbR \rightarrow \bbC$ satisfy $\vert \varphi(\xi)\vert \lesssim \vert \xi \vert^{\frac{1}{2}}$. Then, for all $t > 0$,
\begin{equation}\label{eqn: local_smoothing_estimate_x}
	\left \Vert \int_\bbR e^{is\xi^2 } \varphi(\xi) a(\xi) h(\xi) \, \ud \xi \right \Vert_{L_s^2([0,t])} \lesssim \Vert h \Vert_{L_\xi^2},	
\end{equation}
and
\begin{equation}\label{eqn: local_smoothing_dual_estimate_x}
\left \Vert \int_0^t e^{is\xi^2} \varphi(\xi)a(\xi)F(s)  \,\ud s \right \Vert_{L_\xi^2} \lesssim \Vert F \Vert_{L_s^2([0,t])}.	
\end{equation}
\end{corollary}
\begin{proof}
The estimate \eqref{eqn: local_smoothing_estimate_x} is obtained from the estimate \eqref{eqn:local_smoothing_estimate} in the previous lemma, and \eqref{eqn: local_smoothing_dual_estimate_x} is deduced from \eqref{eqn: local_smoothing_estimate_x} by duality. 
\end{proof}

\section{Setting up the Analysis} \label{sec:setting_up}

In this section we introduce the setup for the analysis of the asymptotic stability of the solitary waves \eqref{equ:intro_family_2parameter} for the focusing cubic Schr\"odinger equation \eqref{equ:cubic_NLS}.
For small even perturbations of \eqref{equ:intro_family_2parameter} we first decompose the corresponding solution to \eqref{equ:cubic_NLS} into a modulated solitary wave and a radiation term. 
We derive a system of first-order differential equations for the modulation parameters and a Schr\"odinger evolution equation for the radiation term.
Then we introduce the profile of the radiation term and we prepare the evolution equation for the distorted Fourier transform of the profile for the nonlinear analysis in the subsequent sections. 
In particular, we exploit a remarkable null structure to renormalize a resonant quadratic term in the evolution equation for the profile via a suitable normal form.

\subsection{Local existence}

The following local existence result for \eqref{equ:cubic_NLS} suffices for our purposes.
It can be proven via a standard fixed point argument, see for instance \cite{GV78}.

\begin{lemma}[Local existence] \label{lem:setup_local_existence}
 For any $\psi_0 \in H^1_x(\bbR) \cap L^{2,1}_x(\bbR)$ there exists a unique solution $\psi \in \calC([0,T_\ast); H^1_x \cap L^{2,1}_x)$ to \eqref{equ:cubic_NLS} with initial condition $\psi(0) = \psi_0$ defined on a maximal interval of existence $[0,T_\ast)$ for some $T_\ast > 0$. Moreover, the following continuation criterion holds
 \begin{equation} \label{equ:setup_continuation_criterion}
  T_\ast < \infty \quad \Rightarrow \quad \limsup_{t \nearrow T_\ast} \| \psi(t) \|_{H^1_x \cap L^{2,1}_x} = \infty.
 \end{equation}
\end{lemma}

\subsection{Modulation and orbital stability}

We now consider small even perturbations of the solitary waves \eqref{equ:intro_family_2parameter}.
In the next proposition we establish a decomposition of the corresponding even solution to \eqref{equ:cubic_NLS} into a modulated solitary wave and a radiation term that is orthogonal to directions related to the invariance of \eqref{equ:cubic_NLS} under scaling and phase shifts.
Moreover, we derive the modulation equations and the Schr\"odinger evolution equation for the radiation term.

\begin{proposition} \label{prop:modulation_and_orbital}
    Let $\omega_0 \in (0, \infty)$. There exist constants $C_1 \geq 1$ and $0 < \varepsilon_1 \ll 1$ such that the following holds:
    For any $\gamma_0 \in \bbR$ and any even $u_0 \in H^1_x(\bbR) \cap L^{2,1}_x(\bbR)$ with $\|u_0\|_{H^1_x} \leq \varepsilon_1$, denote by $\psi(t,x)$ the $H^1_x \cap L^{2,1}_x$--solution to \eqref{equ:cubic_NLS} with initial condition 
    \begin{equation} \label{equ:setup_modulation_orbital_initial_datum}
        \psi_0(x) = e^{i \gamma_0} \bigl( \phi_{\omega_0}(x) + u_0(x) \bigr)
    \end{equation}
    on its maximal interval of existence $[0, T_\ast)$ furnished by Lemma~\ref{lem:setup_local_existence}.
    Then there exist unique continuously differentiable paths $(\omega, \gamma) \colon [0,T_\ast) \to (0,\infty) \times \bbR$ so that the even solution $\psi(t,x)$ to \eqref{equ:cubic_NLS} can be decomposed as
    \begin{equation} \label{equ:setup_modulation_prop_decomposition}
        \psi(t,x) = e^{i \gamma(t)} \bigl( \phi_{\omega(t)}(x) + u(t,x) \bigr), \quad 0 \leq t < T_\ast,
    \end{equation}
    and the following properties hold:
    \begin{itemize}[leftmargin=1.8em]
        \item[(1)] Evolution equation for the radiation\footnote{By direct inspection we find that the equation \eqref{equ:setup_perturbation_equ} is $\calJ$-invariant. Since $U(0) := \big(u(0),\baru(0)\big)^\top$ is $\calJ$-invariant, this ensures that  $U(t)$ is $\calJ$-invariant for all $t \geq 0$.}:
        \begin{equation} \label{equ:setup_perturbation_equ}
            i\partial_t U - \calH(\omega)U = (\dot{\gamma} - \omega) \sigma_3 U + \calMod + \calN(U), \quad U := \begin{bmatrix} u \\ \overline{u} \end{bmatrix},
        \end{equation}
        with
        \begin{align}
                \calH(\omega) &:= \calH_0(\omega) + \calV(\omega) :=  \begin{bmatrix}
                    -\partial_x^2 + \omega & 0 \\ 0 & \partial_x^2 - \omega
                \end{bmatrix}
                +
                \begin{bmatrix}
                    -2\phi_\omega^2	 & - \phi_\omega^2 \\ \phi_\omega^2 & 2\phi_\omega^2
                \end{bmatrix}, \\
                \calMod &:= -i(\dot{\gamma} - \omega) Y_{1, \omega} - i \dot{\omega} Y_{2, \omega},  \\
                \calN(U) &:= \calQ_\omega(U) + \calC(U) := \begin{bmatrix}
                    -\phi_\omega(u^2 + 2u \bar{u}) \\ \phi_\omega(\baru^2 + 2u \baru)
                \end{bmatrix}
                +
                \begin{bmatrix}
                    -u \baru u \\ \baru u \baru
                \end{bmatrix}. 
        \end{align}

        \item[(2)] Modulation equations\footnote{The matrix $\mathbb{M}$ and the components of the vector on the right-hand side of \eqref{equ:setup_modulation_equation} are real-valued, because $c_\omega \in \bbR$ and because all involved inner products are between two $\calJ$-invariant vectors, and are therefore real-valued.}:
        \begin{equation} \label{equ:setup_modulation_equation}
            \mathbb{M} \begin{bmatrix}
                \dot{\gamma} - \omega \\
                 \dot \omega
            \end{bmatrix}
            = \begin{bmatrix}
                \langle i \calN(U), \sigma_2 Y_{1,\omega} \rangle \\
                \langle i \calN(U), \sigma_2 Y_{2, \omega} \rangle
            \end{bmatrix}
        \end{equation}
        with
        \begin{equation} \label{equ:setup_matrix_modulation_equation}
            \mathbb{M} := \begin{bmatrix}
                0 & c_\omega \\
                c_\omega & 0
            \end{bmatrix}
            +
            \begin{bmatrix}
                \langle U, \sigma_1Y_{1, \omega} \rangle &  \langle U,  \sigma_2 \partial_\omega Y_{1, \omega} \rangle \\
                \langle U, \sigma_1 Y_{2, \omega} \rangle &  \langle U, \sigma_2 \partial_\omega Y_{2, \omega} \rangle
            \end{bmatrix}, 
            \quad c_\omega := \frac{2}{\sqrt{\omega}}.
        \end{equation}
        
        \item[(3)] Orthogonality:
        \begin{equation} \label{equ:setup_orthogonality_radiation}
            \langle U(t), \sigma_2 Y_{1, \omega(t)} \rangle = \langle U(t),  \sigma_2 Y_{2, \omega(t)} \rangle = 0, \quad 0 \leq t < T_\ast.
        \end{equation}

        \item[(4)] Stability:
        \begin{equation} \label{equ:setup_smallness_orbital}
            \sup_{0 \leq t < T_\ast} \, \bigl( \|u(t)\|_{H^1_x} + |\omega(t) - \omega_0| \bigr) \leq C_1 \|u_0\|_{H^1_x}.
        \end{equation}

        \item[(5)] Comparison estimate:
        \begin{equation} \label{equ:setup_comparison_estimate}
            \frac12 \omega_0 \leq \omega(t) \leq 2 \omega_0, \quad 0 \leq t < T_{\ast}.
        \end{equation}
    \end{itemize}
\end{proposition}

The proof of Proposition~\ref{prop:modulation_and_orbital} is based on standard modulation and orbital stability techniques.
Recall that solutions to \eqref{equ:cubic_NLS} conserve the mass
\begin{equation}
 M[\psi] = \int_{\bbR} |\psi|^2 \, \ud x
\end{equation}
and the energy
\begin{equation}
 E[\psi] = \int_{\bbR} \biggl( \frac12 |\px \psi|^2 - \frac14 |\psi|^4 \biggr) \, \ud x.
\end{equation}
We use the following expansions of the mass functional and of the energy functional around a solitary wave.

\begin{lemma} \label{lem:setup_expansions_conservation_laws}
    Let $\omega \in (0, \infty)$ and $u \in H^1_x(\bbR)$. Set
    \begin{equation*}
        u_1 := \Re(u), \quad u_2 := \Im(u), \quad U := \begin{bmatrix} u \\ \baru \end{bmatrix}.
    \end{equation*}
    \begin{itemize}[leftmargin=1.8em]
        \item[(a)] Expansion of the mass around a solitary wave:
        \begin{equation} \label{equ:setup_expansion_mass}
            M[\phi_\omega + u] - M[\phi_\omega] = M[u] - \langle U, \sigma_2 Y_{1,\omega} \rangle.
        \end{equation}
        \item[(b)] Expansion of the energy around a solitary wave:
        \begin{equation} \label{equ:setup_expansion_energy1}
            \begin{aligned}
                E[\phi_\omega + u] - E[\phi_\omega] &= \frac12 \langle (-\px^2 - 3\phi_\omega^2) u_1, u_1 \rangle + \frac12 \langle (-\px^2 - \phi_\omega^2) u_2, u_2 \rangle + \frac{\omega}{2} \langle U, \sigma_2 Y_{1,\omega} \rangle \\ 
                &\quad - \frac12 \int_\bbR \phi_\omega \bigl( u^2 \baru + u \baru^2 \bigr) \, \ud x - \frac14 \int_\bbR |u|^4 \, \ud x
            \end{aligned}
        \end{equation}
        as well as 
        \begin{equation} \label{equ:setup_expansion_energy2}
            \begin{aligned}
                E[\phi_\omega + u] - E[\phi_\omega] &= E[u] + \frac{\omega}{2} \langle U, \sigma_2 Y_{1,\omega} \rangle 
                - \frac14 \int_\bbR \phi_\omega^2 \bigl( u^2 + 4 u \baru + \baru^2 \bigr) \, \ud x \\
                &\quad - \frac12 \int_\bbR \phi_\omega \bigl( u^2 \baru + u \baru^2 \bigr) \, \ud x.
            \end{aligned}
        \end{equation}        
    \end{itemize}
\end{lemma}
\begin{proof}
    The identities follow by direct computation.
\end{proof}

We are now in the position to prove Proposition~\ref{prop:modulation_and_orbital}.

\begin{proof}[Proof of Proposition~\ref{prop:modulation_and_orbital}]
We denote by $\psi(t,x)$ the even $H^1_x \cap L^{2,1}_x$-solution to \eqref{equ:cubic_NLS} with initial condition \eqref{equ:setup_modulation_orbital_initial_datum} on its maximal interval of existence $[0,T_\ast)$ furnished by Lemma~\ref{lem:setup_local_existence}.

\medskip 

\noindent {\bf Step 1.} (Time-independent modulation)
We show that there exist constants $C_1 \equiv C_1(\omega_0) \geq 1$ and $\kappa_1 \equiv \kappa_1(\omega_0)$ with $0 < \kappa_1 \ll 1$ such that if $\|u_0\|_{H^1_x} \leq \kappa_1$, then there exist $\omega(0) \in (0,\infty)$ and $\gamma(0) \in \bbR$ satisfying 
\begin{equation} \label{equ:setup_modulation_proposition_proof_step1}
    |\omega(0)-\omega_0| + |\gamma(0)-\gamma_0| \leq C_1 \|u_0\|_{H^1_x}
\end{equation}
as well as
\begin{equation*}
    \langle U(0), \sigma_2 Y_{1,\omega(0)} \rangle = \langle U(0), \sigma_2 Y_{2,\omega(0)} \rangle = 0, 
    \quad 
\end{equation*}
with 
\begin{equation*}
    U(0,x) 
    = \begin{bmatrix} u(0,x) \\ \overline{u(0,x)} \end{bmatrix}, \quad u(0,x) := e^{-i\gamma(0)} \psi_0(x) - \phi_{\omega(0)}(x).
\end{equation*}
Observe that by construction we then have $\|u(0)\|_{H^1_x} \lesssim \|u_0\|_{H^1_x}$.
To this end we introduce the functional
\begin{equation*}
    \calK\bigl[\psi, \omega, \gamma\bigr] 
    := 
    \begin{bmatrix}
     \left\langle \begin{bmatrix} e^{-i\gamma} \psi - \phi_\omega \\ \overline{e^{-i\gamma} \psi} - \phi_\omega   \end{bmatrix}, \sigma_2 Y_{1,\omega} \right\rangle \\ 
     \left\langle \begin{bmatrix} e^{-i\gamma} \psi - \phi_\omega \\ \overline{e^{-i\gamma} \psi} - \phi_\omega   \end{bmatrix}, \sigma_2 Y_{2,\omega} \right\rangle 
    \end{bmatrix}, \quad (\psi, \omega, \gamma) \in H^1_x(\bbR) \times (0,\infty) \times \bbR.
\end{equation*}
By direct computation we find that 
\begin{equation*}
    \frac{\partial \calK}{\partial(\omega, \gamma)}\bigl[ e^{i\gamma_0} \phi_{\omega_0}, \omega_0, \gamma_0 \bigr] = \begin{bmatrix} c_{\omega_0} & 0 \\ 0 & c_{\omega_0} \end{bmatrix}
\end{equation*}
is invertible. Moreover, an analogous computation shows that the Jacobian matrix of $\calK$ with respect to the modulation parameters is uniformly non-degenerate in a neighborhood of $(e^{i\gamma_0} \phi_{\omega_0}, \omega_0, \gamma_0)$. The assertion now follows from a quantitative version of the implicit function theorem, see for instance \cite[Remark 3.2]{Jendrej15}.

\medskip 

\noindent {\bf Step 2.} (Derivation of perturbation equation and modulation equations)
Next, given continuously differentiable paths $(\omega, \gamma) \colon [0,T_\ast) \to (0,\infty) \times \bbR$ and a decomposition of $\psi(t,x)$ into
\begin{equation} \label{equ:setup_modulation_proposition_proof_decomposition}
    \psi(t,x) = e^{i\gamma(t)} \bigl( \phi_{\omega(t)} + u(t,x) \bigr), \quad 0 \leq t < T_\ast,
\end{equation}
we formally compute, using \eqref{equ:cubic_NLS} and \eqref{equ:intro_phi_omega_equation}, that the perturbation $u(t,x)$ satisfies
\begin{equation}\label{equ: scalar_equation_u}
    i\partial_t u+ \partial_x^2 u - \omega u +\phi_\omega^2(2u + \baru) = (\dot{\gamma}-\omega)(\phi_\omega + u) - i \dot{\omega}\partial_\omega \phi_\omega - \phi_\omega(u^2 + 2u \baru) - \vert u \vert^2 u.
\end{equation}
Correspondingly, we obtain the evolution equation \eqref{equ:setup_perturbation_equ} for the vectorial perturbation $U(t) = \big(u(t),\baru(t)\big)^\top$.
Moreover, formally differentiating the orthogonality conditions
\begin{equation} \label{equ:setup_modulation_proposition_proof_orthogonality}
    \langle U(t), \sigma_2 Y_{1,\omega(t)} \rangle = \langle U(t), \sigma_2 Y_{2,\omega(t)} \rangle = 0, \quad 0 \leq t < T_\ast,
\end{equation}
and inserting the evolution equation \eqref{equ:setup_perturbation_equ}, we find for $j=1,2$ and for $0 \leq t < T_*$ that
\begin{equation*}
    \begin{split}
0 &= \langle \partial_t U,\sigma_2 Y_{j,\omega}\rangle + \langle U, \dot{\omega} \sigma_2 \partial_\omega Y_{j,\omega} \rangle\\ 
  &=-i \Big\langle \Big(\calH(\omega)U + (\dot{\gamma}-\omega)\sigma_3 U - i (\dot{\gamma}-\omega)Y_{1,\omega} - i \dot{\omega} Y_{2,\omega} + \calN(U)\Big), \sigma_2 Y_{j,\omega}\Big \rangle + \dot{\omega} \langle U,\sigma_2 \partial_\omega Y_{j,\omega}\rangle.
    \end{split}
\end{equation*}
Then, using \eqref{eqn:symmetry of calH}, \eqref{eqn:eig-eqn for nullspace of calH-omega}, and \eqref{equ:setup_modulation_proposition_proof_orthogonality}, we conclude
\begin{equation*}
    \langle \calH(\omega)U,\sigma_2 Y_{1,\omega}\rangle = \langle \calH(\omega)U,\sigma_2 Y_{2,\omega}\rangle = 0,
\end{equation*}
and using $i\sigma_3\sigma_2 = \sigma_1$, we have 
\begin{equation*}
    -i\langle (\dot{\gamma}-\omega)\sigma_3 U,\sigma_2 Y_{j,\omega}\rangle = (\dot{\gamma}-\omega)\langle U, \sigma_1 Y_{j,\omega}\rangle, \quad j=1,2.
\end{equation*}
Furthermore, we recall that 
\begin{equation*}
    \langle Y_{1,\omega},\sigma_2 Y_{1,\omega} \rangle = \langle Y_{2,\omega},\sigma_2 Y_{2,\omega} \rangle = 0.
\end{equation*}
Hence, using \eqref{eqn: inner-product_Y_j} and the preceding identities, we obtain the modulation equations \eqref{equ:setup_modulation_equation}.

\medskip 

\noindent {\bf Step 3.} (Cauchy problem for the modulation parameters)
Given the $H^1_x \cap L^{2,1}_x$-solution $\psi(t,x)$ to \eqref{equ:cubic_NLS} with initial condition \eqref{equ:setup_modulation_orbital_initial_datum} on its maximal interval of existence $[0,T_\ast)$, we write $u(t,x) = e^{-i\gamma(t)}\psi(t,x) - \phi_{\omega(t)}(x)$ and view the modulation equations \eqref{equ:setup_modulation_equation} as a system of first-order nonlinear ODEs for $(\omega(t), \gamma(t))$ on $[0,T_\ast)$ with initial conditions $\omega(0)$ and $\gamma(0)$ determined in Step~1. If $|\omega(t) - \omega(0)| + \|u(t)\|_{H^1_x} \ll 1$ are sufficiently small (depending on the size of $\omega_0$), we obtain by the Picard-Lindel\"of theorem a local-in-time continuously differentiable solution $(\omega(t), \gamma(t))$ to this system defined on some time interval $[0,T_0]$ with $0 < T_0 < T_\ast$.
As long as $|\omega(t)-\omega(0)| + \|u(t)\|_{H^1_x} \ll 1$, this solution can be continued to the entire time interval $[0,T_\ast)$. We achieve this control in the next step by using the expansion of the conservation laws around the solitary wave from Lemma~\ref{lem:setup_expansions_conservation_laws} and by invoking the orthogonality conditions \eqref{equ:setup_modulation_proposition_proof_orthogonality}.

\medskip 

\noindent {\bf Step 4.} (Orbital stability and conclusion) We consider the quantity
\begin{equation*}
    \calI(t) := 2 E\bigl[ \psi(t) \bigr] + \omega(t) M\bigl[ \psi(t) \bigr] - 2 E\bigl[ \phi_{\omega(0)} \bigr] - \omega(t) M\bigl[\phi_{\omega(0)}\bigr].
\end{equation*}
Observe that upon inserting the decomposition \eqref{equ:setup_modulation_proposition_proof_decomposition} for $\psi(t)$, we have 
\begin{equation*}
    \calI(t) = 2 E\bigl[ \phi_{\omega(t)} + u(t) \bigr] + \omega(t) M\bigl[ \phi_{\omega(t)} + u(t) \bigr] - 2 E\bigl[ \phi_{\omega(0)} \bigr] - \omega(t) M\bigl[\phi_{\omega(0)}\bigr].
\end{equation*}
We also define 
\begin{equation}
    \calJ[\omega] := 2 E[\phi_\omega] + \omega M[\phi_\omega], \quad \omega \in (0,\infty). 
\end{equation}
On the one hand, using the expansions of the conservation laws \eqref{equ:setup_expansion_mass}, \eqref{equ:setup_expansion_energy1}, the orthogonality conditions \eqref{equ:setup_modulation_proposition_proof_orthogonality}, and the fact that $\calJ'[\omega(0)] = M[\phi_{\omega(0)}]$, we obtain
\begin{equation*}
    \begin{aligned}
        \calI(t) &= \langle L_{+, \omega(t)} u_1(t), u_1(t) \rangle + \langle L_{-,\omega(t)} u_2(t), u_2(t) \rangle \\ 
        &\quad \quad + \calJ[\omega(t)] - \calJ[\omega(0)] - \calJ'[\omega(0)]\bigl( \omega(t) - \omega(0) \bigr) \\ 
        &\quad \quad - \int_\bbR \phi_{\omega(t)} \bigl( u(t)^2 \baru(t) + u(t) \baru(t)^2 \bigr) \, \ud x - \frac12 \int_\bbR |u(t)|^4 \, \ud x,
    \end{aligned}
\end{equation*}
where 
\begin{equation*}
    \begin{aligned}
        L_{+, \omega} := -\px^2 + \omega - 3 \phi_\omega^2, \quad L_{-,\omega} := -\px^2 + \omega - \phi_\omega^2, \quad \omega \in (0,\infty).
    \end{aligned}
\end{equation*}
On the other hand, by conservation of mass and energy for solutions to \eqref{equ:cubic_NLS}, we infer from the expansions \eqref{equ:setup_expansion_mass}, \eqref{equ:setup_expansion_energy2}, and the orthogonality conditions \eqref{equ:setup_modulation_proposition_proof_orthogonality} that
\begin{equation*}
    \begin{aligned}
        \calI(t) &= 2 E[u(0)] + \omega(0) M[u(0)] + M[u(0)] \bigl( \omega(t) - \omega(0) \bigr) \\ 
        &\quad - \frac12 \int_\bbR \phi_\omega^2 \bigl( u^2 + 4 u \baru + \baru^2 \bigr)\big|_{t=0} \, \ud x - \int_\bbR \phi_\omega \bigl( u^2 \baru + \baru^2 u \bigr)\big|_{t=0} \, \ud x.
    \end{aligned}
\end{equation*}
Apart from the orthogonality conditions \eqref{equ:setup_modulation_proposition_proof_orthogonality} that we have imposed, 
the even parity of the radiation term trivially implies
\begin{equation*}
    \langle U(t), \sigma_2 Y_{3,\omega(t)} \rangle = \langle U(t), \sigma_2 Y_{4,\omega(t)}\rangle = 0, \quad 0 \leq t < T_\ast,
\end{equation*}
where the odd generalized eigenfunctions $Y_{3,\omega}$ and $Y_{4,\omega}$ were introduced in \eqref{eqn:nullspace of calH-omega}. Then we have by \cite[Theorem 2.5]{Weinstein85} the coercivity estimate
\begin{equation*}
    \langle L_{+, \omega(t)} u_1(t), u_1(t) \rangle + \langle L_{-,\omega(t)} u_2(t), u_2(t) \rangle \gtrsim \|u(t)\|_{H^1_x}^2.
\end{equation*}
Moreover, we obtain from the strict convexity $\calJ''[\omega(0)] = c_{\omega(0)} > 0$ that 
\begin{equation*}
    \calJ[\omega(t)] - \calJ[\omega(0)] - \calJ'[\omega(0)]\bigl( \omega(t) - \omega(0) \bigr) = \frac{c_{\omega(0)}}{2} \bigl( \omega(t) - \omega(0) \bigr)^2 + \calO\bigl( \bigl( \omega(t) - \omega(0) \bigr)^3 \bigr).
\end{equation*}
Combining all of the preceding estimates and identities, we deduce that 
\begin{equation*}
    \begin{aligned}
        &\|u(t)\|_{H^1_x}^2 + \bigl( \omega(t) - \omega(0) \bigr)^2 - C \Bigl( \|u(t)\|_{H^1_x}^3 + \|u(t)\|_{H^1_x}^4 + |\omega(t)-\omega(0)| \|u(0)\|_{L^2_x}^2 + |\omega(t)-\omega(0)|^3 \Bigr) \\
        &\leq C \Bigl( \|u(0)\|_{H^1_x}^2 + \|u(0)\|_{H^1_x}^4 + |\omega(0)| \|u(0)\|_{L^2_x}^2 \Bigr).
    \end{aligned}
\end{equation*}
For sufficiently small $\|u(0)\|_{H^1_x} \lesssim \|u_0\|_{H^1_x} \ll 1$ depending only on the size of $\omega_0$, a standard continuity argument then gives the desired bound on the time interval $[0,T_0]$,
\begin{equation} \label{equ:setup_modulation_proposition_proof_orbital_stability_bound}
    \|u(t)\|_{H^1_x} + |\omega(t)-\omega(0)| \lesssim \|u(0)\|_{H^1_x} \lesssim \|u_0\|_{H^1_x}, \quad 0 \leq t \leq T_0.
\end{equation}
By \eqref{equ:setup_modulation_proposition_proof_orbital_stability_bound} we can then continue the solution $\omega(t), \gamma(t)$ to the entire interval $[0,T_\ast)$.
Finally, we obtain the comparison estimate \eqref{equ:setup_comparison_estimate} in the statement of Proposition~\ref{prop:modulation_and_orbital} from \eqref{equ:setup_modulation_proposition_proof_step1} and from \eqref{equ:setup_modulation_proposition_proof_orbital_stability_bound} by choosing $\|u_0\|_{H^1_x}~\ll~1$ sufficiently small, if necessary, depending on the size of $\omega_0$.
This finishes the proof of the proposition.
\end{proof}

\subsection{Evolution equation for the profile}

In order to analyze the asymptotic behavior of the radiation term $U(t)$, we need to pass to a time-independent reference operator in the evolution equation~\eqref{equ:setup_perturbation_equ}.
Starting from the setup provided by Proposition~\ref{prop:modulation_and_orbital}, 
we fix $\ulomega \in (0,\infty)$ with $\frac12 \omega_0 \leq \ulomega \leq 2\omega_0$.
Then we write the evolution equation \eqref{equ:setup_perturbation_equ} for $U(t)$ with respect to the time-independent Schr\"odinger operator $\calH(\ulomega)$ as
\begin{equation} \label{equ:setup_evol_equ_aux1}
    i\partial_t U - \calH(\ulomega)U = (\dot{\gamma} - \ulomega) \sigma_3 U + \calQ_\ulomega(U) + \calC(U) + \calMod + \calE_1 + \calE_2, \quad U := \begin{bmatrix} u \\ \baru \end{bmatrix},
\end{equation}
where 
\begin{align}
 \calQ_\ulomega(U) &:= \begin{bmatrix} - \phi_\ulomega \bigl( u^2 + 2 u \baru \bigr) \\ \phi_\ulomega \bigl( \baru^2 + 2 u \baru \bigr) \end{bmatrix}, \label{equ:setup_definition_calQulomega} \\
 \calC(U) &:= \begin{bmatrix} - u \baru u \\ \baru u \baru \end{bmatrix}, \label{equ:setup_definition_calC} \\
 \calMod &:= -i(\dot{\gamma} - \omega) Y_{1, \omega} - i \dot{\omega} Y_{2, \omega}, \label{equ:setup_definition_calMod} 
\end{align}
In \eqref{equ:setup_evol_equ_aux1} we passed to the fixed parameter $\ulomega \in (0,\infty)$ in the time-independent reference operator $\calH(\ulomega)$ and in the quadratic nonlinearity, which produced the error terms
\begin{align}
 \calE_1 &:= \bigl( \calV(\omega) - \calV(\ulomega) \bigr) U := \begin{bmatrix} -2\phi_\omega^2 + 2 \phi_\ulomega^2 & -\phi_\omega^2 + \phi_\ulomega^2 \\ \phi_\omega^2 - \phi_\ulomega^2 & 2 \phi_\omega^2 - 2 \phi_\ulomega^2 \end{bmatrix} U, \label{equ:setup_definition_calE1} \\
 \calE_2 &:= \calQ_\omega(U) - \calQ_{\ulomega}(U) = \begin{bmatrix}
              -(\phi_\omega - \phi_{\ulomega}) (u^2 + 2  u \baru) \\
              (\phi_\omega - \phi_{\ulomega}) (\bar{u}^2 + 2 u \bar{u})
             \end{bmatrix}. \label{equ:setup_definition_calE2}
\end{align}

Denoting by $\ulPd$ the Riesz projection onto the discrete spectrum of $\calH(\ulomega)$ defined in \eqref{equ:definition_Pd}, and denoting by $\ulPe = I - \ulPd$ the corresponding projection to the essential spectrum,
Lemma~\ref{lemma: L2 decomposition} gives the following decomposition of the even radiation term 
\begin{equation} \label{equ:setup_decomposition_radiation}
 U(t,x) = (\ulPe U)(t,x) + (\ulPd U)(t,x) = (\ulPe U)(t,x) + d_{1,\ulomega}(t) Y_{1,\ulomega}(x) + d_{2,\ulomega}(t) Y_{2,\ulomega}(x),
\end{equation}
where
\begin{equation*}
 d_{1,\ulomega}(t) = \frac{\langle U(t), \sigma_2 Y_{2, \ulomega}\rangle}{\langle Y_{1,\ulomega}, \sigma_2 Y_{2,\ulomega}\rangle}, \quad d_{2,\ulomega}(t) = \frac{\langle U(t), \sigma_2 Y_{1,\ulomega}\rangle}{\langle Y_{2,\ulomega}, \sigma_2 Y_{1,\ulomega}\rangle}.
\end{equation*}
We recall from \cite[Remark 9.5]{KS06} that the Riesz projection $\ulPe$ preserves the space of $\calJ$-invariant functions. Hence, it holds that $\sigma_1 \overline{(\ulPe U)} = \ulPe U$, and we can therefore write
\begin{equation*}
 \ulPe U = \begin{bmatrix} \usube \\ \barusube \end{bmatrix}.
\end{equation*}

Next, we introduce the profile $F_{\ulomega}(t)$ of the radiation term with respect to the Schr\"odinger operator $\calH(\ulomega)$.
Filtering out the linear evolution from $(\ulPe U)(t)$, we define
\begin{equation} \label{equ:setup_definition_profile_Fulomega}
    F_{\ulomega}(t) := e^{it\calH(\ulomega)} (\ulPe U)(t).
\end{equation}
It follows from \eqref{equ:setup_evol_equ_aux1} that the evolution equation for the profile $F_\ulomega(t)$ is given by
\begin{equation} \label{equ:setup_evol_equ_aux2}
    i\pt F_{\ulomega}(t) = e^{it\calH(\ulomega)} \ulPe \Bigl( (\dot{\gamma} - \ulomega) \sigma_3 U + \calQ_\ulomega(U) + \calC(U) + \calMod + \calE_1 + \calE_2 \Bigr).
\end{equation}
Recall from \eqref{eqn:def-dFT} that the (vector-valued) distorted Fourier transform of the profile $F_\ulomega(t)$ relative to $\calH(\ulomega)$ is
\begin{equation} \label{equ:setup_definition_distFT_of_profile_Fulomega}
 \wtilcalF_\ulomega\bigl[ F_{\ulomega}(t,\cdot) \bigr](\xi) 
 =
 \begin{bmatrix}
     \wtilcalF_{+, \ulomega}\bigl[ F_{\ulomega}(t,\cdot) \bigr](\xi) \\ \wtilcalF_{-, \ulomega}\bigl[ F_{\ulomega}(t,\cdot) \bigr](\xi)
 \end{bmatrix}
 =
 \begin{bmatrix} \langle F_{\ulomega}(t,\cdot), \sigma_3 \Psi_{+, \ulomega}(\cdot,\xi) \rangle \\ \langle F_{\ulomega}(t,\cdot), \sigma_3 \Psi_{-, \ulomega}(\cdot,\xi)\rangle \end{bmatrix}. 
\end{equation}
We use the following notation for its components
\begin{equation} \label{equ:setup_definition_tilfpm_components}
 \tilf_{\pm, \ulomega}(t,\xi) := \wtilcalF_{\pm, \ulomega}\bigl[ F_{\ulomega}(t,\cdot) \bigr](\xi).
\end{equation}
By Proposition~\ref{prop: representation formula} and by \eqref{eqn: representation e-itH} we have the representation formula
\begin{equation} \label{equ:setup_ulPeU_representation_formula}
 \begin{aligned}
    (\ulPe U)(t,x) = \bigl( e^{-it\calH(\ulomega)} F_\ulomega(t) \bigr)(x) &= \int_\bbR e^{-it(\xi^2+\ulomega)} \tilf_{+, \ulomega}(t,\xi) \Psi_{+, \ulomega}(x,\xi) \, \ud \xi \\
    &\quad - \int_\bbR e^{it(\xi^2+\ulomega)} \tilf_{-, \ulomega}(t,\xi) \Psi_{-, \ulomega}(x,\xi) \, \ud \xi,
 \end{aligned}
\end{equation}
or in components
\begin{align}
 \usube(t,x) &= \int_\bbR e^{-it(\xi^2+\ulomega)} \tilf_{+, \ulomega}(t,\xi) \Psi_{1, \ulomega}(x,\xi) \, \ud \xi - \int_\bbR e^{it(\xi^2+\ulomega)} \tilf_{-, \ulomega}(t,\xi) \Psi_{2, \ulomega}(x,\xi) \, \ud \xi, \label{equ:setup_usube_representation_formula}  \\
 \barusube(t,x) &= \int_\bbR e^{-it(\xi^2+\ulomega)} \tilf_{+, \ulomega}(t,\xi) \Psi_{2, \ulomega}(x,\xi) \, \ud \xi - \int_\bbR e^{it(\xi^2+\ulomega)} \tilf_{-, \ulomega}(t,\xi) \Psi_{1, \ulomega}(x,\xi) \, \ud \xi. \label{equ:setup_barusube_representation_formula}
\end{align}

Using Corollary~\ref{cor:distFT_of_propagator} we infer from \eqref{equ:setup_evol_equ_aux2} that the evolution equations for the components $\tilf_{\pm, \ulomega}(t,\xi)$ of the distorted Fourier transform of the profile read
\begin{equation} \label{equ:setup_evol_equ_tilfpm1}
    \begin{aligned}
        \pt \tilf_{\pm, \ulomega}(t,\xi) &= -i e^{\pm i t(\xi^2+\ulomega)} \wtilcalF_{\pm, \ulomega}\Bigl[ (\dot{\gamma}(t) - \ulomega) \sigma_3 U(t) \Bigr](\xi) \\ 
        &\quad -i e^{\pm i t(\xi^2+\ulomega)} \wtilcalF_{\pm, \ulomega}\Bigl[ \calQ_\ulomega\bigl(U(t)\bigr) + \calC\bigl(U(t)\bigr) \Bigr](\xi) \\
        &\quad -i e^{\pm i t(\xi^2+\ulomega)} \wtilcalF_{\pm, \ulomega}\Bigl[ \calMod(t) + \calE_1(t) + \calE_2(t) \Bigr](\xi).
    \end{aligned}
\end{equation}
We now recast the right-hand side of \eqref{equ:setup_evol_equ_tilfpm1} into a form that is more suitable for the nonlinear analysis in the next sections. 
For the first term on the right-hand side of \eqref{equ:setup_evol_equ_tilfpm1} we observe that by Lemma~\ref{lem:distFT_applied_to_sigmathree_F},
\begin{align}
 e^{it(\xi^2+\ulomega)} \wtilcalF_{+, \ulomega}\bigl[ \sigma_3 U(t) \bigr](\xi) &= e^{it(\xi^2+\ulomega)} \wtilcalF_{+,\ulomega}\bigl[U(t)\bigr](\xi) + e^{it(\xi^2+\ulomega)} \calL_{\baru, \ulomega}(t,\xi), \label{equ:distFT_applied_to_sigmathree_F_first_term1} \\
 e^{-it(\xi^2+\ulomega)} \wtilcalF_{\ulomega,-}\bigl[ \sigma_3 U(t) \bigr](\xi) &= - e^{-it(\xi^2+\ulomega)} \wtilcalF_{-,\ulomega}\bigl[U(t)\bigr](\xi) +  e^{-it(\xi^2+\ulomega)} \calL_{u,\ulomega}(t,\xi), \label{equ:distFT_applied_to_sigmathree_F_first_term2}
\end{align}
where
\begin{equation} \label{equ:setup_definition_calL}
    \begin{aligned}
        \calL_{\baru, \ulomega}(t,\xi) &:= 2 \bigl\langle \baru(t,\cdot), \Psi_{2, \ulomega}(\cdot,\xi) \bigr\rangle,  \\
        \calL_{u,\ulomega}(t,\xi) &:= 2 \bigl\langle u(t,\cdot), \Psi_{2, \ulomega}(\cdot,\xi) \bigr\rangle. 
    \end{aligned}
\end{equation}
Using Corollary~\ref{cor:distFT_of_propagator}, we find 
\begin{equation*}
    \begin{aligned}
    \wtilcalF_{\pm,\ulomega}\bigl[U(t)\bigr](\xi) 
    = \wtilcalF_{\pm,\ulomega}\bigl[ \ulPe U(t)\bigr](\xi) 
    = \wtilcalF_{\pm,\ulomega}\bigl[ e^{-it\calH(\ulomega)} F_\ulomega(t) \bigr](\xi) 
    = e^{\mp it(\xi^2+\ulomega)} \tilf_{\pm, \ulomega}(t,\xi),
    \end{aligned}
\end{equation*}
whence the first terms on the right-hand sides of \eqref{equ:distFT_applied_to_sigmathree_F_first_term1} and \eqref{equ:distFT_applied_to_sigmathree_F_first_term2} simply read
\begin{equation*}
    \pm e^{\pm it(\xi^2+\ulomega)} \wtilcalF_{\pm,\ulomega}\bigl[U(t)\bigr](\xi) = \pm \tilf_{\pm, \ulomega}(t,\xi).
\end{equation*}
Moreover, in the quadratic and the cubic nonlinearities on the right-hand side of \eqref{equ:setup_evol_equ_tilfpm1}, we insert the decomposition \eqref{equ:setup_decomposition_radiation} of the radiation term $U$ into its projection to the essential spectrum and into its discrete components. Correspondingly, we write
\begin{equation*}
 \calQ_\ulomega\bigl(U(t)\bigr) + \calC\bigl(U(t)\bigr) = \calQ_\ulomega\bigl((\ulPe U)(t)\bigr) + \calC\bigl((\ulPe U)(t)\bigr) + \calE_3(t)
\end{equation*}
with a remainder term
\begin{equation} \label{equ:setup_definition_calE3}
 \calE_3(t) := \calQ_\ulomega\bigl(U(t)\bigr) - \calQ_\ulomega\bigl((\ulPe U)(t)\bigr) + \calC\bigl(U(t)\bigr) - \calC\bigl((\ulPe U)(t)\bigr).
\end{equation}

The evolution equation for $\tilf_{+, \ulomega}(t,\xi)$ now reads 
\begin{equation} \label{equ:setup_evol_equ_ftilplus}
 \begin{aligned}
  \pt \tilf_{+, \ulomega}(t,\xi) &= -i(\dot{\gamma}(t)-\ulomega) \tilf_{+, \ulomega}(t,\xi) -i (\dot{\gamma}(t)-\ulomega) e^{i t(\xi^2+\ulomega)} \calL_{\baru, \ulomega}(t,\xi) \\
  &\quad -i e^{i t(\xi^2+\ulomega)} \wtilcalF_{+, \ulomega}\Bigl[ \calQ_\ulomega\bigl( (\ulPe U)(t)\bigr) + \calC\bigl((\ulPe U)(t)\bigr)\Bigr](\xi) \\
  &\quad -i e^{i t(\xi^2+\ulomega)} \wtilcalF_{+, \ulomega}\Bigl[ \calMod(t) + \calE_1(t) + \calE_2(t) + \calE_3(t) \Bigr](\xi),
 \end{aligned}
\end{equation}
and the evolution equation for $\tilfminusulo(t,\xi)$ is given by 
\begin{equation} \label{equ:setup_evol_equ_ftilminus}
 \begin{aligned}
  \pt \tilf_{-, \ulomega}(t,\xi) &= i(\dot{\gamma}-\ulomega) \tilf_{-, \ulomega}(t,\xi) -i (\dot{\gamma}-\ulomega) e^{-i t(\xi^2+\ulomega)} \calL_{u, \ulomega}(t,\xi) \\
  &\quad -i e^{-i t(\xi^2+\ulomega)} \wtilcalF_{-, \ulomega}\Bigl[ \calQ_\ulomega\bigl( (\ulPe U)(t)\bigr) + \calC\bigl((\ulPe U)(t)\bigr)\Bigr](\xi) \\
  &\quad -i e^{-i t(\xi^2+\ulomega)} \wtilcalF_{-, \ulomega}\Bigl[ \calMod(t) + \calE_1(t) + \calE_2(t) + \calE_3(t) \Bigr](\xi).
 \end{aligned}
\end{equation}

Lastly, inspired by \cite[Proposition 9.5]{CollotGermain23}, we observe that the first term on the right-hand side of the evolution equation \eqref{equ:setup_evol_equ_ftilplus} for $\tilfplusulo(t,\xi)$ can be absorbed into the phase via an integrating factor.
Setting
\begin{equation} \label{equ:setup_definition_theta}
 \theta(t) := \int_0^t \bigl( \dot{\gamma}(s) - \ulomega \bigr) \, \ud s,
\end{equation}
we write \eqref{equ:setup_evol_equ_ftilplus} as 
\begin{equation} \label{equ:setup_evol_equ_ftilplus1}
 \begin{aligned}
  &\pt \bigl( e^{i\theta(t)} \tilf_{+, \ulomega}(t,\xi) \bigr) \\
  &\quad = -i e^{i\theta(t)} e^{i t(\xi^2+\ulomega)} \Bigl( \wtilcalF_{+, \ulomega}\bigl[ \calQ_\ulomega\bigl( (\ulPe U)(t)\bigr) \bigr](\xi) + \wtilcalF_{+, \ulomega}\bigl[\calC\bigl( (\ulPe U)(t)\bigr)\bigr](\xi) \\
  &\qquad \qquad \qquad \qquad \quad \quad + (\dot{\gamma} -\ulomega) \calL_{\baru, \ulomega}(t,\xi) + \wtilcalF_{+, \ulomega}\bigl[\calMod(t)\bigr](\xi) \\
  &\qquad \qquad \qquad \qquad \quad \quad + \wtilcalF_{+, \ulomega}\bigl[\calE_1(t)\bigr](\xi) + \wtilcalF_{+, \ulomega}\bigl[\calE_2(t)\bigr](\xi) + \wtilcalF_{+, \ulomega}\bigl[\calE_3(t)\bigr](\xi) \Bigr).
 \end{aligned}
\end{equation}
The evolution equation \eqref{equ:setup_evol_equ_ftilminus} can be rewritten similarly, 
but in the next sections we will mostly be working with the evolution equations for $\tilfplusulo(t,\xi)$ to derive suitable bounds for both components of the distorted Fourier transform \eqref{equ:setup_definition_distFT_of_profile_Fulomega} of the profile.
The reason is that thanks to Lemma~\ref{lem:distFT_components_relation} the components $\tilfplusulo(t,\xi)$ and $\tilfminusulo(t,\xi)$ are related by the simple identity
\begin{equation} \label{equ:setup_distFT_components_relation}
 \tilf_{-, \ulomega}(t,\xi) = - \frac{(|\xi|-i\sqrt{\ulomega})^2}{(|\xi|+i\sqrt{\ulomega})^2} \overline{\tilf_{+,\ulomega}(t,-\xi)},
\end{equation}
so that bounds for $\tilfplusulo(t,\xi)$ can be transferred to $\tilfminusulo(t,\xi)$ via \eqref{equ:setup_distFT_components_relation}.

\subsection{Normal form transformation} \label{subsec:normal_form}

As a final preparation for the nonlinear analysis in the next sections, we exploit a remarkable null structure to recast the quadratic nonlinearity $\calQ_\ulomega(\ulPe U)$ in the evolution equation \eqref{equ:setup_evol_equ_ftilplus1} for $\tilfplusulo(t,\xi)$ into a better form.
To this end we implement a version of a variable coefficient normal form introduced in \cite{LLS2, LS1}.

In view of the spatial localization of the coefficient $\phi_\ulomega(x)$, the leading order behavior of the quadratic nonlinearity $\calQ_\ulomega(\ulPe U)$ is governed by the local decay of the radiation term, which is slow due to the two threshold resonances on the edges of the essential spectrum of the linearized operator $\calH(\ulomega)$.
In order to peel off the leading order local decay behavior of the radiation term, we therefore single out the low frequency contributions in the representation formulas \eqref{equ:setup_usube_representation_formula} and \eqref{equ:setup_barusube_representation_formula} by writing
\begin{equation} \label{equ:setup_decompositions_usube_local_decay}
 \begin{aligned}
    \usube(t,x) &= h_{1,\ulomega}(t) \Phi_{1,\ulomega}(x) - h_{2,\ulomega}(t) \Phi_{2,\ulomega}(x) + R_{u,\ulomega}(t,x), \\
    \barusube(t,x) &= h_{1,\ulomega}(t) \Phi_{2,\ulomega}(x) - h_{2,\ulomega}(t) \Phi_{1,\ulomega}(x) + R_{\baru,\ulomega}(t,x),
 \end{aligned}
\end{equation}
where we define
\begin{equation} \label{equ:setup_definitions_Phijulomega}
    \begin{aligned}
        \Phi_{1,\ulomega}(x) &:= \Psi_{1,\ulomega}(x,0) = \frac{1}{\sqrt{2\pi}} \tanh^2(\sqrt{\ulomega}x), \\
        \Phi_{2,\ulomega}(x) &:= \Psi_{2,\ulomega}(x,0) = - \frac{1}{\sqrt{2\pi}} \sech^2(\sqrt{\omega}x), 
    \end{aligned}
\end{equation}
and
\begin{equation} \label{equ:setup_definitions_h12ulomega}
 \begin{aligned}
  h_{1,\ulomega}(t) &:= e^{-it\ulomega} \int_\bbR e^{-it\xi^2} \chi_0(\xi) \tilf_{+,\ulomega}(t,\xi) \, \ud \xi, \\
  h_{2,\ulomega}(t) &:= e^{it\ulomega} \int_\bbR e^{it\xi^2} \chi_0(\xi) \tilf_{-,\ulomega}(t,\xi) \, \ud \xi.
 \end{aligned}
\end{equation}
Here, we denote by $\chi_0(\xi)$ a smooth even non-negative bump function  with $\chi_0(\xi) = 1$ for $|\xi| \leq 1$ and $\chi_0(\xi) = 0$ for $|\xi| \geq 2$. 
Inserting the decompositions \eqref{equ:setup_decompositions_usube_local_decay} into the quadratic nonlinearity, we obtain the expansion
\begin{equation} \label{equ:setup_expansion_quadratic_nonlinearity}
    \begin{aligned}
        &\calQ_\ulomega\bigl( (\ulPe U)(t)\bigr)(x) \\
        &\quad = h_{1,\ulomega}(t)^2 \calQ_{1,\ulomega}(x) + h_{1,\ulomega}(t) h_{2,\ulomega}(t) \calQ_{2,\ulomega}(x) + h_{2,\ulomega}(t)^2 \calQ_{3,\ulomega}(x) + \calQsubrom\bigl((\ulPe U)(t)\bigr)(x)
    \end{aligned}
\end{equation}
with
\begin{align*}
 \calQ_{1,\ulomega}(x) &:= \begin{bmatrix}
                            -\phi_{\ulomega}(x) \bigl( \Phi_{1,\ulomega}(x)^2 + 2 \Phi_{1,\ulomega}(x) \Phi_{2,\ulomega}(x) \bigr) \\
                            \phi_{\ulomega}(x) \bigl( \Phi_{2,\ulomega}(x)^2 + 2 \Phi_{1,\ulomega}(x) \Phi_{2,\ulomega}(x) \bigr)
                           \end{bmatrix}, \\
 \calQ_{2,\ulomega}(x) &:= \begin{bmatrix}
                            2 \phiulomega(x) \bigl( \Phioneulomega(x)^2 + \Phitwoulomega(x)^2 + \Phioneulomega(x) \Phitwoulomega(x) \bigr) \\
                            -2 \phiulomega(x) \bigl( \Phioneulomega(x)^2 + \Phitwoulomega(x)^2 + \Phioneulomega(x) \Phitwoulomega(x) \bigr)
                           \end{bmatrix}, \\
 \calQ_{3,\ulomega}(x) &:= \begin{bmatrix}
                            - \phiulomega(x) \bigl( \Phitwoulomega(x)^2 + 2 \Phioneulomega(x) \Phitwoulomega(x) \bigr) \\ \phiulomega(x) \bigl( \Phioneulomega(x)^2 + 2 \Phioneulomega(x) \Phitwoulomega(x) \bigr)
                           \end{bmatrix},
\end{align*}
and the remainder term
\begin{align}
 \calQsubrom\bigl((\ulPe U)(t)\bigr)(x) &:= \calQ_{\mathrm{r},1,\ulomega}\bigl((\ulPe U)(t)\bigr)(x) + \calQ_{\mathrm{r},2,\ulomega}\bigl((\ulPe U)(t)\bigr)(x) + \calQ_{\mathrm{r},3,\ulomega}\bigl((\ulPe U)(t)\bigr)(x) \label{eqn:def_renormalized_quadratic}
\end{align}
consisting of the terms 
\begin{align*}
 \calQ_{\mathrm{r},1,\ulomega}\bigl((\ulPe U)(t)\bigr)(x) &:= \begin{bmatrix} - 2\phiulomega(x) \bigl( h_{1,\ulomega}(t) \Phi_{1,\ulomega}(x) - h_{2,\ulomega}(t) \Phi_{2,\ulomega}(x)\bigr) \bigl( R_{u,\ulomega}(t,x) + R_{\baru,\ulomega}(t,x) \bigr)\\ 2 \phiulomega(x) \bigl( h_{1,\ulomega}(t) \Phi_{1,\ulomega}(x) - h_{2,\ulomega}(t) \Phi_{2,\ulomega}(x) \bigr) R_{\baru,\ulomega}(t,x) \end{bmatrix},  \\
 \calQ_{\mathrm{r},2,\ulomega}\bigl((\ulPe U)(t)\bigr)(x) &:= \begin{bmatrix} - 2 \phiulomega(x) \bigl( h_{1,\ulomega}(t) \Phi_{2,\ulomega}(x) - h_{2,\ulomega}(t) \Phi_{1,\ulomega}(x)\bigr) R_{\baru,\ulomega}(t,x) \\ 2 \phiulomega(x) \bigl( h_{1,\ulomega}(t) \Phi_{2,\ulomega}(x) - h_{2,\ulomega}(t) \Phi_{1,\ulomega}(x)\bigr) \bigl( R_{u,\ulomega}(t,x) + R_{\baru,\ulomega}(t,x) \bigr) \end{bmatrix}, \\
 \calQ_{\mathrm{r},3,\ulomega}\bigl((\ulPe U)(t)\bigr)(x) &:= \begin{bmatrix} -\phiulomega(x) \bigl( R_{u,\ulomega}(t,x)^2 + 2 R_{u,\ulomega}(t,x) R_{\bar{u},\ulomega}(t,x) \bigr) \\ \phiulomega(x) \bigl( R_{\baru,\ulomega}(t,x)^2 + 2 R_{u,\ulomega}(t,x) R_{\bar{u},\ulomega}(t,x) \bigr) \end{bmatrix}.
\end{align*}
Observe that the components of $\calQ_{1,\ulomega}(x)$, $\calQ_{2,\ulomega}(x)$, and $\calQ_{3,\ulomega}(x)$ are real-valued Schwartz functions.
By a formal stationary phase analysis, we expect $h_{1,\ulomega}(t) \sim t^{-\frac12} e^{-it\ulomega}$ and $h_{2,\ulomega}(t) \sim t^{-\frac12} e^{it\ulomega}$ as $t \to \infty$. Correspondingly, we expect that $\pt ( e^{it\ulomega} h_{1,\ulomega}(t) )$ and $\pt ( e^{-it\ulomega} h_{2,\ulomega}(t) )$ have stronger time decay.
Filtering out the phases, we write
\begin{equation} \label{equ:wtilcalF_of_Q1}
 \begin{aligned}
  -i e^{i t(\xi^2+\ulomega)} \wtilcalF_{+, \ulomega}\bigl[ \calQ_\ulomega\bigl((\ulPe U)(t)\bigr) \bigr](\xi) &= -i e^{i t(\xi^2-\ulomega)} \bigl( e^{it\ulomega} h_{1,\ulomega}(t) \bigr)^2 \widetilde{\calQ}_{1,\ulomega}(\xi) \\
  &\quad -i e^{i t(\xi^2+\ulomega)} \bigl( e^{it\ulomega} h_{1,\ulomega}(t) \bigr) \bigl( e^{-it\ulomega} h_{2,\ulomega}(t) \bigr) \widetilde{\calQ}_{2,\ulomega}(\xi) \\
  &\quad -i e^{i t(\xi^2+3\ulomega)} \bigl( e^{-it\ulomega} h_{2,\ulomega}(t) \bigr)^2 \widetilde{\calQ}_{3,\ulomega}(\xi) \\
  &\quad -i e^{i t(\xi^2+\ulomega)} \wtilcalF\bigl[ \calQ_{\mathrm{r},\ulomega}\bigl((\ulPe U)(t)\bigr) \bigr](\xi),
 \end{aligned}
\end{equation}
where we use the short-hand notation
\begin{equation*}
 \widetilde{\calQ}_{j,\ulomega}(\xi) := \wtilcalF_{+,\ulomega}\bigl[\calQ_{j,\ulomega}\bigr](\xi), \quad 1 \leq j \leq 3.
\end{equation*}

While the phases $e^{i t(\xi^2+\ulomega)}$ and $e^{i t(\xi^2+3\ulomega)}$ have no time resonances, the phase $e^{i t(\xi^2-\ulomega)}$ does have time resonances at the frequencies $\xi = \pm \sqrt{\ulomega}$.
However, the corresponding coefficient $\widetilde{\calQ}_{1,\ulomega}(\xi)$ exhibits the following remarkable null structure, which has already been observed in \cite[Lemma 1.6]{Li23}.
\begin{lemma} \label{lem:null_structure_radiation}
  For any $\ulomega \in (0,\infty)$ it holds that
 \begin{equation*}
  \widetilde{\calQ}_{1,\ulomega}(\xi) = (\ulomega-\xi^2) \frac{1}{24\sqrt{\pi}} \frac{\xi^2 (\ulomega+\xi^2 )}{\ulomega^2(\vert \xi \vert+i\sqrt{\ulomega})^2} \sech\left( \frac{\pi}{2} \frac{\xi}{\sqrt{\ulomega}}\right).
 \end{equation*}
 In particular, we have
 \begin{equation} \label{equ:null_structure_radiation_bad_freq_vanishing}
  \widetilde{\calQ}_{1,\ulomega}(\pm \sqrt{\ulomega}) = 0.
 \end{equation}
\end{lemma}
\begin{proof}
Starting from the definitions \eqref{eqn: Psi_pm_omega}, \eqref{equ:setup_definition_calQulomega}, and \eqref{equ:setup_definitions_Phijulomega}, we compute
\begin{equation*}
	\begin{split}
		\widetilde{\calQ}_{1,\ulomega}(\xi) 
        &= - \bigl\langle \phi_{\ulomega}(\Phi_{1,\ulomega}^2+2\Phi_{1,\ulomega}\Phi_{2,\ulomega}),\Psi_{1, \ulomega}(\cdot,\xi) \bigr\rangle - \bigl\langle \phi_{\ulomega}(\Phi_{2,\ulomega}^2+2\Phi_{1,\ulomega}\Phi_{2,\ulomega}),\Psi_{2, \ulomega}(\cdot,\xi) \bigr\rangle \\
		&= -\frac{\sqrt{2\ulomega}}{(2\pi)^{\frac32}}\int_\bbR e^{-ix\xi} \frac{\big(\xi-i\sqrt{\ulomega}\tanh(\sqrt{\ulomega}x)\big)^2}{(\vert \xi \vert+i\sqrt{\ulomega})^2} \sech(\sqrt{\ulomega}x) \\ 
        &\qquad \qquad \qquad \qquad \qquad \qquad \qquad \times \Bigl( \tanh^4(\sqrt{\ulomega}x)-2\tanh^2(\sqrt{\ulomega}x)\sech^2(\sqrt{\ulomega}x) \Bigr) \, \ud x \\
		&\quad -\frac{\sqrt{2\ulomega}}{(2\pi)^{\frac32}}\int_\bbR e^{-ix\xi} \frac{\big(\sqrt{\ulomega} \sech(\sqrt{\ulomega}x)\big)^2}{(\vert \xi \vert+i\sqrt{\ulomega})^2} \sech(\sqrt{\ulomega}x) \\ 
        &\qquad \qquad \qquad \qquad \qquad \qquad \qquad \times \Bigl( \sech^4(\sqrt{\ulomega}x)-2\tanh^2(\sqrt{\ulomega}x)\sech^2(\sqrt{\ulomega}x) \Bigr) \, \ud x.
	\end{split}
\end{equation*}
Changing variables to $y := \sqrt{\ulomega}x$ in both integrals, we obtain
\begin{equation*}
	\begin{split}
        \widetilde{\calQ}_{1,\ulomega}(\xi)= -\frac{\sqrt{2}}{(2\pi)^{\frac32}} \frac{1}{(\vert \ulxi\vert+i)^2}\int_\bbR e^{-iy\ulxi}H(y,\ulxi) \,\ud y,
	\end{split}
\end{equation*}
where we use the short-hand notation $\ulxi := \ulomega^{-\frac12}\xi$ and where
\begin{equation*}
\begin{split}
H(y,\ulxi) &:= \sech(y)\Big( \big(\ulxi-i\tanh(y)\big)^2\big(\tanh^4(y)-2\tanh^2(y)\sech^2(y)\big) \\
&\qquad \qquad \qquad \qquad + \sech^2(y)\big(\sech^4(y)-2\tanh^2(y)\sech^2(y)\big) \Big).
\end{split}
\end{equation*}
Using the identity $\tanh^2(y) = 1 - \sech^2(y)$, we now write $H(y,\ulxi)$ as a linear combination of the functions $\sech^\ell(y)$ and $\sech^\ell(y)\tanh(y)$ for $\ell \geq 1$,
\begin{equation*}
	\begin{aligned}
        H(y,\ulxi)  &= (\ulxi^2-1)\sech(y)-(4\ulxi^2-5)\sech^3(y)+(3\ulxi^2-9)\sech^5(y)+6\sech^7(y) \\
                    &\quad  -2i\ulxi\sech(y)\tanh(y) + 8i\ulxi \sech^3(y)\tanh(y) - 6i\ulxi \sech^5(y)\tanh(y).
	\end{aligned}
\end{equation*}
Invoking the Fourier transform identities \eqref{eqn: FT-sech}, \eqref{equ:FT_sech3}--\eqref{equ:FT_sech5tanh} from Appendix~\ref{appendix-FT-formulas}, we find that 
\begin{equation*}
	\begin{aligned}
    \widetilde{\calQ}_{1,\ulomega}(\xi) &=-\frac{\sqrt{2}}{2\pi} \frac{1}{(\vert \ulxi\vert+i)^2} \Bigg( (\ulxi^2-1) - \frac12 (1+\ulxi^2)(4\ulxi^2-5) + \frac{1}{24} (1+\ulxi^2)(9+\ulxi^2)(3\ulxi^2-9) \\
    &\qquad \qquad \qquad \qquad \qquad + \frac{1}{120} (1+\ulxi^2)(9+\ulxi^2)(25+\ulxi^2) - 2\ulxi^2 \\
    &\qquad \qquad \qquad \qquad \qquad + \frac{4}{3} \ulxi^2(1+\ulxi^2) - \frac{1}{20} \ulxi^2 (1+\ulxi^2) (9+\ulxi^2) \Bigg)\sqrt{\frac{\pi}{2}} \sech\Bigl( \frac{\pi}{2} \ulxi \Bigr) \\
    &= \frac{1}{24\sqrt{\pi}} \frac{\ulxi^2(1-\ulxi^4)}{(\vert \ulxi \vert+i)^2} \sech\Bigl( \frac{\pi}{2} \ulxi \Bigr) \\
    &= \frac{1}{24\sqrt{\pi}} \frac{\xi^2 (\ulomega+\xi^2 )(\ulomega-\xi^2)}{\ulomega^2(\vert \xi \vert+i\sqrt{\ulomega})^2} \sech\Bigl( \frac{\pi}{2} \frac{\xi}{\sqrt{\ulomega}} \Bigr),
	\end{aligned}
\end{equation*}
where in the last line we changed back to the original frequency variable $\xi$.
\end{proof}

\begin{remark} \label{rem:setup_decay_calQ_coefficients}
    In view of Lemma~\ref{lem:null_structure_radiation} the coefficients $\widetilde{\calQ}_1(\xi)$ as well as $(\xi^2-\ulomega)^{-1} \widetilde{\calQ}_1(\xi)$ are rapidly decaying and smooth up to the factor $(|\xi| + i\sqrt{\ulomega})^{-2}$, which stems from the definition of the distorted Fourier basis elements.
    The same conclusions hold for the coefficients $\widetilde{\calQ}_{2, \ulomega}(\xi)$, $(\xi^2+\ulomega)^{-1} \widetilde{\calQ}_{2,\ulomega}(\xi)$ as well as for $\widetilde{\calQ}_3(\xi)$, $(\xi^2+3\ulomega)^{-1} \widetilde{\calQ}_3(\xi)$.
\end{remark}

Thanks to the null structure \eqref{equ:null_structure_radiation_bad_freq_vanishing} uncovered in Lemma~\ref{lem:null_structure_radiation}, we can rewrite \eqref{equ:wtilcalF_of_Q1} as
\begin{equation} \label{equ:setup_quadratic_rewritten}
 \begin{aligned}
  &-i e^{i t(\xi^2+\ulomega)} \wtilcalF_{+, \ulomega}\bigl[ \calQ_\ulomega\bigl((\ulPe U)(t)\bigr) \bigr](\xi) \\
  &\quad = - \pt \bigl( \wtilB_\ulomega(t,\xi) \bigr) -i e^{it(\xi^2+\ulomega)} \Bigl( \calR_{q,\ulomega}(t,\xi) + \wtilcalF\bigl[ \calQ_{\mathrm{r},\ulomega}\bigl((\ulPe U)(t)\bigr) \bigr](\xi) \Bigr)
 \end{aligned}
\end{equation}
with
\begin{equation}\label{equ:setup_definition_wtilBulomega}
 \begin{aligned}
  \wtilB_\ulomega(t,\xi) &:= e^{it(\xi^2-\ulomega)} \bigl( e^{it\ulomega} h_{1,\ulomega}(t) \bigr)^2 (\xi^2-\ulomega)^{-1} \wtilQ_{1,\ulomega}(\xi) \\
  &\quad + e^{it(\xi^2+\ulomega)} \bigl( e^{it\ulomega} h_{1,\ulomega}(t) \bigr) \bigl( e^{-it\ulomega} h_{2,\ulomega}(t) \bigr)  (\xi^2+\ulomega)^{-1} \wtilQ_{2,\ulomega}(\xi) \\
  &\quad + e^{it(\xi^2+3\ulomega)} \bigl( e^{-it\ulomega} h_{2,\ulomega}(t) \bigr)^2 (\xi^2+3\ulomega)^{-1} \wtilQ_{3,\ulomega}(\xi)
 \end{aligned}
\end{equation}
and
\begin{equation}\label{equ:setup_definition_wtilcalR_qulomega}
 \begin{aligned}
  \widetilde{\calR}_{q,\ulomega}(t,\xi) &:= 2i e^{-2it\ulomega} \bigl( e^{it\ulomega} h_{1,\ulomega}(t) \bigr) \pt \bigl( e^{it\ulomega} h_{1,\ulomega}(t) \bigr) (\xi^2-\ulomega)^{-1} \wtilQ_{1,\ulomega}(\xi) \\
  &\quad + i \pt \bigl( e^{it\ulomega} h_{1,\ulomega}(t) \bigr) \bigl( e^{-it\ulomega} h_{2,\ulomega}(t) \bigr) (\xi^2+\ulomega)^{-1} \wtilQ_{2,\ulomega}(\xi) \\
  &\quad + i \bigl( e^{it\ulomega} h_{1,\ulomega}(t) \bigr) \pt \bigl( e^{-it\ulomega} h_{2,\ulomega}(t) \bigr) (\xi^2+\ulomega)^{-1} \wtilQ_{2,\ulomega}(\xi) \\
  &\quad + 2i e^{2it\ulomega} \bigl( e^{-it\ulomega} h_{2,\ulomega}(t) \bigr) \pt \bigl( e^{-it\ulomega} h_{2,\ulomega}(t) \bigr) (\xi^2+3\ulomega)^{-1} \wtilQ_{3,\ulomega}(\xi).
 \end{aligned}
\end{equation}
Inserting \eqref{equ:setup_quadratic_rewritten} back into \eqref{equ:setup_evol_equ_ftilplus1}, we arrive at the renormalized evolution equation 
\begin{equation}\label{equ:setup_evol_equ_renormalized_tilfplus}
 \begin{aligned}
  \pt \Bigl( e^{i\theta(t)} \bigl( \tilf_{+,\ulomega}(t,\xi) + \wtilB_\ulomega(t,\xi) \bigr) \Bigr) &= -i e^{i\theta(t)} e^{i t(\xi^2+\ulomega)} \wtilcalF_{+, \ulomega}\bigl[\calC\bigl((\ulPe U)(t)\bigr)\bigr](\xi) \\
  &\quad \, -i e^{i\theta(t)} e^{i t(\xi^2+\ulomega)} \widetilde{\calR}_\ulomega(t,\xi) \\
  &\quad \, + i e^{i\theta(t)} ( \dot{\gamma}(t) - \ulomega ) \wtilB_\ulomega(t,\xi)
 \end{aligned}
\end{equation}
with
\begin{equation}\label{equ:setup_definition_wtilcalRulomega}
 \begin{aligned}
  \widetilde{\calR}_\ulomega(t,\xi) &:= \widetilde{\calR}_{q,\ulomega}(t,\xi) + \wtilcalF_{+,\ulomega}\bigl[ \calQ_{\mathrm{r},\ulomega}\bigl((\ulPe U)(t)\bigr) \bigr](\xi) \\
  &\quad \, + (\dot{\gamma}(t) -\ulomega) \calL_{\baru, \ulomega}(t,\xi) + \wtilcalF_{+, \ulomega}\bigl[\calMod(t)\bigr](\xi) \\
  &\quad \, + \wtilcalF_{+, \ulomega}\bigl[\calE_1(t)\bigr](\xi) + \wtilcalF_{+, \ulomega}\bigl[\calE_2(t)\bigr](\xi) + \wtilcalF_{+,\ulomega}\bigl[\calE_3(t)\bigr](\xi).
 \end{aligned}
\end{equation}
For the analysis in the next sections it will be useful to keep in mind that all terms contained in $\widetilde{\calR}_\ulomega(t,\xi)$ can be thought of as spatially localized with at least cubic-type $\jt^{-\frac32+\delta}$ time decay.

\section{Nonlinear Spectral Distributions} \label{sec:spectral_distributions}

In this section we determine the structure of cubic spectral distributions, which arise in the evolution equations for the profile, and of quadratric spectral distributions that occur in the modulation equations.

\subsection{Cubic spectral distributions for the profile equation} \label{subsec:cubic_spectral_distributions}
Fix $\ulomega \in (0,\infty)$.
In view of \eqref{eqn: Psi_pm_omega}, \eqref{eqn:def-dFT}, and \eqref{equ:setup_definition_calC}, we have 
\begin{equation*}
	\wtilcalF_{+, \ulomega}\bigl[\calC\big((\ulPe U)(t)\big)\bigr](\xi) = -\langle \usube \barusube \usube,\Psi_{1,\ulomega}(\cdot,\xi)\rangle - \langle \barusube \usube\barusube,\Psi_{2,\ulomega}(\cdot,\xi)\rangle.
\end{equation*}
Inserting the representation formula \eqref{equ:setup_usube_representation_formula} for $\usube(t,x)$, and using the complex conjugate of \eqref{equ:setup_usube_representation_formula} for $\barusube(t,x)$, leads to the following expression involving 16 terms
\begin{equation*}
	\begin{split}
&\wtilcalF_{+, \ulomega}\bigl[\calC\big((\ulPe U)(t)\big)\bigr](\xi) \\
&=\sum_{S}\fraks \iiint e^{it\Phi_{j_1 j_2 j_3}(\xi_1,\xi_2,\xi_3)}	\tilf_{j_1,\ulomega}(t,\xi_1)\overline{\tilf_{j_2,\ulomega}}(t,\xi_2)\tilf_{j_3,\ulomega}(t,\xi_3)\mu_{1,k_1,k_2,k_3,\ulomega}(\xi,\xi_1,\xi_2,\xi_3)\,\ud\xi_1 \,\ud\xi_2 \,\ud\xi_3 \\
&\quad +\sum_{S} \fraks  \iiint e^{-it\Phi_{j_1 j_2 j_3}(\xi_1,\xi_2,\xi_3)}	\overline{\tilf_{j_1,\ulomega}}(t,\xi_1)\tilf_{j_2,\ulomega}(t,\xi_2)\overline{\tilf_{j_3,\ulomega}}(t,\xi_3)\mu_{2,k_1,k_2,k_3,\ulomega}(\xi,\xi_1,\xi_2,\xi_3)\,\ud\xi_1 \,\ud\xi_2 \,\ud\xi_3 .
	\end{split}
\end{equation*}
The following notations are used in the preceding expression: The set $S$ is given by
\begin{equation*}
	S := \big\{ (j_1,k_1),(j_2,k_2),(j_3,k_3) \mid (j_\ell,k_\ell) \in  \{(+,1),(-,2)\}, \quad 1\leq \ell \leq 3 \big\},
\end{equation*}
and $\fraks = \fraks(j_1,j_2,j_3) \in \{-1,+1\}$ are signs defined by
\begin{equation*}
		\fraks(j_1,j_2,j_3) := -(j_1 1)(j_2 1)(j_3 1).
\end{equation*}
The phase takes the form
\begin{equation*}
	\Phi_{j_1 j_2 j_3}(\xi_1,\xi_2,\xi_3) := -j_1(\xi_1^2+\ulomega)+j_2(\xi_2^2+\ulomega)-j_3(\xi_3^2+\ulomega),
\end{equation*}
where $j_1,j_2,j_3 \in \{+,-\}$. The 16 cubic spectral distributions are formally given by
\begin{align}
		\mu_{1,k_1,k_2,k_3,\ulomega}(\xi,\xi_1,\xi_2,\xi_3) &:= \int_\bbR \overline{\Psi_{1,\ulomega}}(x,\xi) \Psi_{k_1,\ulomega}(x,\xi_1)\overline{\Psi_{k_2,\ulomega}}(x,\xi_2)\Psi_{k_3,\ulomega}(x,\xi_3) \,\ud x,\label{eqn: mu_1k1k2k3}\\
		\mu_{2,k_1,k_2,k_3,\ulomega}(\xi,\xi_1,\xi_2,\xi_3) &:= \int_\bbR \overline{\Psi_{2,\ulomega}}(x,\xi) \overline{\Psi_{k_1,\ulomega}}(x,\xi_1)\Psi_{k_2,\ulomega}(x,\xi_2)\overline{\Psi_{k_3,\ulomega}}(x,\xi_3) \,\ud x,\label{eqn: mu_2k1k2k3}
\end{align}
where $k_1,k_2,k_3 \in \{1,2\}$. We observe from the explicit expressions \eqref{eqn:m-1,omega}--\eqref{eqn:m-2,omega} that only $\mu_{1,1,1,1,\ulomega}$ is singular, while the other 15 cubic spectral distributions are regular in the sense of the following definition.

\begin{definition} \label{def:reg-cubic}
A cubic spectral distribution $\mu(\xi,\xi_1,\xi_2,\xi_3)$ is regular if $\mu$ can be written as a linear combination of terms of the form
\begin{equation*}
	\overline{\frakb_0}(\xi)	\frakb_1(\xi_1)	\overline{\frakb_2}(\xi_2)	\frakb_3(\xi_3) \kappa(\xi-\xi_1+\xi_2-\xi_3) \quad \text{or}\quad \overline{\frakb_0}(\xi)	\overline{\frakb_1}(\xi_1)	{\frakb_2}(\xi_2)	\overline{\frakb_3}(\xi_3) \kappa(\xi+\xi_1-\xi_2+\xi_3),
\end{equation*}
where the multipliers $\frakb_0,\frakb_1,\frakb_2,\frakb_3 \in W^{1,\infty}(\bbR)$ are of the form
\begin{equation}\label{eqn: symbol_frakb}
	\frac{1}{(\vert \xi \vert - i\sqrt{\ulomega})^2},\quad 	\frac{\xi}{(\vert \xi \vert - i\sqrt{\ulomega})^2},\quad \text{or} \quad 	\frac{\xi^2}{(\vert \xi \vert - i\sqrt{\ulomega})^2},
\end{equation}
and where $\kappa(\xi) = \widehat{\calF}[\varphi](\xi)$ for  some Schwartz function $\varphi(x)$.
\end{definition}

Next, we provide a complete breakdown of the singular cubic spectral distribution $\mu_{1,1,1,1,\ulomega}$.

\begin{lemma} \label{lemma:cubic_NSD}
We have
\begin{equation}
\begin{split}
&	\mu_{1,1,1,1,\ulomega}(\xi,\xi_1,\xi_2,\xi_3) = \mu_{\delta_0,\ulomega}(\xi,\xi_1,\xi_2,\xi_3) + \mu_{\pvdots,\ulomega}(\xi_1,\xi_2,\xi_3,\xi_4) + \mu_{\mathrm{reg},\ulomega}(\xi,\xi_1,\xi_2,\xi_3),
\end{split}
\end{equation}
where $\mu_{\mathrm{reg},\ulomega}$ is a regular cubic spectral distribution in the sense of Definition~\ref{def:reg-cubic},  and as tempered distributions,
\begin{equation}
\mu_{\delta_0,\ulomega}(\xi,\xi_1,\xi_2,\xi_3) :=	\frac{1}{2\pi\sqrt{\ulomega}} \frac{\frakp_1\Big(\tfrac{\xi}{\sqrt{\ulomega}},\tfrac{\xi_1}{\sqrt{\ulomega}},\tfrac{\xi_2}{\sqrt{\ulomega}},\tfrac{\xi_3}{\sqrt{\ulomega}}\Big)}{\frakp\Big(\tfrac{\xi}{\sqrt{\ulomega}},\tfrac{\xi_1}{\sqrt{\ulomega}},\tfrac{\xi_2}{\sqrt{\ulomega}},\tfrac{\xi_3}{\sqrt{\ulomega}}\Big)}\delta_0\Big(\tfrac{\xi}{\sqrt{\ulomega}}-\tfrac{\xi_1}{\sqrt{\ulomega}}+\tfrac{\xi_2}{\sqrt{\ulomega}}-\tfrac{\xi_3}{\sqrt{\ulomega}}\Big),
\end{equation}
and
\begin{equation}
\mu_{\pvdots,\ulomega}(\xi,\xi_1,\xi_2,\xi_3) :=	\frac{1}{(2\pi)^{\frac32}}	\frac{\frakp_2\Big(\tfrac{\xi}{\sqrt{\ulomega}},\tfrac{\xi_1}{\sqrt{\ulomega}},\tfrac{\xi_2}{\sqrt{\ulomega}},\tfrac{\xi_3}{\sqrt{\ulomega}}\Big)}{\frakp\Big(\tfrac{\xi}{\sqrt{\ulomega}},\tfrac{\xi_1}{\sqrt{\ulomega}},\tfrac{\xi_2}{\sqrt{\ulomega}},\tfrac{\xi_3}{\sqrt{\ulomega}}\Big)} \sqrt{\frac{\pi}{2\ulomega}}\pvdots \cosech\left(\frac{\pi(\xi-\xi_1+\xi_2-\xi_3)}{2\sqrt{\ulomega}}\right),
\end{equation}
with $\frakp,\frakp_1,\frakp_2$ given by 
\begin{align}
\frakp(\xi,\xi_1,\xi_2,\xi_3) &:= (\vert \xi \vert+i)^2(\vert \xi_1 \vert-i)^2 (\vert \xi_2 \vert+i)^2(\vert \xi_3 \vert - i)^3,	\label{eqn: cubic-frakp}\\
\frakp_1(\xi,\xi_1,\xi_2,\xi_3)
&:= (\xi^2-1)(\xi_1^2-1)(\xi_2^2-1)(\xi_3^2-1)+ 16\xi\xi_1\xi_2\xi_3 \label{eqn: cubic-frakp1}\\
&\quad + 4 \Big(\xi\xi_1(\xi_2^2-1)(\xi_3^2-1) - \xi(\xi_1^2-1)\xi_2(\xi_3^2-1) +\xi(\xi_1^2-1)(\xi_2^2-1)\xi_3 \nonumber \\
&\quad \qquad +(\xi^2-1)\xi_1\xi_2(\xi_3^2-1) -(\xi^2-1)\xi_1(\xi_2^2-1)\xi_3+(\xi^2-1)(\xi_1^2-1)\xi_2\xi_3\Big),\nonumber \\
\frakp_2(\xi,\xi_1,\xi_2,\xi_3) &:= 2\Big((\xi^2-1)\xi_1(\xi_2^2-1)(\xi_3^2-1) -(\xi^2-1)(\xi_1^2-1)\xi_2(\xi_3^2-1)\label{eqn: cubic-frakp2}\\
&\quad \qquad +(\xi^2-1)(\xi_1^2-1)(\xi_2^2-1)\xi_3-\xi(\xi_1^2-1)(\xi_2^2-1)(\xi_3^2-1)\Big)\nonumber	\\
&\quad + 8\Big( (\xi^2-1)\xi_1\xi_2\xi_3 - \xi(\xi_1^2-1)\xi_2\xi_3 +\xi\xi_1(\xi_2^2-1)\xi_3 -\xi\xi_1\xi_2(\xi_3^2-1)\Big).\nonumber
\end{align}

Furthermore, we have the following properties:
\begin{enumerate}
	\item Tensorized structure: both symbols $(\frakp_1/\frakp)(\xi,\xi_1,\xi_2,\xi_3)$ and $(\frakp_2/\frakp)(\xi,\xi_1,\xi_2,\xi_3)$ are a linear combination of the form $	\overline{\frakb_0}(\xi)	\frakb_1(\xi_1)	\overline{\frakb_2}(\xi_2)	\frakb_3(\xi_3)$ with  multipliers $\frakb_0,\frakb_1,\frakb_2,\frakb_3 \in W^{1,\infty}(\bbR)$,
	\item Diagonal property: for any $\xi \in \bbR$, we have
	\begin{equation}\label{eqn:cubic-diagonal-property}
		\frac{\frakp_1(\xi,\xi,\xi,\xi)}{\frakp(\xi,\xi,\xi,\xi)} = 1 \quad \text{and} \quad \frac{\frakp_2(\xi,\xi,\xi,\xi)}{\frakp(\xi,\xi,\xi,\xi)} = 0,
	\end{equation}
	\item Vanishing property for $\mu_{\pvdots,\ulomega}$: in each tensorized term there is at least one frequency variable that vanishes at zero, that is, for every product $\overline{\frakb_0}(\xi)	\frakb_1(\xi_1) \overline{\frakb_2}(\xi_2) \frakb_3(\xi_3)$ in $(\frakp_2/\frakp)(\xi,\xi_1,\xi_2,\xi_3)$, there exists some $j \in \{0,1,2,3\}$ for which $\frakb_j(0) = 0$.
\end{enumerate}
\end{lemma}
\begin{proof} 
First, we introduce some short-hand notation for the rescaled frequency variables
	\begin{equation*}
		(\ulxi,\ulxi_1,\ulxi_2,\ulxi_3) := \Big(\tfrac{\xi}{\sqrt{\ulomega}},\tfrac{\xi_1}{\sqrt{\ulomega}},\tfrac{\xi_2}{\sqrt{\ulomega}},\tfrac{\xi_3}{\sqrt{\ulomega}}\Big).
	\end{equation*}
We then insert the explicit expressions \eqref{eqn:m-1,omega}--\eqref{eqn:m-2,omega} into \eqref{eqn: mu_1k1k2k3} to find that
\begin{equation*}
\begin{split}
&	\mu_{1,1,1,1,\ulomega}(\xi,\xi_1,\xi_2,\xi_3) = \frac{1}{(2\pi)^{2}} \frac{1}{\frakp(\ulxi,\ulxi_1,\ulxi_2,\ulxi_3)}  \int_\bbR e^{-ix(\xi-\xi_1+\xi_2-\xi_3)} H\big(\sqrt{\ulomega}x,\ulxi,\ulxi_1,\ulxi_2,\ulxi_3\big)  \,\ud x,
\end{split}
\end{equation*}
where
\begin{equation*}
	\begin{split}
&H\big(\sqrt{\ulomega}x,\ulxi,\ulxi_1,\ulxi_2,\ulxi_3\big)\\
&:=\big(\ulxi-i\tanh(\sqrt{\ulomega}x)\big)^2\big(\ulxi_1+i\tanh(\sqrt{\ulomega}x)\big)^2\big(\ulxi_2-i\tanh(\sqrt{\ulomega}x)\big)^2\big(\ulxi_3+i\tanh(\sqrt{\ulomega}x)\big)^2.
	\end{split}
\end{equation*}
After a change of variables $y := \sqrt{\ulomega}x$, we have
\begin{equation*}
\mu_{1,1,1,1,\ulomega}(\xi,\xi_1,\xi_2,\xi_3) = \frac{1}{(2\pi)^{2}} \frac{1}{\frakp(\ulxi,\ulxi_1,\ulxi_2,\ulxi_3)} \frac{1}{\sqrt{\ulomega}}\int_\bbR e^{-iy (\ulxi-\ulxi_1+\ulxi_2-\ulxi_3)}H\big(y,\ulxi,\ulxi_1,\ulxi_2,\ulxi_3\big) \, \ud y.
\end{equation*}
Then, we repeatedly use the identities
\begin{equation*}
	\big(\ulxi \pm i \tanh(y)\big)^2 = (\ulxi^2-1) \pm 2i\ulxi \tanh(y) + \sech^2(y), \quad \tanh^2(y) = 1-\sech^2(y),
\end{equation*}
to isolate the constant-in-$y$ term and the $\tanh(y)$ term in the full expansion of $H\big(y,\ulxi,\ulxi_1,\ulxi_2,\ulxi_3\big)$. By direct computation we find  that
\begin{equation*}
H\big(y,\ulxi,\ulxi_1,\ulxi_2,\ulxi_3\big) = \frakp_1(\ulxi,\ulxi_1,\ulxi_2,\ulxi_3) \cdot 1 + i \frakp_2(\ulxi,\ulxi_1,\ulxi_2,\ulxi_3) \cdot \tanh(y) + \{\text{higher order terms}\},
\end{equation*}
where by ``higher order terms'' we mean that they are a linear combination of terms of the form
\begin{equation*}
	\sech^{2\ell}(y),\quad \text{or}\quad \sech^{2\ell}(y)\tanh(y), \quad \ell \geq 1.
\end{equation*}
These are the terms that contribute to $\mu_{\mathrm{reg},\ulomega}$, while the constant $1$ term and the $\tanh(y)$ term on the right-hand side of $H\big(y,\ulxi,\ulxi_1,\ulxi_2,\ulxi_3\big)$ lead to the spectral distributions $\mu_{\delta_0,\ulomega}$, respectively $\mu_{\pvdots,\ulomega}$, after taking a Fourier transform (in the distributional sense) evaluated at $(\ulxi-\ulxi_1+\ulxi_2-\ulxi_3)$. As a reminder, see \eqref{equ:preliminaries_FT_one} and \eqref{equ:preliminaries_FT_tanh}, we have in the sense of tempered distributions
\begin{equation*}
\widehat{\calF}[1](\xi) = \sqrt{2\pi}\delta_0(\xi), \quad \text{and} \quad  \widehat{\calF}[\tanh(\cdot)](\xi) = -i \sqrt{\frac{\pi}{2}} \pvdots \cosech\left(\frac{\pi \xi}{2}\right).
\end{equation*}
Finally, the further properties exhibited by the symbols $\frakp_1/\frakp$ and $\frakp_2/\frakp$ can be observed directly from their explicit expressions \eqref{eqn: cubic-frakp}--\eqref{eqn: cubic-frakp2}.
\end{proof}

In view of Definition~\ref{def:reg-cubic} and Lemma~\ref{lemma:cubic_NSD}, we may now decompose
\begin{equation} \label{eqn: decomposition_cubic}
\wtilcalF_{+, \ulomega}\bigl[\calC\big((\ulPe U)(t)\big)\bigr](\xi) = \wtilcalC_{+,\delta_0,\ulomega}(t,\xi) + \wtilcalC_{+,\pvdots,\ulomega}(t,\xi)+\wtilcalC_{+,\mathrm{reg},\ulomega}(t,\xi),
\end{equation}
where
\begin{equation*}
	\begin{split}
\wtilcalC_{+,\delta_0,\ulomega}(t,\xi) &:=- \iiint e^{-it(\xi_1^2-\xi_2^2+\xi_3^2+\ulomega)}\tilf_{+,\ulomega}(t,\xi_1)\overline{\tilf_{+,\ulomega}}(t,\xi_2)\tilf_{+,\ulomega}(t,\xi_3)\mu_{\delta_0,\ulomega}(\xi,\xi_1,\xi_2,\xi_3) \,\ud\xi_1 \,\ud \xi_2 \,\ud \xi_3,\\
\wtilcalC_{+,\pvdots,\ulomega}(t,\xi) &:=- \iiint e^{-it(\xi_1^2-\xi_2^2+\xi_3^2+\ulomega)}\tilf_{+,\ulomega}(t,\xi_1)\overline{\tilf_{+,\ulomega}}(t,\xi_2)\tilf_{+,\ulomega}(t,\xi_3)\mu_{\pvdots,\ulomega}(\xi,\xi_1,\xi_2,\xi_3) \,\ud\xi_1 \,\ud \xi_2 \,\ud \xi_3,
	\end{split}
\end{equation*}
and where $\wtilcalC_{+,\mathrm{reg},\ulomega}(t,\xi)$ is a linear combination of terms of the form
\begin{equation*}
	\overline{\frakb_0}(\xi) \int_\bbR w_{1,\ulomega}(t,x)\overline{w_{2,\ulomega}}(t,x)w_{3,\ulomega}(t,x) \varphi(x) e^{-ix\xi} \, \ud x,
\end{equation*}
or
\begin{equation*}
    \overline{\frakb_0}(\xi) \int_\bbR \overline{w_{1,\ulomega}}(t,x)w_{2,\ulomega}(t,x)\overline{w_{3,\ulomega}}(t,x) \varphi(x) e^{-ix\xi} \, \ud x,
\end{equation*}
where $\varphi(x)$ is some Schwartz function and where $w_{j,\ulomega}(s,x)$, $1\leq j \leq 3$, are given by either 
\begin{equation*}
\int_\bbR e^{ix\eta} e^{-it(\eta^2+\ulomega)}\frakb_j(\eta)\tilf_{+,\ulomega}(t,\eta)\, \ud\eta, \quad \text{or} \quad \int_\bbR e^{ix\eta} e^{it(\eta^2+\ulomega)}\frakb_j(\eta)\tilf_{-,\ulomega}(t,\eta)\, \ud\eta,
\end{equation*}
with $\frakb_j \in W^{1,\infty}(\bbR)$, $1 \leq j \leq 3$.

\subsection{Quadratic spectral distributions for the modulation equations} \label{subsec:quadratic_spectral_distributions_modulation}

In the next lemma we compute quadratic spectral distributions that arise in the modulation equations. We uncover null structures that play a pivotal role to control the modulation parameters in Section~\ref{sec:modulation_parameters}.

\begin{lemma} \label{lem:null_structure_modulation}
Fix $\ulomega \in (0,\infty)$, and define
\begin{equation*}
	\begin{aligned}
		\nu_{++,\ulomega}(\xi_1, \xi_2) &:= \int_\bbR \bigl( \Psi_{1,\ulomega}(x,\xi_1) \Psi_{1,\ulomega}(x,\xi_2) - \Psi_{2,\ulomega}(x,\xi_1) \Psi_{2,\ulomega}(x,\xi_2) \bigr) \phi_\ulomega(x)^2 \, \ud x, \\
		\nu_{--,\ulomega}(\xi_1, \xi_2) &:= \int_\bbR \bigl( \Psi_{2,\ulomega}(x,\xi_1) \Psi_{2,\ulomega}(x,\xi_2) - \Psi_{1,\ulomega}(x,\xi_1) \Psi_{1,\ulomega}(x,\xi_2) \bigr) \phi_\ulomega(x)^2 \, \ud x, \\
		\nu_{+-,\ulomega}(\xi_1, \xi_2) &:= \int_\bbR \bigl( \Psi_{1,\ulomega}(x,\xi_1) \Psi_{2,\ulomega}(x,\xi_2) - \Psi_{2,\ulomega}(x,\xi_1) \Psi_{1,\ulomega}(x,\xi_2) \bigr) \phi_\ulomega(x)^2 \, \ud x.
	\end{aligned}
\end{equation*}
Then we have
\begin{align}
\nu_{++,\ulomega}(\xi_1, \xi_2) &= \bigl(\xi_1^2+\xi_2^2+2\ulomega\bigr) \frac{\sqrt{\ulomega}}{12}\frac{(\xi_1^2-4\xi_1\xi_2 + \xi_2^2-2\ulomega)}{(\vert \xi_1\vert - i\sqrt{\ulomega})^2( \vert \xi_2\vert-i\sqrt{\ulomega} )^2} \frac{\xi_1+\xi_2}{\sqrt{\ulomega}} \cosech\left( \frac{\pi}{2} \frac{\xi_1+\xi_2}{\sqrt{\ulomega}} \right)	 \label{eqn: nu++},	\\
\nu_{+-,\ulomega}(\xi_1, \xi_2) &= \bigl( \xi_1^2-\xi_2^2 \bigr) \frac{\sqrt{\ulomega}}{12}\frac{(\xi_1^2+2\xi_1\xi_2+\xi_2^2+4\ulomega)}{(\vert \xi_1\vert - i\sqrt{\ulomega})^2( \vert \xi_2\vert-i\sqrt{\ulomega} )^2} \frac{\xi_1+\xi_2}{\sqrt{\ulomega}} \cosech\left( \frac{\pi}{2} \frac{\xi_1+\xi_2}{\sqrt{\ulomega}} \right)	\label{eqn: nu+-},
\end{align}
and $\nu_{--,\ulomega}(\xi_1, \xi_2) = - \nu_{++,\ulomega}(\xi_1, \xi_2)$.
\end{lemma}

\begin{proof} 
We first note that $\nu_{--,\ulomega}(\xi_1, \xi_2) = - \nu_{++,\ulomega}(\xi_1, \xi_2)$ follows by inspection from the definitions. Hence, it suffices to determine the expressions for $\nu_{++,\ulomega}(\xi_1, \xi_2)$ and $\nu_{+-,\ulomega}(\xi_1,\xi_2)$. Recall from \eqref{eqn: Psi_pm_omega} that $\Psi_{j,\ulomega}(x,\xi) = \Psi_{j,1}(\sqrt{\ulomega}x,\xi/\sqrt{\ulomega})$. Moreover, $\phi_\ulomega(x) = \sqrt{\ulomega}Q(\sqrt{\ulomega}x)$ with $Q(x) = \sqrt{2}\sech(x)$. 
Upon rescaling, we have
\begin{equation} \label{equ:quadratic_spectral_modulation_rescaled}
	\nu_{\sigma_1\sigma_2,\ulomega}(\xi_1,\xi_2) = \sqrt{\ulomega} \nu_{\sigma_1\sigma_2} \Bigl( \frac{\xi_1}{\sqrt{\ulomega}}, \frac{\xi_2}{\sqrt{\ulomega}} \Bigr), \quad \sigma_1, \sigma_2\in  \{-,+\},
\end{equation}
where
\begin{equation*}
\begin{split}
\nu_{++}(\xi_1,\xi_2) &:= \int_\bbR \bigl( \Psi_{1,1}(x,\xi_1) \Psi_{1,1}(x,\xi_2) - \Psi_{2,1}(x,\xi_1) \Psi_{2,1}(x,\xi_2) \bigr) Q(x)^2 \, \ud x,\\
\nu_{+-}(\xi_1,\xi_2) &:= \int_\bbR \bigl( \Psi_{1,1}(x,\xi_1) \Psi_{2,1}(x,\xi_2) - \Psi_{2,1}(x,\xi_1) \Psi_{1,1}(x,\xi_2) \bigr) Q(x)^2 \, \ud x.
\end{split}
\end{equation*}
Now we compute the expressions for $\nu_{++,\ulomega}(\xi_1, \xi_2)$ and $\nu_{+-,\ulomega}(\xi_1, \xi_2)$. Inserting the formulas \eqref{eqn:m-1,omega}, \eqref{eqn:m-2,omega} for $\Psi_{1,1}$, $\Psi_{2,1}$ and $Q(x) = \sqrt{2} \sech(x)$, we obtain
\begin{equation*}
	\begin{split}
		\nu_{++}(\xi_1,\xi_2) &= \frac{1}{2\pi}\frac{1}{(\vert \xi_1\vert-i)^2(\vert \xi_2\vert-i)^2}\int_\bbR e^{ix(\xi_1+\xi_2)}H_{++}(x,\xi_1,\xi_2) \,\ud x,\\
		\nu_{+-}(\xi_1,\xi_2) &= \frac{1}{2\pi}\frac{1}{(\vert \xi_1\vert-i)^2(\vert \xi_2\vert-i)^2}\int_\bbR e^{ix(\xi_1+\xi_2)}H_{+-}(x,\xi_1,\xi_2) \,\ud x,
	\end{split}
\end{equation*}
where
\begin{equation*}
	\begin{split}
H_{++}(x,\xi_1,\xi_2) &:= 2\sech^2(x)\left( (\xi_1+i\tanh(x))^2(\xi_2+i\tanh(x))^2 - \sech^4(x) \right),\\
H_{+-}(x,\xi_1,\xi_2) &:=2\sech^2(x)\left( (\xi_1+i\tanh(x))^2\sech^2(x)  - \sech^2(x)(\xi_2+i\tanh(x))^2 \right).
	\end{split}
\end{equation*}
Expanding these expressions and using the identity $\tanh^2(x) = 1-\sech^2(x)$, we obtain
\begin{equation*}
\begin{split}
&H_{++}(x,\xi_1,\xi_2) \\
&=2\Big( \big((\xi_1^2-1)(\xi_2^2-1)-4\xi_1\xi_2\big)\sech^2(x) + 2i\big((\xi_1^2-1)\xi_2 + \xi_1(\xi_2^2-1)\big)\sech^2(x)\tanh(x)\\
&\qquad \qquad \qquad \quad + (\xi_1^2+\xi_2^2+4\xi_1\xi_2-2)\sech^4(x) + 2i(\xi_1+\xi_2)\sech^4(x)\tanh(x) \Big),
\end{split}
\end{equation*}
and
\begin{equation*}
H_{+-}(x,\xi_1,\xi_2) = 2\Big((\xi_1^2-\xi_2^2)\sech^4(x) + 2i(\xi_1-\xi_2)\sech^4(x)\tanh(x)\Big).
\end{equation*}
Hence, by \eqref{eqn: FT-sechsech}, \eqref{equ:FT_sech4}, \eqref{equ:FT_sech2tanh}, and  \eqref{equ:FT_sech4tanh}, we compute 
\begin{equation*}
	\begin{split}
		\nu_{++}(\xi_1,\xi_2)
		&=\frac{1}{(\vert \xi_1\vert-i)^2(\vert \xi_2\vert-i)^2}(\xi_1+\xi_2) \cosech\left(\frac{\pi}{2} (\xi_1+\xi_2)\right)\\
		&\quad \quad \times \biggl( \big((\xi_1^2-1)(\xi_2^2-1)-4\xi_1\xi_2\big) - (\xi_1+\xi_2) \big((\xi_1^2-1)\xi_2 + \xi_1(\xi_2^2-1)\big) \\
		&\qquad \qquad + \frac{1}{6} \bigl( 4+(\xi_1+\xi_2)^2 \bigr) \bigl( \xi_1^2+\xi_2^2+4\xi_1\xi_2-2 \bigr) - \frac{1}{12} (\xi_1+\xi_2)^2 \bigl(4+(\xi_1+\xi_2)^2\bigr) \biggr)\\
		&=\frac{1}{12} \frac{\xi_1^4 - 4\xi_1^3\xi_2 + 2\xi_1^2\xi_2^2 - 4\xi_1\xi_2^3 + \xi_2^4 -4- 8\xi_1 \xi_2 }{(\vert \xi_1\vert-i)^2(\vert \xi_2\vert-i)^2} (\xi_1+\xi_2) \cosech\left(\frac{\pi}{2} (\xi_1+\xi_2)\right) \\
		&= \bigl(\xi_1^2+\xi_2^2+2\bigr) \frac{1}{12} \frac{\xi_1^2-4\xi_1\xi_2 + \xi_2^2 - 2}{(\vert \xi_1\vert-i)^2(\vert \xi_2\vert-i)^2} (\xi_1+\xi_2) \cosech\left(\frac{\pi}{2} (\xi_1+\xi_2)\right),
	\end{split}
\end{equation*}
and
\begin{equation*}
	\begin{split}
		\nu_{+-}(\xi_1,\xi_2) 
		&= \frac{1}{(\vert \xi_1\vert-i)^2(\vert \xi_2\vert-i)^2}(\xi_1+\xi_2) \cosech\left(\frac{\pi}{2} (\xi_1+\xi_2)\right) \\
		&\quad \quad \times \biggl( \frac{1}{6} \bigl(4+(\xi_1+\xi_2)^2\bigr) (\xi_1^2-\xi_2^2) - \frac{1}{12} \bigl(4+(\xi_1+\xi_2)^2\bigr) (\xi_1^2-\xi_2^2) \biggr) \\
		&= \bigl(\xi_1^2-\xi_2^2\bigr) \frac{1}{12}\frac{\xi_1^2+2\xi_1\xi_2+\xi_2^2+4}{(\vert \xi_1\vert-i)^2(\vert \xi_2\vert-i)^2} (\xi_1+\xi_2) \cosech\left(\frac{\pi}{2} (\xi_1+\xi_2)\right).
	\end{split}
\end{equation*}
Then we arrive at the asserted expressions \eqref{eqn: nu++} and \eqref{eqn: nu+-} after rescaling according to \eqref{equ:quadratic_spectral_modulation_rescaled}.

\end{proof}

\section{Bootstrap Setup and Proof of Theorem~\ref{thm:main_theorem}} \label{sec:bootstrap_setup}

In this section we formulate the two main bootstrap propositions, and we establish the proof of Theorem~\ref{thm:main_theorem} as a consequence. 
The proofs of these two propositions will then occupy the remainder of the paper.

\subsection{The main bootstrap propositions}

Here we begin in earnest with the analysis of the long-time behavior of small even perturbations of the solitary waves \eqref{equ:intro_family_2parameter}.
Our starting point is the decomposition~\eqref{equ:setup_modulation_prop_decomposition} of the corresponding solution to \eqref{equ:cubic_NLS} on its maximal time interval of existence $[0,T_\ast)$ into a modulated solitary wave and a radiation term satisfying $(1)$--$(5)$ in the statement of Proposition~\ref{prop:modulation_and_orbital}. 
We now have to simultaneously prove decay of the radiation term and control the evolution of the paths $\omega(t)$ and $\gamma(t)$. In particular, we need to establish the convergence of the path $\omega(t)$ to a final scaling parameter as $t \to \infty$. 
We break this task into proving two key bootstrap propositions, which combined lead to a succinct proof of Theorem~\ref{thm:main_theorem}.

In order to infer decay of the radiation term, we seek to control the following norms of the components $\bigl( \tilfplusulo(t,\xi), \tilfminusulo(t,\xi) \bigr)$ of the distorted Fourier transform \eqref{equ:setup_definition_distFT_of_profile_Fulomega} of the profile of the radiation term relative to a reference operator $\calH(\ulomega)$ for some fixed $\frac12 \omega_0 \leq \ulomega \leq 2\omega_0$.
Denoting by $0 < \delta \ll 1$ a small absolute constant, we define for $0 < T < T_\ast$,
\begin{equation} \label{equ:definition_bootstrap_norm_XT}
    \begin{aligned}
        &\bigl\| \bigl( \tilf_{+, \ulomega}(t), \tilf_{-, \ulomega}(t) \bigr) \bigr\|_{X(T)} \\
        &\quad := \sup_{0 \leq t \leq T} \, \Bigl( \bigl\| \bigl( \tilf_{+, \ulomega}(t,\xi), \tilf_{-, \ulomega}(t,\xi) \bigr) \bigr\|_{L^\infty_\xi} + \jt^{-\delta} \bigl\|  \bigl( \pxi \tilf_{+, \ulomega}(t,\xi), \pxi \tilf_{-, \ulomega}(t,\xi) \bigr) \bigr\|_{L^2_\xi} \Bigr).
    \end{aligned}
\end{equation}

In the first bootstrap proposition we establish decay of the modulation parameter $\omega(t)$ to its final value $\omega(T)$ on a given time interval $[0,T]$.

\begin{proposition}[Control of modulation parameters] \label{prop:modulation_parameters}
    Let $\omega_0 \in (0,\infty)$ and let $0 < \varepsilon_1 \ll 1$ be the small constant from the statement of Proposition~\ref{prop:modulation_and_orbital}.
    There exist constants $0 < \varepsilon_0 \ll \varepsilon_1 \ll 1$ and $C_0 \geq 1$ with the following properties:
    Let $\gamma_0 \in \bbR$ and let $u_0 \in H^1_x(\bbR) \cap L^{2,1}_x(\bbR)$ be even with $\varepsilon := \|u_0\|_{H^1_x \cap L^{2,1}_x} \leq \varepsilon_0$.
    Denote by $\psi(t,x)$ the even solution to \eqref{equ:cubic_NLS} with initial condition
    \begin{equation}
        \psi_0(x) = e^{i \gamma_0} \bigl( \phi_{\omega_0}(x) + u_0(x) \bigr)
    \end{equation}
    on its maximal interval of existence $[0,T_\ast)$ furnished by Lemma~\ref{lem:setup_local_existence}.
    Let $(\omega, \gamma) \colon [0,T_\ast) \to (0,\infty) \times \bbR$ be the unique continuously differentiable paths so that the decomposition
    \begin{equation}
        \psi(t,x) = e^{i \gamma(t)} \bigl( \phi_{\omega(t)}(x) + u(t,x) \bigr), \quad 0 \leq t < T_\ast,
    \end{equation}
    satisfies (1)--(5) in the statement of Proposition~\ref{prop:modulation_and_orbital}.
    Fix $0 < T < T_\ast$ and $\ulomega \in (0,\infty)$ with $\frac12 \omega_0 \leq \ulomega \leq 2 \omega_0$.
    Denote by $\bigl(\tilf_{+,\ulomega}(t,\xi), \tilf_{-,\ulomega}(t,\xi)\bigr)$ the components of the distorted Fourier transform \eqref{equ:setup_definition_distFT_of_profile_Fulomega} of the profile defined in \eqref{equ:setup_definition_profile_Fulomega}.
    Suppose
    \begin{align}
        \sup_{0 \leq t \leq T} \, \jt^{1-\delta} |\omega(t) - \ulomega| &\leq 2 C_0 \varepsilon, \label{equ:prop_modulation_parameters_assumption1} \\
        \bigl\| \bigl( \tilf_{+, \ulomega}(t), \tilf_{-, \ulomega}(t) \bigr) \bigr\|_{X(T)} &\leq 2 C_0 \varepsilon. \label{equ:prop_modulation_parameters_assumption2}
    \end{align}
    Then it follows that
    \begin{equation} \label{equ:prop_modulation_parameters_conclusion}
        \sup_{0 \leq t \leq T} \, \jt^{1-\delta} |\omega(t) - \omega(T)| \leq C_0 \varepsilon.
    \end{equation}
\end{proposition}

We defer the proof of Proposition~\ref{prop:modulation_parameters} to Section~\ref{sec:modulation_parameters}.
In the second bootstrap proposition we obtain control of the norm \eqref{equ:definition_bootstrap_norm_XT} of the components of the distorted Fourier transform of the profile of the radiation term.

\begin{proposition}[Profile bounds] \label{prop:profile_bounds}
    Let $\omega_0 \in (0,\infty)$ and let $0 < \varepsilon_1 \ll 1$ be the small constant from the statement of Proposition~\ref{prop:modulation_and_orbital}.
    There exist constants $0 < \varepsilon_0 \ll \varepsilon_1 \ll 1$ and $C_0 \geq 1$ with the following properties:
    Let $\gamma_0 \in \bbR$ and let $u_0 \in H^1_x(\bbR) \cap L^{2,1}_x(\bbR)$ be even with $\varepsilon := \|u_0\|_{H^1_x \cap L^{2,1}_x} \leq \varepsilon_0$.
    Denote by $\psi(t,x)$ the even solution to \eqref{equ:cubic_NLS} with initial condition
    \begin{equation}
        \psi_0(x) = e^{i \gamma_0} \bigl( \phi_{\omega_0}(x) + u_0(x) \bigr)
    \end{equation}
    on its maximal interval of existence $[0,T_\ast)$ furnished by Lemma~\ref{lem:setup_local_existence}.
    Let $(\omega, \gamma) \colon [0,T_\ast) \to (0,\infty) \times \bbR$ be the unique continuously differentiable paths so that the decomposition
    \begin{equation}
        \psi(t,x) = e^{i \gamma(t)} \bigl( \phi_{\omega(t)}(x) + u(t,x) \bigr), \quad 0 \leq t < T_\ast,
    \end{equation}
    satisfies (1)--(5) in the statement of Proposition~\ref{prop:modulation_and_orbital}.
    Fix $0 < T < T_\ast$ and $\ulomega \in (0,\infty)$ with $\frac12 \omega_0 \leq \ulomega \leq 2 \omega_0$.
    Denote by $\bigl(\tilf_{+,\ulomega}(t,\xi), \tilf_{-,\ulomega}(t,\xi)\bigr)$ the components of the distorted Fourier transform \eqref{equ:setup_definition_distFT_of_profile_Fulomega} of the profile defined in \eqref{equ:setup_definition_profile_Fulomega}.
    Suppose
    \begin{align}
        \sup_{0 \leq t \leq T} \, \jt^{1-\delta} |\omega(t) - \ulomega| &\leq 2 C_0 \varepsilon, \label{equ:prop_profile_bounds_assumption1} \\
        \bigl\| \bigl( \tilf_{+, \ulomega}(t), \tilf_{-,\ulomega}(t) \bigr) \bigr\|_{X(T)} &\leq 2 C_0 \varepsilon. \label{equ:prop_profile_bounds_assumption2}
    \end{align}
    Then it follows that
    \begin{equation} \label{equ:prop_profile_bounds_conclusion}
        \bigl\| \bigl( \tilf_{+,\ulomega}(t), \tilf_{-,\ulomega}(t) \bigr) \bigr\|_{X(T)} \leq C_0 \varepsilon.
    \end{equation}
\end{proposition}

The proof of Proposition~\ref{prop:profile_bounds} follows by a standard continuity argument from Proposition~\ref{prop: weighted-energy-estimate}, Proposition~\ref{prop:pointwise_estimate}, and the local existence theory.
From the bootstrap assumptions in the statements of Proposition~\ref{prop:modulation_parameters} and of Proposition~\ref{prop:profile_bounds}, we now deduce several decay estimates and some auxiliary bounds that will be used again and again throughout the remainder of the paper.

\begin{corollary} \label{cor:consequences}
    Let $\omega_0 \in (0,\infty)$, and let $0 < \varepsilon_0 \ll \varepsilon_1 \ll 1$ and $C_0 \geq 1$
    be the constants from the statements of Proposition~\ref{prop:modulation_and_orbital}, Proposition~\ref{prop:modulation_parameters}, and Proposition~\ref{prop:profile_bounds}.
    Let $\gamma_0 \in \bbR$ and let $u_0 \in H^1_x(\bbR) \cap L^{2,1}_x(\bbR)$ be even with $\varepsilon := \|u_0\|_{H^1_x \cap L^{2,1}_x} \leq \varepsilon_0$.
    Denote by $\psi(t,x)$ the even solution to \eqref{equ:cubic_NLS} with initial condition
    \begin{equation}
        \psi_0(x) = e^{i \gamma_0} \bigl( \phi_{\omega_0}(x) + u_0(x) \bigr)
    \end{equation}
    on its maximal interval of existence $[0,T_\ast)$ furnished by Lemma~\ref{lem:setup_local_existence}.
    Let $(\omega, \gamma) \colon [0,T_\ast) \to (0,\infty) \times \bbR$ be the unique continuously differentiable paths so that the decomposition
    \begin{equation}
        \psi(t,x) = e^{i \gamma(t)} \bigl( \phi_{\omega(t)}(x) + u(t,x) \bigr), \quad 0 \leq t < T_\ast,
    \end{equation}
    satisfies (1)--(5) in the statement of Proposition~\ref{prop:modulation_and_orbital}.
    Fix $0 < T < T_\ast$ and $\ulomega \in (0,\infty)$ with $\frac12 \omega_0 \leq \ulomega \leq 2 \omega_0$.
    Denote by $\bigl(\tilf_{+,\ulomega}(t,\xi), \tilf_{-,\ulomega}(t,\xi)\bigr)$ the components of the distorted Fourier transform \eqref{equ:setup_definition_distFT_of_profile_Fulomega} of the profile defined in \eqref{equ:setup_definition_profile_Fulomega}.
    Suppose
    \begin{align}
        \sup_{0 \leq t \leq T} \, \jt^{1-\delta} |\omega(t) - \ulomega| &\leq 2 C_0 \varepsilon, \label{equ:consequences_assumption1} \\
        \bigl\| \bigl( \tilf_{+, \ulomega}(t), \tilf_{-, \ulomega}(t) \bigr) \bigr\|_{X(T)} &\leq 2 C_0 \varepsilon. \label{equ:consequences_assumption2}
    \end{align}
    Then the following estimates hold:

    \begin{itemize}[leftmargin=1.8em]
        \item[(1)] Sobolev bound for the profile:
        \begin{equation} \label{equ:consequences_sobolev_bound_profile}
            \sup_{0 \leq t \leq T} \, \Bigl( \bigl\|\jxi \tilf_{+, \ulomega}(t)\bigr\|_{L^2_\xi} + \bigl\|\jxi \tilf_{-, \ulomega}(t)\bigr\|_{L^2_\xi} \Bigr) \lesssim \varepsilon.
        \end{equation}

        \item[(2)] Decomposition of the radiation term:
        \begin{equation} \label{equ:consequences_decomposition_radiation}
            U(t) = (\ulPe U)(t) + d_{1,\ulomega}(t) Y_{1,\ulomega} + d_{2,\ulomega}(t) Y_{2,\ulomega}, \quad 0 \leq t \leq T,
        \end{equation}
        with
        \begin{align}
            \sup_{0 \leq t \leq T} \, \jt^{\frac12} \| (\ulPe U)(t) \|_{L^\infty_x} &\lesssim \varepsilon, \label{equ:consequences_ulPe_U_disp_decay} \\
            \sup_{0 \leq t \leq T} \, \jt^{\frac32-\delta} \bigl( |d_{1,\ulomega}(t)| + |d_{2,\ulomega}(t)| \bigr) &\lesssim  \varepsilon. \label{equ:consequences_discrete_components_decay}
        \end{align}

        \item[(3)] Dispersive decay:
        \begin{equation} \label{equ:consequences_U_disp_decay}
            \sup_{0 \leq t \leq T} \, \jt^{\frac12} \|U(t)\|_{L^\infty_x} \lesssim \varepsilon.
        \end{equation}

        \item[(4)] Auxiliary decay estimates for the modulation parameters:
        \begin{align}
            \sup_{0 \leq t \leq T} \, \jt \bigl( |\dot{\omega}(t)| + |\dot{\gamma}(t) - \omega(t)| \bigr) &\lesssim \varepsilon, \label{equ:consequences_aux_bound_modulation1} \\
            \sup_{0 \leq t \leq T} \, \jt^{1-\delta} |\dot{\gamma}(t) - \ulomega| &\lesssim \varepsilon. \label{equ:consequences_aux_bound_modulation2}
        \end{align}

        \item[(5)] Growth bound for the phase:
        \begin{equation} \label{equ:consequences_growth_bound_theta}
            \sup_{0 \leq t \leq T} \, \jt^{-\delta} |\theta(t)| \lesssim \varepsilon.
        \end{equation}

        \item[(6)] Auxiliary bounds for remainder terms in the evolution equations for the profiles:
        \begin{align}
         \sup_{0 \leq t \leq T} \, \jt^{\frac12} \Bigl( \bigl\| \calL_{\baru, \ulomega}(t,\xi) \bigr\|_{H_\xi^1} + \bigl\| \calL_{u,\ulomega}(t,\xi) \bigr\|_{H_\xi^1} \Bigr) &\lesssim \varepsilon, \label{equ:consequences_calL_bounds} \\
         \sup_{0 \leq t \leq T} \, \jt^{2-\delta} \bigl\| \ulPe \calMod(t) \bigr\|_{L^{2,1}_x} &\lesssim \varepsilon^2, \label{equ:consequences_ulPe_Mod_bounds} \\
         \sup_{0 \leq t \leq T} \, \jt^{\frac32 - \delta} \bigl\| \calE_1(t) \bigr\|_{L^{2,1}_x} &\lesssim \varepsilon^2, \label{equ:consequences_calE1_bounds} \\
         \sup_{0 \leq t \leq T} \, \jt^{2 - \delta} \Bigl( \bigl\| \calE_2(t) \bigr\|_{L^{2,1}_x} + \bigl\| \calE_3(t) \bigr\|_{L^{2,1}_x} \Bigr) &\lesssim \varepsilon^2. \label{equ:consequences_calE_2and3_bounds}
        \end{align}

        \item[(7)] Leading order local decay:
        \begin{align}
         \usube(t,x) &= h_{1,\ulomega}(t) \Phi_{1,\ulomega}(x) - h_{2,\ulomega}(t) \Phi_{2,\ulomega}(x) + R_{u,\ulomega}(t,x), \quad 0 \leq t \leq T, \label{equ:consequences_usube_leading_order_local_decay_decomp} \\
         \barusube(t,x) &= h_{1,\ulomega}(t) \Phi_{2,\ulomega}(x) - h_{2,\ulomega}(t) \Phi_{1,\ulomega}(x) + R_{\baru,\ulomega}(t,x), \quad 0 \leq t \leq T, \label{equ:consequences_barusube_leading_order_local_decay_decomp}
        \end{align}
        with
        \begin{align}
         \sup_{0 \leq t \leq T} \, \jt^{\frac12} \bigl( |h_{1,\ulomega}(t)| + |h_{2,\ulomega}(t)| \bigr) &\lesssim \varepsilon, \label{equ:consequences_h12_decay} \\
         \sup_{0 \leq t \leq T} \, \jt^{1-\delta} \Bigl( \bigl|\pt \bigl( e^{i t \ulomega} h_{1,\ulomega}(t) \bigr) \bigr| + \bigl|\pt \bigl( e^{-i t \ulomega} h_{2,\ulomega}(t) \bigr) \bigr| \Bigr) &\lesssim \varepsilon, \label{equ:consequences_h12_phase_filtered_decay}  \\
         \sup_{0 \leq t \leq T} \, \jt^{1-\delta} \Bigl( \bigl\|\jx^{-2} R_{u,\ulomega}(t,x)\bigr\|_{L^\infty_x} + \bigl\| \jx^{-2} R_{\bar{u},\ulomega}(t,x) \bigr\|_{L^\infty_x} \Bigr) &\lesssim \varepsilon, \label{equ:consequences_Ru_local_decay} \\
         \sup_{0 \leq t \leq T} \, \jt^{1-\delta} \Bigl( \bigl\|\jx^{-3} \px R_{u,\ulomega}(t,x)\bigr\|_{L^2_x} + \bigl\| \jx^{-3} \px R_{\bar{u},\ulomega}(t,x) \bigr\|_{L^2_x} \Bigr) &\lesssim \varepsilon. \label{equ:consequences_px_Ru_local_decay}
        \end{align}

        \item[(8)] Bounds for auxiliary free Schr\"odinger evolutions:

        \noindent Denote by $\chi_0(\xi)$ a smooth even non-negative cut-off function with $\chi_0(\xi) = 1$ for $|\xi| \leq 1$ and $\chi_0(\xi) = 0$ for $|\xi| \geq 2$. Given symbols $\fraka, \frakb \in W^{1,\infty}(\bbR)$, we define
        \begin{equation*}
        \begin{aligned}
            v_{+,\ulomega}(t,x) &:= e^{-i t \ulomega} \int_\bbR e^{ix\xi_1} e^{-it\xi_1^2} \fraka(\xi_1) \tilfplusulo(t,\xi_1) \, \ud \xi_1, \\
            v_{-,\ulomega}(t,x) &:= e^{i t \ulomega} \int_\bbR e^{ix\xi_2} e^{it\xi_2^2} \frakb(\xi_2) \tilfminusulo(t,\xi_2) \, \ud \xi_2.
        \end{aligned}
        \end{equation*}
        Morover, we introduce the decompositions
        \begin{align*}
            v_{+,\ulomega}(t,x) &= \tilde{h}_{1,\ulomega}(t) + R_{v_+,\ulomega}(t,x), \\
            v_{-,\ulomega}(t,x) &= \tilde{h}_{2,\ulomega}(t) + R_{v_-,\ulomega}(t,x),
        \end{align*}
        with
        \begin{align*}
            \tilde{h}_{1,\ulomega}(t) &= e^{-i t \ulomega} \int_\bbR e^{-i t \xi_1^2} \chi_0(\xi_1) \fraka(\xi_1) \tilfplusulo(t,\xi_1) \, \ud \xi_1, \\
            \tilde{h}_{2,\ulomega}(t) &= e^{i t \ulomega} \int_\bbR e^{i t \xi_2^2} \chi_0(\xi_2) \frakb(\xi_2) \tilfminusulo(t,\xi_2) \, \ud \xi_2.
        \end{align*}
        Then the following estimates hold:
        \begin{align}
            \sup_{0 \leq t \leq T} \, \jt^{\frac12} \bigl( \|v_{+,\ulomega}(t)\|_{L^\infty_x} + \|v_{-,\ulomega}(t)\|_{L^\infty_x} \bigr) &\lesssim \varepsilon, \label{equ:preparation_flat_Schrodinger_wave_bound1} \\
            \sup_{0 \leq t \leq T} \, \bigl( \|v_{+,\ulomega}(t)\|_{L^2_x} + \|v_{-,\ulomega}(t)\|_{L^2_x} \bigr) &\lesssim \varepsilon, \label{equ:preparation_flat_Schrodinger_wave_bound_L2} \\
            \sup_{0 \leq t \leq T} \, \jt^{1-\delta} \bigl( \|\jx^{-1} \px v_{+,\ulomega}(t)\|_{L^2_x} + \|\jx^{-1} \px v_{-,\ulomega}(t)\|_{L^2_x} \bigr) &\lesssim \varepsilon, \label{equ:preparation_flat_Schrodinger_wave_bound2} \\
            \sup_{0 \leq t \leq T} \, \jt^{\frac12} \bigl( |\tilde{h}_{1,\ulomega}(t)| + |\tilde{h}_{2,\ulomega}(t)| \bigr) &\lesssim \varepsilon, \label{equ:preparation_flat_Schrodinger_wave_bound3} \\
            \sup_{0 \leq t \leq T} \, \jt^{1-\delta} \Bigl( \bigl| \pt \bigl( e^{it\ulomega} \tilde{h}_{1,\ulomega}(t) \bigr) \bigr| + \bigl| \pt \bigl( e^{-it\ulomega} \tilde{h}_{2,\ulomega}(t) \bigr) \bigr| \Bigr) &\lesssim \varepsilon, \label{equ:preparation_flat_Schrodinger_wave_bound4} \\
            \sup_{0 \leq t \leq T} \, \jt^{1-\delta} \Bigl( \bigl\| \jx^{-2} R_{v_+,\ulomega}(t,x) \bigr\|_{L^\infty_x} + \bigl\| \jx^{-2} R_{v_-,\ulomega}(t,x) \bigr\|_{L^\infty_x} \Bigr) &\lesssim \varepsilon. \label{equ:preparation_flat_Schrodinger_wave_bound5}
        \end{align}

    \end{itemize}
\end{corollary}

\begin{proof}[Proof of Corollary~\ref{cor:consequences}]
We prove the asserted estimates item by item.

\noindent \underline{Proof of (1).}
Using Corollary~\ref{cor:distFT_of_propagator}, the boundedness of the distorted Fourier transform as an operator $H^1_x \to L^{2,1}_\xi$ by Proposition~\ref{prop:mapping_properties_dist_FT}, and the stability bound \eqref{equ:setup_smallness_orbital}, we obtain for all $0 \leq t \leq T$ that
\begin{equation*}
 \begin{aligned}
  \bigl\| \jxi \tilde{f}_{\pm,\ulomega}(t,\xi) \bigr\|_{L^2_\xi} = \bigl\| \jxi \widetilde{\calF}_{\pm,\ulomega}\bigl[ e^{it\calH(\ulomega)} (\ulPe U)(t) \bigr] \bigr\|_{L^2_\xi} = \bigl\| \jxi \wtilcalF_{\pm, \ulomega}\bigl[ U(t) \bigr](\xi) \bigr\|_{L^2_\xi} \lesssim \|U(t)\|_{H^1_x} \lesssim \varepsilon.
 \end{aligned}
\end{equation*}

\noindent \underline{Proof of (2).}
The decomposition \eqref{equ:consequences_decomposition_radiation} of the radiation term $U(t)$ follows from Lemma~\ref{lemma: L2 decomposition}.
We first prove the dispersive decay estimate \eqref{equ:consequences_ulPe_U_disp_decay}. For short times $0 \leq t \leq 1$ the bound is just a consequence of the Sobolev embedding $H^1_x(\bbR) \hookrightarrow L^\infty_x(\bbR)$, the $H^1_x \to H^1_x$ boundedness of the projection $\ulPe$ by Lemma~\ref{lemma: L2 decomposition}, and the stability bound \eqref{equ:setup_smallness_orbital}. 
For times $1 \leq t \leq T$ the decay estimate $\|(\ulPe U)(t)\|_{L^\infty_x} \lesssim \varepsilon t^{-\frac12}$ follows by Lemma~\ref{lem:linear_dispersive_decay} from the representation formula \eqref{equ:setup_ulPeU_representation_formula} for $(\ulPe U)(t)$ and the bounds~\eqref{equ:consequences_assumption2}, \eqref{equ:consequences_sobolev_bound_profile}.

Next, we turn to the decay estimates \eqref{equ:consequences_discrete_components_decay} for the discrete components.
Testing the decomposition \eqref{equ:consequences_decomposition_radiation} against $\sigma_2 Y_{j,\omega(t)}$ for $j = 1, 2$ and using the orthogonality properties \eqref{equ:setup_orthogonality_radiation} of the radiation term, we find that
\begin{equation} \label{equ:consequences_proof_decomposition_radiation1}
 \begin{aligned}
  \begin{bmatrix}
   \langle Y_{1,\ulomega}, \sigma_2 Y_{1,\omega(t)} \rangle & \langle Y_{2,\ulomega}, \sigma_2 Y_{1,\omega(t)} \rangle \\
   \langle Y_{1,\ulomega}, \sigma_2 Y_{2,\omega(t)} \rangle & \langle Y_{2,\ulomega}, \sigma_2 Y_{2,\omega(t)} \rangle
  \end{bmatrix}
  \begin{bmatrix}
   d_{1,\ulomega}(t) \\ d_{2,\ulomega}(t)
  \end{bmatrix}
  &=
  \begin{bmatrix}
   - \langle \ulPe U(t), \sigma_2 Y_{1,\omega(t)} \rangle \\ - \langle \ulPe U(t), \sigma_2 Y_{2,\omega(t)} \rangle
  \end{bmatrix} \\
  &=
  \begin{bmatrix}
   - \langle \ulPe U(t), \sigma_2 (Y_{1,\omega(t)} - Y_{1,\ulomega}) \rangle \\ - \langle \ulPe U(t), \sigma_2 (Y_{2,\omega(t)} - Y_{2,\ulomega}) \rangle
  \end{bmatrix}.
 \end{aligned}
\end{equation}
Using that $\langle Y_{j,\ulomega}, \sigma_2 Y_{j,\ulomega} \rangle = 0$ for $j = 1, 2$ and that $\langle Y_{1,\ulomega}, \sigma_2 Y_{2,\ulomega} \rangle = \langle Y_{2,\ulomega}, \sigma_2 Y_{1,\ulomega} \rangle = -c_\ulomega$,
we can rewrite the matrix on the left-hand side of \eqref{equ:consequences_proof_decomposition_radiation1} as
\begin{equation*}
 \begin{aligned}
  &\begin{bmatrix}
   \langle Y_{1,\ulomega}, \sigma_2 Y_{1,\omega(t)} \rangle & \langle Y_{2,\ulomega}, \sigma_2 Y_{1,\omega(t)} \rangle \\
   \langle Y_{1,\ulomega}, \sigma_2 Y_{2,\omega(t)} \rangle & \langle Y_{2,\ulomega}, \sigma_2 Y_{2,\omega(t)} \rangle
  \end{bmatrix} \\
  &=
  \begin{bmatrix}
   0 & \langle Y_{2,\ulomega}, \sigma_2 Y_{1,\ulomega} \rangle \\
   \langle Y_{1,\ulomega}, \sigma_2 Y_{2,\ulomega} \rangle & 0
  \end{bmatrix}
  +
  \begin{bmatrix}
   \langle Y_{1,\ulomega}, \sigma_2 (Y_{1,\omega(t)} - Y_{1,\ulomega}) \rangle & \langle Y_{2,\ulomega}, \sigma_2 (Y_{1,\omega(t)} - Y_{1,\ulomega}) \rangle \\
   \langle Y_{1,\ulomega}, \sigma_2 (Y_{2,\omega(t)} - Y_{2,\ulomega}) \rangle & \langle Y_{2,\ulomega}, \sigma_2 (Y_{2,\omega(t)} - Y_{2,\ulomega}) \rangle
  \end{bmatrix} \\
  &=
  c_\ulomega \begin{bmatrix}
   0 & -1 \\ -1 & 0
  \end{bmatrix}
  +
  \calO\bigl(|\omega(t) - \ulomega|\bigr) \begin{bmatrix} 1 & 1 \\ 1 & 1 \end{bmatrix}.
 \end{aligned}
\end{equation*}
In view of \eqref{equ:consequences_assumption1} and \eqref{equ:setup_comparison_estimate}, the matrix on the left-hand side of \eqref{equ:consequences_proof_decomposition_radiation1} is invertible with a uniform-in-time upper bound on its operator norm whose size depends only on $\omega_0$.
In combination with \eqref{equ:consequences_assumption1} and \eqref{equ:consequences_ulPe_U_disp_decay}, we conclude the desired bound
\begin{equation*}
 \begin{aligned}
  |d_{1,\ulomega}(t)| + |d_{2,\ulomega}(t)| &\lesssim_{\omega_0} \sum_{j=1,2} \, \bigl| \langle \ulPe U(t), \sigma_2 (Y_{j,\omega(t)} - Y_{j,\ulomega}) \rangle \bigr| \lesssim \|(\ulPe U)(t)\|_{L^\infty_x} |\omega(t)-\ulomega| \lesssim \jt^{-\frac32+\delta} \varepsilon^2.
 \end{aligned}
\end{equation*}

\noindent \underline{Proof of (3).}
The decay estimate \eqref{equ:consequences_U_disp_decay} for the radiation term $U(t)$ follows immediately from the decomposition~\eqref{equ:consequences_decomposition_radiation} and the decay estimates \eqref{equ:consequences_ulPe_U_disp_decay}, \eqref{equ:consequences_discrete_components_decay}.

\noindent \underline{Proof of (4).}
We begin with the proof of the first bound \eqref{equ:consequences_aux_bound_modulation1}.
To this end we observe that by \eqref{equ:setup_smallness_orbital}, the time-dependent matrix \eqref{equ:setup_matrix_modulation_equation} on the left-hand side of the modulation equations \eqref{equ:setup_modulation_equation} is of the form
\begin{equation*}
 \bbM(t) = \frac{2}{\sqrt{\omega(t)}} \begin{bmatrix} 0 & 1 \\ 1 & 0 \end{bmatrix} + \calO(\varepsilon) \begin{bmatrix} 1 & 1 \\ 1 & 1 \end{bmatrix}.
\end{equation*}
Hence, by \eqref{equ:setup_comparison_estimate} this matrix is invertible and the operator norm of its inverse has a uniform-in-time upper bound whose size depends only on $\omega_0$.
The asserted bound \eqref{equ:consequences_aux_bound_modulation1} then follows from the modulation equations \eqref{equ:setup_modulation_equation} using \eqref{equ:setup_comparison_estimate}, \eqref{equ:consequences_U_disp_decay}, and the Cauchy-Schwarz inequality,
\begin{equation*}
 \begin{aligned}
  |\dot{\gamma}(t)-\omega(t)| + |\dot{\omega}(t)| \lesssim \bigl\| \bbM(t)^{-1} \bigr\| \bigl\| \calN(U(t) \bigr\|_{L^\infty_x} \bigl( \|Y_{1,\omega(t)}\|_{L^1_x} + \|Y_{2,\omega(t)}\|_{L^1_x} \bigr) \lesssim_{\omega_0} \varepsilon^2 \jt^{-1}.
 \end{aligned}
\end{equation*}
The second asserted bound \eqref{equ:consequences_aux_bound_modulation2} is an immediate consequence of the bootstrap assumption \eqref{equ:consequences_assumption1} and the first bound \eqref{equ:consequences_aux_bound_modulation1}.

\noindent \underline{Proof of (5).}
The growth bound \eqref{equ:consequences_growth_bound_theta} for the phase $\theta(t)$ follows immediately from its definition \eqref{equ:setup_definition_theta} using \eqref{equ:consequences_aux_bound_modulation2}.

\noindent \underline{Proof of (6).}
    We start with the proof of \eqref{equ:consequences_calL_bounds}. From the definitions \eqref{eqn:m-2,omega}, \eqref{equ:setup_definition_calL} together with \eqref{equ:setup_comparison_estimate}, \eqref{equ:consequences_U_disp_decay}, we conclude uniformly for all $\xi \in \bbR$ that
\begin{equation*}
\begin{aligned}
& \bigl|\calL_{\baru,\ulomega}(t,\xi)\bigr| + \bigl|\calL_{u,\ulomega}(t,\xi)| +  \bigl|\pxi \calL_{\baru,\ulomega}(t,\xi)\bigr| + \bigl| \pxi\calL_{u,\ulomega}(t,\xi)|  \\
&\lesssim \frac{\ulomega}{(\xi^2 + \ulomega)} \int_\bbR |u(t,x)| \jx \sech^2(\sqrt{\ulomega} x) \, \ud x \lesssim_{\omega_0} \jxi^{-2} \|U(t)\|_{L^\infty_x} \lesssim \jxi^{-2} \varepsilon \jt^{-\frac12}.
\end{aligned}
\end{equation*}
The asserted bound \eqref{equ:consequences_calL_bounds} immediately follows.

Next, we establish the bound \eqref{equ:consequences_ulPe_Mod_bounds}.
Using Lemma~\ref{lemma: L2 decomposition}, we obtain from the definition \eqref{equ:setup_definition_calMod} of $\calMod(t)$ that 
\begin{equation*}
 \begin{aligned}
  \ulPe \calMod(t) = -i ( \dot{\gamma}(t) - \omega(t) ) \ulPe \bigl( Y_{1,\omega(t)} - Y_{1,\ulomega} \bigr) - i \dot{\omega}(t) \ulPe \bigl( Y_{2,\omega(t)} - Y_{2,\ulomega} \bigr).
 \end{aligned}
\end{equation*}
Thus, invoking \eqref{equ:setup_comparison_estimate}, \eqref{equ:consequences_assumption1}, \eqref{equ:consequences_aux_bound_modulation1}, the $L^{2,1}_x$ boundedness of $\ulPe$ by Lemma~\ref{lemma: L2 decomposition}, and the spatial localization of the generalized eigenfunctions $Y_{j,\omega}$, we infer the desired bound
\begin{equation*}
 \bigl\| \ulPe \calMod(t) \bigr\|_{L^{2,1}_x} \lesssim_{\omega_0} \bigl( |\dot{\gamma}(t) - \omega(t)| + |\dot{\omega}(t)| \bigr) |\omega(t) - \ulomega| \lesssim_{\omega_0} \varepsilon^2 \jt^{-2+\delta}.
\end{equation*}

Now we turn to the proof of \eqref{equ:consequences_calE1_bounds}.
Using \eqref{equ:consequences_assumption1}, \eqref{equ:consequences_U_disp_decay}, and the spatial localization of $\phi_\omega(x)$, we deduce straight from the definition \eqref{equ:setup_definition_calE1} of $\calE_1(t)$ that
\begin{equation*}
 \|\calE_1(t)\|_{L^{2,1}_x} \lesssim |\omega(t) - \ulomega| \|U(t)\|_{L^\infty_x} \lesssim \varepsilon \jt^{-1+\delta} \cdot \varepsilon \jt^{-\frac12} \lesssim \varepsilon^2 \jt^{-\frac32+\delta},
\end{equation*}
as desired.

Finally, we prove \eqref{equ:consequences_calE_2and3_bounds}. Again using \eqref{equ:consequences_assumption1}, \eqref{equ:consequences_U_disp_decay}, and the spatial localization of $\phi_\omega(x)$, we infer from the definition \eqref{equ:setup_definition_calE2} of $\calE_2(t)$ that
\begin{equation*}
 \|\calE_2(t)\|_{L^{2,1}_x} \lesssim |\omega(t) - \ulomega| \|U(t)\|_{L^\infty_x}^2 \lesssim \varepsilon \jt^{-1+\delta} \cdot \varepsilon^2 \jt^{-1} \lesssim \varepsilon^3 \jt^{-2+\delta}.
\end{equation*}
For the proof of the corresponding bound for $\calE_3(t)$ defined in \eqref{equ:setup_definition_calE3}, we insert the decomposition \eqref{equ:consequences_decomposition_radiation} of $U(t)$ into its projection to the essential spectrum and its discrete components. Then the key observation is that in every resulting nonlinear term at least one input is given by $d_{1,\ulomega}(t) Y_{1,\ulomega}(x)$ or $d_{2,\ulomega}(t) Y_{2,\ulomega}(x)$, which enjoy faster decay and spatial localization.
Correspondingly, we obtain by \eqref{equ:consequences_ulPe_U_disp_decay} and \eqref{equ:consequences_discrete_components_decay} that
\begin{equation*}
 \begin{aligned}
  \bigl\| \calQ_{\ulomega}\bigl( U(t) \bigr) - \calQ_{\ulomega}\bigl( (\ulPe U)(t) \bigr) \bigr\|_{L_x^{2,1}} &\lesssim \bigl( \|(\ulPe U)(t)\|_{L^\infty_x} + |d_{1,\ulomega}(t)| + |d_{2,\ulomega}(t)| \bigr) \bigl( |d_{1,\ulomega}(t)| + |d_{2,\ulomega}(t)| \bigr) \\
  &\lesssim \varepsilon \jt^{-\frac12} \cdot \varepsilon \jt^{-\frac32+\delta} \lesssim \varepsilon^2 \jt^{-2+\delta},
 \end{aligned}
\end{equation*}
as well as
\begin{equation*}
 \begin{aligned}
  \bigl\| \calC\bigl( U(t) \bigr) - \calC\bigl( (\ulPe U)(t) \bigr) \bigr\|_{L_x^{2,1}} &\lesssim \bigl( \|(\ulPe U)(t)\|_{L^\infty_x} + |d_{1,\ulomega}(t)| + |d_{2,\ulomega}(t)| \bigr)^2 \bigl( |d_{1,\ulomega}(t)| + |d_{2,\ulomega}(t)| \bigr) \\
  &\lesssim \varepsilon^2 \jt^{-1} \cdot \varepsilon \jt^{-\frac32+\delta} \lesssim \varepsilon^3 \jt^{-\frac52+\delta}.
 \end{aligned}
\end{equation*}
Combining the preceding estimates yields the asserted bound \eqref{equ:consequences_calE_2and3_bounds}.

\noindent \underline{Proof of (7).}
In what follows we establish the estimates \eqref{equ:consequences_h12_decay}, \eqref{equ:consequences_h12_phase_filtered_decay} for $h_{1,\ulomega}(t)$ and the estimates \eqref{equ:consequences_Ru_local_decay}, \eqref{equ:consequences_px_Ru_local_decay} for $R_{u,\ulomega}(t,x)$. The corresponding estimates for $h_{2,\ulomega}(t)$ and $R_{\bar{u},\ulomega}(t,x)$ follow analogously.

We begin with the proof of \eqref{equ:consequences_h12_decay} for $h_{1,\ulomega}(t)$ defined in \eqref{equ:setup_definitions_h12ulomega}. For short times $0 \leq t \leq 1$ we use \eqref{equ:consequences_assumption2} and crudely bound
\begin{equation*}
 |h_{1,\ulomega}(t)| \leq \int_\bbR |\chi_0(\xi)| \bigl| \tilf_{+,\ulomega}(t,\xi) \bigr| \, \ud \xi \lesssim \sup_{0 \leq t \leq 1} \, \bigl\| \tilf_{+,\ulomega}(t,\xi) \bigr\|_{L^\infty_\xi} \lesssim \varepsilon.
\end{equation*}
For times $t \geq 1$ the decay estimate \eqref{equ:consequences_h12_decay} for $h_{1,\ulomega}(t)$ follows from Lemma~\ref{lem:linear_dispersive_decay} and the bounds \eqref{equ:consequences_assumption2}, \eqref{equ:consequences_sobolev_bound_profile}.

Next we turn to the proof of \eqref{equ:consequences_h12_phase_filtered_decay} for $h_{1,\ulomega}(t)$.
From the definition \eqref{equ:setup_definitions_h12ulomega} we obtain
\begin{equation*}
 \begin{aligned}
  \pt \bigl( e^{it\ulomega} h_{1,\ulomega}(t) \bigr) &= -i \int_\bbR \xi^2 e^{-it\xi^2} \chi_0(\xi) \tilf_{+,\ulomega}(t,\xi) \, \ud \xi + \int_\bbR e^{-it\xi^2} \chi_0(\xi) \pt \tilf_{+,\ulomega}(t,\xi) \, \ud \xi \\
  &=: I(t) + II(t).
 \end{aligned}
\end{equation*}
The first term $I(t)$ is easily bounded for short times $0 \leq t \leq 1$, while for times $t \geq 1$ we integrate by parts in the frequency variable $\xi$,
\begin{equation*}
 I(t) = -\frac{1}{2t} \int_\bbR e^{-it\xi^2} \pxi \bigl( \xi \chi_0(\xi) \tilf_{+,\ulomega}(t,\xi) \bigr) \, \ud \xi,
\end{equation*}
and use \eqref{equ:consequences_assumption2} to conclude that
\begin{equation*}
 |I(t)| \lesssim t^{-1} \Bigl( \bigl\| \tilf_{+,\ulomega}(t,\xi) \bigr\|_{L^\infty_\xi} + \bigl\| \pxi \tilf_{+,\ulomega}(t,\xi) \bigr\|_{L^2_\xi} \Bigr) \lesssim \varepsilon t^{-1+\delta}.
\end{equation*}
To estimate the second term $II(t)$ we insert the evolution equation \eqref{equ:setup_evol_equ_ftilplus} for $\tilfplusulo(t,\xi)$. This gives
\begin{equation*}
 \begin{aligned}
  II(t) &= -i (\dot{\gamma}(t)-\ulomega) \int_\bbR e^{-it\xi^2} \chi_0(\xi) \tilf_{+,\ulomega}(t,\xi) \, \ud \xi \\
  &\quad -i (\dot{\gamma}(t)-\ulomega) \int_\bbR e^{-it\xi^2} \chi_0(\xi) \calL_{\baru, \ulomega}(t,\xi) \, \ud \xi \\
  &\quad -i e^{it\ulomega} \int_\bbR \chi_0(\xi) \wtilcalF_{+, \ulomega}\Bigl[ \calQ_\ulomega\bigl( (\ulPe U)(t)\bigr) + \calC\bigl((\ulPe U)(t)\bigr)\Bigr](\xi) \, \ud \xi \\
  &\quad -i e^{it\ulomega} \int_\bbR \chi_0(\xi) \wtilcalF_{+, \ulomega}\Bigl[ \calMod(t) + \calE_1(t) + \calE_2(t) + \calE_3(t) \Bigr](\xi) \, \ud \xi \\
  &=: II_{(a)}(t) + II_{(b)}(t) + II_{(c)}(t) + II_{(d)}(t).
 \end{aligned}
\end{equation*}
Now we estimate each term on the right-hand side separately.
By \eqref{equ:consequences_assumption2} and \eqref{equ:consequences_aux_bound_modulation2} we obtain that
\begin{equation*}
 |II_{(a)}(t)| \lesssim \bigl| \dot{\gamma}(t) - \ulomega \bigr| \|\chi_0(\xi)\|_{L^1_\xi} \bigl\| \tilfplusulo(t,\xi) \bigr\|_{L^\infty_\xi} \lesssim \varepsilon^2 \jt^{-1+\delta}.
\end{equation*}
Using \eqref{equ:consequences_aux_bound_modulation2} and \eqref{equ:consequences_calL_bounds}, we conclude
\begin{equation*}
 |II_{(b)}(t)| \lesssim |\dot{\gamma}(t)-\ulomega| \|\chi_0(\xi)\|_{L^2_\xi} \bigl\| \calL_{\baru,\ulomega}(t,\xi) \bigr\|_{L^2_\xi} \lesssim \varepsilon^2 \jt^{-\frac32+\delta}.
\end{equation*}
Next, using the $L^2_x \to L^2_\xi$ boundedness of the distorted Fourier transform by Proposition~\ref{prop:mapping_properties_dist_FT}, the $L^2_x \to L^2_x$ boundedness of the projection $\ulPe$ by Lemma~\ref{lemma: L2 decomposition} along with the estimates \eqref{equ:setup_smallness_orbital},\eqref{equ:consequences_ulPe_U_disp_decay}, we infer that
\begin{equation*}
 \begin{aligned}
 |II_{(c)}(t)| &\lesssim \|\chi_0(\xi)\|_{L^2_\xi} \Bigl( \bigl\| \wtilcalF_{+, \ulomega}\bigl[ \calQ_\ulomega\bigl( (\ulPe U)(t)\bigr) \bigr](\xi) \bigr\|_{L^2_\xi} + \bigl\| \wtilcalF_{+, \ulomega}\bigl[ \calC\bigl((\ulPe U)(t)\bigr)\bigr](\xi) \bigr\|_{L^2_\xi} \Bigr) \\
 &\lesssim \|\phi_\ulomega\|_{L^2_x} \bigl\|(\ulPe U)(t)\bigr\|_{L^\infty_x}^2 + \bigl\|(\ulPe U)(t)\bigr\|_{L^2_x} \bigl\|(\ulPe U)(t)\bigr\|_{L^\infty_x}^2 \\
 &\lesssim \varepsilon^2 \jt^{-1}.
 \end{aligned}
\end{equation*}
Finally, writing $\wtilcalF_{+,\ulomega}[\calMod(t)](\xi) = \wtilcalF_{+,\ulomega}[\ulPe \calMod(t)](\xi)$ thanks to \eqref{equ:wtilcalF_applied_to_P} and invoking the $L^2_x \to L^2_\xi$ boundedness of the distorted Fourier transform by Proposition~\ref{prop:mapping_properties_dist_FT} together with the estimates \eqref{equ:consequences_ulPe_Mod_bounds}, \eqref{equ:consequences_calE1_bounds}, \eqref{equ:consequences_calE_2and3_bounds} yields
\begin{equation*}
 \begin{aligned}
 |II_{(d)}(t)| &\lesssim \|\chi_0(\xi)\|_{L^2_\xi} \biggl( \bigl\| \wtilcalF_{+, \ulomega}\bigl[ \ulPe \calMod(t) \bigr](\xi) \bigr\|_{L^2_\xi} + \sum_{j=1}^3 \, \bigl\| \wtilcalF_{+, \ulomega}\bigl[\calE_j(t)\bigr] \bigr\|_{L^2_\xi} \biggr) \\
 &\lesssim \bigl\| \ulPe \calMod(t) \bigr\|_{L^2_x} + \sum_{j=1}^3 \, \bigl\| \calE_j(t) \bigr\|_{L^2_x}
 \lesssim \varepsilon^2 \jt^{-\frac32+\delta}.
 \end{aligned}
\end{equation*}
Combining the preceding estimates finishes the proof of \eqref{equ:consequences_h12_phase_filtered_decay}.

Now we prove \eqref{equ:consequences_Ru_local_decay} for the remainder term $R_{u,\ulomega}(t,x)$. In view of the representation formula \eqref{equ:setup_usube_representation_formula} for $\usube(t,x)$ and the decomposition \eqref{equ:setup_decompositions_usube_local_decay}, the precise expression for $R_{u,\ulomega}(t,x)$ is given by
\begin{equation*}
 \begin{aligned}
  R_{u,\ulomega}(t,x) &= \int_\bbR e^{-it(\xi^2+\ulomega)} \tilfplusulo(t,\xi) \bigl( \Psi_{1,\ulomega}(x,\xi) - \Psi_{1,\ulomega}(x,0) \bigr) \, \ud \xi \\
  &\quad + \Phi_{1,\ulomega}(x) \int_\bbR e^{-it(\xi^2+\ulomega)} \bigl( 1 - \chi_0(\xi) \bigr) \tilfplusulo(t,\xi) \, \ud \xi \\
  &\quad - \int_\bbR e^{it(\xi^2+\ulomega)} \tilfminusulo(t,\xi) \bigl( \Psi_{2,\ulomega}(x,\xi) - \Psi_{2,\ulomega}(x,0) \bigr) \, \ud \xi \\
  &\quad - \Phi_{2,\ulomega}(x) \int_\bbR e^{it(\xi^2+\ulomega)} \bigl( 1 - \chi_0(\xi) \bigr) \tilfminusulo(t,\xi) \, \ud \xi \\
  &=: R_{u,\ulomega}^{(a)}(t,x) + R_{u,\ulomega}^{(b)}(t,x) + R_{u,\ulomega}^{(c)}(t,x) + R_{u,\ulomega}^{(d)}(t,x).
 \end{aligned}
\end{equation*}
By Lemma~\ref{lem:improved_local_decay_difference_Psis} and the bounds \eqref{equ:consequences_assumption2}, \eqref{equ:consequences_sobolev_bound_profile}, we obtain for all $t \geq 0$ that
\begin{equation*}
 \bigl\| \jx^{-2} R_{u,\ulomega}^{(a)}(t,x) \bigr\|_{L^\infty_x} + \bigl\| \jx^{-2} R_{u,\ulomega}^{(c)}(t,x) \bigr\|_{L^\infty_x} \lesssim \jt^{-1+\delta} \varepsilon.
\end{equation*}
Moreover, by Lemma~\ref{lem:improved_local_decay_simple} and the bounds \eqref{equ:consequences_assumption2}, \eqref{equ:consequences_sobolev_bound_profile}, we obtain for all $t \geq 0$ that
\begin{equation*}
 \begin{aligned}
  &\bigl\| \jx^{-2} R_{u,\ulomega}^{(b)}(t,x) \bigr\|_{L^\infty_x} + \bigl\| \jx^{-2} R_{u,\ulomega}^{(d)}(t,x) \bigr\|_{L^\infty_x} \\
  &\lesssim \Bigl( \bigl\| \jx^{-2} \Phi_{1,\ulomega}(x) \bigr\|_{L^\infty_x} + \bigl\| \jx^{-2} \Phi_{2,\ulomega}(x) \bigr\|_{L^\infty_x} \Bigr) \sum_\pm \, \biggl| \int_\bbR e^{\mp it(\xi^2+\ulomega)} \bigl( 1 - \chi_0(\xi) \bigr) \tilf_{\pm, \ulomega}(t,\xi) \, \ud \xi \biggr| \\
  &\lesssim \varepsilon \jt^{-1+\delta}.
 \end{aligned}
\end{equation*}
Combining the preceding estimates gives \eqref{equ:consequences_Ru_local_decay}.

Finally, we establish \eqref{equ:consequences_px_Ru_local_decay} for the remainder term $R_{u,\ulomega}(t,x)$. By direct computation
\begin{equation*}
 \begin{aligned}
  \bigl(\px R_{u,\ulomega}\bigr)(t,x) &= \int_\bbR e^{-it(\xi^2+\ulomega)} \tilfplusulo(t,\xi) \px \bigl( \Psi_{1,\ulomega}(x,\xi) - \Psi_{1,\ulomega}(x,0) \bigr) \, \ud \xi \\
  &\quad + \bigl(\px \Phi_{1,\ulomega}\bigr)(x) \int_\bbR e^{-it(\xi^2+\ulomega)} \bigl( 1 - \chi_0(\xi) \bigr) \tilfplusulo(t,\xi) \, \ud \xi \\
  &\quad - \int_\bbR e^{it(\xi^2+\ulomega)} \tilfminusulo(t,\xi) \px \bigl( \Psi_{2,\ulomega}(x,\xi) - \Psi_{2,\ulomega}(x,0) \bigr) \, \ud \xi \\
  &\quad - \bigl( \px \Phi_{2,\ulomega}\bigr)(x) \int_\bbR e^{it(\xi^2+\ulomega)} \bigl( 1 - \chi_0(\xi) \bigr) \tilfminusulo(t,\xi) \, \ud \xi \\
  &=: \bigl(\px R_{u,\ulomega}\bigr)^{(a)}(t,x) + \bigl(\px R_{u,\ulomega}\bigr)^{(b)}(t,x) +  \bigl(\px R_{u,\ulomega}\bigr)^{(c)}(t,x) +  \bigl(\px R_{u,\ulomega}\bigr)^{(d)}(t,x).
 \end{aligned}
\end{equation*}
Similarly as before, Lemma~\ref{lem:improved_local_decay_difference_Psis} together with the bounds \eqref{equ:consequences_assumption2} and \eqref{equ:consequences_sobolev_bound_profile} implies for all $t \geq 0$ that
\begin{equation*}
 \bigl\| \jx^{-3} \bigl(\px R_{u,\ulomega}\bigr)^{(a)}(t,x) \bigr\|_{L^2_x} + \bigl\| \jx^{-3} \bigl(\px R_{u,\ulomega}\bigr)^{(c)}(t,x) \bigr\|_{L^2_x} \lesssim \jt^{-1+\delta} \varepsilon,
\end{equation*}
while Lemma~\ref{lem:improved_local_decay_simple} combined with the bounds \eqref{equ:consequences_assumption2} and \eqref{equ:consequences_sobolev_bound_profile} gives for all $t \geq 0$,
\begin{equation*}
 \begin{aligned}
  &\bigl\| \jx^{-3} \bigl(\px R_{u,\ulomega}\bigr)^{(b)}(t,x) \bigr\|_{L^2_x} + \bigl\| \jx^{-3} \bigl(\px R_{u,\ulomega}\bigr)^{(d)}(t,x) \bigr\|_{L^2_x} \\
  &\lesssim \Bigl( \bigl\| \jx^{-3} \bigl(\px \Phi_{1,\ulomega}\bigr)(x) \bigr\|_{L^2_x} + \bigl\| \jx^{-3} \bigl( \px \Phi_{2,\ulomega}(x) \bigr) \bigr\|_{L^2_x} \Bigr) \sum_\pm \, \biggl| \int_\bbR e^{\mp it(\xi^2+\ulomega)} \bigl( 1 - \chi_0(\xi) \bigr) \tilf_{\pm, \ulomega}(t,\xi) \, \ud \xi \biggr| \\
  &\lesssim \varepsilon \jt^{-1+\delta}.
 \end{aligned}
\end{equation*}
The estimate \eqref{equ:consequences_px_Ru_local_decay} now follows from the preceding bounds.

\noindent \underline{Proof of (8).}
In what follows we freely assume that $t \geq 1$, and we omit the details for the short time bounds.
 The dispersive decay estimate \eqref{equ:preparation_flat_Schrodinger_wave_bound1} follows directly from Lemma~\ref{lem:linear_dispersive_decay} and \eqref{equ:consequences_assumption2}, \eqref{equ:consequences_sobolev_bound_profile}, while the $L^2_x$ bound \eqref{equ:preparation_flat_Schrodinger_wave_bound_L2} follows from Plancherel's theorem and \eqref{equ:consequences_sobolev_bound_profile}.

 To prove \eqref{equ:preparation_flat_Schrodinger_wave_bound2} for $v_{+,\ulomega}(t,x)$ we compute
 \begin{equation*}
  (\px v_{+,\ulomega})(t,x) = e^{-is\ulomega} \int_\bbR i \xi_1 e^{ix\xi_1} e^{-it\xi_1^2} \fraka(\xi_1) \tilfplusulo(t,\xi_1) \, \ud \xi_1,
 \end{equation*}
 and then integrate by parts in $\xi_1$ to find
 \begin{equation*}
  \begin{aligned}
   (\px v_{+,\ulomega})(t,x) &= \frac{1}{2t} e^{-is\ulomega} \int_\bbR e^{-it\xi_1^2} \partial_{\xi_1} \Bigl( e^{ix\xi_1} \fraka(\xi_1) \tilfplusulo(t,\xi_1) \Bigr) \, \ud \xi_1.
  \end{aligned}
 \end{equation*}
 Using Plancherel's theorem and \eqref{equ:consequences_assumption2}, \eqref{equ:consequences_sobolev_bound_profile}, we obtain the desired estimate
 \begin{equation*}
  \bigl\| \jx^{-1} (\px v_{+,\ulomega})(t,x) \bigr\|_{L^2_x} \lesssim t^{-1} \Bigl( \bigl\| \tilfplusulo(t,\xi_1) \bigr\|_{L^2_{\xi_1}} + \bigl\| \partial_{\xi_1} \tilfplusulo(t,\xi_1) \bigr\|_{L^2_{\xi_1}} \Bigr) \lesssim \varepsilon \jt^{-1+\delta}.
 \end{equation*}
 The proof of \eqref{equ:preparation_flat_Schrodinger_wave_bound2} for $v_{-,\ulomega}(t,x)$ is analogous.

 Next, we observe that the proofs of \eqref{equ:preparation_flat_Schrodinger_wave_bound3} and \eqref{equ:preparation_flat_Schrodinger_wave_bound4} are analogous to the proofs of \eqref{equ:consequences_h12_decay} and \eqref{equ:consequences_h12_phase_filtered_decay}. We omit the details.

 Finally, we turn to the proof of the improved local decay estimate \eqref{equ:preparation_flat_Schrodinger_wave_bound5} for $R_{v_+,\ulomega}(t,x)$, the other case being analogous. From the definitions of $v_{+,\ulomega}(t,x)$ and $\tilde{h}_{1,\ulomega}(t)$ we infer that
 \begin{equation*}
  \begin{aligned}
   R_{v_+,\ulomega}(t,x) &= e^{-it\ulomega} \int_\bbR \bigl( e^{ix\xi_1} - 1 \bigr) e^{-it\xi_1^2} \fraka(\xi_1) \tilfplusulo(t,\xi_1) \, \ud \xi_1 \\
   &\quad + e^{-it\ulomega} \int_\bbR e^{-it\xi_1^2} \bigl( 1 - \chi_0(\xi_1) \bigr) \fraka(\xi_1) \tilfplusulo(t,\xi_1) \, \ud \xi_1.
  \end{aligned}
 \end{equation*}
 Thus, the estimate \eqref{equ:preparation_flat_Schrodinger_wave_bound5} follows from Lemma~\ref{lem:improved_local_decay_simple} and the bounds \eqref{equ:consequences_assumption2}, \eqref{equ:consequences_sobolev_bound_profile}.
\end{proof}

\subsection{Proof of Theorem~\ref{thm:main_theorem}}

Now we are in the position to prove Theorem~\ref{thm:main_theorem} based on the conclusions of Proposition~\ref{prop:modulation_parameters}, Proposition~\ref{prop:profile_bounds}, and Corollary~\ref{cor:consequences}.

\begin{proof}[Proof of Theorem~\ref{thm:main_theorem}]
Fix $\omega_0 \in (0,\infty)$. Let $0 < \varepsilon_0 \ll 1$ and $C_0 \geq 1$ be the constants from the statements of Proposition~\ref{prop:modulation_parameters} and Proposition~\ref{prop:profile_bounds}.
Let $\gamma_0 \in \bbR$ and let $u_0 \in H^1_x(\bbR) \cap L^{2,1}_x(\bbR)$ be even satisfying the smallness condition \eqref{equ:theorem_statement_smallness_initial_condition} in the statement of Theorem~\ref{thm:main_theorem}.
Then Lemma~\ref{lem:setup_local_existence} furnishes a local-in-time $H^1_x \cap L^{2,1}_x$--solution $\psi(t,x)$ to the focusing cubic Schr\"odinger equation \eqref{equ:cubic_NLS} for the initial condition \eqref{equ:theorem_statement_initial_condition} given by 
\begin{equation*}
    \psi(0,x) = e^{i\gamma_0} \bigl( \phi_{\omega_0}(x) + u_0(x) \bigr).
\end{equation*}
The solution $\psi(t,x)$ exists on a maximal interval of existence $[0,T_\ast)$ for some $0 < T_\ast \leq \infty$, and the continuation criterion~\eqref{equ:setup_continuation_criterion} holds.
Then Proposition~\ref{prop:modulation_and_orbital} provides unique continuously differentiable paths $(\omega, \gamma) \colon [0, T_\ast) \to (0,\infty) \times \bbR$ so that the decomposition of the solution
\begin{equation*}
    \psi(t,x) = e^{i\gamma(t)} \bigl( \phi_{\omega(t)}(x) + u(t,x) \bigr), \quad 0 \leq t < T_\ast,
\end{equation*}
satisfies $(1)$--$(5)$ in the statement of Proposition~\ref{prop:modulation_and_orbital}.
We consider the exit time
\begin{equation*}
    \begin{aligned}
        T_0 &:= \sup \, \biggl\{ 0 \leq T < T_\ast \, \bigg| \, \bigl\| \bigl( \tilf_{+, \omega(T)}(t), \tilf_{-, \omega(T)}(t) \bigr) \bigr\|_{X(T)} \leq C_0 \varepsilon, \, \sup_{0 \leq t \leq T} \, \jt^{1-\delta} \bigl|\omega(t) - \omega(T)\bigr| \leq C_0 \varepsilon \biggr\},
    \end{aligned}
\end{equation*}
where $\bigl( \tilf_{+, \omega(T)}(t), \tilf_{-, \omega(T)}(t) \bigr)$ denote the components of the distorted Fourier transform~\eqref{equ:setup_definition_distFT_of_profile_Fulomega} of the profile of the radiation term relative to the linearized operator $\calH\bigl(\omega(T)\big)$.
By the local existence theory we have $T_0 > 0$. 

Our first goal is to show that $T_0 = T_\ast$. We argue by contradiction and suppose $0 < T_0 < T_\ast$. 
Then we pick a strictly monotone increasing sequence $\{T_n\}_{n \in \bbN} \subset [0,T_0)$ with $T_n \nearrow T_0$ as $n \to \infty$. Since the paths $(\omega, \gamma) \colon [0,T_\ast) \to (0,\infty) \times \bbR$ are continuous, it follows that
\begin{equation*}
    \sup_{0 \leq t \leq T_0} \, \jt^{1-\delta} \bigl|\omega(t) - \omega(T_0)\bigr| \leq C_0 \varepsilon.
\end{equation*}
Using Proposition~\ref{prop:profile_bounds} with $\ulomega = \omega(T_0)$ fixed, and noting that by the mapping properties for the distorted Fourier transform from Proposition~\ref{prop:mapping_properties_dist_FT} we have 
\begin{equation*} 
    \begin{aligned}
        \bigl\| \bigl( \tilf_{+, \omega(T_0)}(0,\xi), \tilf_{-, \omega(T_0)}(0,\xi) \bigr) \bigr\|_{L^\infty_\xi} + \bigl\|  \bigl( \pxi \tilf_{+, \omega(T_0)}(0,\xi), \pxi \tilf_{-, \omega(T_0)}(0,\xi) \bigr) \bigr\|_{L^2_\xi} \lesssim_{\omega_0} \|u_0\|_{H^1_x \cap L^{2,1}_x} \lesssim \varepsilon,
    \end{aligned}
\end{equation*}
we conclude via a continuity argument that
\begin{equation*}
    \bigl\| \bigl( \tilf_{+, \omega(T_0)}(t), \tilf_{-, \omega(T_0)}(t) \bigr) \bigr\|_{X(T_0)} \leq C_0 \varepsilon.
\end{equation*}
Keeping $\omega(T_0)$ fixed, it follows by continuity that there exists $T_0 < \widetilde{T} < T_\ast$ such that
\begin{align} 
    \bigl\| \bigl( \tilf_{+, \omega(T_0)}(t), \tilf_{-, \omega(T_0)}(t) \bigr) \bigr\|_{X(\wtilT)} &\leq 2C_0\varepsilon, \label{equ:proof_theorem_aux1} \\
    \sup_{0 \leq t \leq \wtilT} \, \jt^{1-\delta} \bigl|\omega(t) - \omega(T_0)\bigr| &\leq 2C_0\varepsilon. \label{equ:proof_theorem_aux2}
\end{align}
Then we first invoke Proposition~\ref{prop:modulation_parameters} on the time interval $[0, \wtilT]$ with $\ulomega = \omega(T_0)$ fixed to infer from \eqref{equ:proof_theorem_aux1}, \eqref{equ:proof_theorem_aux2} that
\begin{equation} \label{equ:theorem_proof_wtilT_bound1}
    \begin{aligned}
        \sup_{0 \leq t \leq \wtilT} \, \jt^{1-\delta} \bigl|\omega(t) - \omega(\wtilT)\bigr| &\leq C_0 \varepsilon.
    \end{aligned}
\end{equation}
Subsequently, equipped with \eqref{equ:theorem_proof_wtilT_bound1} we use Proposition~\ref{prop:profile_bounds} with $\ulomega = \omega(\wtilT)$ fixed to deduce via a continuity argument on the time interval $[0, \wtilT]$ that
\begin{equation} \label{equ:theorem_proof_wtilT_bound2}
    \bigl\| \bigl( \tilf_{+, \omega(\wtilT)}(t), \tilf_{-, \omega(\wtilT)}(t) \bigr) \bigr\|_{X(\wtilT)} \leq C_0 \varepsilon.
\end{equation}
But then \eqref{equ:theorem_proof_wtilT_bound1} and \eqref{equ:theorem_proof_wtilT_bound2} are a contradiction to $\wtilT > T_0$. Hence, we must have $T_0 = T_\ast$.

By the definition of the exit time we therefore have uniformly for all $0 < T < T_\ast$ that
\begin{equation*}
    \bigl\| \bigl( \tilf_{+, \omega(T)}(t), \tilf_{-, \omega(T)}(t) \bigr) \bigr\|_{X(T)} \leq C_0 \varepsilon, \quad \sup_{0 \leq t \leq T} \, \jt^{1-\delta} \bigl|\omega(t) - \omega(T)\bigr| \leq C_0 \varepsilon.
\end{equation*}
In combination with \eqref{equ:consequences_sobolev_bound_profile}, \eqref{equ:consequences_decomposition_radiation}, and \eqref{equ:consequences_discrete_components_decay}, the mapping properties of the distorted Fourier transform from Proposition~\ref{prop:mapping_properties_dist_FT} lead to the uniform bound
\begin{equation*}
    \sup_{0 \leq t < T_\ast} \, \jt^{-\delta} \|u_0\|_{H^1_x \cap L^{2,1}_x} \lesssim_{\omega_0} \varepsilon,
\end{equation*}
whence the continuation criterion \eqref{equ:setup_continuation_criterion} implies $T_\ast = \infty$.

Next, we pick a strictly increasing sequence $\{T_n\}_{n\in\bbN} \subset (0, \infty)$ with $T_n \nearrow \infty$ as $n \to \infty$. Using that for each $n \in \bbN$ we have
\begin{equation}  \label{equ:proof_theorem_aux3}
    \sup_{0 \leq t \leq T_n} \, \jt^{1-\delta} |\omega(t) - \omega(T_n)| \leq C_0 \varepsilon,
\end{equation}
we conclude that $\{ \omega(T_n) \}_n \subset (0,\infty)$ is a Cauchy sequence satisfying $|\omega(T_n) - \omega_0| \lesssim \varepsilon$ by \eqref{equ:setup_comparison_estimate}. Thus, there exists $\omega_\infty \in (0, \infty)$ with $|\omega_\infty - \omega_0| \lesssim \varepsilon$ such that $\omega(T_n) \to \omega_\infty$ as $n \to \infty$. 
The value $\omega_\infty$ is independent of the chosen sequence $\{T_n\}_{n\in\bbN}$ since \eqref{equ:proof_theorem_aux3} implies 
\begin{equation} \label{equ:proof_theorem_aux4}
    \sup_{0 \leq t < \infty} \, \jt^{1-\delta} |\omega(t) - \omega_\infty| \leq C_0 \varepsilon.
\end{equation}
Then from \eqref{equ:proof_theorem_aux4} and from Proposition~\ref{prop:profile_bounds} with $\ulomega = \omega_\infty$ fixed, we conclude via one more continuity argument that 
\begin{equation} \label{equ:proof_theorem_uniform_profile_bounds}
 \sup_{0 \leq t < \infty} \, \Bigl( \bigl\| \bigl( \tilf_{+, \omega_\infty}(t), \tilf_{-, \omega_\infty}(t) \bigr) \bigr\|_{L^\infty_\xi} + \jt^{-\delta} \bigl\|  \bigl( \pxi \tilf_{+, \omega_\infty}(t), \pxi \tilf_{-, \omega_\infty}(t) \bigr) \bigr\|_{L^2_\xi} \Bigr) \leq C_0 \varepsilon.
\end{equation}

From \eqref{equ:proof_theorem_aux4} and \eqref{equ:proof_theorem_uniform_profile_bounds}, we obtain \eqref{equ:consequences_U_disp_decay} uniformly for all $T > 0$, which implies the dispersive decay estimate \eqref{equ:theorem_statement_decay_radiation} for the radiation term in the statement of Theorem~\ref{thm:main_theorem}.
Moreover, the asserted decay estimates \eqref{equ:theorem_statement_decay_modulation_parameters} for the modulation parameters follow from \eqref{equ:proof_theorem_aux4} and \eqref{equ:consequences_aux_bound_modulation2}.

Finally, we derive the asymptotics for the radiation term building on the computations and estimates in Section~\ref{sec:pointwise_profile}. We begin by invoking Lemma~\ref{lemma: L2 decomposition} to decompose the even radiation term into its projection to the essential spectrum and its discrete components with respect to the reference operator $\calH(\omega_\infty)$,
\begin{equation}
    U(t,x) = \ulPe U(t,x) + d_{1,\omega_\infty}(t)Y_{1,\omega_\infty}(x) + d_{2,\omega_\infty}(t)Y_{2,\omega_\infty}(x), \quad U(t,x) = \begin{bmatrix}
        u(t,x) \\ \baru(t,x)
    \end{bmatrix}.
\end{equation}
By the representation formula \eqref{eqn: representation e-itH} and the definition of the profile \eqref{equ:setup_definition_profile_Fulomega}, we have 
\begin{equation*}
 \begin{aligned}
    (\ulPe U)(t,x) = \bigl( e^{-it\calH(\omega_\infty)} F_{\omega_\infty}(t) \bigr)(x) &= \int_\bbR e^{-it(\xi^2+\omega_\infty)} \tilf_{+, \omega_\infty}(t,\xi) \Psi_{+, \omega_\infty}(x,\xi) \, \ud \xi \\
    &\quad - \int_\bbR e^{it(\xi^2+\omega_\infty)} \tilf_{-, \omega_\infty}(t,\xi) \Psi_{-, \omega_\infty}(x,\xi) \, \ud \xi.
 \end{aligned}
\end{equation*}
Using \eqref{equ:consequences_discrete_components_decay}, we find that the discrete parameters satisfy the decay estimate
\begin{equation}\label{eqn: proof-discrete-decay}
    \vert d_{1,\omega_\infty}(t) \vert + \vert d_{2,\omega_\infty}(t) \vert \lesssim \varepsilon \jt^{-\frac{3}{2}+\delta}.
\end{equation}
Next, we determine the asymptotics of the components $\tilf_{+, \omega_\infty}(t,\xi)$ and $\tilf_{-, \omega_\infty}(t, \xi)$ of the distorted Fourier transform of the profile. In view of Proposition~\ref{prop:pointwise_estimate}, we define
\begin{align}
w_{+,\omega_\infty}(t,\xi) &:= \tilf_{+,\omega_\infty}(t,\xi)e^{i\theta_\infty(t)}\exp\left(\frac{i}{2}\int_1^t \frac{1}{s} \, \vert\tilf_{+,\omega_\infty}(s,\xi)\vert^2 \, \ud s \right),\label{eqn: def-w_+}\\
w_{-,\omega_\infty}(t,\xi) &:= \tilf_{-,\omega_\infty}(t,\xi)e^{-i\theta_\infty(t)}\exp\left(-\frac{i}{2}\int_1^t \frac{1}{s} \, \vert\tilf_{-,\omega_\infty}(s,\xi)\vert^2 \, \ud s \right),\label{eqn: def-w_-}
\end{align}
where
\begin{equation}
    \theta_\infty(t) := \int_0^t (\dot{\gamma}(s)-\omega_\infty)\,\ud s.
\end{equation}
By \eqref{eqn: Cauchy-in-time-estimate} and \eqref{eqn: Cauchy-in-time-estimate2}, we find that for all $1 \leq t_1 < t_2 <\infty$, 
\begin{equation*}
    \Vert w_{\pm,\omega_\infty}(t_2,\xi) - w_{\pm,\omega_\infty}(t_1,\xi) \Vert_{L_\xi^\infty} \lesssim \varepsilon^3 \jap{t_1}^{- \frac{1}{10}+\delta}.
\end{equation*}
Hence, there exist profiles $V_\pm \in L_\xi^\infty(\bbR)$ such that $\lim_{t \rightarrow \infty }w_{\pm,\omega_\infty}(t,\xi) = V_\pm(\xi)$ for every $\xi \in \bbR$ and such that
\begin{equation} \label{equ:proof_asymptotics_first_limit_profile}
    \Vert w_{\pm,\omega_\infty}(t,\xi) - V_\pm(\xi)\Vert_{L_\xi^\infty} \lesssim \varepsilon^3 \jap{t}^{- \frac{1}{10}+\delta}, \quad t\geq 1.
\end{equation}
Repeating the arguments in the proof of Proposition~\ref{prop:pointwise_estimate} using \eqref{eqn: def-w_+}, \eqref{eqn: def-w_-}, \eqref{equ:proof_asymptotics_first_limit_profile}, and the fact that $\vert w_{\pm,\omega_\infty}(t,\xi)\vert = \vert \tilf_{\pm,\omega_\infty}(t,\xi)\vert$, 
we conclude that there exist asymptotic profiles $W_\pm \in L^\infty_\xi(\bbR)$ with $|W_\pm(\xi)| = |V_\pm(\xi)|$ such that for all $t \geq 1$,
\begin{align}
\left \Vert \tilf_{+,\omega_\infty}(t,\xi) - W_+(\xi) e^{-i\theta_\infty(t)}\exp\left(-\frac{i}{2}\log(t)\vert W_+(\xi)\vert^2\right) \right \Vert_{L_\xi^\infty}     &\lesssim \varepsilon^3 \jap{t}^{- \frac{1}{10}+\delta},\label{eqn: proof-f_+-asymptotics}\\
\left \Vert \tilf_{-,\omega_\infty}(t,\xi) - W_-(\xi) e^{i\theta_\infty(t)}\exp\left(\frac{i}{2}\log(t)\vert W_-(\xi)\vert^2\right) \right \Vert_{L_\xi^\infty}     &\lesssim \varepsilon^3 \jap{t}^{- \frac{1}{10}+\delta}\label{eqn: proof-f_--asymptotics}.    
\end{align}
We return to the vector $(\ulPe U)(t,x) := \big(u_\mathrm{e}(t,x),\baru_{\mathrm{e}}(t,x)\big)^\top$ whose components are given by 
\begin{align*}
 \usube(t,x) &:= \int_\bbR e^{-it(\xi^2+\omega_\infty)} \tilf_{+, \omega_\infty}(t,\xi) \Psi_{1, \omega_\infty}(x,\xi) \, \ud \xi - \int_\bbR e^{it(\xi^2+\omega_\infty)} \tilf_{-, \omega_\infty}(t,\xi) \Psi_{2, \omega_\infty}(x,\xi) \, \ud \xi,  \\
 \barusube(t,x) &:= \int_\bbR e^{-it(\xi^2+\omega_\infty)} \tilf_{+, \omega_\infty}(t,\xi) \Psi_{2, \omega_\infty}(x,\xi) \, \ud \xi - \int_\bbR e^{it(\xi^2+\omega_\infty)} \tilf_{-, \omega_\infty}(t,\xi) \Psi_{1, \omega_\infty}(x,\xi) \, \ud \xi. 
\end{align*}
Upon inserting the asymptotics formula \eqref{eqn:linear_dispersive_decay} from Lemma~\ref{lem:linear_dispersive_decay}, we deduce for $t\geq 1$ that 
\begin{equation*}
\begin{split}
\usube(t,x) &= \frac{1}{\sqrt{2t}} e^{-it\omega_\infty}e^{i \frac{x^2}{4t}} e^{-i \frac{\pi}{4}}m_{1,\omega_\infty}(x,\tfrac{x}{2t})\tilf_{+,\omega_\infty}(t,\tfrac{x}{2t})\\
&\quad -\frac{1}{\sqrt{2t}} e^{it\omega_\infty}e^{-i \frac{x^2}{4t}} e^{i \frac{\pi}{4}}m_{2,\omega_\infty}(x,-\tfrac{x}{2t})\tilf_{-,\omega_\infty}(t,-\tfrac{x}{2t})  + \calO_{L_x^\infty}(t^{-\frac34 + \delta}\varepsilon ).
\end{split}
\end{equation*}
This implies via \eqref{eqn: proof-discrete-decay}, \eqref{eqn: proof-f_+-asymptotics}, and \eqref{eqn: proof-f_--asymptotics} the asserted asymptotics \eqref{eqn: theorem-u-asymptotics} for the radiation term $u(t,x)$ and finishes the proof of the theorem.
\end{proof}

\section{Control of the Modulation Parameters} \label{sec:modulation_parameters}

In this section we prove Proposition~\ref{prop:modulation_parameters}.
The proof exploits the oscillations of $\dot{\omega}$ and crucially relies on a remarkable quadratic null structure in the modulation equations. 

\begin{proof}[Proof of Proposition~\ref{prop:modulation_parameters}]
Our goal is to show that uniformly for all $0 \leq t \leq T$ it holds that
\begin{equation} \label{equ:modulation_proof_goal}
    |\omega(t)-\omega(T)| \lesssim \varepsilon^2 \jt^{-1+\delta}.
\end{equation}
The asserted improved bound \eqref{equ:prop_modulation_parameters_conclusion} in the statement of Proposition~\ref{prop:modulation_parameters} then follows directly from \eqref{equ:modulation_proof_goal}.
We begin with an application of the fundamental theorem of calculus to write for $0 \leq t \leq T$,
\begin{equation} \label{equ:modulation_proof_writing_out_omega_difference}
 \omega(t) - \omega(T) = - \int_t^T \dot{\omega}(s) \, \ud s.
\end{equation}
We would now like to insert the equation for $\dot{\omega}(s)$ resulting from the first-order modulation equations \eqref{equ:setup_modulation_equation}. Making the time-dependence of all variables explicit, we recall that \eqref{equ:setup_modulation_equation} reads for $0 \leq s \leq T$,
        \begin{equation} \label{equ:modulation_proof_recall_mod_equns}
            \mathbb{M}(s) \begin{bmatrix}
                \dot{\gamma}(s) - \omega(s) \\
                 \dot{\omega}(s)
            \end{bmatrix}
            = \begin{bmatrix}
                \langle i \calN\bigl(U(s)\bigr), \sigma_2 Y_{1,\omega(s)} \rangle \\
                \langle i \calN\bigl(U(s)\bigr), \sigma_2 Y_{2, \omega(s)} \rangle
              \end{bmatrix}
        \end{equation}
        with
        \begin{equation} \label{equ:modulation_proof_M_definition}
            \mathbb{M}(s) = \frac{2}{\sqrt{\omega(s)}} \begin{bmatrix}
                0 & 1 \\
                1 & 0
            \end{bmatrix}
            +
            \begin{bmatrix}
                \langle U(s), \sigma_1 Y_{1, \omega(s)} \rangle &  \langle U(s),  \sigma_2 Z_{1,\omega(s)} \rangle \\
                \langle U(s), \sigma_1 Y_{2, \omega(s)} \rangle &  \langle U(s), \sigma_2 Z_{2,\omega(s)} \rangle
            \end{bmatrix},
        \end{equation}
        and where we set $Z_{j,\omega} := \partial_\omega Y_{j,\omega}$ for $j \in \{1,2\}$.
        The matrix $\bbM(s)$ is invertible since the first matrix on the right-hand side of \eqref{equ:modulation_proof_M_definition} is invertible with $\frac12 \omega_0 \leq \omega(s) \leq 2\omega_0$, while the second matrix on the right-hand side of \eqref{equ:modulation_proof_M_definition} is of size $\calO(\varepsilon)$ by \eqref{equ:setup_smallness_orbital} with $0 < \varepsilon \leq \varepsilon_0 \ll 1$ sufficiently small depending on the size of $\omega_0$.
        Correspondingly, we have the following uniform-in-time operator norm bound for the inverse
        \begin{equation} \label{equ:modulation_proof_crude_Mt_op_norm_bound}
         \sup_{0 \leq s \leq T} \, \bigl\|[ \bbM(s) ]^{-1}\bigr\| \lesssim_{\omega_0} 1.
        \end{equation}

Since we need to exploit the oscillatory properties of $\dot{\omega}(s)$ in \eqref{equ:modulation_proof_writing_out_omega_difference}, this crude bound does not suffice for our purposes. Instead, we need to expand the inverse $[\bbM(s)]^{-1}$ into leading order terms and a sufficiently fast decaying remainder term.
To this end we pass to the fixed modulation parameter $\ulomega$ and we insert the decomposition \eqref{equ:consequences_decomposition_radiation} of $U(s)$ into its projection to the essential spectrum and its discrete components.
We obtain
\begin{equation*}
  \bbM(s) = \bbM_{\ulomega} + \bbA_{\ulomega}(s) + \bbB(s)
\end{equation*}
with
\begin{equation*}
 \begin{aligned}
  \bbM_{\ulomega} := \begin{bmatrix} 0 & c_{\ulomega} \\ c_{\ulomega} & 0 \end{bmatrix}, \quad
  \bbA_{\ulomega}(s) := \begin{bmatrix}
                \bigl\langle \ulPe U(s), \sigma_1 Y_{1, \ulomega} \bigr\rangle &  \bigl\langle \ulPe U(s),  \sigma_2 Z_{1,\ulomega} \bigr\rangle \\
                \bigl\langle \ulPe U(s), \sigma_1 Y_{2, \ulomega} \bigr\rangle &  \bigl\langle \ulPe U(s), \sigma_2 Z_{2,\ulomega} \bigr\rangle
            \end{bmatrix},
 \end{aligned}
\end{equation*}
and $\bbB(s) := \bbM(s) - \bbM_{\ulomega} - \bbA_{\ulomega}(s)$. Using \eqref{equ:consequences_ulPe_U_disp_decay}, we obtain the simple operator norm decay estimate
\begin{equation} \label{equ:modulation_proof_At_op_norm_bound}
 \sup_{0 \leq s \leq T} \, \js^{\frac12} \bigl\| \bbA_\ulomega(s) \bigr\| \lesssim_{\omega_0} \varepsilon.
\end{equation}
Moreover, using \eqref{equ:prop_modulation_parameters_assumption1}, \eqref{equ:consequences_ulPe_U_disp_decay}, \eqref{equ:consequences_U_disp_decay}, and the faster decay of the discrete components \eqref{equ:consequences_discrete_components_decay}, it is straightforward to infer the following operator norm decay bound
\begin{equation} \label{equ:modulation_proof_Bt_op_norm_bound}
 \sup_{0 \leq s \leq T} \, \js^{1-\delta} \bigl\| \bbB(s) \bigr\| \lesssim_{\omega_0} \varepsilon.
\end{equation}
By direct computation we find that
\begin{equation} \label{equ:modulation_proof_expansion_Mt_inverse}
 [\bbM(s)]^{-1} = \bbM_{\ulomega}^{-1} - \bbM_{\ulomega}^{-1} \bbA_{\ulomega}(s) \bbM_{\ulomega}^{-1} + \bbD(s)
\end{equation}
with
\begin{equation*}
 \bbD(s) := [\bbM(s)]^{-1} \Bigl( - \bbB(s) \bbM_{\ulomega}^{-1} + \bbA_{\ulomega}(s) \bbM_{\ulomega}^{-1} \bbA_{\ulomega}(s) \bbM_{\ulomega}^{-1} + \bbB(s) \bbM_{\ulomega}^{-1} \bbA_{\ulomega}(s) \bbM_{\ulomega}^{-1} \Bigr).
\end{equation*}
Using \eqref{equ:modulation_proof_crude_Mt_op_norm_bound}, \eqref{equ:modulation_proof_At_op_norm_bound}, and \eqref{equ:modulation_proof_Bt_op_norm_bound}, we obtain that
\begin{equation} \label{equ:modulation_proof_Dt_op_norm_bound}
 \sup_{0 \leq s \leq T} \, \js^{1-\delta} \bigl\| \bbD(s) \bigr\| \lesssim_{\omega_0} \varepsilon.
\end{equation}

Next, we expand the right-hand side of the modulation equations \eqref{equ:modulation_proof_recall_mod_equns} into leading order quadratic and cubic terms and a remainder term. Passing to the fixed modulation parameter $\ulomega$ and inserting the decomposition \eqref{equ:consequences_decomposition_radiation} of $U(t)$ into its projection to the essential spectrum and its discrete components, we write
\begin{equation} \label{equ:modulation_proof_expansion_nonlinearity}
 \begin{aligned}
    \begin{bmatrix}
                \bigl\langle i \calN\bigl(U(s)\bigr), \sigma_2 Y_{1,\omega(s)} \bigr\rangle \\
                \bigl\langle i \calN\bigl(U(s)\bigr), \sigma_2 Y_{2, \omega(s)} \bigr\rangle
    \end{bmatrix}
  =
    \begin{bmatrix}
                \bigl\langle i \calQ_{\ulomega}\bigl(\ulPe U(s)\bigr), \sigma_2 Y_{1,\ulomega} \bigr\rangle \\
                \bigl\langle i \calQ_{\ulomega}\bigl(\ulPe U(s)\bigr), \sigma_2 Y_{2,\ulomega} \bigr\rangle
    \end{bmatrix}
  + \begin{bmatrix}
                \bigl\langle i \calC\bigl(\ulPe U(s)\bigr), \sigma_2 Y_{1,\ulomega} \bigr\rangle \\
                \bigl\langle i \calC\bigl(\ulPe U(s)\bigr), \sigma_2 Y_{2,\ulomega} \bigr\rangle
    \end{bmatrix}
  + H(s).
 \end{aligned}
\end{equation}
Using \eqref{equ:prop_modulation_parameters_assumption1}, \eqref{equ:consequences_ulPe_U_disp_decay}, \eqref{equ:consequences_discrete_components_decay}, and \eqref{equ:consequences_U_disp_decay}, it follows that the remainder term $H(s)$ has sufficient decay
\begin{equation} \label{equ:modulation_proof_remainder_nonlinearity_decay}
 \sup_{0 \leq s \leq T} \, \js^{2-\delta} |H(s)| \lesssim \varepsilon^3.
\end{equation}

Now we rewrite the modulation equations \eqref{equ:modulation_proof_recall_mod_equns} by inserting the expansions \eqref{equ:modulation_proof_expansion_Mt_inverse} and \eqref{equ:modulation_proof_expansion_nonlinearity},
\begin{equation} \label{equ:modulation_proof_expansion_of_RHS}
 \begin{aligned}
  &\begin{bmatrix}
   \dot{\gamma}(s) - \dot{\omega}(s) \\ \dot{\omega}(s)
  \end{bmatrix} \\
  &=
   [ \bbM(s) ]^{-1} \begin{bmatrix}
                \langle i \calN\bigl(U(s)\bigr), \sigma_2 Y_{1,\omega(s)} \rangle \\
                \langle i \calN\bigl(U(s)\bigr), \sigma_2 Y_{2, \omega(s)} \rangle
                \end{bmatrix} \\
 &=  \bbM_{\ulomega}^{-1} \begin{bmatrix}
                \bigl\langle i \calQ_{\ulomega}\bigl(\ulPe U(s)\bigr), \sigma_2 Y_{1,\ulomega} \bigr\rangle \\
                \bigl\langle i \calQ_{\ulomega}\bigl(\ulPe U(s)\bigr), \sigma_2 Y_{2,\ulomega} \bigr\rangle
    \end{bmatrix}
    + \bbM_{\ulomega}^{-1} \begin{bmatrix}
                \bigl\langle i \calC\bigl(\ulPe U(s)\bigr), \sigma_2 Y_{1,\ulomega} \bigr\rangle \\
                \bigl\langle i \calC\bigl(\ulPe U(s)\bigr), \sigma_2 Y_{2,\ulomega} \bigr\rangle
    \end{bmatrix} \\
 &\quad - \bbM_{\ulomega}^{-1} \bbA_{\ulomega}(s) \bbM_{\ulomega}^{-1} \begin{bmatrix}
                \bigl\langle i \calQ_{\ulomega}\bigl(\ulPe U(s)\bigr), \sigma_2 Y_{1,\ulomega} \bigr\rangle \\
                \bigl\langle i \calQ_{\ulomega}\bigl(\ulPe U(s)\bigr), \sigma_2 Y_{2,\ulomega} \bigr\rangle
    \end{bmatrix}
    - \bbM_{\ulomega}^{-1} \bbA_{\ulomega}(s) \bbM_{\ulomega}^{-1} \begin{bmatrix}
                \bigl\langle i \calC\bigl(\ulPe U(s)\bigr), \sigma_2 Y_{1,\ulomega} \bigr\rangle \\
                \bigl\langle i \calC\bigl(\ulPe U(s)\bigr), \sigma_2 Y_{2,\ulomega} \bigr\rangle
    \end{bmatrix} \\
&\quad + \bbD(s) \begin{bmatrix}
                \bigl\langle i \calQ_{\ulomega}\bigl(\ulPe U(s)\bigr), \sigma_2 Y_{1,\ulomega} \bigr\rangle \\
                \bigl\langle i \calQ_{\ulomega}\bigl(\ulPe U(s)\bigr), \sigma_2 Y_{2,\ulomega} \bigr\rangle
    \end{bmatrix}
    + \bbD(s) \begin{bmatrix}
                \bigl\langle i \calC\bigl(\ulPe U(s)\bigr), \sigma_2 Y_{1,\ulomega} \bigr\rangle \\
                \bigl\langle i \calC\bigl(\ulPe U(s)\bigr), \sigma_2 Y_{2,\ulomega} \bigr\rangle
    \end{bmatrix} + [\bbM(s)]^{-1} H(s) \\
&=: I(s) + II(s) + III(s) + IV(s) + V(s) + VI(s) + VII(s).
 \end{aligned}
\end{equation}
In what follows we estimate the delicate contributions of the (second components) of the vectorial terms $I(s)$, $II(s)$, and $III(s)$ to the integral on the right-hand side of \eqref{equ:modulation_proof_writing_out_omega_difference} by exploiting their oscillatory properties. The treatment of the term $I(s)$ additionally hinges on a crucial null structure for a resonant quadratic term.
Instead the contributions of the terms $IV(s)$, $V(s)$, $VI(s)$, and $VII(s)$ can be bounded quite crudely.
We proceed term by term.

\medskip
\noindent \underline{Contribution of $I(s)$.}
Only the second component of the vector $I(s)$ contributes to $\dot{\omega}(s)$. We have
\begin{equation*}
 \begin{aligned}
   \bbM_{\ulomega}^{-1} \begin{bmatrix}
                \bigl\langle i \calQ_{\ulomega}\bigl(\ulPe U(s)\bigr), \sigma_2 Y_{1,\ulomega} \bigr\rangle \\
                \bigl\langle i \calQ_{\ulomega}\bigl(\ulPe U(s)\bigr), \sigma_2 Y_{2,\ulomega} \bigr\rangle
    \end{bmatrix}
 \end{aligned}
 =  c_{\ulomega}^{-1} \begin{bmatrix}
                \bigl\langle i \calQ_{\ulomega}\bigl(\ulPe U(s)\bigr), \sigma_2 Y_{2,\ulomega} \bigr\rangle \\
                \bigl\langle i \calQ_{\ulomega}\bigl(\ulPe U(s)\bigr), \sigma_2 Y_{1,\ulomega} \bigr\rangle
    \end{bmatrix},
\end{equation*}
whence the contribution of $I(s)$ to the integral on the right-hand side of \eqref{equ:modulation_proof_writing_out_omega_difference} is
\begin{equation} \label{equ:modulation_proof_termI_contribution_integral}
 \begin{aligned}
  - c_{\ulomega}^{-1} \int_t^T \bigl\langle i \calQ_{\ulomega}\bigl(\ulPe U(s)\bigr), \sigma_2 Y_{1,\ulomega} \bigr\rangle \, \ud s = -i c_{\ulomega}^{-1} \int_t^T \bigl\langle \bigl( \usube(s)^2 - \barusube(s)^2 \bigr), \phi_\ulomega^2 \bigr\rangle \, \ud s.
 \end{aligned}
\end{equation}
Inserting the representation formulas \eqref{equ:setup_usube_representation_formula} for $\usube(s,x)$, respectively \eqref{equ:setup_barusube_representation_formula} for $\barusube(s,x)$, we find by direct computation
\begin{equation*}
 \begin{aligned}
  &\usube(s,x)^2 - \barusube(s,x)^2 \\
  &= \iint e^{-is(\xi_1^2 + \xi_2^2 + 2\ulomega)} \tilfplusulo(s,\xi_1) \tilfplusulo(s,\xi_2) \bigl( \Psioneulomega(x,\xi_1) \Psioneulomega(x,\xi_2) - \Psitwoulomega(x,\xi_1) \Psitwoulomega(x,\xi_2) \bigr) \, \ud \xi_1 \, \ud \xi_2 \\
  &\quad + \iint e^{i s(\xi_1^2 + \xi_2^2 + 2\ulomega)} \tilfminusulo(s,\xi_1) \tilfminusulo(s,\xi_2) \bigl( \Psitwoulomega(x,\xi_1) \Psitwoulomega(x,\xi_2) - \Psioneulomega(x,\xi_1) \Psioneulomega(x,\xi_2) \bigr) \, \ud \xi_1 \, \ud \xi_2 \\
  &\quad - 2 \iint e^{i s (-\xi_1^2 + \xi_2^2)} \tilfplusulo(s,\xi_1) \tilfminusulo(s,\xi_2) \bigl( \Psioneulomega(x,\xi_1) \Psitwoulomega(x,\xi_2) - \Psitwoulomega(x,\xi_1) \Psioneulomega(x,\xi_2) \bigr) \, \ud \xi_1 \, \ud \xi_2.
 \end{aligned}
\end{equation*}
Thus, we have
\begin{equation} \label{equ:modulation_quadratic_leading_order_expansion}
 \begin{aligned}
  &\bigl\langle \bigl( \usube(s)^2 - \barusube(s)^2 \bigr), \phi_\ulomega^2 \bigr\rangle \\
  &= \iint e^{-is(\xi_1^2 + \xi_2^2 + 2\ulomega)} \tilfplusulo(s,\xi_1) \tilfplusulo(s,\xi_2) \, \nu_{++,\ulomega}(\xi_1, \xi_2) \, \ud \xi_1 \, \ud \xi_2 \\
  &\quad + \iint e^{is(\xi_1^2 + \xi_2^2 + 2\ulomega)} \tilfminusulo(s,\xi_1) \tilfminusulo(s,\xi_2) \, \nu_{--,\ulomega}(\xi_1, \xi_2) \, \ud \xi_1 \, \ud \xi_2 \\
  &\quad -2 \iint e^{i s(-\xi_1^2 + \xi_2^2)} \tilfplusulo(s,\xi_1) \tilfminusulo(s,\xi_2) \, \nu_{+-,\ulomega}(\xi_1, \xi_2) \, \ud \xi_1 \, \ud \xi_2 \\
  &=: I_{(a)}(s) + I_{(b)}(s) + I_{(c)}(s),
 \end{aligned}
\end{equation}
where
\begin{equation*}
 \begin{aligned}
  \nu_{++,\ulomega}(\xi_1, \xi_2) &:= \int_\bbR \bigl( \Psioneulomega(x,\xi_1) \Psioneulomega(x,\xi_2) - \Psitwoulomega(x,\xi_1) \Psitwoulomega(x,\xi_2) \bigr) \phi_\ulomega(x)^2 \, \ud x, \\
  \nu_{--,\ulomega}(\xi_1, \xi_2) &:= \int_\bbR \bigl( \Psitwoulomega(x,\xi_1) \Psitwoulomega(x,\xi_2) - \Psioneulomega(x,\xi_1) \Psioneulomega(x,\xi_2) \bigr) \phi_\ulomega(x)^2 \, \ud x, \\
  \nu_{+-,\ulomega}(\xi_1, \xi_2) &:= \int_\bbR \bigl( \Psioneulomega(x,\xi_1) \Psitwoulomega(x,\xi_2) - \Psitwoulomega(x,\xi_1) \Psioneulomega(x,\xi_2) \bigr) \phi_\ulomega(x)^2 \, \ud x.
 \end{aligned}
\end{equation*}
These quadratic spectral distributions have been computed explicitly in Subsection~\ref{subsec:quadratic_spectral_distributions_modulation}.
We observe that the term $I_{(c)}(s)$ has time resonances, while the terms $I_{(a)}(s)$ and $I_{(b)}(s)$ do not have time resonances. In what follows, we describe in detail how to control the contributions of the resonant term $I_{(c)}(s)$ to the time integral on the right-hand side of \eqref{equ:modulation_proof_termI_contribution_integral}. A remarkable null structure of the quadratic spectral distribution $\nu_{+-,\ulomega}(\xi_1, \xi_2)$ is key for this step. Subsequently, we outline how to control the contributions of the non-resonant terms $I_{(a)}(s)$ and $I_{(b)}(s)$.

\medskip
\noindent \underline{Contribution of the resonant term $I_{(c)}(s)$.}
Estimating the contribution of $I_{(c)}(s)$ to the time integral on the right-hand side of \eqref{equ:modulation_proof_termI_contribution_integral} amounts to bounding for $0 \leq t \leq T$,
\begin{equation}
 \int_t^T \iint e^{i s(-\xi_1^2 + \xi_2^2)} \tilfplusulo(s,\xi_1) \tilfminusulo(s,\xi_2) \, \nu_{+-,\ulomega}(\xi_1, \xi_2) \, \ud \xi_1 \, \ud \xi_2 \, \ud s.
\end{equation}
Remarkably, Lemma~\ref{lem:null_structure_radiation} uncovers that the quadratic spectral distribution $\nu_{+-,\ulomega}(\xi_1, \xi_2)$ comes with a factor of $(-\xi_1^2+\xi_2^2)$, whence we can integrate by parts in time.
A closer inspection of the expression \eqref{eqn: nu+-} shows that $\nu_{+-,\ulomega}(\xi_1, \xi_2)$ is a linear combination of terms of the form
\begin{equation*}
 (-\xi_1^2 + \xi_2^2) \fraka_{\ulomega}(\xi_1) \frakb_{\ulomega}(\xi_2) \kappa_{\ulomega}(\xi_1+\xi_2),
\end{equation*}
where the multipliers $\fraka_\ulomega(\eta)$, $\frakb_\ulomega(\eta)$ are of the form
\begin{equation} \label{equ:modulation_proof_multipliers_types}
 \frac{1}{(|\eta|-i\sqrt{\omega})^2}, \quad \frac{\eta}{(|\eta|-i\sqrt{\omega})^2}, \quad \text{or} \quad \frac{\eta^2}{(|\eta|-i\sqrt{\omega})^2},
\end{equation}
and where
\begin{equation*}
 \kappa_\ulomega(\xi_1+\xi_2) := \frac{\xi_1+\xi_2}{\sqrt{\omega}} \cosech\Bigl( \frac{\pi}{2} \frac{\xi_1+\xi_2}{\sqrt{\omega}} \Bigr) = \int_\bbR e^{ix(\xi_1+\xi_2)} q_\ulomega(x) \, \ud x, \quad q_\ulomega(x) := \sqrt{\frac{2\ulomega}{\pi}} \sech^2(\sqrt{\ulomega} x).
\end{equation*}
We emphasize that $\fraka_\ulomega, \frakb_\ulomega \in W^{1,\infty}(\bbR)$.
Without loss of generality, it therefore suffices to consider for times $0 \leq t \leq T$,
\begin{equation*}
 \calI_\ulomega(t;T) := \int_t^T \iint e^{i s(-\xi_1^2 + \xi_2^2)} (-\xi_1^2 + \xi_2^2) \, \fraka_{\ulomega}(\xi_1) \tilfplusulo(s,\xi_1) \, \frakb_{\ulomega}(\xi_2) \tilfminusulo(s,\xi_2) \, \kappa_\ulomega(\xi_1+\xi_2) \, \ud \xi_1 \, \ud \xi_2 \, \ud s.
\end{equation*}
Integrating by parts in time, we find
\begin{equation*}
 \begin{aligned}
   \calI_\ulomega(t;T) &= -i \biggl[ \iint e^{i s(-\xi_1^2 + \xi_2^2)} \, \fraka_{\ulomega}(\xi_1) \tilfplusulo(s,\xi_1) \, \frakb_{\ulomega}(\xi_2) \tilfminusulo(s,\xi_2) \, \kappa_\ulomega(\xi_1+\xi_2) \, \ud \xi_1 \, \ud \xi_2 \biggr]_{s=t}^{s=T} \\
   &\quad + i \int_t^T \iint e^{i s(-\xi_1^2 + \xi_2^2)} \, \fraka_{\ulomega}(\xi_1) \ps \tilfplusulo(s,\xi_1) \, \frakb_{\ulomega}(\xi_2) \tilfminusulo(s,\xi_2) \, \kappa_\ulomega(\xi_1+\xi_2) \, \ud \xi_1 \, \ud \xi_2 \, \ud s \\
   &\quad + i \int_t^T \iint e^{i s(-\xi_1^2 + \xi_2^2)} \, \fraka_{\ulomega}(\xi_1) \tilfplusulo(s,\xi_1) \, \frakb_{\ulomega}(\xi_2) \ps \tilfminusulo(s,\xi_2) \, \kappa_\ulomega(\xi_1+\xi_2) \, \ud \xi_1 \, \ud \xi_2 \, \ud s \\
   &=: \calI_{\ulomega}^{1}(t;T) + \calI_{\ulomega}^{2}(t;T) + \calI_{\ulomega}^{3}(t;T).
 \end{aligned}
\end{equation*}
In order to estimate the preceding terms, it is useful to introduce the free Schr\"odinger evolutions
\begin{align}
  v_{+,\ulomega}(s,x) &:= e^{-is\ulomega} \int_\bbR e^{ix\xi_1} e^{-is\xi_1^2} \fraka_{\ulomega}(\xi_1) \tilfplusulo(s,\xi_1) \, \ud \xi_1, \label{equ:modulation_definition_free_Schrodinger_plus} \\
  v_{-,\ulomega}(s,x) &:= e^{is\ulomega} \int_\bbR e^{ix\xi_2} e^{is\xi_2^2} \frakb_{\ulomega}(\xi_2) \tilfminusulo(s,\xi_2) \, \ud \xi_2. \label{equ:modulation_definition_free_Schrodinger_minus}
\end{align}
Occasionally, we will need to isolate their leading order local decay behavior. To this end we introduce the decompositions
\begin{align}
 v_{+,\ulomega}(s,x) &= \tilde{h}_{1,\ulomega}(s) + R_{v_+,\ulomega}(s,x), \label{equ:modulation_free_Schrodinger_plus_decomposition} \\
 v_{-,\ulomega}(s,x) &= \tilde{h}_{2,\ulomega}(s) + R_{v_-,\ulomega}(s,x), \label{equ:modulation_free_Schrodinger_minus_decomposition}
\end{align}
with
\begin{align*}
 \tilde{h}_{1,\ulomega}(s) &= e^{-is\ulomega} \int_\bbR e^{-is\xi_1^2} \chi_0(\xi_1) \fraka_{\ulomega}(\xi_1) \tilfplusulo(s,\xi_1) \, \ud \xi_1, \\
 \tilde{h}_{2,\ulomega}(s) &= e^{is\ulomega} \int_\bbR e^{is\xi_2^2} \chi_0(\xi_2) \frakb_{\ulomega}(\xi_2) \tilfminusulo(s,\xi_2) \, \ud \xi_2.
\end{align*}
We observe that the bounds \eqref{equ:preparation_flat_Schrodinger_wave_bound1} -- \eqref{equ:preparation_flat_Schrodinger_wave_bound5} apply to the Schr\"odinger evolutions $v_{+,\ulomega}(s,x)$, $v_{-,\ulomega}(s,x)$, and their components.

The bound for the first term $\calI_{\ulomega}^{1}(t;T)$ is straightforward. We rewrite $\calI_{\ulomega}^{1}(t;T)$ in terms of the Schr\"odinger waves \eqref{equ:modulation_definition_free_Schrodinger_plus}, \eqref{equ:modulation_definition_free_Schrodinger_minus} as
\begin{equation*}
 \begin{aligned}
  \calI_{\ulomega}^{1}(t;T) &= -i \biggl[ \int_\bbR q_\ulomega(x) v_{+,\ulomega}(s,x) v_{-,\ulomega}(s,x) \, \ud x \biggr]_{s=t}^{s=T}.
 \end{aligned}
\end{equation*}
Then we obtain by \eqref{equ:preparation_flat_Schrodinger_wave_bound1} the sufficient bound
\begin{equation*}
 \begin{aligned}
  \bigl| \calI_{\ulomega}^{1}(t;T) \bigr| &\lesssim \sup_{t \leq s \leq T} \, \|q_{\ulomega}\|_{L^1_x} \|v_{+,\ulomega}(s)\|_{L^\infty_x} \|v_{-,\ulomega}(s)\|_{L^\infty_x} \lesssim \sup_{t \leq s \leq T} \, \varepsilon^2 \js^{-1} \lesssim \varepsilon^2 \jt^{-1}.
 \end{aligned}
\end{equation*}

In order to estimate the terms $\calI_{\ulomega}^{2}(t;T)$ and $\calI_{\ulomega}^{3}(t;T)$, we insert the evolution equations \eqref{equ:setup_evol_equ_ftilplus} for the profile $\tilfplusulo(s,\xi_1)$, respectively \eqref{equ:setup_evol_equ_ftilminus} for the profile $\tilfminusulo(s,\xi_2)$. We provide the details for the term $\calI_{\ulomega}^{2}(t;T)$, and we leave the analogous treatment of the term $\calI_{\ulomega}^{3}(t;T)$ to the reader.
Inserting \eqref{equ:setup_evol_equ_ftilplus} we find that
\begin{equation*}
 \begin{aligned}
  \calI_{\ulomega}^2(t;T) &= \int_t^T \bigl( \dot{\gamma}(s)-\ulomega \bigr) \iint e^{is(-\xi_1^2+\xi_2^2)} \fraka_\ulomega(\xi_1) \tilfplusulo(s,\xi_1) \, \frakb_{\ulomega}(\xi_2) \tilfminusulo(s,\xi_2) \, \kappa_\ulomega(\xi_1+\xi_2) \, \ud \xi_1 \, \ud \xi_2 \, \ud s \\
  &\quad + \int_t^T \bigl( \dot{\gamma}(s)-\ulomega \bigr) \iint e^{is(\xi_2^2+\ulomega)} \fraka_\ulomega(\xi_1) \calL_{\baru, \ulomega}(s,\xi_1) \, \frakb_{\ulomega}(\xi_2) \tilfminusulo(s,\xi_2) \, \kappa_\ulomega(\xi_1+\xi_2) \, \ud \xi_1 \, \ud \xi_2 \, \ud s \\
  &\quad + \int_t^T \iint e^{is(\xi_2^2+\ulomega)} \fraka_{\ulomega}(\xi_1) \wtilcalF_{+, \ulomega}\bigl[ \calQ_\ulomega\bigl( (\ulPe U)(s)\bigr) \bigr](\xi_1) \, \frakb_{\ulomega}(\xi_2) \tilfminusulo(s,\xi_2) \, \kappa_\ulomega(\xi_1+\xi_2) \, \ud \xi_1 \, \ud \xi_2 \, \ud s \\
  &\quad + \int_t^T \iint e^{is(\xi_2^2+\ulomega)} \fraka_{\ulomega}(\xi_1) \wtilcalF_{+, \ulomega}\bigl[ \calC\bigl((\ulPe U)(s)\bigr) \bigr](\xi_1) \, \frakb_{\ulomega}(\xi_2) \tilfminusulo(s,\xi_2) \, \kappa_\ulomega(\xi_1+\xi_2) \, \ud \xi_1 \, \ud \xi_2 \, \ud s \\
  &\quad + \int_t^T \iint e^{is(\xi_2^2+\ulomega)} \fraka_{\ulomega}(\xi_1) \wtilcalF_{+, \ulomega}\bigl[ \calMod(s) + \calE_1(s) + \calE_2(s) + \calE_3(s) \bigr](\xi_1) \\
  &\qquad \qquad \qquad \qquad \qquad \qquad \qquad \qquad \qquad \qquad \quad \times \frakb_{\ulomega}(\xi_2) \tilfminusulo(s,\xi_2) \, \kappa_\ulomega(\xi_1+\xi_2) \, \ud \xi_1 \, \ud \xi_2 \, \ud s \\
  &=: \calI_{\ulomega}^{2,1}(t;T) + \calI_{\ulomega}^{2,2}(t;T) + \calI_{\ulomega}^{2,3}(t;T) + \calI_{\ulomega}^{2,4}(t;T) + \calI_{\ulomega}^{2,5}(t;T).
 \end{aligned}
\end{equation*}
We write the term $\calI_{\ulomega}^{2,1}(t;T)$ as
\begin{equation*}
 \begin{aligned}
  \calI_{\ulomega}^{2,1}(t;T) = \int_t^T \bigl( \dot{\gamma}(s)-\ulomega \bigr) \biggl( \int_\bbR q_\ulomega(x) v_{+,\ulomega}(s,x) v_{-,\ulomega}(s,x) \, \ud x \biggr) \, \ud s.
 \end{aligned}
\end{equation*}
Then \eqref{equ:consequences_aux_bound_modulation2} and \eqref{equ:preparation_flat_Schrodinger_wave_bound1} imply
\begin{equation*}
 \begin{aligned}
  \bigl| \calI_{\ulomega}^{2,1}(t;T) \bigr| &\lesssim \int_t^T |\dot{\gamma}(s)-\ulomega| \|q_\ulomega\|_{L^1_x} \|v_{+,\ulomega}(s)\|_{L^\infty_x} \|v_{-,\ulomega}(s)\|_{L^\infty_x} \, \ud s \\
  &\lesssim \int_t^T \varepsilon \js^{-1+\delta} \cdot \varepsilon^2 \js^{-1} \, \ud s \lesssim \varepsilon^3 \jt^{-1+\delta}.
 \end{aligned}
\end{equation*}
Similarly, we write the term $\calI_{\ulomega}^{2,2}(t;T)$ as
\begin{equation*}
 \begin{aligned}
  \calI_{\ulomega}^{2,2}(t;T) = \int_t^T \bigl( \dot{\gamma}(s)-\ulomega \bigr) \biggl( \int_\bbR q_\ulomega(x) \biggl( \int_\bbR e^{ix\xi_2} \fraka_{\ulomega}(\xi_1) \calL_{\baru, \ulomega}(s,\xi_1) \, \ud \xi_1 \biggr)   v_{-,\ulomega}(s,x) \, \ud x \biggr) \, \ud s.
 \end{aligned}
\end{equation*}
Using Plancherel's theorem and the bounds \eqref{equ:consequences_aux_bound_modulation2}, \eqref{equ:consequences_calL_bounds}, \eqref{equ:preparation_flat_Schrodinger_wave_bound1}, we conclude
\begin{equation*}
 \begin{aligned}
  \bigl| \calI_{\ulomega}^{2,2}(t;T) \bigr| &\lesssim \int_t^T |\dot{\gamma}(s) - \ulomega| \|q_\ulomega\|_{L^2_x} \bigl\|\calL_{\baru, \ulomega}(s,\xi_1)\bigr\|_{L^2_{\xi_1}} \|v_{-,\ulomega}(s)\|_{L^\infty_x} \, \ud s \\
  &\lesssim \int_t^T \varepsilon \js^{-1+\delta} \cdot \varepsilon \js^{-\frac12} \cdot \varepsilon \js^{-\frac12} \, \ud s \lesssim \varepsilon^3 \jt^{-1+\delta}.
 \end{aligned}
\end{equation*}
Now we turn to the term $\calI_{\ulomega}^{2,3}(t;T)$, which we write as
\begin{equation*}
 \begin{aligned}
    \calI_{\ulomega}^{2,3}(t;T)
    &= \int_t^T \biggl( \int_\bbR q_\ulomega(x) \biggl( \int_\bbR e^{ix\xi_1} \fraka_{\ulomega}(\xi_1) \wtilcalF_{+, \ulomega}\bigl[ \calQ_\ulomega\bigl(\ulPe U(s)\bigr) \bigr](\xi_1) \, \ud \xi_1 \biggr)  v_{-,\ulomega}(s,x) \, \ud x \biggr) \, \ud s.
 \end{aligned}
\end{equation*}
In view of the expression for the quadratic nonlinearity \eqref{equ:setup_definition_calQulomega} and the definition of the distorted Fourier transform in \eqref{eqn: Psi_pm_omega}, \eqref{eqn:def-dFT}, we have
\begin{equation} \label{equ:modulation_proof_wtilcalFplus_of_quadratic}
 \begin{aligned}
    &\wtilcalF_{+, \ulomega}\bigl[ \calQ_\ulomega\bigl(\ulPe U(s)\bigr) \bigr](\xi_1) \\
    &\qquad = - \int_\bbR \phi_\ulomega(y) \Bigl( \bigl( \usube^2 + 2 \usube \barusube \bigr)(s,x) \overline{\Psi_{1,\ulomega}(x,\xi_1)} + \bigl( \barusube^2 + 2 \usube \barusube \bigr)(s,x) \overline{\Psi_{2,\ulomega}(y,\xi_1)} \Bigr) \, \ud y.
 \end{aligned}
\end{equation}
Inserting the decomposition \eqref{equ:consequences_usube_leading_order_local_decay_decomp} of $\usube(s,x)$, respectively the decomposition \eqref{equ:consequences_barusube_leading_order_local_decay_decomp} of $\barusube(s,x)$, and inserting the decomposition \eqref{equ:modulation_free_Schrodinger_minus_decomposition} of $v_{-,\ulomega}(s,x)$,  we find that $\calI_{\ulomega}^{2,3}(t;T)$ is a sum of terms of the form
\begin{equation} \label{equ:modulation_proof_quadratic_reinserted_type1}
 \begin{aligned}
  C_{\ulomega, k_1, k_2, \ell} \biggl( \int_t^T h_{j_1,\ulomega}(s) h_{j_2,\ulomega}(s) \tilde{h}_{2,\ulomega}(s) \, \ud s \biggr)
 \end{aligned}
\end{equation}
with $j_1, j_2, k_1, k_2, \ell \in \{1,2\}$ and
\begin{equation*}
 C_{\ulomega, k_1, k_2, \ell} := \int_\bbR q_\ulomega(x) \biggl( \int_\bbR e^{ix\xi_1} \fraka_\ulomega(\xi_1) \int_\bbR \phi_\ulomega(y) \Phi_{k_1,\ulomega}(y)  \Phi_{k_2,\ulomega}(y)  \overline{\Psi_{\ell,\ulomega}(y,\xi_1)} \, \ud y \, \ud \xi_1 \biggr) \, \ud x,
\end{equation*}
and of terms of the form
\begin{equation} \label{equ:modulation_proof_quadratic_reinserted_type2}
 \begin{aligned}
  \int_t^T \biggl( \int_\bbR q_\ulomega(x) \biggl( \int e^{ix\xi_1} \fraka_\ulomega(\xi_1) \int \phi_\ulomega(y) g_1(s,y) g_2(s,y) \overline{\Psi_{\ell,\ulomega}(y,\xi_1)} \, \ud y \, \ud \xi_1 \biggr) g_3(s,x) \, \ud x \biggr) \, \ud s,
 \end{aligned}
\end{equation}
where at least one of the inputs $g_1(s,y)$, $g_2(s,y)$ is given by $R_{u,\ulomega}(s,y)$ or $R_{\baru,\ulomega}(s,y)$, or if not, then $g_3(s,x)$ is given by $R_{v_-,\ulomega}(s,x)$, while the other inputs are arbitray among $h_{j,\ulomega}(s) \Phi_{k,\ulomega}(y)$, $j, k \in \{1,2\}$, $R_{u,\ulomega}(s,y)$, $R_{\baru,\ulomega}(s,y)$ for $g_1(s,y)$, $g_2(s,y)$ and among $\tilde{h}_{2,\ulomega}(s)$, $R_{v_-,\ulomega}(s,x)$ for $g_3(s,x)$.

We begin with the discussion of the contributions of the terms of the form \eqref{equ:modulation_proof_quadratic_reinserted_type1}. First, we point out that by a repeated application of Plancherel's theorem, making use of the tensorized product structure of $\Psi_{\ell,\ulomega}(y,\xi_1)$, see Lemma~\ref{lemma: PDO on m12}, we have 
\begin{equation*}
 \begin{aligned}
 |C_{\ulomega, k_1, k_2, \ell}| &\lesssim \|q_\ulomega\|_{L^2_x} \biggl\| \int_\bbR e^{ix\xi_1} \fraka_\ulomega(\xi_1) \int_\bbR \phi_\ulomega(y) \Phi_{k_1,\ulomega}(y)  \Phi_{k_2,\ulomega}(y)  \overline{\Psi_{\ell,\ulomega}(y,\xi_1)} \, \ud y \, \ud \xi_1 \biggr\|_{L^2_x} \\
 &\lesssim \|q_\ulomega\|_{L^2_x} \|\phi_\ulomega\|_{L^2_y}.
 \end{aligned}
\end{equation*}
Then the idea is to observe that the leading order behavior of $h_{j_1,\ulomega}(s) h_{j_2,\ulomega}(s) \tilde{h}_{2,\ulomega}(s)$ is schematically given by 
\begin{equation*}
 \bigl(s^{-\frac12} e^{\pm i \ulomega s}\bigr)^2 \bigl( s^{-\frac12} e^{is\ulomega} \bigr) = s^{-\frac32} e^{i m \ulomega s}, \quad m \in \{ -1, 1, 3 \}.
\end{equation*}
Thus, there are no time resonances. So we can integrate by parts in time and then infer sufficient decay using the bounds \eqref{equ:consequences_h12_decay}, \eqref{equ:consequences_h12_phase_filtered_decay}, \eqref{equ:preparation_flat_Schrodinger_wave_bound3}, and \eqref{equ:preparation_flat_Schrodinger_wave_bound4}.
Concretely, let us provide the details for the contribution
\begin{equation*}
 \int_t^T h_{1,\ulomega}(s) h_{1,\ulomega}(s) \tilde{h}_{2,\ulomega}(s) \, \ud s.
\end{equation*}
Filtering out the leading order phase and integrating by parts in time, we find
\begin{equation*}
 \begin{aligned}
  &\int_t^T h_{1,\ulomega}(s) h_{1,\ulomega}(s) \tilde{h}_{2,\ulomega}(s) \, \ud s \\
  &= \int_t^T e^{-i\ulomega s} \bigl( e^{i\ulomega s} h_{1,\ulomega}(s) \bigr) \bigl( e^{i\ulomega s} h_{1,\ulomega}(s) \bigr) \bigl( e^{-i\ulomega s} \tilde{h}_{2,\ulomega}(s) \bigr) \, \ud s \\
  &= i \ulomega^{-1} \biggl( \Bigl[ e^{-i\ulomega s} \bigl( e^{i\ulomega s} h_{1,\ulomega}(s) \bigr) \bigl( e^{i\ulomega s} h_{1,\ulomega}(s) \bigr) \bigl( e^{-i\ulomega s} \tilde{h}_{2,\ulomega}(s) \bigr) \Bigr]_{s=t}^{s=T} \\
  &\quad \quad \quad \quad - \int_t^T e^{-i\ulomega s} \ps \bigl( e^{i\ulomega s} h_{1,\ulomega}(s) \bigr) \bigl( e^{i\ulomega s} h_{1,\ulomega}(s) \bigr) \bigl( e^{-i\ulomega s} \tilde{h}_{2,\ulomega}(s) \bigr) \, \ud s  + \bigl\{ \text{similar terms} \bigr\} \biggr).
 \end{aligned}
\end{equation*}
Using the estimates \eqref{equ:consequences_h12_decay}, \eqref{equ:consequences_h12_phase_filtered_decay}, and \eqref{equ:preparation_flat_Schrodinger_wave_bound3}, we obtain the sufficient bounds
\begin{equation*}
 \begin{aligned}
  \biggl| \Bigl[ i \ulomega^{-1} e^{-i\ulomega s} \bigl( e^{i\ulomega s} h_{1,\ulomega}(s) \bigr) \bigl( e^{i\ulomega s} h_{1,\ulomega}(s) \bigr) \bigl( e^{-i\ulomega s} \tilde{h}_{2,\ulomega}(s) \bigr) \Bigr]_{s=t}^{s=T} \biggr| &\lesssim \sup_{t \leq s \leq T} \, |h_{1,\ulomega}(s)|^2 |\tilde{h}_{2,\ulomega}(s)| \lesssim \varepsilon^3 \jt^{-\frac32},
 \end{aligned}
\end{equation*}
as well as
\begin{equation*}
 \begin{aligned}
  \biggl| \int_t^T e^{-i\ulomega s} \ps \bigl( e^{i\ulomega s} h_{1,\ulomega}(s) \bigr) \bigl( e^{i\ulomega s} h_{1,\ulomega}(s) \bigr) \bigl( e^{-i\ulomega s} \tilde{h}_{2,\ulomega}(s) \bigr) \, \ud s \biggr| &\lesssim \int_t^T \varepsilon \js^{-1+\delta} \cdot \varepsilon^2 \js^{-1} \, \ud s \lesssim \varepsilon^3 \jt^{-1+\delta}.
 \end{aligned}
\end{equation*}
This finishes the discussion of the contribution of the terms of the form \eqref{equ:modulation_proof_quadratic_reinserted_type1}, and it remains to consider the terms of the form \eqref{equ:modulation_proof_quadratic_reinserted_type2}. Concretely, we provide the details for the case where $g_1(s,y)$ is given by $R_{u,\ulomega}(s,y)$, $g_2(s,y)$ by $h_{1,\ulomega}(s) \Phi_{2,\ulomega}(y)$, and $g_3(s,x)$ by $\tilde{h}_{2,\ulomega}(s)$.
Applying the Cauchy-Schwarz inequality and Plancherel's theorem repeatedly, and then using the bounds \eqref{equ:consequences_Ru_local_decay}, \eqref{equ:consequences_h12_decay}, \eqref{equ:preparation_flat_Schrodinger_wave_bound3}, we obtain
\begin{equation*}
 \begin{aligned}
  &\biggl| \int_t^T \biggl( \int_\bbR q_\ulomega(x) \biggl( \int e^{ix\xi_1} \fraka_\ulomega(\xi_1) \int \phi_\ulomega(y) R_{u,\ulomega}(s,y) h_{1,\ulomega}(s) \Phi_{2,\ulomega}(y) \overline{\Psi_{\ell,\ulomega}(y,\xi_1)} \, \ud y \, \ud \xi_1 \biggr) \tilde{h}_{2,\ulomega}(s) \, \ud x \biggr) \, \ud s \biggr| \\
  &\lesssim \int_t^T \|q_\ulomega\|_{L^2_x} \bigl\|\jap{y}^2 \phi_\ulomega\bigr\|_{L^2_y} \bigl\| \jap{y}^{-2} R_{u,\ulomega}(s,y) \bigr\|_{L^\infty_y} \bigl|h_{1,\ulomega}(s)\bigr| \bigl|\tilde{h}_{2,\ulomega}(s)\bigr| \, \ud s \\
  &\lesssim \int_t^T \varepsilon \js^{-1+\delta} \cdot \varepsilon^2 \js^{-1} \, \ud s \lesssim \varepsilon^3 \jt^{-1+\delta}.
 \end{aligned}
\end{equation*}
This is a sufficient bound and finishes the discussion of the term $\calI_{\ulomega}^{2,3}(t;T)$.

Next, we discuss the term $\calI_{\ulomega}^{2,4}(t;T)$ given by
\begin{equation*}
 \begin{aligned}
  \calI_{\ulomega}^{2,4}(t;T)
  &= \int_t^T \biggl( \int_\bbR q_\ulomega(x) \biggl( \int_\bbR e^{ix\xi_1} \fraka_{\ulomega}(\xi_1) \wtilcalF_{+, \ulomega}\bigl[ \calC\bigl(\ulPe U(s)\bigr) \bigr](\xi_1) \, \ud \xi_1 \biggr) v_{-,\ulomega}(s,x) \, \ud x \biggr) \, \ud s.
 \end{aligned}
\end{equation*}
Recall that the multiplier $\fraka_\ulomega(\xi_1)$ is of one of the forms listed in \eqref{equ:modulation_proof_multipliers_types}. We first observe that we may in fact assume that $\fraka_\ulomega(\xi_1)$ is of the form
\begin{equation} \label{equ:modulation_proof_multipliers_reduced_types}
 \frac{1}{(|\xi_1|-i\sqrt{\ulomega})^2} \quad \text{or} \quad \frac{\xi_1}{(|\xi_1|-i\sqrt{\ulomega})^2}.
\end{equation}
If instead $\fraka_\ulomega(\xi_1)$ is of the form $\xi_1^2 (|\xi_1|-i\sqrt{\ulomega})^{-2}$,
we can turn one factor of $\xi_1$ into a spatial derivative with respect to $x$ and then integrate by parts in $x$. If the derivative falls onto the Schwartz function $q_\ulomega(x)$, we are effectively in the above setting. If instead the derivative falls onto $v_{-,\ulomega}(s,x)$, we use its improved local decay \eqref{equ:preparation_flat_Schrodinger_wave_bound2}. In that case, we use Plancherel's theorem and the $L^2$ boundedness of $\wtilcalF_{+, \ulomega}$ by Proposition~\ref{prop:mapping_properties_dist_FT} along with the bounds \eqref{equ:setup_smallness_orbital}, \eqref{equ:consequences_U_disp_decay}, \eqref{equ:preparation_flat_Schrodinger_wave_bound2} to infer
\begin{equation*}
 \begin{aligned}
  &\biggl| \int_t^T \biggl( \int_\bbR q_\ulomega(x) \biggl( \int_\bbR e^{ix\xi_1} \fraka_{\ulomega}(\xi_1) \wtilcalF_{+, \ulomega}\bigl[ \calC\bigl(\ulPe U(s)\bigr) \bigr](\xi_1) \, \ud \xi_1 \biggr) (\px v_{-,\ulomega})(s,x) \, \ud x \biggr) \, \ud s \biggr| \\
  &\lesssim \int_t^T \|\jx^2 q_\ulomega\|_{L^\infty_x} \bigl\| \calC\bigl(\ulPe U(s)\bigr) \bigr\|_{L^2_x} \bigl\| \jx^{-2} (\px v_{-,\ulomega})(s,x) \bigr\|_{L^2_x} \, \ud s \\
  &\lesssim \int_t^T \|U(s)\|_{L^2_x} \|U(s)\|_{L^\infty_x}^2 \bigl\| \jx^{-2} (\px v_{-,\ulomega})(s,x) \bigr\|_{L^2_x} \, \ud s \\
  &\lesssim \int_t^T \varepsilon \cdot \varepsilon^2 \js^{-1} \cdot \varepsilon \js^{-1+\delta} \, \ud s \lesssim \varepsilon^4 \jt^{-1+\delta},
 \end{aligned}
\end{equation*}
as desired.
In what follows, we can therefore assume that $\fraka_\ulomega(\xi_1)$ is of the form \eqref{equ:modulation_proof_multipliers_reduced_types}.
Hence, given another symbol $\frakm \in W^{1,\infty}(\bbR)$, both $\fraka_\ulomega(\xi_1) \frakm(\xi_1)$ and $\partial_{\xi_1} ( \fraka_\ulomega(\xi_1) \frakm(\xi_1) )$ are in $L^2_{\xi_1}$, whence $\jx \widehat{\calF}^{-1}[\fraka_\ulomega(\cdot) \frakm(\cdot)](x)$ is $L^2_x$ integrable, and thus $\widehat{\calF}^{-1}[\fraka_\ulomega(\cdot) \frakm(\cdot)](x)$ is in $L^1_x$.

In view of the structure of the cubic spectral distributions determined in Subsection~\ref{subsec:cubic_spectral_distributions}, the integral
\begin{equation*}
 \int_\bbR e^{ix\xi_1} \fraka_{\ulomega}(\xi_1) \wtilcalF_{+, \ulomega}\bigl[ \calC\bigl(\ulPe U(s)\bigr) \bigr](\xi_1) \, \ud \xi_1
\end{equation*}
in the expression for $\calI_{\ulomega}^{2,4}(t;T)$
can be written as a linear combination of three types of contributions: regular contributions, principal value contributions, and Dirac delta contributions.
To specify these contributions, we denote by $\frakm, \frakn_1, \frakn_2, \frakn_3 \in W^{1,\infty}(\bbR)$ symbols of the form
\begin{equation}
 \frac{1}{(|\eta|-i\sqrt{\omega})^2}, \quad \frac{\eta}{(|\eta|-i\sqrt{\omega})^2}, \quad \text{or} \quad \frac{\eta^2}{(|\eta|-i\sqrt{\omega})^2}.
\end{equation}
Then a regular contribution is of the form
\begin{equation*}
 I_{\mathrm{reg}}(s,x) := \int_\bbR e^{i x \xi_1} \fraka_\ulomega(\xi_1) \frakm(\xi_1) \widehat{\calF}\bigl[ \varphi(\cdot) w_{1,\ulomega}(s,\cdot) w_{2,\ulomega}(s,\cdot) w_{3,\ulomega}(s,\cdot) \bigr](\xi_1) \, \ud \xi_1
\end{equation*}
with $\varphi \in \calS(\bbR)$ a Schwartz function and $w_{j,\ulomega}(s,y)$, $1 \leq j \leq 3$, given by
\begin{equation*}
 \int_\bbR e^{i y \eta} e^{\mp it(\eta^2+\ulomega)} \frakn_j(\eta) \tilde{f}_{\pm,\ulomega}(t,\eta) \, \ud \eta,
\end{equation*}
or complex conjugates thereof.
A principal value contribution is of the form
\begin{equation*}
 I_{\pvdots}(s,x) := \int_\bbR e^{i x \xi_1} \fraka_\ulomega(\xi_1) \frakm(\xi_1) \widehat{\calF}\bigl[ \tanh(\sqrt{\ulomega} \cdot) v_{+, \frakn_1,\ulomega}(s,\cdot) \overline{v_{+, \frakn_2, \ulomega}(s,\cdot)} v_{+, \frakn_3, \ulomega}(s,\cdot) \bigr](\xi_1) \, \ud \xi_1
\end{equation*}
with $v_{+, \frakn_j,\ulomega}(s,y)$, $1 \leq j \leq 3$, given by
\begin{equation*}
 v_{+, \frakn_j,\ulomega}(s,y) := \int_\bbR e^{i y \eta} e^{-it(\eta^2+\ulomega)} \frakn_j(\eta) \tilfplusulo(t,\eta) \, \ud \eta.
\end{equation*}
A Dirac delta contribution is of the form
\begin{equation*}
 I_{\delta_0}(s,x) := \int_\bbR e^{i x \xi_1} \fraka_\ulomega(\xi_1) \frakm(\xi_1) \widehat{\calF}\bigl[ v_{+, \frakn_1,\ulomega}(s,\cdot) \overline{v_{+, \frakn_2, \ulomega}(s,\cdot)} v_{+, \frakn_3, \ulomega}(s,\cdot) \bigr](\xi_1) \, \ud \xi_1
\end{equation*}
with $v_{+, \frakn_j,\ulomega}(s,y)$, $1 \leq j \leq 3$, given by
\begin{equation*}
 v_{+, \frakn_j,\ulomega}(s,y) := \int_\bbR e^{i y \eta} e^{-it(\eta^2+\ulomega)} \frakn_j(\eta) \tilfplusulo(t,\eta) \, \ud \eta.
\end{equation*}
We emphasize that the bounds \eqref{equ:preparation_flat_Schrodinger_wave_bound1}--\eqref{equ:preparation_flat_Schrodinger_wave_bound5} apply to the Schr\"odinger evolutions $w_{j,\ulomega}(s,y)$, $1 \leq j \leq 3$, and $v_{+, \frakn_j,\ulomega}(s,y)$, $1 \leq j \leq 3$.
Then in order to estimate the contribution of $I_{\mathrm{reg}}(s,x)$ to $\calI_{\ulomega}^{2,4}(t;T)$, we use Plancherel's theorem repeatedly along with the bounds \eqref{equ:preparation_flat_Schrodinger_wave_bound1}. This gives the sufficient bound
\begin{equation*}
 \begin{aligned}
    &\biggl| \int_t^T \biggl( \int_\bbR q_\ulomega(x) I_{\mathrm{reg}}(s,x) v_{-,\ulomega}(s,x) \, \ud x \biggr) \, \ud s \biggr| \\
    &\lesssim \int_t^T \|q_\ulomega\|_{L^2_x} \bigl\| \varphi(x) w_{1,\ulomega}(s) w_{2,\ulomega}(s) w_{3,\ulomega}(s) \bigr\|_{L^2_x} \|v_{-,\ulomega}(s,x)\|_{L^\infty_x} \, \ud s \\
    &\lesssim \int_t^T \|q_\ulomega\|_{L^2_x} \|\varphi\|_{L^2_x} \|w_{1,\ulomega}(s)\|_{L^\infty_x} \|w_{2,\ulomega}(s)\|_{L^\infty_x} \|w_{3,\ulomega}(s)\|_{L^\infty_x} \|v_{-,\ulomega}(s,x)\|_{L^\infty_x} \, \ud s \\
    &\lesssim \int_t^T \varepsilon^4 \js^{-2} \, \ud s \lesssim \varepsilon^4 \jt^{-1}.
 \end{aligned}
\end{equation*}
Next, for the contribution of $I_{\pvdots}(s,x)$ to $\calI_{\ulomega}^{2,4}(t;T)$, we use Young's inequality and the bounds \eqref{equ:preparation_flat_Schrodinger_wave_bound1} to infer
\begin{align*}
&\biggl| \int_t^T \biggl( \int_\bbR q_\ulomega(x) I_{\pvdots}(s,x) v_{-,\ulomega}(s,x) \, \ud x \biggr) \, \ud s \biggr| \\
&\lesssim \int_t^T \|q_\ulomega\|_{L^1_x} \Bigl\| \widehat{\calF}^{-1}\Bigl[ \fraka_\ulomega(\xi_1) \frakm(\xi_1) \widehat{\calF}\bigl[ \tanh(\sqrt{\ulomega} \cdot) v_{+, \frakn_1,\ulomega}(s,\cdot) \overline{v_{+, \frakn_2, \ulomega}(s,\cdot)} v_{+, \frakn_3, \ulomega}(s,\cdot)\bigr] \Bigr](x) \Bigr\|_{L^\infty_x}\\
&\hspace{35em} \times \|v_{-,\ulomega}(s,x)\|_{L^\infty_x} \, \ud s \\
&\lesssim \int_t^T \|q_\ulomega\|_{L^1_x} \bigl\| \widehat{\calF}^{-1}\bigl[ \fraka_\ulomega(\cdot) \frakm(\cdot) \bigr] \bigr\|_{L^1_x} \bigl\| \tanh(\sqrt{\ulomega} \cdot) v_{+, \frakn_1,\ulomega}(s,\cdot) \overline{v_{+, \frakn_2, \ulomega}(s,\cdot)} v_{+, \frakn_3, \ulomega}(s,\cdot) \bigr\|_{L^\infty_x} \|v_{-,\ulomega}(s,x)\|_{L^\infty_x} \, \ud s \\
&\lesssim \int_t^T \| v_{+, \frakn_1,\ulomega}(s) \|_{L^\infty_x} \|v_{+, \frakn_2, \ulomega}(s)\|_{L^\infty_x} \|v_{+, \frakn_3, \ulomega}(s)\|_{L^\infty_x} \|v_{-,\ulomega}(s,x)\|_{L^\infty_x} \, \ud s \\
&\lesssim \int_t^T \varepsilon^4 \js^{-2} \, \ud s \lesssim \varepsilon^4 \jt^{-1}.
 \end{align*}
Similarly, we bound the contribution of $I_{\delta_0}(s,x)$ to $\calI_{\ulomega}^{2,4}(t;T)$ using Young's inequality and \eqref{equ:preparation_flat_Schrodinger_wave_bound1},
\begin{equation*}
 \begin{aligned}
    &\biggl| \int_t^T \biggl( \int_\bbR q_\ulomega(x) I_{\delta_0}(s,x) v_{-,\ulomega}(s,x) \, \ud x \biggr) \, \ud s \biggr| \\
    &\lesssim \int_t^T \|q_\ulomega\|_{L^1_x} \bigl\| \widehat{\calF}^{-1}\bigl[ \fraka_\ulomega(\cdot) \frakm(\cdot) \bigr] \bigr\|_{L^1_x} \bigl\| v_{+, \frakn_1,\ulomega}(s,\cdot) \overline{v_{+, \frakn_2, \ulomega}(s,\cdot)} v_{+, \frakn_3, \ulomega}(s,\cdot) \bigr\|_{L^\infty_x} \|v_{-,\ulomega}(s,x)\|_{L^\infty_x} \, \ud s \\
    &\lesssim \int_t^T \| v_{+, \frakn_1,\ulomega}(s) \|_{L^\infty_x} \|v_{+, \frakn_2, \ulomega}(s)\|_{L^\infty_x} \|v_{+, \frakn_3, \ulomega}(s)\|_{L^\infty_x} \|v_{-,\ulomega}(s,x)\|_{L^\infty_x} \, \ud s \\
    &\lesssim \int_t^T \varepsilon^4 \js^{-2} \, \ud s \lesssim \varepsilon^4 \jt^{-1}.
 \end{aligned}
\end{equation*}
This finishes the discussion of the term $\calI_{\ulomega}^{2,4}(t;T)$.

Finally, we write the term $\calI_{\ulomega}^{2,5}(t;T)$ as
\begin{equation*}
 \begin{aligned}
  \calI_{\ulomega}^{2,5}(t;T) = \int_t^T \biggl( \int_\bbR q_\ulomega(x) \biggl( \int_\bbR e^{ix\xi_2} \fraka_{\ulomega}(\xi_1) \wtilcalF_{+, \ulomega}\bigl[ \calMod(s) + \calE_1(s) &+ \calE_2(s) + \calE_3(s) \bigr](\xi_1) \, \ud \xi_1 \biggr) \\
  &\quad \quad \quad \quad \quad \times v_{-,\ulomega}(s,x) \, \ud x \biggr) \, \ud s.
 \end{aligned}
\end{equation*}
Using the Cauchy-Schwarz inequality combined with Plancherel's identity and the $L^2$ boundedness of $\wtilcalF_{+,\ulomega}$ by Proposition~\ref{prop:mapping_properties_dist_FT} along with the estimates \eqref{equ:consequences_ulPe_Mod_bounds}, \eqref{equ:consequences_calE1_bounds}, \eqref{equ:consequences_calE_2and3_bounds}, \eqref{equ:preparation_flat_Schrodinger_wave_bound1} gives the sufficient bound
\begin{equation*}
 \begin{aligned}
  \bigl| \calI_{\ulomega}^{2,5}(t;T) \bigr| &\lesssim \int_t^T \|q_\ulomega\|_{L^2_x} \Bigl( \bigl\| \ulPe \calMod(s) \bigr\|_{L^2_x} + \|\calE_1(s)\|_{L^2_x} + \|\calE_2(s)\|_{L^2_x} + \|\calE_3(s)\|_{L^2_x} \Bigr) \|v_{-,\ulomega}(s)\|_{L^\infty_x} \, \ud s \\
  &\lesssim \int_t^T \varepsilon^2 \js^{-\frac32+\delta} \varepsilon \js^{-\frac12} \, \ud s \lesssim \varepsilon^3 \jt^{-1+\delta}.
 \end{aligned}
\end{equation*}

\medskip
\noindent \underline{Contribution of the non-resonant terms $I_{(a)}(s)$ and $I_{(b)}(s)$.}
Here we outline how to control the contributions of the non-resonant term $I_{(a)}(s)$. The other non-resonant term $I_{(b)}(s)$ can be handled exactly in the same manner.
For $0 \leq t \leq T$ we therefore consider
\begin{equation*}
 \int_t^T \iint e^{-is(\xi_1^2 + \xi_2^2 + 2\ulomega)} \tilfplusulo(s,\xi_1) \tilfplusulo(s,\xi_2) \, \nu_{++,\ulomega}(\xi_1, \xi_2) \, \ud \xi_1 \, \ud \xi_2 \, \ud s.
\end{equation*}
The phase has no time resonances, so we can certainly integrate by parts in time. Remarkably, it turns out that the quadratic spectral distribution $\nu_{++,\ulomega}(\xi_1, \xi_2)$ computed in \eqref{eqn: nu+-} comes with a factor $(\xi_1^2 + \xi_2^2 + 2\ulomega)$, analogously to the structure \eqref{eqn: nu++} of the quadratic spectral distribution $\nu_{+-,\ulomega}(\xi_1, \xi_2)$. The presence of this factor is not strictly necessary for controlling the contribution of the term $I_{(a)}(s)$ since the phase has no time resonances. Nevertheless, as a further inspection of \eqref{eqn: nu+-} shows, it puts us exactly in the setting of the treatment of the resonant term $I_{(c)}(s)$ in the preceding step. We can therefore omit further details.

\medskip
\noindent \underline{Contribution of $II(s)$.}
Again, only the second component of the vectorial term $II(s)$ contributes to $\dot{\omega}(s)$.
In view of
\begin{equation*}
  \bbM_{\ulomega}^{-1} \begin{bmatrix}
                \bigl\langle i \calC\bigl(\ulPe U(s)\bigr), \sigma_2 Y_{1,\ulomega} \bigr\rangle \\
                \bigl\langle i \calC\bigl(\ulPe U(s)\bigr), \sigma_2 Y_{2,\ulomega} \bigr\rangle
    \end{bmatrix}
 =  c_{\ulomega}^{-1} \begin{bmatrix}
                \bigl\langle i \calC\bigl(\ulPe U(s)\bigr), \sigma_2 Y_{2,\ulomega} \bigr\rangle \\
                \bigl\langle i \calC\bigl(\ulPe U(s)\bigr), \sigma_2 Y_{1,\ulomega} \bigr\rangle
    \end{bmatrix},
\end{equation*}
the contribution of $II(s)$ to the integral on the right-hand side of \eqref{equ:modulation_proof_writing_out_omega_difference} is given by
\begin{equation*}
 \begin{aligned}
  -c_{\ulomega}^{-1} \int_t^T \bigl\langle i \calC\bigl(\ulPe U(s)\bigr), \sigma_2 Y_{1,\ulomega} \bigr\rangle \, \ud s = -i c_{\ulomega}^{-1} \int_t^T \bigl\langle \bigl( \usube \barusube \usube - \barusube \usube \barusube \bigr)(s), \phi_{\ulomega} \bigr\rangle \, \ud s.
 \end{aligned}
\end{equation*}
Inserting the decomposition \eqref{equ:consequences_usube_leading_order_local_decay_decomp} of $\usube(s,x)$, respectively the decomposition \eqref{equ:consequences_barusube_leading_order_local_decay_decomp} of $\barusube(s,x)$, we find that $\bigl\langle \bigl( \usube \barusube \usube - \barusube \usube \barusube \bigr)(s), \phi_{\ulomega} \bigr\rangle$ is a sum of terms of the form
\begin{equation} \label{equ:modulation_proof_cubic_terms_type1}
 h_{j_1,\ulomega}(s) h_{j_2,\ulomega}(s) h_{j_3,\ulomega}(s) \bigl\langle \Phi_{k_1,\ulomega} \Phi_{k_2,\ulomega} \Phi_{k_3,\ulomega}, \phi_{\ulomega} \bigr\rangle, \quad j_1, j_2, j_3, k_1, k_2, k_3 \in \{1,2\},
\end{equation}
and of terms of the form
\begin{equation} \label{equ:modulation_proof_cubic_terms_type2}
 \bigl\langle g_1(s,\cdot) g_2(s,\cdot) g_3(s,\cdot), \phi_{\ulomega}(\cdot) \bigr\rangle,
\end{equation}
where at least one of the inputs $g_1(s,x)$, $g_2(s,x)$, or $g_3(s,x)$ is given by $R_{u,\ulomega}(s,x)$ or $R_{\baru,\ulomega}(s,x)$, while the other inputs are arbitrary among $h_{j,\ulomega}(s) \Phi_{k,\ulomega}(x)$, $j, k \in \{1,2\}$, and $R_{u,\ulomega}(s,x)$, $R_{\baru,\ulomega}(s,x)$.

We first discuss the contributions of the terms of the form \eqref{equ:modulation_proof_cubic_terms_type1} to the integral on the right-hand side of \eqref{equ:modulation_proof_writing_out_omega_difference}.
Since the terms $\langle \Phi_{k_1,\ulomega} \Phi_{k_2,\ulomega} \Phi_{k_3,\ulomega}, \phi_{\ulomega} \rangle$ with $k_1, k_2, k_3 \in \{1, 2\}$ are time-independent, the idea is to observe that the leading order behavior of $h_{j_1,\ulomega}(s) h_{j_2,\ulomega}(s) h_{j_3,\ulomega}(s)$ is schematically given by 
\begin{equation*}
 \bigl(s^{-\frac12} e^{\pm i \ulomega s}\bigr)^3 = s^{-\frac32} e^{i m \ulomega s}, \quad m \in \{ -3, -1, 1, 3 \}.
\end{equation*}
Thus, there are no time resonances. So we can integrate by parts in time and then infer sufficient decay using the bounds \eqref{equ:consequences_h12_decay}, \eqref{equ:consequences_h12_phase_filtered_decay}.
Concretely, let us provide the details for the contribution
\begin{equation*}
 \int_t^T h_{1,\ulomega}(s) h_{2,\ulomega}(s) h_{1,\ulomega}(s) \, \ud s.
\end{equation*}
Filtering out the leading order phase and integrating by parts in time, we find
\begin{equation*}
 \begin{aligned}
  &\int_t^T h_{1,\ulomega}(s) h_{2,\ulomega}(s) h_{1,\ulomega}(s) \, \ud s \\
  &= \int_t^T e^{-i\ulomega s} \bigl( e^{i\ulomega s} h_{1,\ulomega}(s) \bigr) \bigl( e^{-i\ulomega s} h_{2,\ulomega}(s) \bigr) \bigl( e^{i\ulomega s} h_{1,\ulomega}(s) \bigr) \, \ud s \\
  &= i\ulomega^{-1} \biggl( \Bigl[ e^{-i\ulomega s} \bigl( e^{i\ulomega s} h_{1,\ulomega}(s) \bigr) \bigl( e^{-i\ulomega s} h_{2,\ulomega}(s) \bigr) \bigl( e^{i\ulomega s} h_{1,\ulomega}(s) \bigr) \Bigr]_{s=t}^{s=T} \\
  &\qquad \qquad - \int_t^T \ps\bigl( e^{i\ulomega s} h_{1,\ulomega}(s) \bigr) \bigl( e^{-i\ulomega s} h_{2,\ulomega}(s) \bigr) \bigl( e^{i\ulomega s} h_{1,\ulomega}(s) \bigr) \, \ud s + \bigl\{ \text{similar terms} \bigr\} \biggr).
 \end{aligned}
\end{equation*}
Hence, invoking the estimates \eqref{equ:consequences_h12_decay}, \eqref{equ:consequences_h12_phase_filtered_decay} we arrive at the desired bound
\begin{equation*}
 \begin{aligned}
  \biggl| \int_t^T h_{1,\ulomega}(s) h_{2,\ulomega}(s) h_{1,\ulomega}(s) \, \ud s \biggr| &\lesssim \varepsilon^3 \jt^{-\frac32} + \int_t^T \varepsilon \js^{-1+\delta} \cdot \varepsilon^2 \js^{-1} \, \ud s \lesssim \varepsilon^3 \jt^{-1+\delta}.
 \end{aligned}
\end{equation*}

Now we turn to the contribution of the milder terms of the form \eqref{equ:modulation_proof_cubic_terms_type2} to the integral on the right-hand side of \eqref{equ:modulation_proof_writing_out_omega_difference}. Thanks to the spatial localization provided by $\phi_{\ulomega}(x)$ their contributions can all be estimated using \eqref{equ:consequences_h12_decay}, \eqref{equ:consequences_Ru_local_decay} by
\begin{equation*}
\begin{aligned}
 &\int_t^T \bigl| \bigl\langle g_1(s,\cdot) g_2(s,\cdot) g_3(s,\cdot), \phi_{\ulomega}(\cdot) \bigr\rangle \bigr| \, \ud s \\
 &\lesssim \int_t^T \Bigl( |h_{1,\ulomega}(s)| + |h_{2,\ulomega}(s)| + \bigl\| \jx^{-2} R_{u,\ulomega}(s,x) \bigr\|_{L^\infty_x} + \bigl\| \jx^{-2} R_{\baru,\ulomega}(s,x) \bigr\|_{L^\infty_x} \Bigr)^2 \\
 &\qquad \qquad \qquad \qquad \qquad \qquad \qquad \times \Bigl( \bigl\| \jx^{-2} R_{u,\ulomega}(s,x) \bigr\|_{L^\infty_x} + \bigl\| \jx^{-2} R_{\baru,\ulomega}(s,x) \bigr\|_{L^\infty_x} \Bigr) \, \ud s \\
 &\lesssim \int_t^T \varepsilon^2 \js^{-1} \cdot \varepsilon \js^{-1+\delta} \, \ud s \lesssim \varepsilon^3 \jt^{-1+\delta},
\end{aligned}
\end{equation*}
as desired.

\medskip
\noindent \underline{Contribution of $III(s)$.}
This case can be treated similarly to the contributions of $II(s)$. We describe schematically how to proceed, and leave the details to the reader.
Recall that
\begin{equation*}
    III(s) := - \mathbb{M}_{\ulomega}^{-1}\mathbb{A}_\ulomega(s)\mathbb{M}_{\ulomega}^{-1} \begin{bmatrix}
                \langle i \calQ_{\ulomega}\bigl(\ulPe U(s)\bigr),\sigma_2 Y_{1,\ulomega} \rangle \\
                \langle i \calQ_{\ulomega}\bigl(\ulPe U(s)\bigr),\sigma_2 Y_{2,\ulomega} \rangle
                \end{bmatrix}.
\end{equation*}
In view of the definition of $\bbA_\ulomega(s)$ and $\calQ_{\ulomega}\bigl(\ulPe U(s)\bigr)$, the contributions of the (second component) of $III(s)$ to the integral on the right-hand side of \eqref{equ:modulation_proof_writing_out_omega_difference} is given by a sum of terms of the form
\begin{equation*}
  \langle g_1(s,\cdot), q_1(\cdot) \rangle \langle \phi_\ulomega(\cdot) g_2(s,\cdot) g_3(s,\cdot), q_2(\cdot) \rangle,
\end{equation*}
where $q_1(x)$, $q_2(x)$ are spatially localized and where the inputs $g_1(s,x)$, $g_2(s,x)$, $g_3(s,x)$ are given by $\usube(s,x)$ or $\barusube(s,x)$.
In the preceding line the term $\langle g_1(s,\cdot), q_1(\cdot) \rangle$ stems from one of the entries of $\bbA_\ulomega(s)$, while $\langle \phi_\ulomega(\cdot) g_2(s,\cdot) g_3(s,\cdot), q_2(\cdot) \rangle$ is a component of one of the entries $\langle i \calQ_{\ulomega}\bigl(\ulPe U(s)\bigr),\sigma_2 Y_{j,\ulomega} \rangle$, $1 \leq j \leq 2$.
Upon inserting the decompositions \eqref{equ:consequences_usube_leading_order_local_decay_decomp} for $\usube(s,x)$, respectively \eqref{equ:consequences_barusube_leading_order_local_decay_decomp} for $\barusube(s,x)$, it is clear that estimating the contributions of $III(s)$ effectively reduces to the two cases carried out in detail for $II(s)$.

\medskip
\noindent \underline{Contributions of $IV(s)$, $V(s)$, $VI(s)$, and $VII(s)$.}
Using \eqref{equ:modulation_proof_At_op_norm_bound} and \eqref{equ:consequences_ulPe_U_disp_decay}, we have
\begin{equation*}
 |IV(s)| \lesssim_{\omega_0} \bigl\|\bbA_{\ulomega}(s)\bigr\| \|(\ulPe U)(s)\|_{L^\infty_x}^3 \lesssim_{\omega_0} \varepsilon^4 \js^{-2}.
\end{equation*}
Similarly, using \eqref{equ:modulation_proof_Dt_op_norm_bound} and \eqref{equ:consequences_ulPe_U_disp_decay}, we find that
\begin{equation*}
 |V(s)| \lesssim_{\omega_0} \bigl\| \bbD(s) \bigr\| \|(\ulPe U)(s)\|_{L^\infty_x}^2 \lesssim_{\omega_0} \varepsilon^3 \js^{-2+\delta},
\end{equation*}
as well as
\begin{equation*}
 |VI(s)| \lesssim_{\omega_0} \bigl\| \bbD(s) \bigr\| \|(\ulPe U)(s)\|_{L^\infty_x}^3 \lesssim_{\omega_0} \varepsilon^4 \js^{-\frac52+\delta}.
\end{equation*}
Finally, \eqref{equ:modulation_proof_crude_Mt_op_norm_bound} and \eqref{equ:modulation_proof_remainder_nonlinearity_decay} imply
\begin{equation*}
 |VII(s)| \lesssim_{\omega_0} \bigl\| [\bbM(s)]^{-1} \bigr\| |H(s)| \lesssim_{\omega_0} \varepsilon^3 \js^{-2+\delta}.
\end{equation*}
It follows that for $0 \leq t \leq T$,
\begin{equation*}
 \int_t^T \bigl( |IV(s)| + |V(s)| + |VI(s)| + |VII(s)| \bigr) \, \ud s \lesssim_{\omega_0} \int_t^T \varepsilon^3 \js^{-2+\delta} \, \ud s \lesssim_{\omega_0} \varepsilon^3 \jt^{-1+\delta},
\end{equation*}
which is an acceptable contribution to \eqref{equ:modulation_proof_writing_out_omega_difference}.

\medskip

\noindent \underline{Conclusion.}
By combining all of the preceding estimates, we arrive at the bound \eqref{equ:modulation_proof_goal}, which finishes the proof of Proposition~\ref{prop:modulation_parameters}.
\end{proof}

\section{Weighted Energy Estimates} \label{sec:energy_estimates}

In this section we obtain slowly growing weighted energy estimates for the distorted Fourier transform of the profile.

\subsection{Preparations}
In the next few lemmas, we collect several estimates that will be used frequently in the derivation of the weighted energy estimates. We begin with a local decay estimate for the spatial derivative of the radiation term. 
\begin{lemma}
Suppose the assumptions in the statement of Proposition~\ref{prop:profile_bounds} are in place. Then we have 
\begin{equation}\label{equ:consequences_px_U_local_decay}
	\sup_{0 \leq t \leq T} \jt^{\frac{1}{2}} \Vert \jx^{-3}\partial_x U (t)\Vert_{L_x^2}\lesssim \varepsilon,
\end{equation}
\end{lemma}
\begin{proof}
We recall from \eqref{equ:setup_decomposition_radiation} the decomposition
\begin{equation*}
    U(t,x) = (\ulPe U)(t,x) + d_{1,\ulomega}(t) Y_{1,\ulomega}(x) + d_{2,\ulomega}(t) Y_{2,\ulomega}(x),
\end{equation*}
where $\ulPe U = (\usube,\barusube)^\top$. Using \eqref{equ:consequences_h12_decay}, \eqref{equ:consequences_px_Ru_local_decay}, and the fact that $\partial_x \Phi_{1,\ulomega}(x)$ and $\partial_x \Phi_{2,\ulomega}(x)$ are bounded, we find that 
\begin{equation}\label{equ:consequences_px_usube_local_decay}
	\sup_{0 \leq t \leq T} \jt^{\frac{1}{2}}\big(\Vert \jx^{-3}\partial_x \usube (t)\Vert_{L_x^2}+\Vert \jx^{-3}\partial_x \barusube (t)\Vert_{L_x^2}\big) \lesssim \varepsilon,
\end{equation}
Since $Y_{1,\ulomega}(x)$ and $Y_{2,\ulomega}(x)$ are Schwartz functions, we obtain the asserted bound by combining the estimates \eqref{equ:consequences_discrete_components_decay} and \eqref{equ:consequences_px_usube_local_decay}.
\end{proof}

Next, we prove pointwise and weighted energy bounds for the terms $\wtilB_{ \ulomega}(t,\xi)$ and $\wtilcalR_{q,\ulomega}(t,\xi)$ that appear in the renormalized profile equation \eqref{equ:setup_evol_equ_renormalized_tilfplus}.

\begin{lemma}\label{lemma:preparatory_lemma_on_wtilB} Let $\wtilB_{ \ulomega}(t,\xi)$ and $\wtilcalR_{q,\ulomega}(t,\xi)$ be given by \eqref{equ:setup_definition_wtilBulomega} and \eqref{equ:setup_definition_wtilcalR_qulomega} respectively. 
 Suppose the assumptions in the statement of Proposition~\ref{prop:profile_bounds} are in place. Then we have
\begin{align}
\sup_{0 \leq t \leq T} \Vert \pxi \wtilB_{ \ulomega}(t,\xi)\Vert_{L_\xi^2} &\lesssim \varepsilon^2,   \label{eqn:weighted_estimate_bilinear}\\
\sup_{0 \leq t \leq T} \jt \Vert \wtilB_{ \ulomega}(t,\xi)\Vert_{L_\xi^\infty} &\lesssim \varepsilon^2 ,\label{eqn:pointwise_estimate_bilinear}\\
\sup_{0 \leq t \leq T} \jt^{\frac{3}{2}-\delta} \Vert \pxi \wtilcalR_{q, \ulomega}(t,\xi)\Vert_{L_\xi^2} &\lesssim \varepsilon^2,   \label{eqn:weighted_estimate_bilinear2}\\
\sup_{0 \leq t \leq T} \jt^{\frac{3}{2}-\delta} \Vert \wtilcalR_{q,\ulomega}(t,\xi)\Vert_{L_\xi^\infty} &\lesssim \varepsilon^2. \label{eqn:pointwise_estimate_bilinear2}
\end{align}
 
\end{lemma}
\begin{proof}
	From its definition, $\wtilB_{ \ulomega}(t,\xi)$ is a sum of terms of the form $		e^{it(\xi^2+\ulomega)}\frakc(t)\frakd(\xi)$,	where $\frakc(t)\frakd(\xi)$ is of the form 
	\begin{equation*}
		\big(h_{1,\ulomega}(t)\big)^2(\xi^2-\ulomega)^{-1}\wtilQ_{1,\ulomega}(\xi), \quad \big(h_{1,\ulomega}(t)\big)	\big(h_{2,\ulomega}(t)\big)(\xi^2+\ulomega)^{-1}\wtilQ_{2,\ulomega}(\xi)	, \quad \text{or} \quad 	\big(h_{2,\ulomega}(t)\big)^2(\xi^2+3\ulomega)^{-1}\wtilQ_{3,\ulomega}(\xi).
	\end{equation*}

In view of Lemma~\ref{lem:null_structure_radiation} and Remark~\ref{rem:setup_decay_calQ_coefficients}, we may place $\frakd(\xi)$ and its derivative in $L_\xi^2$ or in $L_\xi^\infty$ with bounds $\lesssim 1$. Using the bound \eqref{equ:consequences_h12_decay} from Corollary~\ref{cor:consequences}, we find that 
\begin{equation*}
	\vert \frakc(t)\vert \lesssim \varepsilon^2 \jt^{-1}.
\end{equation*}
Thus, the bounds \eqref{eqn:weighted_estimate_bilinear}--\eqref{eqn:pointwise_estimate_bilinear} follow.

Next, from the definition of $\wtilcalR_{q,\ulomega}$, we find that it is a linear combination of terms of the form $e^{imt\ulomega}\frake(t)\frakf(t)\frakd(\xi)$, $m \in \{-1,+1\}$
where $\frake(t)$ is either $h_{1,\ulomega}(t)$ or $h_{2,\ulomega}(t)$, $\frakf(t)$ is either $\partial_t(e^{it\ulomega}h_{1,\ulomega}(t))$ or $\partial_t(e^{-it\ulomega}h_{2,\ulomega}(t))$, and $\frakd(\xi)$ is exactly the same as before. Hence, the estimates \eqref{eqn:weighted_estimate_bilinear2} and \eqref{eqn:pointwise_estimate_bilinear2} follow directly from  \eqref{equ:consequences_h12_decay}, \eqref{equ:consequences_h12_phase_filtered_decay}, and the regularity of $\frakd(\xi)$.
\end{proof}

The following lemma provides cubic-type decay estimates for the spatial derivative of the remainder terms in \eqref{equ:setup_evol_equ_renormalized_tilfplus}.

\begin{lemma}  Suppose the assumptions in the statement of Proposition~\ref{prop:profile_bounds} are in place. Then we have
\begin{equation}
\sup_{0 \leq t \leq T} \jt^{\frac{3}{2}-\delta}   \Vert \calQ_{\mathrm{r},\ulomega}\big((\ulPe U)(t)\big) \Vert_{L_x^{2,1}} \lesssim \varepsilon^2, \label{equ:consequences_renormalized_quadratic}
\end{equation}
as well as 
\begin{align}
\sup_{0 \leq t \leq T} \jt^{\frac{3}{2}-\delta}   \Vert \px \calQ_{\mathrm{r},\ulomega}\big((\ulPe U)(t)\big) \Vert_{L_x^{2,1}} &\lesssim \varepsilon^2, \label{equ:consequences_px_Qr_L21bound} \\
\sup_{0 \leq t \leq T} \jt^{2-\delta} \Vert \px \ulPe \calMod(t)\Vert_{L_x^{2,1}}&\lesssim \varepsilon^2, \label{equ:consequences_px_PeMod}\\
\sup_{0 \leq t \leq T} \jt^{\frac{3}{2}-\delta} \Vert \px \calE_1(t)\Vert_{L_x^{2,1}}&\lesssim \varepsilon^2, \label{equ:consequences_px_calE1_L21bound}\\
\sup_{0 \leq t \leq T} \jt^{2-\delta} \big(\Vert \px \calE_2(t)\Vert_{L_x^{2,1}}+\Vert \px \calE_3(t)\Vert_{L_x^{2,1}}\big)&\lesssim \varepsilon^2.\label{equ:consequences_px_calE23_L21bound}
\end{align}
\end{lemma}
\begin{proof}
We start with the proof of \eqref{equ:consequences_renormalized_quadratic} and of \eqref{equ:consequences_px_Qr_L21bound}. From \eqref{eqn:def_renormalized_quadratic}, we observe that the components of the vector $\calQ_{\mathrm{r},\ulomega}\big((\ulPe U)(t)\big)$ are sums of terms of the form
\begin{equation*}
    \phi_{\ulomega}(x) g_1(t,x) g_2(t,x),
\end{equation*}
where the input $g_1(t,x)$ is given by $R_{u,\ulomega}(t,x)$ or $R_{\baru,\ulomega}(t,x)$ and the other input $g_2(t,x)$ is given by either $h_{1,\ulomega}(t)\Phi_{2,\ulomega}(x)$, $h_{2,\ulomega}(t)\Phi_{1,\ulomega}(x)$, $R_{u,\ulomega}(t,x)$ or $R_{\baru,\ulomega}(t,x)$. Since $\phi_{\ulomega}(x)$ is Schwartz and since $\Phi_{j,\ulomega}(x)$, $j \in \{1,2\}$, are smooth functions, the bounds \eqref{equ:consequences_renormalized_quadratic} and \eqref{equ:consequences_px_Qr_L21bound} can be inferred from the estimates \eqref{equ:consequences_h12_decay}, \eqref{equ:consequences_Ru_local_decay}, or \eqref{equ:consequences_px_Ru_local_decay}.

Next, we establish \eqref{equ:consequences_px_PeMod}. 
We recall specifically from the proof of \eqref{equ:consequences_ulPe_Mod_bounds} that we have 
\begin{equation*}
  \ulPe \calMod(t) = -i ( \dot{\gamma}(t) - \omega(t) ) \ulPe \bigl( Z_1(t,x) \bigr) - i \dot{\omega}(t) \ulPe \bigl( Z_2(t,x) \bigr),
\end{equation*}
with short-hand notation $Z_j(t,x) := Y_{j,\omega(t)}(x) - Y_{j,\ulomega}(x)$, $j \in \{1,2\}$. In view of \eqref{equ:spectral_decomposition_L2L2_even} and the fact that the eigenfunctions are Schwartz functions, we obtain by the Cauchy-Schwarz inequality, \eqref{equ:consequences_aux_bound_modulation1}, and \eqref{equ:consequences_aux_bound_modulation2} that 
\begin{equation*}
\begin{split}
    \Vert \px \ulPe \calMod(t) \Vert_{L_x^{2,1}} &\lesssim (\vert \dot{\gamma}(t)-\omega(t)\vert + \vert \dot{\omega}(t)\vert) \big( \Vert \px Z_1(t)\Vert_{L_x^{2,1}} +\Vert \px Z_2(t)\Vert_{L_x^{2,1}} \big)\\
    &\lesssim (\vert \dot{\gamma}(t)-\omega(t)\vert + \vert \dot{\omega}(t)\vert)\vert \omega(t) - \ulomega\vert \lesssim \varepsilon^2 \jt^{-2+\delta}. 
\end{split}
\end{equation*}

Now we prove \eqref{equ:consequences_px_calE1_L21bound} and \eqref{equ:consequences_px_calE23_L21bound}. Using \eqref{equ:consequences_assumption1}, \eqref{equ:consequences_U_disp_decay}, \eqref{equ:consequences_px_U_local_decay}, and the spatial localization of $\phi_\omega(x)$, we deduce from the definition \eqref{equ:setup_definition_calE1} of $\calE_1(t)$ that
\begin{equation*}
\Vert \px \calE_{1}(t)\Vert_{L_x^{2,1}} \lesssim \vert \omega(t) - \ulomega \vert \big(\Vert U(t) \Vert_{L_x^\infty} +\Vert \jx^{-3} \partial_x U \Vert_{L_x^2} \big) \lesssim \varepsilon\jt^{-1+\delta} \cdot \varepsilon \jt^{-\frac{1}{2}} \lesssim \varepsilon^2 \jt^{-\frac{3}{2}+\delta}.
\end{equation*}
Similarly, we infer from the definition \eqref{equ:setup_definition_calE2} of $\calE_2(t)$ that
\begin{equation*}
\Vert \px \calE_{2}(t)\Vert_{L_x^{2,1}} \lesssim \vert \omega(t) - \ulomega \vert \big(\Vert U(t) \Vert_{L_x^\infty} +\Vert \jx^{-3} \partial_x U \Vert_{L_x^2} \big)^2 \lesssim \varepsilon\jt^{-1+\delta} \cdot \varepsilon^2 \jt^{-1} \lesssim \varepsilon^3 \jt^{-2+\delta}.
\end{equation*}
We recall from the discussion of the proof of \eqref{equ:consequences_calE_2and3_bounds} for $\calE_3(t)$ that $\calE_3(t)$ is a sum of nonlinear terms, each of which has at least one input given by $d_{1,\ulomega}(t)Y_{1,\ulomega}(x)$ or $d_{2,\ulomega}(t)Y_{2,\ulomega}(x)$. Using again that the eigenfunctions $Y_{1,\ulomega}(x)$ and $Y_{2,\ulomega}(x)$ are Schwartz, we obtain the asserted bound \eqref{equ:consequences_px_calE23_L21bound} for $\calE_3(t)$ by \eqref{equ:consequences_ulPe_U_disp_decay},  \eqref{equ:consequences_discrete_components_decay}, and \eqref{equ:consequences_px_Ru_local_decay}.
\end{proof}

\subsection{Main weighted energy estimates}
We are now prepared for the proofs of the main weighted energy estimates for the distorted Fourier transform of the profile. 
\begin{proposition}\label{prop: weighted-energy-estimate}
 Suppose the assumptions in the statement of Proposition~\ref{prop:profile_bounds} are in place. Then we have for all $0 \leq t \leq T$ that
 \begin{equation*}
  \bigl\| \pxi \tilf_{+, \ulomega}(t,\xi) \bigr\|_{L^2_\xi} + \bigl\| \pxi \tilf_{-, \ulomega}(t,\xi) \bigr) \bigr\|_{L^2_\xi} \lesssim \varepsilon + \varepsilon^2 \jt^{\delta}.
 \end{equation*}
\end{proposition}

\begin{proof}
First, we observe that by \eqref{equ:setup_distFT_components_relation},
 \begin{equation*}
  \tilf_{-, \ulomega}(t,\xi) = - \frac{(|\xi| - i\sqrt{\ulomega})^2}{(|\xi| + i\sqrt{\ulomega})^2} \overline{\tilf_{+, \ulomega}(t,-\xi)},
 \end{equation*}
 whence
 \begin{equation*}
  \bigl\| \pxi \tilf_{-, \ulomega}(t,\xi) \bigr) \bigr\|_{L^2_\xi} \lesssim \bigl\| \tilf_{+,\ulomega}(t,\xi) \bigr\|_{L^2_\xi} + \bigl\| \pxi \tilf_{+, \ulomega}(t,\xi) \bigr) \bigr\|_{L^2_\xi}.
 \end{equation*}
 By the stability bound \eqref{equ:setup_smallness_orbital}, the mapping properties of the Fourier transform from Proposition~\ref{prop:mapping_properties_dist_FT}, and the $L^2_x$ boundedness of the projection $\ulPe$ from Lemma~\ref{lemma: L2 decomposition}, we have 
 \begin{equation*}
     \bigl\| \tilf_{+,\ulomega}(t,\xi) \bigr\|_{L^2_\xi} \leq C \|u(t)\|_{L^2_x} \leq C \varepsilon.
 \end{equation*}
Although the boundedness of these operators depend on the reference value $\ulomega$, the constant $C$ above only depends on $\omega_0$ since $\ulomega \in (\frac{1}{2}\omega_0,2\omega_0)$ by \eqref{equ:setup_comparison_estimate}, which is independent of the value of the constant $C_0$ in the statement of Proposition~\ref{prop:profile_bounds}.

In what follows it therefore suffices to show for all $0 \leq t \leq T$ that
 \begin{equation*}
  \bigl\| \pxi \tilf_{+, \ulomega}(t,\xi) \bigr) \bigr\|_{L^2_\xi} \lesssim \varepsilon + \varepsilon^2 \jt^{\delta}.
 \end{equation*}
Integrating the normalized evolution equation \eqref{equ:setup_evol_equ_renormalized_tilfplus}, we obtain for $0 \leq t \leq T$,
\begin{equation*}
	\begin{split}
		&e^{i\theta(t)}\big( \tilfplusulo(t,\xi) + \wtilB_\ulomega(t,\xi)\big) - \big(\tilfplusulo(0,\xi) + \wtilB_\ulomega(0,\xi)\big)\\
		&= i\int_0^t e^{i\theta(s)}(\dot{\gamma}(s)-\ulomega)\wtilB_{\ulomega}(s,\xi) \, \ud s - i\int_0^t  e^{i\theta(s)} e^{is(\xi^2+\ulomega)}\wtilcalR_\ulomega(s,\xi) \, \ud s\\
		&\quad -i\int_0^t  e^{i\theta(s)} e^{is(\xi^2+\ulomega)} \wtilcalF_{+, \ulomega}[\calC\big((\ulPe U)(s)\big)](\xi) \, \ud s\\
		&=: \calJ_\ulomega^{1}(t,\xi)+\calJ_\ulomega^{2}(t,\xi)+\calJ_\ulomega^{3}(t,\xi).
	\end{split}
\end{equation*}
By taking derivatives in $\xi$, we find that 
\begin{equation*}
  \begin{split}
  \bigl\| \pxi \tilf_{+, \ulomega}(t,\xi)  \bigr\|_{L^2_\xi} &\lesssim   \bigl\| \pxi \tilf_{+, \ulomega}(0,\xi)  \bigr\|_{L^2_\xi} +  \bigl\| \pxi \wtilB_{ \ulomega}(0,\xi)  \bigr\|_{L^2_\xi} + \bigl\| \pxi \wtilB_{ \ulomega}(t,\xi)  \bigr\|_{L^2_\xi} \\
&\quad +  \bigl\| \pxi \calJ_{\ulomega}^{1}(t,\xi)  \bigr\|_{L^2_\xi} +  \bigl\| \pxi \calJ_{\ulomega}^{2}(t,\xi)  \bigr\|_{L^2_\xi} +  \bigl\| \pxi \calJ_{\ulomega}^{3}(t,\xi)  \bigr\|_{L^2_\xi}.   	
  \end{split}
\end{equation*}
The first four terms on the right-hand side of the preceding expression are relatively straightforward to estimate. For the profile at the initial time, the mapping properties of the distorted Fourier transform and the $L_x^{2,1}$ boundedness of $\ulPe$  ensure that 
\begin{equation*}
	\Vert \partial_\xi \tilfplusulo(0,\xi)\Vert_{L_\xi^2} \leq C\Vert u_0 \Vert_{L_x^{2,1}} \leq C\varepsilon,
\end{equation*}
where the constant $C$ only depends on $\omega_0$. Using the bound \eqref{eqn:weighted_estimate_bilinear}, we have 
\begin{equation*}
\bigl\| \pxi \wtilB_{ \ulomega}(0,\xi)  \bigr\|_{L^2_\xi} + \bigl\| \pxi \wtilB_{ \ulomega}(t,\xi)  \bigr\|_{L^2_\xi} \lesssim \sup_{0 \leq t \leq T} \bigl\Vert \pxi \wtilB_{ \ulomega}(t,\xi) \bigr\Vert_{L_\xi^2} \lesssim \varepsilon^2.
\end{equation*}
In addition, using the decay estimate \eqref{equ:consequences_aux_bound_modulation2} from Corollary~\ref{cor:consequences} for the modulation parameter, we find that 
\begin{equation*}
\bigl\| \pxi \calJ_{\ulomega}^{1}(t,\cdot)  \bigr\|_{L^2_\xi} \lesssim \int_0^t \vert \dot{\gamma}(s)-\ulomega\vert  \Vert \pxi \wtilB_{ \ulomega}(s,\cdot)\Vert_{L_\xi^2} \,\ud s \lesssim \varepsilon^3 \int_0^t \js^{-1+\delta} \,\ud s \lesssim \varepsilon^3 \jt^{\delta}.
\end{equation*}

The weighted energy estimates for the remaining terms $\calJ_{\ulomega}^2(t,\xi)$ and $\calJ_{\ulomega}^3(t,\xi)$ will occupy the rest of the proof. We note that each of these terms carries a phase $e^{is(\xi^2+\ulomega)}$ and that a factor of $2i s \xi$ is produced whenever the $\xi$ derivative hits this phase. The growing factor of $s$ can be problematic but the additional factor of $\xi$ actually provides a local smoothing effect, at least for the low frequencies. In the high frequencies, we turn the multiplication by $\xi$ into a spatial derivative through the distorted (or flat) Fourier transform, and we are then able to apply local smoothing estimates. This observation will be key for analyzing the weighted energy contributions of the remainder term $\calJ_{\ulomega}^2(t,\xi)$ and of the regular cubic terms in $\calJ_{\ulomega}^3(t,\xi)$. For the analysis of the  singular cubic terms, we will further exploit the structure of the cubic spectral distributions 
given by Lemma~\ref{lemma:cubic_NSD}.

\medskip
\noindent \underline{Contribution of the remainder term $\calJ_{\ulomega}^2(t,\xi)$:}
By differentiating in $\xi$, we find that 
\begin{equation}\label{eqn:proof_pxi_calJ2}
	\partial_\xi \calJ_{\ulomega}^2(t,\xi) = 2\int_0^t e^{i\theta(s)}e^{is(\xi^2+\ulomega)} s \xi \wtilcalR_\ulomega(s,\xi)\,\ud s - i\int_0^t e^{i\theta(s)}e^{is(\xi^2+\ulomega)}\partial_\xi \wtilcalR_{\ulomega}(s,\xi) \,\ud s.
\end{equation}
For the latter integral in \eqref{eqn:proof_pxi_calJ2}, we insert the definition  \eqref{equ:setup_definition_wtilcalRulomega} of $\wtilcalR_{\ulomega}(s,\xi)$ and write 
\begin{equation}\label{eqn:proof_remainderterms_1_123}
    \begin{split}
 \int_0^t e^{i\theta(s)}e^{is(\xi^2+\ulomega)}\partial_\xi \wtilcalR_{\ulomega}(s,\xi) \,\ud s
 &=\int_0^t e^{i\theta(s)}e^{is(\xi^2+\ulomega)}\partial_\xi \wtilcalR_{q,\ulomega}(s,\xi) \,\ud s\\
 &\quad+\int_0^t e^{i\theta(s)}e^{is(\xi^2+\ulomega)}(\dot{\gamma}(s)-\ulomega) \partial_\xi \calL_{\baru,\ulomega}(s,\xi) \,\ud s\\
  &\quad+\int_0^t e^{i\theta(s)}e^{is(\xi^2+\ulomega)}\partial_\xi \wtilcalF_{+,\ulomega}[F(s,\cdot)](\xi) \,\ud s\\
  &=: \calI_{\mathrm{r}}^{1,1}(t,\xi)+\calI_{\mathrm{r}}^{1,2}(t,\xi)+\calI_{\mathrm{r}}^{1,3}(t,\xi),
    \end{split}
\end{equation}
where 
\begin{equation}
    F(s,x) := \calQ_{\mathrm{r},\ulomega}\big((\ulPe U)(s)\big) + \ulPe\calMod(s) + \calE_1(s) + \calE_2(s) + \calE_3(s).
\end{equation}
Note that we are using \eqref{equ:wtilcalF_applied_to_P} to write $\ulPe \calMod(s)$ instead of $\calMod(s)$ in the above expression.

The terms in \eqref{eqn:proof_remainderterms_1_123} can be crudely bounded as follows. Using \eqref{eqn:weighted_estimate_bilinear2}, we have 
\begin{equation*}
\bigl\Vert \calI_{\mathrm{r}}^{1,1}(t,\cdot) \bigr\Vert_{L_\xi^2} \lesssim \int_0^t \bigl\Vert \partial_\xi \wtilcalR_{q,\ulomega}(s,\xi) \bigr\Vert_{L_\xi^2} \,\ud s \lesssim \int_0^t \varepsilon^2 \js^{-\frac{3}{2}+\delta}\,\ud s \lesssim \varepsilon^2 .
\end{equation*}
Using \eqref{equ:consequences_aux_bound_modulation2} and \eqref{equ:consequences_calL_bounds}, we have 
\begin{equation*}
    \bigl\Vert \calI_{\mathrm{r}}^{1,2}(t,\cdot) \bigr\Vert_{L_\xi^2} \lesssim \int_0^t \vert \dot{\gamma}(s)-\ulomega\vert \bigl\Vert \calL_{\baru,\ulomega}(s,\cdot) \bigr\Vert_{H_\xi^1} \,\ud s \lesssim \int_0^t \varepsilon \js^{-1+\delta} \cdot \varepsilon\js^{-\frac{1}{2}}\,\ud s \lesssim \varepsilon^2 .
\end{equation*}
Using \eqref{equ:consequences_ulPe_Mod_bounds}, \eqref{equ:consequences_calE1_bounds}, \eqref{equ:consequences_calE_2and3_bounds}, \eqref{equ:consequences_renormalized_quadratic}, and the $L_x^{2,1}\rightarrow H_\xi^{1}$ boundedness of the distorted Fourier transform, we find that 
\begin{equation*}
    \bigl\Vert \calI_{\mathrm{r}}^{1,3}(t,\cdot) \bigr\Vert_{L_\xi^2} \lesssim \int_0^t \bigl\Vert F(s,x) \bigr\Vert_{L_x^{2,1}}\,\ud s \lesssim \int_0^t \varepsilon^2 \js^{-\frac{3}{2}+\delta}\,\ud s \lesssim \varepsilon^2 .
\end{equation*}

For the first term on the right-hand side of \eqref{eqn:proof_pxi_calJ2}, we again insert the definition \eqref{equ:setup_definition_wtilcalRulomega} of $\wtilcalR_{\ulomega}(s,\xi)$ and write 
\begin{equation*}
    \begin{split}
     \int_0^t e^{i\theta(s)}e^{is(\xi^2+\ulomega)} s \xi \wtilcalR_\ulomega(s,\xi)\,\ud s & =        \int_0^t e^{i\theta(s)}e^{is(\xi^2+\ulomega)} s \xi \wtilcalR_{q,\ulomega}(s,\xi)\,\ud s \\
     &\quad +      \int_0^t e^{i\theta(s)}e^{is(\xi^2+\ulomega)} s \xi (\dot{\gamma}(s)-\ulomega) \calL_{\baru,\ulomega}(s,\xi)\,\ud s \\
     &\quad +      \int_0^t e^{i\theta(s)}e^{is(\xi^2+\ulomega)} s \xi \wtilcalF_{+,\ulomega}[F(s,\cdot)](\xi)\,\ud s\\
     &=: \calI_{\mathrm{r}}^{2,1}(t,\xi)+\calI_{\mathrm{r}}^{2,2}(t,\xi)+\calI_{\mathrm{r}}^{2,3}(t,\xi).
    \end{split}
\end{equation*}
Observe that only $\wtilcalR_{q,\ulomega}(s,\xi)$ is tensorized in the sense that it is a sum of terms of the form $\frakc(s)\frakd(\xi)$ with $\frakc(s)$ decaying like $\varepsilon^2 \js^{-\frac{3}{2}+\delta}$ and $\frakd(\xi)$ is a rapidly decaying function. The local smoothing estimate \eqref{eqn: local_smoothing_dual_estimate_x}  then implies that
\begin{equation*}
\bigl\Vert \calI_{\mathrm{r}}^{2,1}(t,\cdot) \bigr\Vert_{L_\xi^2} \lesssim \Bigl\Vert \varepsilon^2 s \js^{-\frac{3}{2}+\delta} \Bigr\Vert_{L_s^2([0,t])} \lesssim \varepsilon^2 \jt^\delta.
\end{equation*}

Recalling the definition of $\calL_{\baru,\ulomega}$ from \eqref{equ:setup_definition_calL}, we have
\begin{equation*}
    \calI_{\mathrm{r}}^{2,2}(t,\xi) = 2\int_0^t \int_\bbR e^{i\theta(s)} e^{is(\xi^2+\ulomega)} e^{-ix\xi} \varphi_{\ulomega}(\xi)s(\dot{\gamma}(s)-\ulomega)\alpha_{\ulomega}(x) \baru(s,x)\,\ud x \,\ud s,
\end{equation*}
where $\varphi_{\ulomega}(\xi) := \frac{\xi }{(\vert \xi \vert + i \sqrt{\ulomega})^2} $ and $\alpha_{\ulomega}(x) := (2\pi)^{-\frac{1}{2}} \ulomega \sech^2(\sqrt{\ulomega}x) \in \calS(\bbR)$. Since $\vert \varphi_\ulomega(\xi) \vert \lesssim \vert \xi \vert^{\frac{1}{2}}$ and since $\alpha_\ulomega(x)$ is spatially localized, we may apply the local smoothing estimate \eqref{eqn:local_smoothing_dual_estimate} from Lemma~\ref{lemma:local_smoothing} and use the dispersive bound \eqref{equ:consequences_U_disp_decay} along with \eqref{equ:consequences_aux_bound_modulation2} to obtain 
\begin{equation*}
\begin{split}
\bigl\Vert \calI_{\mathrm{r}}^{2,2}(t,\cdot) \bigr\Vert_{L_\xi^2} &\lesssim \bigl\Vert s\cdot (\dot{\gamma}(s)-\ulomega)\alpha_{\ulomega}(x) \baru(s,x) \bigr\Vert_{L_x^1L_s^2([0,t])} \lesssim \bigl\Vert s \cdot(\dot\gamma(s) -\ulomega) \cdot \Vert \baru(s,\cdot) \Vert_{L_x^\infty} \bigr\Vert_{L_s^2([0,t])}\\
&\lesssim \Bigl\Vert s \cdot \varepsilon  \js^{-1+\delta} \cdot\varepsilon\js^{-\frac{1}{2}} \Bigr\Vert_{L_s^2([0,t])}   \lesssim \varepsilon^2 \jt^\delta.
\end{split}
\end{equation*}

For the contribution of the term $\calI_{\mathrm{r}}^{2,3}(t,\xi)$, we split the integral into a low frequency and a high frequency part by inserting a smooth cut-off $\chi_0(\xi)$ localized around the origin. We then apply the explicit definition of the distorted Fourier transform from \eqref{eqn:def-dFT} to write 
\begin{equation*}
    \begin{split}
\calI_{\mathrm{r}}^{2,3}(t,\xi) &=\int_0^t e^{i\theta(s)}e^{is(\xi^2+\ulomega)} s \xi \wtilcalF_{+,\ulomega}[F(s,\cdot)](\xi)\,\ud s \\
&=\int_0^t\int_\bbR e^{i\theta(s)}e^{is(\xi^2+\ulomega)}  s \xi \chi_0(\xi) e^{-ix\xi} F(s,x) \cdot \sigma_3 \overline{\frakm_{+,\ulomega}}(x,\xi)\,\ud x\,\ud s \\
&\quad + \int_0^t\int_\bbR e^{i\theta(s)}e^{is(\xi^2+\ulomega)}  s \xi \big(1-\chi_0(\xi)\big) e^{-ix\xi} F(s,x) \cdot \sigma_3 \overline{\frakm_{+,\ulomega}}(x,\xi)\,\ud x\,\ud s  \\
&=: \calI_{\mathrm{r},\mathrm{low}}^{2,3}(t,\xi)+\calI_{\mathrm{r},\mathrm{high}}^{2,3}(t,\xi),
    \end{split}
\end{equation*}
where we define $\frakm_{+,\ulomega} = \frac{1}{\sqrt{2\pi}}(m_{1,\ulomega},m_{2,\ulomega})^\top$ with $m_{1,\ulomega}$, $m_{2,\ulomega}$ given in \eqref{eqn:m-1,omega}, \eqref{eqn:m-2,omega} respectively. By Lemma~\ref{lemma: PDO on m12}, we find that 
\begin{equation}\label{eqn:proof_bound_frakm}
    \sup_{x,\xi \in \bbR} \vert e^{ix\xi}\frakm_{+,\ulomega}(x,\xi) \vert + \sup_{x,\xi \in \bbR} \vert e^{ix\xi}\px \frakm_{+,\ulomega}(x,\xi) \vert \lesssim_{\ulomega} 1.
\end{equation}
Since $\vert \xi \chi_0(\xi) \vert \lesssim \vert \xi \vert^{\frac{1}{2}}$ and since \eqref{eqn:proof_bound_frakm} holds, we can apply Lemma~\ref{lemma:local_smoothing} to the low frequency term. Hence, by using the local smoothing estimate \eqref{eqn:local_smoothing_dual_estimate} and the Cauchy-Schwarz inequality along with the estimates \eqref{equ:consequences_renormalized_quadratic}, \eqref{equ:consequences_ulPe_Mod_bounds}, \eqref{equ:consequences_calE1_bounds}, \eqref{equ:consequences_calE_2and3_bounds}, we obtain that 
\begin{equation*}
\big\Vert \calI_{\mathrm{r},\mathrm{low}}^{2,3}(t,\cdot) \big\Vert_{L_\xi^2} \lesssim \big\Vert s F(s,x) \big\Vert_{L_x^1L_s^2([0,t])}\lesssim \big\Vert \jx s F(s,x) \big\Vert_{L_x^2L_s^2([0,t])} \lesssim \varepsilon^2 \Big\Vert s \js^{-\frac{3}{2}+\delta}\Big\Vert_{L_s^2([0,t])} \lesssim \varepsilon^2 \jt^{\delta}.
\end{equation*}
For the high frequency part, we integrate by parts using $i\partial_x(e^{-ix\xi}) = \xi e^{-ix\xi}$ to write 
\begin{equation*}
    \begin{split}
\calI_{\mathrm{r},\mathrm{high}}^{2,3}(t,\xi) &= -i\int_0^t\int_\bbR e^{i\theta(s)}e^{is(\xi^2+\ulomega)}  s \big(1-\chi_0(\xi)\big) e^{-ix\xi} (\px F)(s,x) \cdot \sigma_3 \overline{\frakm_{+,\ulomega}}(x,\xi)\,\ud x\,\ud s     \\
&\quad -i\int_0^t\int_\bbR e^{i\theta(s)}e^{is(\xi^2+\ulomega)}  s \big(1-\chi_0(\xi)\big) e^{-ix\xi} F(s,x) \cdot \sigma_3 \overline{\px \frakm_{+,\ulomega}}(x,\xi)\,\ud x\,\ud s\\
&=: \calI_{\mathrm{r},\mathrm{high}}^{2,3,1}(t,\xi) + \calI_{\mathrm{r},\mathrm{high}}^{2,3,2}(t,\xi).
    \end{split}
\end{equation*}
Since $\vert 1-\chi_0(\xi) \vert \lesssim \vert \xi \vert^{\frac{1}{2}}$, the contribution of the second term $\calI_{\mathrm{r},\mathrm{high}}^{2,3,2}(t,\xi)$ can be handled similarly to $\calI_{\mathrm{r},\mathrm{low}}^{2,3}(t,\xi)$. For the first term $\calI_{\mathrm{r},\mathrm{high}}^{2,3,1}(t,\xi)$, we use the estimates \eqref{equ:consequences_px_Qr_L21bound}--\eqref{equ:consequences_px_calE23_L21bound} to obtain 
\begin{equation*}
\big\Vert \calI_{\mathrm{r},\mathrm{high}}^{2,3,1}(t,\cdot) \big\Vert_{L_\xi^2} \lesssim \big\Vert s \px F(s,x) \big\Vert_{L_x^1L_s^2([0,t])}\lesssim \big\Vert \jx s \px F(s,x) \big\Vert_{L_x^2L_s^2([0,t])} \lesssim \varepsilon^2 \Big\Vert s \js^{-\frac{3}{2}+\delta} \Big\Vert_{L_s^2([0,t])} \lesssim \varepsilon^2 \jt^{\delta}.
\end{equation*}
Thus, we conclude the discussion for the weighted energy contribution of the remainder term $\calJ_{\ulomega}^2(t,\xi)$.

\medskip
\noindent \underline{Contribution of the cubic term $\calJ_{\ulomega}^3(t,\xi)$:} In view of the cubic spectral distributions in Lemma~\ref{lemma:cubic_NSD}, we insert the decomposition \eqref{eqn: decomposition_cubic} for $\wtilcalF_{+, \ulomega}[\calC\big((\ulPe U)(s)\big)]$ and we write  
\begin{equation*}
	\calJ_{\ulomega}^3(t,\xi) = \calJ_{\ulomega}^{3,\delta_0}(t,\xi)+\calJ_{\ulomega}^{3,\pvdots}(t,\xi)+\calJ_{\ulomega}^{3,\mathrm{reg}}(t,\xi),
\end{equation*}
where the Dirac delta and principal value contributions are given respectively by 
\begin{align*}
\calJ_{\ulomega}^{3,\delta_0}(t,\xi) &:= 	\frac{i}{2\pi} \int_0^t e^{i\theta(s)}\iint  e^{2is(\xi-\xi_1)(\xi_1-\xi_2)} \tilfplusulo(s,\xi_1) \overline{\tilfplusulo}(s,\xi_2)\tilfplusulo(s,\xi-\xi_1+\xi_2)\\
&\qquad\qquad\qquad\qquad\qquad\qquad\qquad\qquad\qquad\quad \times \frac{\frakp_1\left(\tfrac{\xi}{\sqrt{\ulomega}},\tfrac{\xi_1}{\sqrt{\ulomega}},\tfrac{\xi_2}{\sqrt{\ulomega}},\tfrac{\xi-\xi_1+\xi_2}{\sqrt{\ulomega}}\right)}{\frakp\left(\tfrac{\xi}{\sqrt{\ulomega}},\tfrac{\xi_1}{\sqrt{\ulomega}},\tfrac{\xi_2}{\sqrt{\ulomega}},\tfrac{\xi-\xi_1+\xi_2}{\sqrt{\ulomega}}\right)}\,\ud\xi_1 \,\ud \xi_2 \,\ud s,\\
\calJ_{\ulomega}^{3,\pvdots}(t,\xi) &:=			\frac{i}{4\pi\sqrt{\ulomega}} \int_0^t e^{i\theta(s)}\iiint e^{is\Phi(\xi,\xi_1,\xi_2,\xi_4)} \tilfplusulo(s,\xi_1) \overline{\tilfplusulo}(s,\xi_2)\tilfplusulo(s,\xi-\xi_1+\xi_2-\xi_4)\\
&\qquad\qquad\qquad\qquad  \times \frac{\frakp_2\left(\tfrac{\xi}{\sqrt{\ulomega}},\tfrac{\xi_1}{\sqrt{\ulomega}},\tfrac{\xi_2}{\sqrt{\ulomega}},\tfrac{\xi - \xi_1 + \xi_2 - \xi_4}{\sqrt{\ulomega}}\right)}{\frakp\left(\tfrac{\xi}{\sqrt{\ulomega}},\tfrac{\xi_1}{\sqrt{\ulomega}},\tfrac{\xi_2}{\sqrt{\ulomega}},\tfrac{\xi - \xi_1 + \xi_2 - \xi_4}{\sqrt{\ulomega}}\right)} \pvdots \cosech\left(\frac{\pi \xi_4}{2\sqrt{\ulomega}}\right)\,\ud\xi_1 \,\ud \xi_2 \,\ud\xi_4 \,\ud s,
\end{align*}
with $\frakp,\frakp_1,\frakp_2$ defined in \eqref{eqn: cubic-frakp}--\eqref{eqn: cubic-frakp2}, and 
\begin{equation}\label{eqn:proof_pv_phase}
\Phi(\xi,\xi_1,\xi_2,\xi_4) := \xi^2 - \xi_1^2 + \xi_2^2 - (\xi-\xi_1+\xi_2-\xi_4)^2.	
\end{equation}
The regular contribution $\calJ_{\ulomega}^{3,\mathrm{reg}}(t,\xi)$ is a linear combination of terms of the form 
\begin{equation}\label{eqn: proof_calI_reg}
\int_0^t  e^{i\theta(s)} e^{is(\xi^2+\ulomega)} \frakb_0(\xi) \widehat{\calF}[H(s,\cdot)](\xi) \,\ud s,
\end{equation}
where
\begin{equation}
	H(s,x) := \varphi(x)w_{1,\ulomega}(s,x)w_{2,\ulomega}(s,x)w_{3,\ulomega}(s,x)
\end{equation}
with $\varphi \in \calS(\bbR)$ a Schwartz function and $w_{j,\ulomega}(s,x)$, $1 \leq j \leq 3$, given by 
\begin{equation*}
	\int_\bbR e^{ix \eta}e^{\mp is(\eta^2+\ulomega)}\frakb_j(\eta)\tilf_{\pm,\ulomega}(s,\eta)\,\ud \eta,
\end{equation*}
or complex conjugates thereof, and where the symbols $\frakb_j(\xi) \in W^{1,\infty}(\bbR)$ are given by \eqref{eqn: symbol_frakb}. We observe that the functions $w_{j,\ulomega}(s,x)$ enjoy the free Schr\"odinger decay estimates \eqref{equ:preparation_flat_Schrodinger_wave_bound1}--\eqref{equ:preparation_flat_Schrodinger_wave_bound5}.

\medskip
\noindent\underline{Contribution of the regular cubic term $\calJ_{\ulomega}^{3,\mathrm{reg}}(t,\xi)$:} We proceed with the weighted energy estimates of the regular cubic contribution term, treating only the generic term \eqref{eqn: proof_calI_reg}. By taking a derivative in $\xi$ in \eqref{eqn: proof_calI_reg}, we obtain the terms  
\begin{equation}\label{eqn: proof_pxi_I}
	\begin{split}
&\int_0^t  (2is\xi)e^{i\theta(s)}e^{is(\xi^2+\ulomega)} \frakb_0(\xi)  \widehat{\calF}[H(s,\cdot)](\xi) \,\ud s
+\int_0^t e^{i\theta(s)}e^{is(\xi^2+\ulomega)} (\pxi\frakb_0)(\xi) \widehat{\calF}[H(s,\cdot)](\xi) \,\ud s\\
&+ \int_0^t e^{i\theta(s)}e^{is(\xi^2+\ulomega)} \frakb_0(\xi) \partial_\xi \widehat{\calF}[H(s,\cdot)](\xi)\,\ud s\\
&=: \calI_{\mathrm{reg}}^{1}(t,\xi)+\calI_{\mathrm{reg}}^{2}(t,\xi)+\calI_{\mathrm{reg}}^{3}(t,\xi).
	\end{split}
\end{equation}
The second and third terms in \eqref{eqn: proof_pxi_I} can be bounded crudely. By using the free Schr\"odinger evolution bounds \eqref{equ:preparation_flat_Schrodinger_wave_bound1} on $w_{j,\ulomega}$, $1\leq j \leq 3$, and the fact that $\varphi$ is Schwartz, we obtain the bound 
\begin{equation*}
\begin{split}
\big\Vert \widehat{\calF}[H(s,\cdot)](\xi) \big\Vert_{H_\xi^1} \lesssim \big\Vert \jx H(s,x) \big\Vert_{L_x^2}  \lesssim \Vert w_{1,\ulomega}(s,x)\Vert_{L_x^\infty}\Vert w_{2,\ulomega}(s,x)\Vert_{L_x^\infty}\Vert w_{3,\ulomega}(s,x)\Vert_{L_x^\infty} \lesssim \varepsilon^3 \js^{-\frac{3}{2}}.
	\end{split}
\end{equation*}
Since the multiplier $\frakb_0$ is in $W^{1,\infty}(\bbR)$, we obtain that
\begin{equation*}
	\begin{split}
\big\Vert \calI_{\mathrm{reg}}^{2}(t,\cdot) \big\Vert_{L_\xi^2}+\big\Vert \calI_{\mathrm{reg}}^{3}(t,\cdot) \big\Vert_{L_\xi^2}
\lesssim \Vert \frakb_0 \Vert_{W^{1,\infty}} \int_0^t \big\Vert \widehat{\calF}[H(s,\cdot)](\xi)\big\Vert_{H_\xi^1} \,\ud s \lesssim \varepsilon^3 \int_0^t \js^{-\frac{3}{2}} \,\ud s \lesssim \varepsilon^3.
	\end{split}
\end{equation*} 
For the first term on the right hand side of \eqref{eqn: proof_pxi_I}, we split the integral into low and high frequency parts by writing
\begin{equation*}
\begin{split}
\calI_{\mathrm{reg}}^{1}(t,\xi) 		&= \int_0^t  (2is\xi)\chi_0(\xi)e^{i\theta(s)}e^{is(\xi^2+\ulomega)} \frakb_0(\xi)  \widehat{\calF}[H(s,\cdot)](\xi) \,\ud s \\
&\quad + \int_0^t  (2is\xi)\big(1-\chi_0(\xi)\big)e^{i\theta(s)}e^{is(\xi^2+\ulomega)} \frakb_0(\xi)  \widehat{\calF}[H(s,\cdot)](\xi) \,\ud s\\
&=: \calI_{\mathrm{reg}}^{\mathrm{low}}(t,\xi) + \calI_{\mathrm{reg}}^{\mathrm{high}}(t,\xi),
\end{split}
\end{equation*}
with a cutoff function $\chi_0(\xi)$ defined as before. By applying the local smoothing estimate \eqref{eqn:local_smoothing_dual_estimate} from Lemma~\ref{lemma:local_smoothing}, we find that 
\begin{equation*}
	\begin{split}
\big\Vert \calI_{\mathrm{reg}}^{\mathrm{low}}(t,\cdot) \big\Vert_{L_\xi^2} \lesssim \big\Vert s  H(s,x) \big\Vert_{L_x^1L_s^2([0,t])}	\lesssim \big\Vert \jx s H(s,x)\big\Vert_{L_x^2L_s^2{[0,t]}} \lesssim \varepsilon^3 \Big\Vert s \js^{-\frac{3}{2}}\Big\Vert_{L_s^2([0,t])} \lesssim \varepsilon^3 \big(\log(\jt)\big)^{\frac{1}{2}}.
	\end{split}
\end{equation*}
For the high frequency part, we integrate by parts using $i\xi e^{-ix\xi} = -\partial_x (e^{-ix\xi})$ to write 
\begin{equation*}
	\begin{split}
\calI_{\mathrm{reg}}^{\mathrm{high}}(t,\xi) = \int_0^t 2s\big(1-\chi_0(\xi)\big)e^{i\theta(s)}e^{is(\xi^2+\ulomega)} \frakb_0(\xi)\int_\bbR e^{-ix\xi}(\partial_x H)(s,x)\,\ud x \,\ud s.
	\end{split}
\end{equation*}
Since we have 
\begin{equation*}
\left	\vert \big(1-\chi_0(\xi)\big) \frakb_0(\xi) \right\vert \lesssim \vert \xi \vert^{\frac{1}{2}},
\end{equation*}
we may apply the local smoothing estimate \eqref{eqn:local_smoothing_dual_estimate} from Lemma~\ref{lemma:local_smoothing}, in combination with the free evolution bounds \eqref{equ:preparation_flat_Schrodinger_wave_bound1}, \eqref{equ:preparation_flat_Schrodinger_wave_bound3} and the fact that $\varphi(x)$ is Schwartz to obtain that
\begin{equation}\label{eqn: proof_reg_cubic}
\big\Vert\calI_{\mathrm{reg}}^{\mathrm{high}}(t,\cdot) \big\Vert_{L_\xi^2} \lesssim \big\Vert s (\partial_x H)(s,x)\big\Vert_{L_x^1L_s^2([0,t])}\lesssim \big\Vert s \jx(\partial_x H)(s,x)\big\Vert_{L_x^2L_s^2([0,t])} \lesssim \varepsilon^3 \jt^\delta.
\end{equation}
This finishes the discussion of the regular cubic contribution.

\medskip
\noindent\underline{Contribution of the singular cubic term $\calJ_{\ulomega}^{3,\delta_0}(t,\xi)$ with a Dirac kernel:}
We make a change of variables from ($\xi_1,\xi_2$) to $(\sigma,\eta)$ in $\calJ_{\ulomega}^{3,\delta_0}(t,\xi)$ by setting 
\begin{equation}\label{eqn: proof-cubic-cov}
	\sigma \mapsto \xi - \xi_1, \quad \eta \mapsto \xi_1 - \xi_2,
\end{equation}
and we obtain that 
\begin{equation*}
\begin{split}
    \calJ_{\ulomega}^{3,\delta_0}(t,\xi) := 	\frac{i}{2\pi} \int_0^t e^{i\theta(s)} \iint  e^{2is \eta \sigma} \tilfplusulo(t,\xi - \sigma) &\overline{\tilfplusulo}(t,\xi - \sigma - \eta)\tilfplusulo(t,\xi-\eta)\\
    & \times \frac{\frakp_1\left(\tfrac{\xi}{\sqrt{\ulomega}},\tfrac{\xi - \sigma}{\sqrt{\ulomega}},\tfrac{\xi-\sigma-\eta}{\sqrt{\ulomega}},\tfrac{\xi-\eta}{\sqrt{\ulomega}}\right)}{\frakp\left(\tfrac{\xi}{\sqrt{\ulomega}},\tfrac{\xi - \sigma}{\sqrt{\ulomega}},\tfrac{\xi-\sigma-\eta}{\sqrt{\ulomega}},\tfrac{\xi-\eta}{\sqrt{\ulomega}}\right)}\,\ud\sigma \,\ud \eta \,\ud s.
\end{split}
\end{equation*}
Observe that the variable $\xi$ is now absent from the phase $e^{2is\eta \sigma}$ after this change of variables. In view of Item~(1) of  Lemma~\ref{lemma:cubic_NSD}, the term $\calJ_{\ulomega}^{3,\delta_0}(t,\xi)$ is as a linear combination of terms of the form 
\begin{equation}\label{eqn:proof-dirac-generic}
	\int_0^t e^{i\theta(s)}\overline{\frakb_0}(\xi)\iint e^{2is\eta \sigma} g_1(s,\xi-\sigma) \overline{g_2}(s,\xi-\sigma-\eta) g_3(s,\xi-\eta) \,\ud \sigma \,\ud \eta \, \ud s,
\end{equation}
where $g_j(s,\xi_j) := \frakb_j(\xi_j)\tilfplusulo(s,\xi_j)$ and where $\frakb_j(\xi_j)$, $0\leq j \leq 3$, are multipliers in $W^{1,\infty}(\bbR)$. Taking a derivative in $\xi$ in $\eqref{eqn:proof-dirac-generic}$, we obtain the terms 
\begin{equation*}
	\begin{split}
&\int_0^t e^{i\theta(s)}\overline{(\partial_\xi\frakb_0)}(\xi)\iint e^{2is\eta \sigma} g_1(s,\xi-\sigma) \overline{g_2}(s,\xi-\sigma-\eta) g_3(s,\xi-\eta) \,\ud \sigma \,\ud \eta \, \ud s\\
&\quad+\int_0^t e^{i\theta(s)}\overline{\frakb_0}(\xi)\iint e^{2is\eta \sigma} (\partial_\xi g_1)(s,\xi-\sigma) \overline{g_2}(s,\xi-\sigma-\eta) g_3(s,\xi-\eta) \,\ud \sigma \,\ud \eta \, \ud s\\
&\quad+\int_0^t e^{i\theta(s)}\overline{\frakb_0}(\xi)\iint e^{2is\eta \sigma} g_1(s,\xi-\sigma) \overline{(\partial_\xi g_2)}(s,\xi-\sigma-\eta) g_3(s,\xi-\eta) \,\ud \sigma \,\ud \eta \, \ud s\\
&\quad+\int_0^t e^{i\theta(s)}\overline{\frakb_0}(\xi)\iint e^{2is\eta \sigma} g_1(s,\xi-\sigma) \overline{g_2}(s,\xi-\sigma-\eta)(\partial_\xi g_3)(s,\xi-\eta) \,\ud \sigma \,\ud \eta \, \ud s\\
&=: \calI_{\delta_0}^1(t,\xi)  + \calI_{\delta_0}^2(t,\xi) + \calI_{\delta_0}^3(t,\xi) + \calI_{\delta_0}^4(t,\xi).
	\end{split}
\end{equation*}
In each of the terms above, we undo the change of variables from \eqref{eqn: proof-cubic-cov} and express the resulting double integral (in $\xi_1$ and $\xi_2$) as a convolution. 
Hence, we write 
\begin{equation*}
\begin{split}
\calI_{\delta_0}^1(t,\xi) &= \int_0^t e^{i\theta(s)}e^{is\xi^2}\overline{(\partial_\xi\frakb_0)}(\xi)\big(\wtilw_1(s,\cdot) \ast \overline{\wtilw_2}(s,\cdot) \ast \wtilw_3(s,\cdot) \big)(\xi) \,\ud s,\\
\calI_{\delta_0}^2(t,\xi) &= \int_0^t e^{i\theta(s)}e^{is\xi^2}\overline{\frakb_0}(\xi)\big(\wtilw_0(s,\cdot) \ast \overline{\wtilw_2}(s,\cdot) \ast \wtilw_3(s,\cdot) \big)(\xi) \,\ud s,			
\end{split}
\end{equation*}
where we define $\wtilw_0(s,\xi_1) := e^{-is\xi_1^2}\partial_{\xi_1}(\frakb_1(\xi_1)\tilfplusulo(s,\xi_1))$ and $\wtilw_j(s,\xi_j) = e^{-is\xi_j^2}\frakb_j(\xi_j)\tilfplusulo(s,\xi_j) $, $1 \leq j \leq 3$. We omit writing out the expressions for $\calI_{\delta_0}^3(t,\xi)$ and $\calI_{\delta_0}^4(t,\xi)$, which are similar to $\calI_{\delta_0}^2(t,\xi)$. Using the free Schr\"odinger evolution bounds \eqref{equ:preparation_flat_Schrodinger_wave_bound1}, \eqref{equ:preparation_flat_Schrodinger_wave_bound_L2}, we find that 
\begin{equation}\label{eqn: proof-we-disp-bound}
\Big\Vert \widehat{\calF}^{-1}[\wtilw_j(s,\cdot)](x)\Big\Vert_{L_x^\infty} \lesssim \varepsilon \js^{-\frac{1}{2}}, \quad 1 \leq j \leq 3,
\end{equation}
and
\begin{equation}\label{eqn: proof-we-L2-bound}
\Big\Vert \widehat{\calF}^{-1}[\wtilw_j(s,\cdot)](x) \Big\Vert_{L_x^2}   \lesssim \varepsilon , \quad 1 \leq j \leq 3.
\end{equation}
Moreover, by \eqref{equ:prop_profile_bounds_assumption2} and \eqref{equ:consequences_sobolev_bound_profile}, we have 
\begin{equation}\label{eqn: proof-we-H1-bound}
\Big\Vert \widehat{\calF}^{-1}[\wtilw_0(s,\cdot)](x) \Big\Vert_{L_x^2}   \lesssim \big\Vert \tilfplusulo(s,\cdot) \big\Vert_{H_\xi^1}\lesssim   \varepsilon \js^\delta.
\end{equation}

Then using the convolution property of the standard Fourier transform, the Plancherel's identity, and \eqref{eqn: proof-we-disp-bound}, \eqref{eqn: proof-we-L2-bound}, \eqref{eqn: proof-we-H1-bound}, we obtain 
\begin{equation*}
\begin{split}
\big\Vert \calI_{\delta_0}^1(t,\cdot) \big\Vert_{L_\xi^2} &\lesssim \int_0^t \big\Vert \big(\wtilw_1(s,\cdot) \ast \overline{\wtilw_2}(s,\cdot) \ast \wtilw_3(s,\cdot) \big)(\xi) \big\Vert_{L_\xi^2} \,\ud s\\
&\lesssim \int_0^t \Big\Vert \widehat{\calF}^{-1}[\wtilw_1(s,\cdot) \ast \overline{\wtilw_2}(s,\cdot) \ast \wtilw_3(s,\cdot)](x) \Big\Vert_{L_x^2} \,\ud s\\
&\lesssim \int_0^t \Big\Vert \widehat{\calF}^{-1}[\wtilw_1(s,\cdot)] \Big\Vert_{L_x^2} \cdot \Big\Vert \widehat{\calF}^{-1}[\overline{\wtilw_2}(s,\cdot)] \Big\Vert_{L_x^\infty}\cdot \Big\Vert \widehat{\calF}^{-1}[\wtilw_3(s,\cdot)]\Big\Vert_{L_x^\infty} \,\ud s\\
&\lesssim \varepsilon^3 \int_0^t \js^{-1} \,\ud s \lesssim \varepsilon^3 \log(\jt)\lesssim\varepsilon^3\jt^{\delta},
\end{split}
\end{equation*}
and 
\begin{equation*}
\begin{split}
\big\Vert \calI_{\delta_0}^2(t,\cdot) \big\Vert_{L_\xi^2} \lesssim \int_0^t \Big\Vert \widehat{\calF}^{-1}[\wtilw_0(s,\cdot)]\Big\Vert_{L_x^2} \Big\Vert \widehat{\calF}^{-1}[\overline{\wtilw_2}(s,\cdot)] \Big\Vert_{L_x^\infty} \cdot \Big\Vert \widehat{\calF}^{-1}[\wtilw_3(s,\cdot)] \Big\Vert_{L_x^\infty} \,\ud s\lesssim \varepsilon^3\jt^{\delta}.
\end{split}
\end{equation*}
The weighted energy contributions for $\calI_{\delta_0}^3(t,\xi)$ and $\calI_{\delta_0}^4(t,\xi)$ are handled analogously and we omit the details. Thus, we conclude the weighted energy estimates for the Dirac delta contribution.

\medskip
\noindent\underline{Contribution of the principal value cubic term $\calJ_{\ulomega}^{3,\pvdots}(t,\xi)$:}
%
In view of Lemma~\ref{lemma:cubic_NSD}, the term $\calJ_{\ulomega}^{3,\pvdots}(t,\xi)$ is a linear combination of terms of the form 
\begin{equation}\label{eqn:proof-pv-generic}
	\begin{split}
\int_0^t e^{i\theta(s)}\overline{\fraka_0}(\xi) \int_{\bbR^3} e^{is\Phi(\xi,\xi_1,\xi_2,\xi_4)}g_1(s,\xi_1)\overline{g_2}(s,\xi_2)&g_3(s,\xi-\xi_1+\xi_2-\xi_4) \\
		&\times  \widehat{\calF}[\tanh(\sqrt{\ulomega} \cdot)](\xi_4)\,\ud\xi_1\,\ud\xi_2\,\ud\xi_4\,\ud s,
	\end{split}
\end{equation}
where $\Phi$ is defined in \eqref{eqn:proof_pv_phase} and $g_j(s,\xi) := \fraka_j(\xi)\tilfplusulo(s,\xi)$ with symbols $\fraka_j \in W^{1,\infty}(\bbR)$ for $1 \leq j \leq 3$. Moreover, in view of Item~(3) of Lemma~\ref{lemma:cubic_NSD}, at least one of the symbols $\fraka_j(\xi)$, $0 \leq j \leq 3$, satisfies the vanishing property $\fraka_j(0) = 0$. By taking a derivative in $\xi$ in \eqref{eqn:proof-pv-generic}, we get 
\begin{equation*}
	\begin{split}
&\int_0^t e^{i\theta(s)}\overline{\fraka_0}(\xi)\int_{\bbR^3} (is\partial_\xi \Phi)e^{is\Phi(\xi,\xi_1,\xi_2,\xi_4)}g_1(s,\xi_1)\overline{g_2}(s,\xi_2)g_3(s,\xi-\xi_1+\xi_2-\xi_4) \\
&\qquad\qquad\qquad\qquad\qquad\qquad\qquad\qquad\qquad\qquad\qquad\quad\times \widehat{\calF}[\tanh(\sqrt{\ulomega} \cdot)](\xi_4)\,\ud\xi_1\,\ud\xi_2\,\ud\xi_4\,\ud s	\\
&\quad +\int_0^t e^{i\theta(s)}\overline{(\pxi\fraka_0)}(\xi) \int_{\bbR^3} e^{is\Phi(\xi,\xi_1,\xi_2,\xi_4)}g_1(s,\xi_1)\overline{g_2}(s,\xi_2)g_3(s,\xi-\xi_1+\xi_2-\xi_4) \\
&\qquad\qquad\qquad\qquad\qquad\qquad\qquad\qquad\qquad\qquad\qquad\quad\times\widehat{\calF}[\tanh(\sqrt{\ulomega} \cdot)](\xi_4)\,\ud\xi_1\,\ud\xi_2\,\ud\xi_4\,\ud s\\
&\quad +\int_0^t e^{i\theta(s)}\overline{\fraka_0}(\xi) \int_{\bbR^3} e^{is\Phi(\xi,\xi_1,\xi_2,\xi_4)} g_1(s,\xi_1)\overline{g_2}(s,\xi_2)(\pxi g_3)(s,\xi-\xi_1+\xi_2-\xi_4) \\
&\qquad\qquad\qquad\qquad\qquad\qquad\qquad\qquad\qquad\qquad\qquad\quad\times \widehat{\calF}[\tanh(\sqrt{\ulomega} \cdot)](\xi_4)\,\ud\xi_1\,\ud\xi_2\,\ud\xi_4\,\ud s\\
&=: \calI_{\pvdots}^1(t,\xi) +\calI_{\pvdots}^2(t,\xi) +\calI_{\pvdots}^3(t,\xi).
	\end{split}
\end{equation*}
By direct computation, we find the identity
\begin{equation*}
	\partial_\xi \Phi = -\partial_{\xi_1}\Phi - \partial_{\xi_2}\Phi + 2 \xi_4.
\end{equation*}
Inserting this identity into the term $\calI_{\pvdots}^1(t,\xi)$, and integrating by parts via the identity $\partial_{\xi_j}e^{is\Phi} = is(\partial_{\xi_j}\Phi)e^{is\Phi}$, we obtain 
\begin{equation*}
	\begin{split}
\calI_{\pvdots}^1(t,\xi) &=\int_0^t e^{i\theta(s)}\overline{\fraka_0}(\xi) \int_{\bbR^3} e^{is\Phi(\xi,\xi_1,\xi_2,\xi_4)}(\partial_{\xi_1} g_1)(s,\xi_1)\overline{g_2}(s,\xi_2)g_3(s,\xi-\xi_1+\xi_2-\xi_4) \\
&\qquad\qquad\qquad\qquad\qquad\qquad\qquad\qquad\qquad\qquad\qquad\quad\times  \widehat{\calF}[\tanh(\sqrt{\ulomega} \cdot)](\xi_4)\,\ud\xi_1\,\ud\xi_2\,\ud\xi_4\,\ud s	\\
&\quad + \int_0^t e^{i\theta(s)}\overline{\fraka_0}(\xi) \int_{\bbR^3} e^{is\Phi(\xi,\xi_1,\xi_2,\xi_4)} g_1(s,\xi_1)\overline{(\partial_{\xi_2}g_2)}(s,\xi_2)g_3(s,\xi-\xi_1+\xi_2-\xi_4) \\
&\qquad\qquad\qquad\qquad\qquad\qquad\qquad\qquad\qquad\qquad\qquad\quad\times  \widehat{\calF}[\tanh(\sqrt{\ulomega} \cdot)](\xi_4)\,\ud\xi_1\,\ud\xi_2\,\ud\xi_4\,\ud s	\\
&\quad + 2\int_0^t s \cdot e^{i\theta(s)}\overline{\fraka_0}(\xi) \int_{\bbR^3} e^{is\Phi(\xi,\xi_1,\xi_2,\xi_4)}  g_1(s,\xi_1)\overline{g_2}(s,\xi_2)g_3(s,\xi-\xi_1+\xi_2-\xi_4) \\
&\qquad\qquad\qquad\qquad\qquad\qquad\qquad\qquad\qquad\qquad\qquad\qquad\qquad\qquad\times \varphi(\xi_4)\,\ud\xi_1\,\ud\xi_2\,\ud\xi_4\,\ud s	\\
&=: \calI_{\pvdots}^{1,1}(t,\xi) + \calI_{\pvdots}^{1,2}(t,\xi) + \calI_{\pvdots}^{1,3}(t,\xi),
	\end{split}
\end{equation*}
where
\begin{equation*}
 \varphi(\xi_4) := i\xi_4 \widehat{\calF}[\tanh(\sqrt{\ulomega}\cdot)](\xi_4) = \sqrt{\ulomega}\widehat{\calF}[\sech^2(\sqrt{\ulomega}\cdot)](\xi_4) \in \calS(\bbR).   
\end{equation*}

In what follows, we focus on estimating the contributions of $\calI_{\pvdots}^{1,1}(t,\xi)$ and $\calI_{\pvdots}^{1,3}(t,\xi)$ because the strategy to handle the terms $\calI_{\pvdots}^{1,2}(t,\xi)$, $\calI_{\pvdots}^2(t,\xi)$ and $\calI_{\pvdots}^3(t,\xi)$ is similar to the treatment of $\calI_{\pvdots}^{1,1}(t,\xi)$.  By a change of variables (on $\xi_4$), we rewrite $\calI_{\pvdots}^{1,1}(t,\xi)$ as a convolution
\begin{equation*}
	\begin{split}
\calI_{\pvdots}^{1,1}(t,\xi) &= 	\int_0^t e^{i\theta(s)}\overline{\fraka_0}(\xi) \int_{\bbR^3} e^{is(\xi^2-\xi_1^2+\xi_2^2-\xi_3^2)}(\partial_{\xi_1} g_1)(s,\xi_1)\overline{g_2}(s,\xi_2)g_3(s,\xi_3) \\
&\qquad\qquad\qquad\qquad\qquad\qquad\qquad\qquad\qquad\times \widehat{\calF}[\tanh(\sqrt{\ulomega}\cdot)](\xi-\xi_1+\xi_2-\xi_3)\,\ud\xi_1\,\ud\xi_2\,\ud\xi_3\,\ud s\\
&= \int_0^t e^{i\theta(s)}e^{is\xi^2}\overline{\fraka_0}(\xi) \Big(\wtilw_0(s,\cdot)\ast \overline{\wtilw_2}(s,\cdot) \ast \wtilw_3(s,\cdot) \ast \widehat{\calF}[\tanh(\sqrt{\ulomega}\cdot)]\Big)(\xi) \,\ud s,
	\end{split}
\end{equation*}
with   $\wtilw_0(s,\xi_1) := e^{-is\xi_1^2}\partial_{\xi_1}(\fraka_1(\xi_1)\tilfplusulo(s,\xi_1))$ and $\wtilw_j(s,\xi_j) = e^{-is\xi_j^2}\fraka_j(\xi_j)\tilfplusulo(s,\xi_j) $, $1 \leq j \leq 3$. Then, using \eqref{eqn: proof-we-disp-bound} and \eqref{eqn: proof-we-H1-bound} along with Plancherel's identity, we find that
\begin{equation*}
	\begin{split}
\big\Vert \calI_{\pvdots}^{1,1}(t,\cdot) \big\Vert_{L_\xi^2} &\lesssim \int_0^t \left\Vert \widehat{\calF}^{-1}\left[\Big(\wtilw_0(s,\cdot)\ast \overline{\wtilw_2}(s,\cdot) \ast \wtilw_3(s,\cdot) \ast \widehat{\calF}\big[\tanh(\sqrt{\ulomega}\cdot)\big]\Big)\right](x)	\right\Vert_{L_x^2} \,\ud s\\
&\lesssim \int_0^t \Big\Vert \widehat{\calF}^{-1}[\wtilw_0(s,\cdot)]\Big\Vert_{L_x^2} \cdot \Big\Vert \widehat{\calF}^{-1}[\overline{\wtilw_2}(s,\cdot)]\Big\Vert_{L_x^\infty}\cdot \Big\Vert \widehat{\calF}^{-1}[\wtilw_3(s,\cdot)]\Big\Vert_{L_x^\infty} \cdot \Big\Vert \tanh(\sqrt{\ulomega}\cdot)\Big\Vert_{L_x^\infty} \,\ud s\\
&\lesssim \varepsilon^3 \int_0^t \js^{-1+\delta} \,\ud s \lesssim \varepsilon^3\jt^{\delta}.
	\end{split}
\end{equation*}

For the term $\calI_{\pvdots}^{1,3}(t,\xi)$, we either have $\fraka_0(0) = 0$ or $\fraka_j(0) = 0$ for some $j \in \{1,2,3\}$ due to the vanishing property aforementioned. In the former case, the analysis will proceed in the same manner as in the term \eqref{eqn: proof_reg_cubic} for the regular cubic contribution. For the latter case, we assume without loss of generality that $\fraka_1(0) = 0$. Again, by a change of variables (on $\xi_4$), we rewrite the term $\calI_{\pvdots}^{1,3}(t,\xi)$ as a convolution
\begin{equation*}
\calI_{\pvdots}^{1,3}(t,\xi) = 2 \int_0^t se^{i\theta(s)}\overline{\fraka_0}(\xi) \big(\wtilw_1(s,\cdot)\ast \wtilw_2(s,\cdot) \ast \wtilw_3(s,\cdot) \ast\varphi(\cdot) \big)(\xi) \, \ud s, 
\end{equation*}
with $\wtilw_j(s,\xi)$, $1 \leq j \leq 3$ defined as before. By a small modification to the proof of \eqref{equ:preparation_flat_Schrodinger_wave_bound2} using $\fraka_1(0)=0$, we find the improved local decay
\begin{equation*}
\left\Vert \jx^{-1} \widehat{\calF}^{-1}[\wtilw_1(s,\cdot)](x) \right\Vert_{L_x^2} \lesssim \varepsilon \js^{-1+\delta}.
\end{equation*}
Hence, using the preceding improved local decay bound and the dispersive bounds \eqref{eqn: proof-we-disp-bound} in combination with the spatial localization of $\varphi$ and Plancherel's identity, we find that 
\begin{equation*}
	\begin{split}
\big\Vert \calI_{\pvdots}^{1,3}(t,\cdot) \big\Vert_{L_\xi^2} &\lesssim \int_0^t \left\Vert s \cdot\widehat{\calF}^{-1}\left[\big(\wtilw_1(s,\cdot)\ast \overline{\wtilw_2}(s,\cdot) \ast \wtilw_3(s,\cdot) \ast \varphi(\cdot)\big)\right](x)	\right\Vert_{L_x^2} \,\ud s\\
&\lesssim \int_0^t s\cdot \Big\Vert \jx^{-1} \widehat{\calF}^{-1}[\wtilw_1(s,\cdot)](x)\Big\Vert_{L_x^2} \cdot \Big\Vert \widehat{\calF}^{-1}[\overline{\wtilw_2}(s,\cdot)](x)\Big\Vert_{L_x^\infty} \cdot \Big\Vert \widehat{\calF}^{-1}[\wtilw_3(s,\cdot)](x)\Big\Vert_{L_x^\infty} \,\ud s\\
&\lesssim \int_0^t s \cdot \varepsilon\js^{-1+\delta}\cdot \varepsilon^2 \js^{-1} \,\ud s \lesssim \varepsilon^3 \jt^{\delta}.
	\end{split}
\end{equation*}
Hence, we conclude the weighted energy estimates for the principal value contribution. 

\medskip
\noindent \underline{Conclusion:} Combining all of the preceding estimates, we find that uniformly for all $0\leq t \leq T$, 
 \begin{equation*}
	\bigl\| \pxi \tilf_{+, \ulomega}(t,\xi) \bigr\|_{L^2_\xi} \lesssim \varepsilon + \varepsilon^2 \jt^{\delta}.
\end{equation*}
This finishes the proof of Proposition~\ref{prop: weighted-energy-estimate}.
\end{proof}

\section{Pointwise Estimates for the Profile} \label{sec:pointwise_profile}
In this section we establish uniform-in-time pointwise bounds on the distorted Fourier transform of the profile.
\subsection{Stationary phase lemmas}
We begin with a few stationary phase lemmas that will be applied to the singular cubic terms in the renormalized evolution equation \eqref{equ:setup_evol_equ_renormalized_tilfplus} for the profile. The proof of the next lemma is essentially contained in Section~2 of \cite{KatPus11} but we state it for the reader's convenience.

\begin{lemma}\label{ref: lemma-stationary-phase-delta}
Let $g_1,g_2,g_3$ satisfy for all $t \in \bbR$
\begin{equation}\label{eqn: hypothesis_on_g_j}
    \Vert g_j(t) \Vert_{L_\xi^\infty} + \Vert \jxi g_j(t) \Vert_{L_\xi^2} + \jt^{-\delta}\Vert \partial_\xi g_j(t) \Vert_{L_\xi^2} \leq \tilde{\varepsilon}, \quad 1 \leq j \leq 3,
\end{equation}
for some $\tilde{\varepsilon} \in (0,1)$, and let 
\begin{equation}
    \calT[g_1,g_2,g_3](t,\xi) := \iint e^{2it\sigma\eta} g_1(t,\xi-\sigma) \overline{g_2}(t,\xi-\sigma-\eta) g_3(t,\xi-\eta)  \,\ud\sigma \,\ud \eta.
\end{equation}    
Then, we have for all $t \geq 1$ 
\begin{equation}
    \calT[g_1,g_2,g_3](t,\xi)  = \frac{\pi}{t}g_1(t,\xi)\overline{g_2}(t,\xi)g_3(t,\xi) + r(t,\xi),
\end{equation}
where the remainder term satisfies
\begin{equation}\label{eqn: remainder-delta-bound}
    \Vert r(t,\xi) \Vert_{L_\xi^\infty} \lesssim \tilde{\varepsilon}^3 \jt^{-\frac65+3\delta}.
\end{equation}
\end{lemma}
\begin{proof}
By using Plancherel's identity along with the identity $\widehat{\calF}[e^{2it\sigma\eta}](x,y) = \frac{1}{2t}e^{-i\frac{xy}{2t}}$, we find that
\begin{equation*}
\begin{split}
\calT[g_1,g_2,g_3](t,\xi) &= \frac{1}{2t} \iint e^{-i\frac{xy}{2t}} \widehat{\calF}_{\sigma,\eta}^{-1}[H(\sigma,\eta;t,\xi)](x,y) \,\ud x \,\ud y\\
&= \frac{\pi}{t}g_1(t,\xi)\overline{g_2}(t,\xi)g_3(t,\xi) + \frac{1}{2t} \iint \big(e^{-i\frac{xy}{2t}}-1\big) \widehat{\calF}_{\sigma,\eta}^{-1}[H(\sigma,\eta;t,\xi)](x,y) \,\ud x \,\ud y\\
&=:\frac{\pi}{t}g_1(t,\xi)\overline{g_2}(t,\xi)g_3(t,\xi) + r(t,\xi),
\end{split}
\end{equation*}
where the function $H$ is defined by
\begin{equation*}
    H(\sigma,\eta;t,\xi) := g_1(t,\xi-\sigma) \overline{g_2}(t,\xi-\sigma-\eta) g_3(t,\xi-\eta).
\end{equation*}
It remains to show that the remainder term $r(t,\xi)$ satisfies the asserted estimate \eqref{eqn: remainder-delta-bound}. By direct computation, we have
\begin{equation*}
\begin{split}
\widehat{\calF}_{\sigma,\eta}^{-1}[H(\sigma,\eta;t,\xi)](x,y) &= \frac{e^{i\xi(x+y)}}{\sqrt{2\pi}} \int_\bbR e^{-iz\xi} \whatg_1(t,z-x)\widehat{\overline{g}}_2(t,z) \whatg_3(t,z-y)\,\ud z.
\end{split}
\end{equation*}
Hence, using the elementary estimate
\begin{equation*}
\vert e^{-i\frac{xy}{2t}}-1\vert \lesssim \vert t \vert^{-\frac{1}{5}} \jx^{\frac15}\jap{y}^{\frac15} \lesssim \vert t \vert^{-\frac15}\Big(\jap{x-z}^{\frac15}+\jap{z}^{\frac15}\Big)\Big(\jap{y-z}^{\frac15}+\jap{z}^{\frac15}\Big),    
\end{equation*}
and the bounds \eqref{eqn: hypothesis_on_g_j}, we conclude that 
\begin{equation*}
\begin{split}
    \Vert r(t,\xi)\Vert_{L_\xi^\infty} \lesssim \jt^{-\frac65} \Vert \jap{\cdot} \whatg_1\Vert_{L^2}\Vert \jap{\cdot} \whatg_2\Vert_{L^2}\Vert \jap{\cdot} \whatg_3\Vert_{L^2}\lesssim \tilde{\varepsilon}^3 \jt^{-\frac65+3\delta}.
\end{split}
\end{equation*}

\end{proof}

The next lemma determines the leading order behavior of the cubic interaction term with a Hilbert-type kernel. It is a special case of Lemma~5.1 in \cite{GermPusRou18}, and we give a streamlined proof for this special case. 

\begin{lemma}\label{ref: lemma-stationary-phase-pv}
Let $\varphi$ be an even Schwartz function, let $g_1,g_2,g_3$ satisfy for all $t \in \bbR$
\begin{equation}\label{eqn: hypothesis_on_g_j1}
\Vert g_j(t) \Vert_{L_\xi^\infty} + \Vert \jxi g_j(t) \Vert_{L_\xi^2} + \jt^{-\delta}\Vert \partial_\xi g_j(t) \Vert_{L_\xi^2} \leq \tilde{\varepsilon}, \quad 1 \leq j \leq 3,
\end{equation}
for some $\tilde{\varepsilon} \in (0,1)$, and let 
\begin{equation}
\calT [g_1,g_2,g_3](t,\xi) := \iiint e^{it\Phi(\xi,\xi_1,\xi_2,\xi_4)} g_1(t,\xi_1) \overline{g_2}(t,\xi_2) g_3(t,\xi-\xi_1+\xi_2-\xi_4) \, \pvdots \frac{\varphi(\xi_4)}{\xi_4}\,\ud \xi_1 \,\ud \xi_2 \,\ud \xi_4,
\end{equation}
where $\Phi(\xi,\xi_1,\xi_2,\xi_4) := \xi^2-\xi_1^2+\xi_2^2-(\xi-\xi_1+\xi_2-\xi_4)^2$. Then, we have for all $t \geq 1$ that
\begin{equation}
\calT[g_1,g_2,g_3](\xi) = \frac{\pi}{t} e^{it\xi^2} \int_\bbR e^{-it\gamma^2 }  g_1(t,\gamma) \overline{g_2}(t,\gamma) g_3(t,\gamma) \, \pvdots \frac{\varphi(\xi - \gamma)}{\xi - \gamma} \,\ud \gamma  + R(t,\xi),
\end{equation}
where the remainder term satisfies
\begin{equation}\label{eqn: remainder-pv-bound}
\Vert R(t,\xi) \Vert_{L_\xi^\infty} \lesssim \tilde{\varepsilon}^3 \jt^{-\frac65+4\delta}. 
\end{equation}
\end{lemma}
\begin{proof}
We apply a change of variables from $(\xi_1,\xi_2,\xi_4)$ to $(\sigma,\eta,\gamma)$ by setting 
\begin{equation}
\sigma \mapsto \xi-\xi_1-\xi_4, \quad \eta \mapsto \xi_1 - \xi_2, \quad \gamma \mapsto \xi - \xi_4.
\end{equation}
Then we have     
\begin{equation*}
\begin{split}
    &\calT[g_1,g_2,g_3](t,\xi) \\
    &= e^{it\xi^2} \int e^{-it\gamma^2 } \pvdots \frac{\varphi(\xi - \gamma)}{\xi - \gamma} \iint e^{2it\sigma\eta} g_1(t,\gamma-\sigma) \overline{g_2}(t,\gamma-\sigma-\eta) g_3(t,\gamma-\eta)\,\ud\sigma \,\ud \eta  \,\ud \gamma\\
    &= \frac{\pi}{t} e^{it\xi^2} \int_\bbR e^{-it\gamma^2 }  g_1(t,\gamma) \overline{g_2}(t,\gamma) g_3(t,\gamma) \pvdots \frac{\varphi(\xi - \gamma)}{\xi - \gamma} \,\ud \gamma + R(t,\xi),
\end{split}    
\end{equation*}
where
\begin{equation*}
R(t,\xi) := e^{it\xi^2} \int_\bbR e^{-it\gamma^2 }  \tilr (t,\gamma) \pvdots \frac{\varphi(\xi - \gamma)}{\xi - \gamma} \,\ud \gamma
\end{equation*}
with
\begin{equation}\label{eqn: proof-pointwise-tilr}
\tilr(t,\gamma) :=  \iint e^{2it\sigma\eta} g_1(t,\gamma-\sigma) \overline{g_2}(t,\gamma-\sigma-\eta) g_3(t,\gamma-\eta)\,\ud\sigma \,\ud \eta - \frac{\pi}{t}g_1(t,\gamma)\overline{g_2}(t,\gamma)g_3(t,\gamma).
\end{equation}
By the previous lemma we have the pointwise estimate
\begin{equation*}
\Vert \tilr(t,\cdot)\Vert_{L^\infty} \lesssim \tilde\varepsilon^3 \jt ^{-\frac65+3\delta}.
\end{equation*}

It remains to show that the remainder term $R(t,\xi)$ satisfies the asserted bound \eqref{eqn: remainder-pv-bound}. By \eqref{eqn: hypothesis_on_g_j} we have the crude bounds
\begin{equation}\label{eqn:proof-japtilr}
\Vert \jap{\gamma} \tilr(t,\gamma)\Vert_{L_\gamma^2} \lesssim \tilde\varepsilon^3
\end{equation}
and
\begin{equation}\label{eqn:proof-ptilr}
\Vert \partial_\gamma \tilr(t,\gamma)\Vert_{L_\gamma^2} \lesssim \tilde\varepsilon^3 \jt ^{\delta}.
\end{equation}
Indeed, by taking absolute values and making a change of variables in $\sigma$ for the expression of $\tilr(t,\gamma)$, we find that
\begin{equation*}
\begin{split}
\vert \tilr(t,\gamma) \vert \lesssim \iint \vert g_1(t,\sigma)\vert \vert \barg_2 (t,\sigma-\eta)\vert \vert g_3(\gamma-\eta)\vert  \,\ud \sigma \,\ud \eta +  \vert g_1(t,\gamma)\vert\vert g_2(t,\gamma)\vert\vert g_3(t,\gamma)\vert.
\end{split}
\end{equation*}
Using the bound
\begin{equation*}
\jap{\gamma} \lesssim \jap{\gamma - \eta} + \jap{\sigma - \eta} + \jap{\sigma},
\end{equation*}
and using Young's convolution inequality twice, we obtain that
\begin{equation*}
\begin{split}
\Vert \jap{\gamma}\tilr(t,\gamma)\Vert_{L^2}&\lesssim \Big\Vert \big(\vert g_1(t,\cdot)\vert  \ast \vert g_2(t,-\cdot)\vert \ast \vert \jap{\cdot}g_3(t,\cdot)\vert\big) (\gamma) \Big\Vert_{L_\gamma^2}+\Big\Vert \big(\vert g_1(t,\cdot)\vert  \ast \vert \jap{-\cdot} g_2(t,-\cdot)\vert \ast \vert g_3(t,\cdot)\vert\big) (\gamma) \Big\Vert_{L_\gamma^2}\\
&\quad+\Big\Vert \big(\vert \jap{\cdot}g_1(t,\cdot)\vert  \ast \vert g_2(t,-\cdot)\vert \ast \vert g_3(t,\cdot)\vert\big) (\gamma) \Big\Vert_{L_\gamma^2}+ \big\Vert \jap{\gamma}g_1(t,\gamma)g_2(t,\gamma)g_3(t,\gamma)\big\Vert_{L_\gamma^2}        \\
&\lesssim \Vert g_1 \Vert_{L^1}\Vert g_2 \Vert_{L^1}\Vert \jap{\cdot}g_3(\cdot)\Vert_{L^2} + \cdots + \Vert \jap{\cdot}g_1\Vert_{L^2} \Vert g_2 \Vert_{L^\infty} \Vert g_3 \Vert_{L^\infty}\\
&\lesssim \Vert \jap{\cdot}g_1\Vert_{L^2} \Vert \jap{\cdot} g_2 \Vert_{L^2} \Vert \jap{\cdot} g_3 \Vert_{L^2} + \Vert \jap{\cdot}g_1\Vert_{L^2} \Vert g_2 \Vert_{L^\infty} \Vert g_3 \Vert_{L^\infty} \lesssim \tilde{\varepsilon}^3.
\end{split}
\end{equation*}
This shows \eqref{eqn:proof-japtilr}, and the proof for \eqref{eqn:proof-ptilr} can be obtained similarly.

Next, we insert frequency cutoffs to rewrite the integral in the expression of $R(t,\xi)$ as 
\begin{equation*}
\begin{split}
    \int_\bbR e^{-it\gamma^2 }  \tilr (t,\gamma) \, \pvdots \frac{\varphi(\xi - \gamma)}{\xi - \gamma} \,\ud \gamma &=\int_\bbR e^{-it(\xi-\gamma)^2}\tilr(t,\xi-\gamma) \, \pvdots \frac{\varphi(\gamma)}{\gamma} \,\ud \gamma \\
    &= \int_\bbR e^{-it(\xi-\gamma)^2}\tilr(t,\xi-\gamma) \chi(t^{10}\gamma) \, \pvdots \frac{\varphi(\gamma)}{\gamma} \,\ud \gamma \\
    &\quad + \int_\bbR e^{-it(\xi-\gamma)^2}\tilr(t,\xi-\gamma) \big(1-\chi(t^{10}\gamma)\big) \frac{\varphi(\gamma)}{\gamma} \,\ud \gamma   \\
    &=: I(t,\xi) + II(t,\xi),
\end{split}
\end{equation*}
where $\chi(\gamma)$ is a non-negative even bump function supported on $[-2,2]$ with $\chi \equiv 1$ on $[-1,1]$. Since $\chi$ and $\varphi$ are even functions, we have 
\begin{equation*}
\int_\bbR e^{-it\xi^2}\tilr(t,\xi) \chi(t^{10} \gamma) \pvdots \frac{\varphi(\gamma)}{\gamma} \,\ud \gamma =e^{-it\xi^2}\tilr(t,\xi) \int_\bbR  \chi(t^{10} \gamma) \pvdots \frac{\varphi(\gamma)}{\gamma} \,\ud \gamma= 0.
\end{equation*}
Hence, we may rewrite the term $I(t,\xi)$ as 
\begin{equation*}
I(t,\xi) = \int_\bbR \big( H(t,\xi-\gamma)-H(t,\xi)\big) \chi(t^{10}\gamma) \pvdots \frac{\varphi(\gamma)}{\gamma} \,\ud \gamma,  
\end{equation*}
with
\begin{equation*}
H(t,\gamma) := e^{-it\gamma^2}\tilr(t,\gamma).
\end{equation*}
Using \eqref{eqn:proof-japtilr} and  \eqref{eqn:proof-ptilr} we find that
\begin{equation*}
\big\Vert \partial_\gamma H(t,\gamma) \big\Vert_{L_\gamma^2} \lesssim \vert t \vert \Vert \gamma \tilr(t,\gamma)\Vert_{L_\gamma^2} + \Vert \partial_\gamma \tilr(t,\gamma)\Vert_{L_\gamma^2} \lesssim \tilde{\varepsilon}^3 \jt.
\end{equation*}
Hence, by fundamental theorem of calculus, we have 
\begin{equation*}
\Vert I(t,\xi)\Vert_{L_\xi^\infty} \lesssim \int \vert \gamma \vert^{\frac12} \big\Vert \partial_\gamma H(t,\cdot)\big\Vert_{L^2} \chi(t^{10}\gamma) \frac{\vert \varphi(y) \vert}{\vert \gamma \vert} \,\ud \gamma \lesssim \tilde{\varepsilon}^3 \jt \int \vert \gamma \vert^{-\frac12}\chi(t^{10}\gamma)\vert \varphi(y) \vert \,\ud \gamma \lesssim \tilde{\varepsilon}^3 \jt^{-4}.    
\end{equation*}

Using \eqref{eqn: proof-pointwise-tilr}, the localization of $\varphi(\gamma)$ and H\"older's inequality, we obtain 
\begin{equation*}
\begin{split}
    \Vert II(t,\xi)  \Vert_{L_\xi^\infty} &\lesssim \Vert \tilr(t,\xi)\Vert_{L_\xi^\infty} \left \Vert \frac{1-\chi(t^{10} \gamma)}{\gamma}\varphi(\gamma)\right\Vert_{L_\gamma^1} \lesssim  \tilde{\varepsilon}^3 \jt^{-\frac65 + 3\delta}\cdot \log(\jt) \lesssim \tilde{\varepsilon}^3 \jt^{-\frac65+4\delta}.
\end{split}
\end{equation*}
Thus, we conclude the asserted estimate \eqref{eqn: remainder-pv-bound} by combining the preceding estimates.
\end{proof}

\subsection{Main pointwise estimates}
We now establish uniform-in-time pointwise estimates for the distorted Fourier transform of the profile.
\begin{proposition}\label{prop:pointwise_estimate}
 Suppose the assumptions in the statement of Proposition~\ref{prop:profile_bounds} are in place. Then we have uniformly for all $1 \leq t \leq T$ that
 \begin{equation}
  \bigl\| \tilf_{+, \ulomega}(t,\xi) \bigr\|_{L^\infty_\xi} + \bigl\| \tilf_{-, \ulomega}(t,\xi) \bigr\|_{L^\infty_\xi} \lesssim \Vert \tilfplusulo(1,\xi) \Vert_{L_\xi^\infty} + \varepsilon^2.
 \end{equation}
Moreover, we obtain for arbitrary times $1 \leq t_1 \leq t_2 \leq T$ that
\begin{align}
\Big\Vert \tilfplusulo(t_2,\xi)e^{i\theta(t_2)}e^{i\Theta_+(t_2,\xi)} - \tilfplusulo(t_1,\xi)e^{i\theta(t_1)}e^{i\Theta_+(t_1,\xi)} \Big\Vert_{L_\xi^\infty} &\lesssim \varepsilon^2 t_1^{-\frac{1}{10}+\delta}, \label{eqn: Cauchy-in-time-estimate}\\
\Big\Vert \tilfminusulo(t_2,\xi)e^{-i\theta(t_2)}e^{-i\Theta_-(t_2,\xi)} - \tilfminusulo(t_1,\xi)e^{-i\theta(t_1)}e^{-i\Theta_-(t_1,\xi)} \Big\Vert_{L_\xi^\infty} &\lesssim \varepsilon^2 t_1^{-\frac{1}{10}+\delta}, \label{eqn: Cauchy-in-time-estimate2}
\end{align}
where
\begin{equation}\label{eqn: integrating-factor-plus}
    \Theta_\pm(t,\xi) = \frac{1}{2}\int_1^t \frac{1}{s} \, \vert \tilf_{\pm,\ulomega}(s,\xi)\vert^2 \,\ud s.
\end{equation}
\end{proposition}
The main part of the proof of Proposition~\ref{prop:pointwise_estimate} is based on the differential equation \eqref{eqn: effective-ODE-profile} in the next lemma. 

\begin{lemma} \label{lemma: effective-ODE-profile}
Suppose the assumptions in the statement of Proposition~\ref{prop:profile_bounds} are in place. Assume $T \geq 1$. Then we have for all $\xi \in \bbR$ and for all $1 \leq t \leq T$ that
 \begin{equation}\label{eqn: effective-ODE-profile}
\partial_t \Big( e^{i\theta(t)}\big(\tilfplusulo(t,\xi)+\wtilB_{\ulomega}(t,\xi)\big) \Big) = e^{i\theta(t)}\Big(\frac{i}{2t} \vert \tilfplusulo(t,\xi)\vert^2 \tilfplusulo(t,\xi) + \wtilcalE_\ulomega(t,\xi)\Big),
 \end{equation}
 where the remainder term  
 \begin{equation}\label{eqn: def-wtilcalE}
 \begin{split}
\wtilcalE_\ulomega(t,\xi) &:= \Big(-\frac{i}{2t}\vert \tilfplusulo(t,\xi)\vert^2 \tilfplusulo(t,\xi) -i e^{it(\xi^2+\ulomega)}\wtilcalF_{+,\ulomega}\big[\calC\big((\ulPe U)(t)\big)\big](\xi) \Big) \\
&\qquad -ie^{it(\xi^2+\ulomega)}\wtilcalR_\ulomega(t,\xi) +i e^{i\theta(t)}(\dot{\gamma}(t)-\ulomega)\wtilB_{\ulomega}(t,\xi)
 \end{split}     
 \end{equation}
 satisfies
\begin{equation}\label{eqn: pointwise_estimate_wtilcalE}
\sup_{1 \leq t \leq T} \jt^{\frac{11}{10}-\delta} \big\Vert \wtilcalE_\ulomega(t,\xi) \big\Vert_{L_\xi^\infty} \lesssim \varepsilon^2.
\end{equation}
\end{lemma}

\begin{proof}[Proof of Proposition~\ref{prop:pointwise_estimate}]
 
 In view of \eqref{equ:setup_distFT_components_relation}, we have 
 \begin{equation*}
  \bigl\| \tilf_{-,\ulomega}(t,\xi) \bigr\|_{L^\infty_\xi} = \bigl\| \tilf_{+,\ulomega}(t,\xi) \bigr\|_{L^\infty_\xi}.
 \end{equation*}
 In what follows it therefore suffices to show for all $1 \leq t \leq T$ that
 \begin{equation*}
  \bigl\| \tilf_{+, \ulomega}(t,\xi) \bigr\|_{L^\infty_\xi} \lesssim \bigl\| \tilf_{+, \ulomega}(1,\xi) \bigr\|_{L^\infty_\xi} + \varepsilon^2.
 \end{equation*}
This bound follows from \eqref{eqn: Cauchy-in-time-estimate} evaluated at times $t_2 = t$ and $t_1=1$, and by triangle inequality. 

The asserted estimate \eqref{eqn: Cauchy-in-time-estimate} is an immediate consequence of Lemma~\ref{lemma: effective-ODE-profile}. We multiply the differential equation \eqref{eqn: effective-ODE-profile} with the integrating factor \eqref{eqn: integrating-factor-plus} to obtain 
\begin{equation}\label{eqn: proof-ode-with-integrating-factor}
    \partial_t \Big( e^{i\theta(t)}e^{i\Theta_+(t,\xi)}\big(\tilfplusulo(t,\xi)+\wtilB_{\ulomega}(t,\xi)\big) \Big) =  e^{i\theta(t)}e^{i\Theta_+(t,\xi)}\Big(\wtilcalE_\ulomega(t,\xi) - \frac{i}{2t}\vert \tilfplusulo(t,\xi)\vert^2\wtilB_{\ulomega}(t,\xi)\Big).
\end{equation}
Then, using the bootstrap assumption \eqref{equ:prop_profile_bounds_assumption2} and the bounds \eqref{eqn:pointwise_estimate_bilinear}, \eqref{eqn: pointwise_estimate_wtilcalE}, we find that the right-hand side of \eqref{eqn: proof-ode-with-integrating-factor} can be bounded in $L_\xi^\infty$ by $\varepsilon^2 \jt^{-\frac{11}{10}+\delta}$, and we obtain the asserted estimate \eqref{eqn: Cauchy-in-time-estimate} by integrating the above differential equation in time. 

Next, we look at the estimate for $\tilfminusulo(t,\xi)$. From \eqref{equ:setup_distFT_components_relation}, we find for any $\xi \in \bbR$ that 
\begin{equation*}
    \vert \tilfminusulo(t,\xi)\vert^2 = \vert \tilfplusulo(t,-\xi)\vert^2.
\end{equation*}
Hence, we compute from \eqref{equ:setup_distFT_components_relation} and \eqref{eqn: proof-ode-with-integrating-factor} that
\begin{equation*}
\begin{split}
&\partial_t \Big( e^{-i\theta(t)}e^{-i\Theta_+(t,-\xi)}\big(\tilfminusulo(t,\xi)- \tfrac{(\vert \xi \vert-i{\sqrt{\ulomega}})^2}{(\vert \xi \vert+i{\sqrt{\ulomega}})^2}\overline{\wtilB_{\ulomega}}(t,-\xi)\big) \Big) \\
&= - \tfrac{(\vert \xi \vert-i{\sqrt{\ulomega}})^2}{(\vert \xi \vert+i{\sqrt{\ulomega}})^2}e^{-i\theta(t)}e^{-i\Theta_+(t,-\xi)}\Big(\overline{\wtilcalE_\ulomega}(t,-\xi) + \frac{i}{2t}\vert \tilfminusulo(t,\xi)\vert^2 \overline{\wtilB_{\ulomega}}(t,-\xi)\Big).
\end{split}
\end{equation*}
Again, the right-hand side of the above differential equation can be bounded in $L_\xi^\infty$ by $\varepsilon^2 \jt^{-\frac{11}{10}+\delta}$. Since $\Theta_-(t,\xi) = \Theta_+(t,-\xi)$, we conclude the asserted estimate \eqref{eqn: Cauchy-in-time-estimate2} by integrating the above differential equation in time and using \eqref{eqn:pointwise_estimate_bilinear}.
\end{proof}

\begin{proof}[Proof of Lemma~\ref{lemma: effective-ODE-profile}]
The equation \eqref{eqn: effective-ODE-profile} is simply a rewritten version of the evolution equation for the renormalized profile  \eqref{equ:setup_evol_equ_renormalized_tilfplus}. 

The main work is now to show that the error term $\wtilcalE_\ulomega(t,\xi)$ defined in \eqref{eqn: def-wtilcalE} satisfies \eqref{eqn: pointwise_estimate_wtilcalE}. We may crudely bound the last two terms on the right-hand side of \eqref{eqn: def-wtilcalE} as follows. 
  
  We recall from \eqref{equ:setup_definition_wtilcalRulomega} that 
  \begin{equation*}
 \begin{aligned}
  \widetilde{\calR}_\ulomega(t,\xi) &= \widetilde{\calR}_{q,\ulomega}(t,\xi) + \wtilcalF_{+,\ulomega}\bigl[ \calQ_{\mathrm{r},\ulomega}\bigl((\ulPe U)(t)\bigr) \bigr](\xi) \\
  &\quad \, + (\dot{\gamma}(t) -\ulomega) \calL_{\baru, \ulomega}(t,\xi) + \wtilcalF_{+, \ulomega}\bigl[\calMod(t)\bigr](\xi) \\
  &\quad \, + \wtilcalF_{+, \ulomega}\bigl[\calE_1(t)\bigr](\xi) + \wtilcalF_{+, \ulomega}\bigl[\calE_2(t)\bigr](\xi) + \wtilcalF_{+,\ulomega}\bigl[\calE_3(t)\bigr](\xi).
 \end{aligned}
\end{equation*}
Using Sobolev embedding, the $L_x^{2,1}\rightarrow H_\xi^1$ boundedness of the distorted Fourier transform along with the bounds \eqref{eqn:pointwise_estimate_bilinear2}, \eqref{equ:consequences_aux_bound_modulation2}, \eqref{equ:consequences_calL_bounds}--\eqref{equ:consequences_calE_2and3_bounds}, we obtain that 
\begin{equation*}
\begin{split}
\big\Vert \widetilde{\calR}_\ulomega(t,\xi) \big\Vert_{L_{\xi}^\infty} &\lesssim \big\Vert \widetilde{\calR}_{q,\ulomega}(t,\xi) \big\Vert_{L_{\xi}^\infty} + \vert \dot{\gamma}(t) - \ulomega \vert \cdot \big\Vert \calL_{\baru,\ulomega}(t,\xi)\big\Vert_{H_\xi^1} + \big\Vert \calQ_{\mathrm{r},\ulomega}\bigl((\ulPe U)(t)\bigr) \big\Vert_{L_x^{2,1}}  \\
&\quad + \big\Vert \ulPe \calMod(t)\big\Vert_{L_x^{2,1}} + \big\Vert \calE_1(t) \big\Vert_{L_x^{2,1}}+ \big\Vert \calE_2(t) \big\Vert_{L_x^{2,1}} + \big\Vert \calE_3(t) \big\Vert_{L_x^{2,1}}\\
&\lesssim \varepsilon^2 \jt^{-\frac{3}{2}}.
\end{split}
\end{equation*}
Similarly, by \eqref{equ:consequences_aux_bound_modulation2} and \eqref{eqn:pointwise_estimate_bilinear}, we find that 
\begin{equation*}
\big\Vert (\dot{\gamma}(t) - \ulomega) \wtilB(t,\xi) \big\Vert_{L_\xi^\infty} \leq \vert \dot{\gamma}(t) - \ulomega\vert \cdot  \big\Vert \wtilB(t,\xi) \big\Vert_{L_\xi^\infty} \lesssim \varepsilon^3 \jt^{-2+\delta}.
\end{equation*}

For the cubic term in $\wtilcalE_\ulomega(t,\xi)$, we insert the decomposition from Lemma~\ref{lemma:cubic_NSD} to write 
\begin{equation*}
-ie^{it(\xi^2+\ulomega)}\wtilcalF_{+,\ulomega}[\calC\big((\ulPe U)(t)\big)](\xi) = \wtilcalJ_{\ulomega}^{\delta_0}(t,\xi) + \wtilcalJ_{\ulomega}^{\pvdots}(t,\xi) + \wtilcalJ_{\ulomega}^{\mathrm{reg}}(t,\xi),
\end{equation*}
where
\begin{align*}
\wtilcalJ_{\ulomega}^{\delta_0}(t,\xi) &:= 	\frac{i}{2\pi} \iint  e^{2it(\xi-\xi_1)(\xi_1-\xi_2)} \tilfplusulo(t,\xi_1) \overline{\tilfplusulo}(t,\xi_2)\tilfplusulo(t,\xi-\xi_1+\xi_2)\\
&\qquad\qquad\qquad\qquad\qquad\qquad\qquad\qquad\qquad\qquad
\quad\times \frac{\frakp_1\left(\tfrac{\xi}{\sqrt{\ulomega}},\tfrac{\xi_1}{\sqrt{\ulomega}},\tfrac{\xi_2}{\sqrt{\ulomega}},\tfrac{\xi-\xi_1+\xi_2}{\sqrt{\ulomega}}\right)}{\frakp\left(\tfrac{\xi}{\sqrt{\ulomega}},\tfrac{\xi_1}{\sqrt{\ulomega}},\tfrac{\xi_2}{\sqrt{\ulomega}},\tfrac{\xi-\xi_1+\xi_2}{\sqrt{\ulomega}}\right)}\,\ud\xi_1 \,\ud \xi_2,    \\
\wtilcalJ_{\ulomega}^{\pvdots}(t,\xi) &:=	\frac{i}{4\pi\sqrt{\ulomega}} \iiint  e^{it\Phi(\xi,\xi_1,\xi_2,\xi_4)} \tilfplusulo(t,\xi_1) \overline{\tilfplusulo}(t,\xi_2)\tilfplusulo(t,\xi-\xi_1+\xi_2-\xi_4)\\
&\qquad\qquad\qquad\qquad\qquad
\times \frac{\frakp_2\left(\tfrac{\xi}{\sqrt{\ulomega}},\tfrac{\xi_1}{\sqrt{\ulomega}},\tfrac{\xi_2}{\sqrt{\ulomega}},\tfrac{\xi - \xi_1 + \xi_2 - \xi_4}{\sqrt{\ulomega}}\right)}{\frakp\left(\tfrac{\xi}{\sqrt{\ulomega}},\tfrac{\xi_1}{\sqrt{\ulomega}},\tfrac{\xi_2}{\sqrt{\ulomega}},\tfrac{\xi - \xi_1 + \xi_2 - \xi_4}{\sqrt{\ulomega}}\right)} \pvdots \cosech\left(\frac{\pi \xi_4}{2\sqrt{\ulomega}}\right)\,\ud\xi_1 \,\ud \xi_2 \,\ud\xi_4,
\end{align*}
with $\frakp,\frakp_1,\frakp_2$ defined in \eqref{eqn: cubic-frakp}--\eqref{eqn: cubic-frakp2} and  $\Phi(\xi,\xi_1,\xi_2,\xi_4)$ given by \eqref{eqn:proof_pv_phase}. The regular cubic term $\wtilcalJ_{\ulomega}^{\mathrm{reg}}(t,\xi)$ is a linear combination of terms of the form 
\begin{equation}\label{eqn:typical-regular-cubic}
 e^{it(\xi^2+\ulomega)} \frakb_0(\xi)\int_\bbR e^{-ix \xi} \varphi(x)w_{1,\ulomega}(t,x)w_{2,\ulomega}(t,x)w_{3,\ulomega}(t,x)\,\ud x 
\end{equation}
with $\varphi \in \calS(\bbR)$ a Schwartz function, and inputs $w_{j,\ulomega}(t,x)$, $1 \leq j \leq 3$, given by 
\begin{equation*}
\int_\bbR e^{ix \eta}e^{\mp it(\eta^2+\ulomega)}\frakb_j(\eta)\tilf_{\pm,\ulomega}(t,\eta)\,\ud \eta,
\end{equation*}
or complex conjugates thereof, and where the symbols $\frakb_j(\xi) \in W^{1,\infty}(\bbR)$, $0 \leq j \leq 3$, are given by \eqref{eqn: symbol_frakb}. The functions $w_{j,\ulomega}(t,x)$, $1 \leq j \leq 3$, satisfy the free Schr\"odinger dispersive estimates \eqref{equ:preparation_flat_Schrodinger_wave_bound1}, and we can quickly infer from \eqref{eqn:typical-regular-cubic} and H\"older's inequality that 
\begin{equation*}
\big\Vert \wtilcalJ_{\ulomega}^{\mathrm{reg}}(t,\xi) \big\Vert_{L_\xi^\infty} \lesssim \Vert \varphi \Vert_{L_x^1} \Vert w_{1,\ulomega}(t)\Vert_{L_x^\infty}\Vert w_{2,\ulomega}(t)\Vert_{L_x^\infty}\Vert w_{3,\ulomega}(t)\Vert_{L_x^\infty}\lesssim \varepsilon^3 \jt^{-\frac{3}{2}}.
\end{equation*}

The pointwise estimates for $\wtilcalJ_{\ulomega}^{\delta_0}(t,\xi)$ and $\wtilcalJ_{\ulomega}^{\pvdots}(t,\xi)$ now occupy the rest of the proof. By stationary phase analysis, the leading order of these integrals is governed by the behavior at their diagonal, i.e., when all input frequencies are aligned at $\xi$. In view of Item~(2) of Lemma~\ref{lemma:cubic_NSD} and the preceding stationary phase lemmas, it is therefore natural to expect the following estimates 
\begin{align}
\left \Vert \wtilcalJ_{\ulomega}^{\delta_0}(t,\xi) -  \frac{i}{2t}\vert \tilfplusulo(t,\xi)\vert^2 \tilfplusulo(t,\xi)\right \Vert_{L_\xi^\infty} &\lesssim \varepsilon^3 \jt^{-\frac65+3\delta},\label{eqn:proof-stationary-phase-delta}\\
\left \Vert \wtilcalJ_{\ulomega}^{\pvdots}(t,\xi) \right \Vert_{L_\xi^\infty}&\lesssim \varepsilon^3 \jt^{-\frac{11}{10}+\delta}.\label{eqn:proof-stationary-phase-pv}
\end{align}

\noindent\underline{Proof of \eqref{eqn:proof-stationary-phase-delta}}: Using the change of variables \eqref{eqn: proof-cubic-cov}, we write 
\begin{equation*}
\wtilcalJ_{\ulomega}^{\delta_0}(t,\xi) = \frac{i}{2\pi} \iint e^{2it \eta \sigma} G_{\ulomega}^{\delta_0}(t,\xi,\sigma,\eta) \,\ud\sigma\,\ud\eta,
\end{equation*}
where
\begin{equation*}
G_{\ulomega}^{\delta_0}(t,\xi,\sigma,\eta):=    \tilfplusulo(t,\xi - \sigma) \overline{\tilfplusulo}(t,\xi -\sigma -\eta)\tilfplusulo(t,\xi - \eta) \frac{\frakp_1\left(\tfrac{\xi}{\sqrt{\ulomega}},\tfrac{\xi - \sigma}{\sqrt{\ulomega}},\tfrac{\xi - \sigma - \eta}{\sqrt{\ulomega}},\tfrac{\xi-\eta}{\sqrt{\ulomega}}\right)}{\frakp\left(\tfrac{\xi}{\sqrt{\ulomega}},\tfrac{\xi - \sigma}{\sqrt{\ulomega}},\tfrac{\xi - \sigma - \eta}{\sqrt{\ulomega}},\tfrac{\xi-\eta}{\sqrt{\ulomega}}\right)}.
\end{equation*}
By Item~(1) of Lemma~\ref{lemma:cubic_NSD}, we find that $G_{\ulomega}^{\delta_0}(t,\xi,\sigma,\eta)$ is a sum of tensorized terms of the form 
\begin{equation*}
    \overline{\frakb_0}(\xi)g_1(t,\xi-\sigma)\overline{g_2}(t,\xi-\sigma-\eta) g_3(t,\xi-\eta),
\end{equation*}
with inputs $g_j(t,\xi) := \frakb_j(\xi)\tilfplusulo(t,\xi)$ and multipliers $\frakb_j \in W^{1,\infty}(\bbR)$ for $0 \leq j \leq 3$. Since the inputs $g_j(t,\xi)$, $1\leq j \leq 3$, satisfy the assumption \eqref{eqn: hypothesis_on_g_j1} in view of the bootstrap assumption \eqref{equ:prop_profile_bounds_assumption2}, we may apply Lemma~\ref{ref: lemma-stationary-phase-delta} 
to conclude that
\begin{equation*}
    \wtilcalJ_{\ulomega}^{\delta_0}(t,\xi) = \frac{i}{2t}G_{\ulomega}^{\delta_0}(t,\xi,0,0) + r(t,\xi),
\end{equation*}
where the remainder satisfies $\Vert r(t,\xi)\Vert_{L_\xi^\infty} \lesssim \varepsilon^3 \jt^{-\frac65+3\delta}$. Using \eqref{eqn:cubic-diagonal-property} we conclude
\begin{equation*}
G_{\ulomega}^{\delta_0}(t,\xi,0,0)  = \vert \tilfplusulo(t,\xi)\vert^2 \tilfplusulo(t,\xi).
\end{equation*}
Hence, we have proven \eqref{eqn:proof-stationary-phase-delta}.

\noindent \underline{Proof of \eqref{eqn:proof-stationary-phase-pv}}: We write 
\begin{equation*}
\wtilcalJ_{\ulomega}^{\pvdots}(t,\xi) :=	\frac{i}{4\pi} \iiint  e^{it\Phi(\xi,\xi_1,\xi_2,\xi_4)}  G_\ulomega^{\pvdots}(t,\xi,\xi_1,\xi_2,\xi-\xi_1+\xi_2-\xi_4) \pvdots \frac{\varphi(\xi_4)}{\xi_4}\,\ud\xi_1 \,\ud \xi_2 \,\ud\xi_4,
\end{equation*}
where $\varphi(\xi) := \frac{\xi}{\sqrt{\ulomega}} \cosech\left(\frac{\pi \xi}{2\sqrt{\ulomega}}\right)$ is an even Schwartz function, and
\begin{equation*}
G_\ulomega^{\pvdots}(t,\xi,\xi_1,\xi_2,\xi_3) :=    \tilfplusulo(t,\xi_1) \overline{\tilfplusulo}(t,\xi_2)\tilfplusulo(t,\xi_3) \frac{\frakp_2\left(\tfrac{\xi}{\sqrt{\ulomega}},\tfrac{\xi_1}{\sqrt{\ulomega}},\tfrac{\xi_2}{\sqrt{\ulomega}},\tfrac{\xi_3}{\sqrt{\ulomega}}\right)}{\frakp\left(\tfrac{\xi}{\sqrt{\ulomega}},\tfrac{\xi_1}{\sqrt{\ulomega}},\tfrac{\xi_2}{\sqrt{\ulomega}},\tfrac{\xi_3}{\sqrt{\ulomega}}\right)}.
\end{equation*}
Again by Item~(1) of Lemma~\ref{lemma:cubic_NSD}, we find that $G_{\ulomega}^\pvdots(t,\xi,\xi_1,\xi_2,\xi_3)$ is a sum of tensorized terms of the form 
\begin{equation*}
    \overline{\frakb_0}(\xi)g_1(t,\xi_1)\overline{g_2}(t,\xi_2) g_3(t,\xi_3),
\end{equation*}
with inputs $g_j(t,\xi) := \frakb_j(\xi)\tilfplusulo(t,\xi)$ and multipliers $\frakb_j \in W^{1,\infty}(\bbR)$ for $0 \leq j \leq 3$. By applying Lemma~\ref{ref: lemma-stationary-phase-pv} to each term, we obtain that 
\begin{equation*}
\wtilcalJ_{\ulomega}^{\pvdots}(t,\xi) =    \frac{i}{4t} e^{it\xi^2} \int_\bbR e^{-it\gamma^2 }  G_{\ulomega}^\pvdots(t,\xi,\gamma,\gamma,\gamma)\pvdots \frac{\varphi(\xi - \gamma)}{\xi - \gamma} \,\ud \gamma  + R(t,\xi),
\end{equation*}
where the remainder satisfies $\Vert R(t,\xi)\Vert_{L_\xi^\infty} \lesssim \varepsilon^3 \jt^{-\frac65+4\delta}$.  For the leading term, we insert low- and high-frequency cut-offs centered at $\xi$ to write
\begin{equation*}
    \begin{split}
&\frac{i}{4t} e^{it\xi^2} \int_\bbR e^{-it\gamma^2 }  G_{\ulomega}^\pvdots(t,\xi,\gamma,\gamma,\gamma)\pvdots \frac{\varphi(\xi - \gamma)}{\xi - \gamma} \,\ud \gamma \\
&=\frac{i}{4t} e^{it\xi^2} \int_\bbR e^{-it\gamma^2 }  G_{\ulomega}^\pvdots(t,\xi,\gamma,\gamma,\gamma) \chi\big(t^{\mu}(\xi-\gamma)\big)\pvdots \frac{\varphi(\xi - \gamma)}{\xi - \gamma} \,\ud \gamma\\
&\quad +\frac{i}{4t} e^{it\xi^2} \int_\bbR e^{-it\gamma^2 }  G_{\ulomega}^\pvdots(t,\xi,\gamma,\gamma,\gamma) \bigg(\frac{1-\chi\big(t^{\mu}(\xi-\gamma)\big)}{\xi-\gamma}\bigg) \varphi(\xi - \gamma) \,\ud \gamma\\
&=: I(t,\xi)+ II(t,\xi),
    \end{split}
\end{equation*}
with $\chi(\cdot)$ as before and $\mu \in (0,1)$ a small positive constant to be determined. 

By \eqref{eqn:cubic-diagonal-property} we have  $G_{\ulomega}^\pvdots(t,\xi,\xi,\xi,\xi) = 0$ for every $\xi \in \bbR$. Using the  fundamental theorem of calculus and the bootstrap assumptions, we find that 
\begin{equation*}
    \vert G_{\ulomega}^\pvdots(t,\xi,\gamma,\gamma,\gamma) \vert \leq \vert \xi - \gamma \vert^{\frac12}\Vert \partial_\gamma G_{\ulomega}^\pvdots(t,\xi,\gamma,\gamma,\gamma) \Vert_{L_\gamma^2} \lesssim \varepsilon^3 \jt^{\delta}\vert \xi - \gamma \vert^{\frac12}.
\end{equation*}
By using this bound along with H\"older's inequality, we obtain that
\begin{equation*}
\begin{split}
\Vert I(t,\xi)\Vert_{L_\xi^\infty} \lesssim   \varepsilon^3 \jt^{-1+\delta} \left \Vert \frac{\chi(t^{\mu}\gamma)}{\vert \gamma \vert^{\frac{1}{2}}} \right \Vert_{L_\gamma^1} \Vert \varphi \Vert_{L_\gamma^\infty}  \lesssim \varepsilon^3 \jt^{-1-\frac{\mu}{2}+\delta}.
\end{split}
\end{equation*}

To estimate the term $II(t,\xi)$, we additionally insert low- and high-frequency cut-offs centered at the origin to write 
\begin{equation*}
\begin{split}
II(t,\xi) &= \frac{i}{4t} e^{it\xi^2} \int_\bbR e^{-it\gamma^2 }  G_{\ulomega}^\pvdots(t,\xi,\gamma,\gamma,\gamma) \chi(t^{\nu}\gamma) \bigg(\frac{1-\chi\big(t^{\mu}(\xi-\gamma)\big)}{\xi-\gamma}\bigg) \varphi(\xi - \gamma) \,\ud \gamma\\
&\quad +\frac{i}{4t} e^{it\xi^2} \int_\bbR e^{-it\gamma^2 }  G_{\ulomega}^\pvdots(t,\xi,\gamma,\gamma,\gamma) \big(1-\chi(t^{\nu}\gamma)\big) \bigg(\frac{1-\chi\big(t^{\mu}(\xi-\gamma)\big)}{\xi-\gamma}\bigg) \varphi(\xi - \gamma) \,\ud \gamma\\
&=: II_1(t,\xi) + II_2(t,\xi).
\end{split}
\end{equation*}
Here $\nu \in (0,1)$ is another small positive constant to be deteremined. For these terms, we will need the following basic bounds
\begin{equation}\label{eqn:proof-time-cutoffs-bound0}
 \Vert \chi(t^\mu\gamma)\Vert_{L_\gamma^p} \lesssim \jt^{-\mu/p}, \quad 1 \leq p < \infty,
\end{equation}
\begin{equation}\label{eqn:proof-time-cutoffs-bound}
\left \Vert \frac{1-\chi(t^{\mu}\gamma)}{\gamma} \right \Vert_{L_\gamma^\infty} \lesssim \jt^\mu, \quad \left \Vert \frac{1-\chi(t^{\mu}\gamma)}{\gamma} \right \Vert_{L_\gamma^2} \lesssim \jt^{\frac{\mu}{2}}, \quad \left \Vert \partial_\gamma \left(\frac{1-\chi(t^{\mu}\gamma)}{\gamma} \right) \right \Vert_{L_\gamma^2} \lesssim \jt^{\frac{3\mu}{2}},
\end{equation}
and
\begin{equation}\label{eqn:proof-Gpv-bound}
\big\Vert G_{\ulomega}^\pvdots(t,\xi,\gamma,\gamma,\gamma)\big\Vert_{L_{\xi}^\infty L_\gamma^\infty} \lesssim \varepsilon^3, \quad \big\Vert \partial_\gamma G_{\ulomega}^\pvdots(t,\xi,\gamma,\gamma,\gamma) \big\Vert_{L_\xi^\infty L_\gamma^2} \lesssim \varepsilon^3 \jt^\delta.
\end{equation}
By H\"older's inequality, we have 
\begin{equation*}
    \Vert II_1(t,\xi) \Vert_{L_\xi^\infty} \lesssim t^{-1} \Vert G_{\ulomega}^\pvdots(t,\xi,\gamma,\gamma,\gamma)\Vert_{L_{\xi}^\infty L_\gamma^\infty} \left \Vert \frac{1-\chi(t^{\mu}\gamma)}{\gamma} \right \Vert_{L_\gamma^\infty} \Vert \chi(t^\nu\gamma)\Vert_{L_\gamma^1} \lesssim \varepsilon \jt^{-1+\mu-\nu}.
\end{equation*}
For $II_2(t,\xi)$, we integrate by parts in $\gamma$ using $\partial_\gamma e^{-it\gamma^2} = -2it\gamma e^{-it\gamma^2}$ and write
\begin{equation*}
    \begin{split}
        II_2(t,\xi) = \frac{e^{it\xi^2}}{8t^2}\int_\bbR e^{-it\gamma^2 }  \partial_\gamma \bigg(G_{\ulomega}^\pvdots(t,\xi,\gamma,\gamma,\gamma) \bigg(\frac{1-\chi(t^{\nu}\gamma)}{\gamma}\bigg) \bigg(\frac{1-\chi\big(t^{\mu}(\xi-\gamma)\big)}{\xi-\gamma}\bigg) \varphi(\xi - \gamma) \bigg) \,\ud \gamma.
    \end{split}
\end{equation*}
Using H\"older's inequality along with the basic bounds \eqref{eqn:proof-time-cutoffs-bound}--\eqref{eqn:proof-Gpv-bound}, and the fact that $\varphi$ is Schwartz, we find that 
\begin{equation*}
\begin{split}
\Vert II_2(t,\xi) \Vert_{L_\xi^\infty} &\lesssim t^{-2}\big\Vert \partial_\gamma G_{\ulomega}^\pvdots(t,\xi,\gamma,\gamma,\gamma) \big\Vert_{L_\xi^\infty L_\gamma^2} \left \Vert \frac{1-\chi(t^{\mu}\gamma)}{\gamma} \right \Vert_{L_\gamma^\infty} \left \Vert \frac{1-\chi(t^{\nu}\gamma)}{\gamma} \right \Vert_{L_\gamma^\infty} \Vert \varphi\Vert_{L^2} \\
&\quad +t^{-2} \big\Vert G_{\ulomega}^\pvdots(t,\xi,\gamma,\gamma,\gamma)\big\Vert_{L_{\xi}^\infty L_\gamma^\infty} \left \Vert \partial_\gamma \left(\frac{1-\chi(t^{\mu}\gamma)}{\gamma} \right) \right \Vert_{L_\gamma^2} \left \Vert \frac{1-\chi(t^{\nu}\gamma)}{\gamma} \right \Vert_{L_\gamma^2}\Vert \varphi\Vert_{L^\infty}\\
&\quad + t^{-2} \big\Vert G_{\ulomega}^\pvdots(t,\xi,\gamma,\gamma,\gamma)\big\Vert_{L_{\xi}^\infty L_\gamma^\infty} \left \Vert \frac{1-\chi(t^{\mu}\gamma)}{\gamma} \right \Vert_{L_\gamma^2}\left \Vert \partial_\gamma \left(\frac{1-\chi(t^{\nu}\gamma)}{\gamma} \right) \right \Vert_{L_\gamma^2}\Vert \varphi\Vert_{L^\infty}\\ 
&\quad + t^{-2} \big\Vert G_{\ulomega}^\pvdots(t,\xi,\gamma,\gamma,\gamma)\big\Vert_{L_{\xi}^\infty L_\gamma^\infty}  \left \Vert \frac{1-\chi(t^{\mu}\gamma)}{\gamma} \right \Vert_{L_\gamma^2} \left \Vert \frac{1-\chi(t^{\nu}\gamma)}{\gamma} \right \Vert_{L_\gamma^2} \Vert \partial_\gamma \varphi\Vert_{L^\infty}\\
&\lesssim \varepsilon^3 \jt^{-2+\mu+\nu+\delta}+\varepsilon^3 \jt^{-2+\frac{3\mu}{2}+\frac{\nu}{2}}+\varepsilon^3 \jt^{-2+\frac{\mu}{2}+\frac{3\nu}{2}}+\varepsilon^3 \jt^{-2+\mu+\nu} \lesssim \varepsilon^3 \jt^{-2+\frac{3\mu}{2}+\frac{3\nu}{2}+\delta}.    
\end{split}
\end{equation*}
Thus, with the choices $\mu = \frac{1}{5}$ and $\nu = \frac{2}{5}$, we conclude that 
\begin{equation*}
    \Vert I(t,\xi) \Vert_{L_\xi^\infty} + \Vert II(t,\xi) \Vert_{L_\xi^\infty} \lesssim \varepsilon^3 \jt^{-\frac{11}{10}+\delta},
\end{equation*}
which implies the asserted estimate \eqref{eqn:proof-stationary-phase-pv}.

\noindent\underline{Conclusion:} By combining all of the preceding estimates, we conclude \eqref{eqn: pointwise_estimate_wtilcalE}.

\end{proof}

\appendix
\section{Fourier transforms of hyperbolic functions} \label{appendix-FT-formulas}

The main purpose of this appendix is to systematically derive explicit expressions for the Fourier transforms of the functions $\sech^\ell(x)$ and $\sech^\ell(x)\tanh(x)$ for $1 \leq \ell \leq 8$, which are needed in this work. 
To this end, we recall that the Fourier transforms of the functions $\sech(x)$ and $\sech^2(x)$ are given by 
\begin{align}
    \widehat{\calF}[\sech(x)](\xi) &= \sqrt{\frac{\pi}{2}} \sech\left(\frac{\pi\xi}{2}\right), \label{eqn: FT-sech} \\
    \widehat{\calF}[\sech^2(x)](\xi) &= \sqrt{\frac{\pi}{2}} \xi \cosech\left(\frac{\pi\xi}{2}\right). \label{eqn: FT-sechsech} 
\end{align}
See for instance \cite[Example 3.3]{SteinShakarchi_Complex} for a proof of \eqref{eqn: FT-sech} and \cite[Corollary 5.7]{LS1} for a proof of \eqref{eqn: FT-sechsech}.

\begin{lemma} \label{lemma:appendix-FT} 
For any integer $\ell \geq 1$, we have
\begin{equation} \label{eqn: FT-sech.2+l}
	\widehat{\calF}[\sech^{2+\ell}(x)](\xi) = \frac{\ell^2+\xi^2}{\ell(\ell+1)}\widehat{\calF}[\sech^\ell(x)](\xi),
\end{equation}
and
\begin{equation}\label{eqn: FT-sech.l-tanh}
	\widehat{\calF}[\sech^\ell(x)\tanh(x)](\xi) = \frac{\xi}{\ell i}\widehat{\calF}[\sech^\ell(x)](\xi).
\end{equation}
\end{lemma}

\begin{proof}
By elementary computations, we find for $\ell\geq1$ that
\begin{equation*}
	\sech^{2+\ell}(x) = \frac{\ell}{\ell+1}\sech^\ell(x) - \frac{1}{\ell(\ell+1)}\partial_x^2\big(\sech^\ell(x)\big),
\end{equation*}
as well as
\begin{equation*}
	\sech^\ell(x)\tanh(x) = -\frac{1}{\ell}\partial_x\big(\sech^\ell(x)\big).
\end{equation*}
Hence, using the Fourier transform property $\widehat{\calF}[\partial_x f](\xi) = i\xi \widehat{\calF}[f](\xi)$, we obtain \eqref{eqn: FT-sech.2+l} and \eqref{eqn: FT-sech.l-tanh}.
\end{proof}

Thus, starting from \eqref{eqn: FT-sech} and \eqref{eqn: FT-sechsech} we can recursively compute the Fourier transforms of higher powers of $\sech(x)$   using the identity \eqref{eqn: FT-sech.2+l}. This then enable us to determine formulas for $\widehat{\calF}[\sech^\ell(x)\tanh(x)]$ using \eqref{eqn: FT-sech.l-tanh}.

Proceeding in this manner, we find for odd integers $\ell$ that
\begin{align} 
\widehat{\calF}[\sech^3(x)](\xi) &= \frac{1+\xi^2}{2} \sqrt{\frac{\pi}{2}}\sech\left(\frac{\pi\xi}{2}\right), \label{equ:FT_sech3} \\
\widehat{\calF}[\sech^5(x)](\xi) &= \frac{ (1+\xi^2)(9+\xi^2)}{24} \sqrt{\frac{\pi}{2}} \sech\left(\frac{\pi\xi}{2}\right), \label{equ:FT_sech5} \\
\widehat{\calF}[\sech^7(x)](\xi) &= \frac{(1+\xi^2)(9+\xi^2)(25+\xi^2)}{720} \sqrt{\frac{\pi}{2}}  \sech\left(\frac{\pi\xi}{2}\right), \label{equ:FT_sech7}
\end{align}
as well as
\begin{align}
\widehat{\calF}[\sech(x)\tanh(x)](\xi) &= \frac{\xi}{i}  \sqrt{\frac{\pi}{2}} \sech\left(\frac{\pi\xi}{2}\right), \label{equ:FT_sech1tanh} \\
\widehat{\calF}[\sech^3(x)\tanh(x)](\xi) &= \frac{\xi(1+\xi^2)}{6i}  \sqrt{\frac{\pi}{2}} \sech\left(\frac{\pi\xi}{2}\right), \label{equ:FT_sech3tanh} \\
\widehat{\calF}[\sech^5(x)\tanh(x)](\xi) &= \frac{\xi(1+\xi^2)(9+\xi^2)}{120i}  \sqrt{\frac{\pi}{2}} \sech\left(\frac{\pi\xi}{2}\right), \label{equ:FT_sech5tanh} \\
\widehat{\calF}[\sech^7(x)\tanh(x)](\xi) &= \frac{\xi(1+\xi^2)(9+\xi^2)(25+\xi^2)}{5040i}  \sqrt{\frac{\pi}{2}} \sech\left(\frac{\pi\xi}{2}\right). \label{equ:FT_sech7tanh}
\end{align}
For even integers $\ell$, we conclude 
\begin{align}
\widehat{\calF}[\sech^4(x)](\xi) &= \frac{4+\xi^2}{6}\sqrt{\frac{\pi}{2}}  \xi \cosech\left(\frac{\pi\xi}{2}\right), \label{equ:FT_sech4} \\
\widehat{\calF}[\sech^6(x)](\xi) &= \frac{(4+\xi^2)(16+\xi^2) }{120}\sqrt{\frac{\pi}{2}} \xi \cosech\left(\frac{\pi\xi}{2}\right), \label{equ:FT_sech6}\\
\widehat{\calF}[\sech^8(x)](\xi) &= \frac{(4+\xi^2)(16+\xi^2)(36+\xi^2)}{5040}\sqrt{\frac{\pi}{2}} \xi \cosech\left(\frac{\pi\xi}{2}\right), \label{equ:FT_sech8}
\end{align}
and
\begin{align}
\widehat{\calF}[\sech^2(x)\tanh(x)](\xi) &= \frac{\xi}{2i} \sqrt{\frac{\pi}{2}} \xi \cosech\left(\frac{\pi\xi}{2}\right), \label{equ:FT_sech2tanh} \\
\widehat{\calF}[\sech^4(x)\tanh(x)](\xi) &= \frac{\xi(4+\xi^2)}{24i} \sqrt{\frac{\pi}{2}} \xi \cosech\left(\frac{\pi\xi}{2}\right), \label{equ:FT_sech4tanh}\\
\widehat{\calF}[\sech^6(x)\tanh(x)](\xi) &= \frac{\xi(4+\xi^2)(16+\xi^2)}{720i} \sqrt{\frac{\pi}{2}} \xi \cosech\left(\frac{\pi\xi}{2}\right), \label{equ:FT_sech6tanh} \\
\widehat{\calF}[\sech^8(x)\tanh(x)](\xi) &= \frac{\xi(4+\xi^2)(16+\xi^2)(36+\xi^2)}{40320i} \sqrt{\frac{\pi}{2}} \xi \cosech\left(\frac{\pi\xi}{2}\right). \label{equ:FT_sech8tanh}
\end{align}

\bibliographystyle{amsplain}
\bibliography{references}

\end{document}